\documentclass[a4paper]{amsart}

\usepackage[T1]{fontenc}
\usepackage[osf]{mathpazo}

\usepackage{eucal} 

\usepackage[utf8]{inputenc}

\makeatletter
\renewcommand \thepart {\@Roman\c@part}
\renewcommand\part{%
  \newpage\thispagestyle{plain}
  \if@noskipsec \leavevmode \fi
  \par
  \addvspace{4ex}%
  \@afterindentfalse
  \secdef\@part\@spart}

\def\@part[#1]#2{%
  \ifnum \c@secnumdepth >\m@ne
    \refstepcounter{part}%
    \addcontentsline{toc}{part}{\thepart\hspace{1em}#1}%
  \else
    \addcontentsline{toc}{part}{#1}%
  \fi
  {\parindent \z@ \raggedright
    \interlinepenalty \@M
    \normalfont
    \ifnum \c@secnumdepth >\m@ne
      \Large\bfseries \partname\nobreakspace\thepart
      \par\nobreak
    \fi
    \huge \bfseries #2%
    \par}%
  \nobreak
  \vskip 3ex
  \@afterheading}
\def\@spart#1{%
  \addcontentsline{toc}{part}{#1}%
  {\parindent \z@ \raggedright
    \interlinepenalty \@M
    \normalfont
    \huge \bfseries #1\par}%
  \nobreak
  \vskip 3ex
  \@afterheading}
\makeatother

\title[Extremal models and direct integrals in affine logic]{Extremal models and direct integrals\\ in affine logic}
\author{Itaï Ben Yaacov}
\address{
  Universit\'e Claude Bernard Lyon 1 \\
  Institut Camille Jordan \\
  Lyon \\
  France}
\urladdr{\url{https://math.univ-lyon1.fr/~begnac/}}

\author{Tomás Ibarlucía}
\address{Universit\'e Paris Cit\'e \\
  CNRS \\
  IMJ-PRG \\
  F-75006 Paris \\
  France.}
\urladdr{\url{https://webusers.imj-prg.fr/~tomas.ibarlucia}}

\author{Todor Tsankov}
\address{
  Universit\'e Claude Bernard Lyon 1 \\
  Institut Camille Jordan \\
  Lyon \\
  France}
\email{\url{tsankov@math.univ-lyon1.fr}}

\thanks{Research supported by ANR projects AGRUME (ANR-17-CE40-0026) and GeoMod (ANR-19-CE40-0022), and by the Institut Universitaire de France.}

\usepackage{hyperref}

\usepackage{stmaryrd}

\usepackage{my-macros}

\ifx\thmnum\undefined
  \newtheorem{thmnum}{}[section]
\fi

\usepackage{aliascnt}

\newcounter{quotethmcnt}

\newcommand{\mynewthm}[3][]{
  \newaliascnt{#2}{thmnum}
  \newtheorem{#2}[#2]{#3}
  \aliascntresetthe{#2}
  \newtheorem*{#2*}{#3}
  \expandafter\newcommand\csname #2autorefname\endcsname{#3}
  \expandafter\renewcommand\csname the#2\endcsname{\thethmnum}
}

\newcommand{\mynewthmdecons}[3][]{
  \newaliascnt{#2}{thmnum}
  \newtheorem{#2}[#2]{#3}
  \aliascntresetthe{#2}
  \newtheorem*{#2*}{#3}
  \expandafter\renewcommand\csname the#2\endcsname{\thethmnum}
}

\def\equationautorefname~#1\null{(#1)}
\def\itemautorefname~#1\null{#1}

\mynewthmdecons{theorem}{Theorem}

\mynewthm{cor}{Corollary}
\mynewthm{lemma}{Lemma}
\mynewthm{prop}{Proposition}
\mynewthm{fact}{Fact}
\mynewthm{conjecture}{Conjecture}

\theoremstyle{definition}
\mynewthm{defn}{Definition}
\mynewthm{ntn}{Notation}

\theoremstyle{remark}
\mynewthm{question}{Question}
\mynewthm{remark}{Remark}
\mynewthm{example}{Example}
\mynewthm{claim}{Claim}

\numberwithin{equation}{section}

\usepackage{enumitem}
\setlist[enumerate,1]{label=(\roman*), font=\normalfont}
\setlist[enumerate,2]{label=(\arabic*), font=\normalfont}

\usepackage{tikz,tikz-cd}

\newcommand{\fJ}{\mathfrak{J}}
\newcommand{\fM}{\mathfrak{M}}
\newcommand{\rest}{{\restriction}}

\newcommand{\fI}{\mathfrak{I}}
\newcommand{\half}{\tfrac{1}{2}}
\newcommand{\Los}{{\L}o{\'s}\xspace}
\newcommand{\density}{\mathfrak{d}}

\newcommand{\cco}{\overline{\mathrm{co}}}

\newcommand{\aff}{\mathrm{aff}}
\newcommand{\ext}{\mathrm{ext}}
\newcommand{\nat}{\textsc{\tiny{CR}}}
\newcommand{\eca}{\mathrm{eca}}
\newcommand{\Bau}{\mathrm{Bau}}
\newcommand{\Mor}{\mathrm{Mor}}
\newcommand{\Rand}{\mathrm{R}}
\newcommand{\Ra}{\mathrm{Ra}}

\newcommand{\CR}{\mathrm{CR}}
\newcommand{\cont}{\mathrm{cont}}
\newcommand{\qf}{\mathrm{qf}}
\newcommand{\APrA}{\mathrm{APrA}}
\newcommand{\PrA}{\mathrm{PrA}}
\newcommand{\RV}{\mathrm{RV}}

\newcommand{\HS}{\mathrm{HS}}
\newcommand{\AHS}{\mathrm{AHS}}

\newcommand{\Class}{{\mathrm{Class}}}
\newcommand{\ClassL}{{\mathrm{Class}_\cL}}
\newcommand{\PMP}{\mathrm{PMP}}
\newcommand{\FPMP}{\mathrm{FPMP}}

\newcommand{\scx}{\textsc{x}}

\newcommand{\fd}{\mathfrak{d}}
\newcommand{\fs}{\mathfrak{s}}

\newcommand{\op}{\mathrm{op}}

\DeclareMathOperator{\tR}{R}
\DeclareMathOperator{\Sub}{Sub}
\DeclareMathOperator{\Stab}{Stab}
\DeclareMathOperator{\Fix}{Fix}
\DeclareMathOperator{\IRS}{IRS}
\DeclareMathOperator{\ot}{ot}

\newcommand{\Cstar}{$\mathrm{C}^\ast$}
\newcommand{\TvN}{\mathrm{TvN}}

\renewcommand{\phi}{\varphi}
\renewcommand{\eps}{\varepsilon}

\begin{document}
\maketitle

\begin{abstract}
  Affine logic is a fragment of continuous logic, introduced by Bagheri, in which only \emph{affine} functions are allowed as connectives. This has the effect of endowing type spaces with the structure of compact convex sets. We study \emph{extremal models} of affine theories (those that only realize extreme types), and the ways and conditions under which all models can be described from the extremal ones.

  We introduce and develop the general theory of \emph{measurable fields} of metric structures and their \emph{direct integrals}.
  One of our main results is an extremal decomposition theorem for models of \emph{simplicial} theories, that is, affine theories whose type spaces form Choquet simplices.
  We prove that every model of a simplicial theory can be (uniquely) decomposed as a direct integral of extremal models.
  This generalizes known decomposition results (ergodic decomposition, tracial von Neumann factor decomposition), and moreover, holds without any separability hypothesis.

  Two extreme kinds of simplicial theories are \emph{Bauer theories}, whose extreme types form a closed set, and \emph{Poulsen theories}, whose extreme types form a dense set.
  We show that Keisler randomizations of continuous theories are, essentially, the same thing as affine Bauer theories.
  We establish a dichotomy result: a complete simplicial theory is either Bauer or Poulsen.

  As part of our analysis, we adapt many results and tools from continuous logic to the affine or extremal contexts (definability, saturation, type isolation, categoricity, etc.).
  We also provide a detailed study of the relations between continuous logic and affine logic.

  Finally, we present several examples of simplicial theories arising from theories in discrete logic, Hilbert spaces, probability measure-preserving systems, and tracial von Neumann algebras.
\end{abstract}

\newpage

\setcounter{tocdepth}{1}
\tableofcontents

\newpage

\addtocontents{toc}{\protect\vspace{5pt}}


\part*{Introduction}

Model theory studies the tripartite relationship binding mathematical structures, first-order formulas, and type spaces.
In classical logic, structures are discrete mathematical objects, and the relationship between formulas and types is given by the Stone duality between Boolean algebras and totally disconnected compact Hausdorff spaces.
Continuous logic is a generalization of classical logic that has been highly successful as a logical setting for real-valued (usually, complete metric) mathematical structures, such as those arising in functional analysis and related areas.
Here, the duality between formulas and types is the Gelfand duality (or its real-valued version) between commutative \Cstar-algebras and compact Hausdorff spaces, and formulas correspond to continuous functions on the compact type spaces.
Many ideas from classical logic can be adapted to continuous model theory, often with new features and some additional complexity.
However, a number of key notions, constructions, and phenomena from functional analysis or measure theory have no analogue in the discrete universe and new model-theoretic ideas are needed to incorporate them.

Not long after the beginnings of continuous logic, S.-M. Bagheri and his co-authors initiated the study of a natural fragment thereof, which they called \emph{linear logic} \cite{Bagheri2010, Bagheri2014a, Bagheri2014, Bagheri2021}.
In this fragment, the connectives are restricted to the \emph{affine} ones: addition, multiplication by real scalars, and the constant formula $\bOne$.
The continuous logic quantifiers, $\sup$ and $\inf$, remain unchanged.
Instead of a real algebra, formulas form an \emph{order unit space} (an ordered vector space with a distinguished element $\bOne$, satisfying certain axioms), and type spaces form \emph{compact convex sets}.
The relationship between formulas and types is governed in this case by another important (albeit maybe less well-known) duality, due to Kadison.
In particular, formulas can be identified with affine functions on appropriate compact convex sets, and for this reason we prefer to call this fragment \emph{affine logic} (terminology that has also been adopted in recent works of Bagheri \cite{Bagheri2024pp,Bagheri2024p}).

The significance of affine logic seems to have been overlooked by a large part of the model theory community.
In this paper, building on the work of Bagheri, we develop some of its fundamentally new features and show how the study of the affine fragment of continuous logic opens the door to a model-theoretic treatment of many aspects of continuous mathematical structures that were previously elusive. As a notable outcome, our work provides a unified framework and generalization of several integral decomposition constructions and results from functional analysis and ergodic theory.

\section*{Background and motivation}
The original motivation of Bagheri's paper \cite{Bagheri2010} was a construction analogous to that of ultraproducts, named \emph{ultramean} in \cite{Bagheri2014a,Bagheri2014}, where the ultrafilter is replaced by a finitely additive probability measure on the algebra of all subsets of the index set, and the value of the predicates is defined by integration.
Affine formulas are preserved by this construction, just as arbitrary continuous formulas are preserved by ultraproducts. In the light of this observation, an appropriate compactness theorem for this logic was established in \cite{Bagheri2014a, Bagheri2014}, and thereafter many affine counterparts of basic fundamental results of continuous logic were deduced. These papers also discussed two important examples of affine theories: the theory of probability algebras, and its expansion with an aperiodic automorphism. Later, in \cite{Bagheri2021}, Bagheri proved a version of the Keisler--Shelah isomorphism theorem for affine equivalence.

However, a crucial new feature of affine logic remained unexplored. As type spaces form compact convex sets, a distinguished class of types as well as a distinguished class of models appear: the \emph{extreme types} and the \emph{extremal models}.

A central role in the theory of compact convex sets is played by their extreme points and by various decomposition theorems that allow to represent any point in the set as the barycenter of a probability measure concentrated on the extreme points. Perhaps the earliest result of this kind is the Krein--Milman theorem which ensures the existence of extreme points and provides a decomposition theorem in terms of measures concentrated on the closure of the set of extreme points (see \cite[\textsection1]{Phelps2001}). The first general decomposition theorem for measures concentrated exactly on the set of extreme points was proved by Choquet, in the case where the compact convex set is metrizable, and it was later shown by Bishop and de Leeuw that the metrizability assumption can be dropped at the expense of some technical complications (as the set of extreme points may no longer be measurable). These decomposition results have proved to be of great importance in analysis and in ergodic theory.

The main goal of this paper is to analyze the structure of extreme types (extreme points of the type spaces), of extremal models of affine theories, i.e., those that only realize extreme types, and of the models that can decomposed into extremal components. This analysis involves a subtle interplay between compactness (as type spaces are compact) and the lack thereof: in general, the set of extreme types is not closed, and ultraproducts of extremal models are not necessarily extremal. This tension is one of the most interesting features of the theory.

The general theory of compact convex sets, which will of course be fundamental in our analysis, is enriched by the presence of quantifiers. A crucial basic instance of this is that if one restricts an extreme type in the variables $xy$ to the variable $x$, one obtains an extreme type. In other words, the variable restriction maps preserve extreme points, which is of course not true for general affine maps between compact convex sets.

Our work is also guided by the introduction of ideas of ergodic theory into model theory. Indeed, extremal models are a generalization of ergodic systems, and many intuitions and methods from this field work in our more general setting. Conversely, we expect that the development of affine logic will find applications in measured group theory, and in this paper, we lay some of the groundwork for this.

\section*{Summary of the main results}

The selected results discussed in the subsections below are to be found in Parts \ref{part:Direct-integrals}-\ref{part:Examples} of the present work. The beginning of the paper (\autoref{part:The-logic}) is devoted to a self-contained presentation of affine logic and the development of the general theory of extreme types and extremal models. Even though the results are mostly analogues to ones in continuous logic, the proofs often require new ideas. Some of the main points of \autoref{part:The-logic} are the following:
\begin{itemize}
\item We give an alternative proof of Bagheri's compactness theorem for affine logic, by a combination of the usual compactness theorem for continuous logic and the Hahn--Banach separation theorem.
\item We prove that extremal models always exist, and enjoy many desirable properties.
\item We obtain affine or extremal versions of several fundamental theorems of continuous logic, notably concerning definability, saturation, omitting types, and atomic models.
\end{itemize}

\subsection*{Direct integrals of metric structures}

A fundamental construction of affine logic is the \emph{convex combination} of several structures.
For example, given $M$ and $N$ we may define
\begin{gather*}
  K = \half M \oplus \half N,
\end{gather*}
whose elements are pairs in $M \times N$ and whose basic predicates are defined by $P^K(a,b) = \half P^M(a) + \half P^N(b)$.
It is easy to observe that affine formulas and affine theories are preserved by convex combinations of models. This construction can be seen as a particular case of Bagheri's ultrameans, but is in a sense orthogonal to the ultraproduct construction, which ultrameans also encompass. We generalize convex combinations in a different manner, by considering \emph{measurable fields} and \emph{direct integrals} of structures over measurable spaces equipped with $\sigma$-additive probability measures (which cannot be seen as a case of Bagheri's construction).

Measurable fields of Hilbert spaces and their direct integrals are important in functional analysis, representation theory of groups, and operator algebras. See, for instance, \cite{Nielsen1980}. With a different viewpoint and terminology (and in a dual form that is peculiar to its setting), direct integrals are also fundamental in ergodic theory, most notably via the ergodic decomposition theorem. Other constructions, such as fields of Polish groups, have also been considered.
In almost all cases appearing in the literature, only fields of separable structures (typically, over standard probability spaces) are studied, as the general non-separable case appears to be at odds with the fact that measures are only $\sigma$-additive.

We introduce in \autoref{sec:direct-integrals} a general notion of a measurable field of $\cL$-structures, for any given language $\cL$, with no restrictions on the size of the structures or the underlying probability space. From a measurable field $(M_\omega : \omega \in \Omega)$ based on a probability space $(\Omega, \mu)$, we can define an associated direct integral,
\begin{gather*}
  K = \int^\oplus_\Omega M_\omega \ud \mu(\omega),
\end{gather*}
which is again an $\cL$-structure. An appropriate version of \Los's theorem holds (see \ref{th:Los}).
\begin{theorem*}
  For every affine $\cL$-formula $\varphi$ and tuple $f\in K^n$, the function $\omega \mapsto \varphi^{M_\omega}\bigl(f(\omega) \bigr)$ is $\mu$-measurable and
  \begin{gather*}
    \varphi^K(f) = \int_\Omega \varphi^{M_\omega}\bigl(f(\omega) \bigr) \ud \mu(\omega).
  \end{gather*}
\end{theorem*}
\noindent Ensuring the measurability of the integrand in the non-separable case is the major technical obstacle, and the condition that makes our definition work is inherently model-theoretic.

Our construction of direct integrals generalizes several of the notions from the literature mentioned above. Also, Bagheri's ultramean can be recovered as a direct integral of a field of ultraproducts; we explain this in \autoref{sec:comp-with-bagheri}. In addition, as we comment further below, direct integrals allow for a streamlined description of the models of \emph{Keisler randomizations}.

Beyond the connections with the existing literature, direct integrals of measurable fields of models constitute one of our main tools for the study of affine theories. A key construction (see \autoref{sec:DirectIntegralExtremalModels}) shows how to produce direct integrals of extremal models of a given theory. The underlying probability space is obtained from a type space $\tS^\aff_\kappa(T)$ with an appropriate boundary measure (i.e., a measures concentrated on the extreme points), and the field of models arises naturally from the crucial observation that a measure on $\tS^\aff_\kappa(T)$ whose barycenter is the type of an enumeration of a model, must itself concentrate on types of enumerations of models. In this way, similarly to the Choquet representation of points of convex sets as barycenters of boundary measures, we have the following (see \ref{prop:DirectIntegralExtremalModelsConstruction}).

\begin{theorem*}
  Let $T$ be an affine theory and let $M$ be a model of $T$. Then there is a measurable field $(M_\omega:\omega\in\Omega)$ of extremal models of $T$ such that
  \begin{gather}\label{intro:eq:M-preceq-int}\tag{*}
    M\preceq^\aff \int_\Omega^\oplus M_\omega \ud\mu(\omega).
  \end{gather}
\end{theorem*}

The investigation of the conditions under which one can get equality in \autoref{intro:eq:M-preceq-int} is one of the main contributions of the paper, which we discuss in the next subsection.

Direct integrals are a fundamental device in the understanding of the connections between affine logic and continuous logic. For instance, a structure is a model of the affine part $Q_\aff$ of a continuous theory $Q$ if and only if it embeds affinely into a direct integral of models of $Q$ (see \ref{th:general-models-Qaff}). See also \ref{th:affine-classes} for a characterization of affinely axiomatizable continuous logic theories in terms of preservation under convex combinations and \emph{direct multiples}.

Finally, let us mention that our construction of direct integrals has been employed by Farah and Ghasemi in \cite{Farah2023p} to prove a version of the Feferman--Vaught theorem for direct integrals of fields of structures (in the separable setting), and thereby deduce a preservation result for elementary equivalence under tensor products of tracial von Neumann algebras. We will revisit their result and prove a generalization of one of their main theorems in \autoref{sec:continuity-direct-integral}.

\subsection*{Simplicial theories, extremal decomposition, and a Bauer--Poulsen dichotomy}

A compact convex set is a \emph{Choquet simplex} if every point can be represented as the barycenter of a \emph{unique} boundary measure.
We will call an affine theory \emph{simplicial} if all its type spaces $\tS^\aff_x(T)$ are Choquet simplices.
As we will show, many important examples of affine theories are indeed simplicial.

One of our main results (see \autoref{sec:ExtremalDecomposition}) is the existence and uniqueness of extremal decompositions for models of simplicial theories.

\begin{theorem*}
  Let $T$ be a simplicial affine theory. Then every model of $T$ is isomorphic to a direct integral of extremal models of $T$.

  Moreover, this extremal decomposition is unique, up to an appropriate notion of isomorphism of measurable fields.
\end{theorem*}

Applied to the theory of probability measure-preserving actions of a given countable group $\Gamma$, our theorem generalizes the ergodic decomposition theorem from $\Gamma$-systems on standard probability spaces to arbitrary ones (see \autoref{sec:PMP}). Similarly, it provides a generalization to the non-separable setting of the factor decomposition theorem of von Neumann algebras, in the tracial case (see \autoref{sec:tracial-von-neumann}). These applications require establishing that the corresponding affine theories $\PMP_\Gamma$ and $\TvN$ are simplicial, which we manage to do via a transfer principle from quantifier-free type spaces, despite the presumably very complicated nature of the full type spaces $\tS^\aff_x(\PMP_\Gamma)$ and $\tS^\aff_x(\TvN)$.

Among Choquet simplices, which can be quite diverse, two classes are of particular significance. On the one hand, \emph{Bauer simplices} are those for which the extreme points form a closed set. More concretely, a compact convex set is a Bauer simplex if and only if it is affinely homeomorphic to $\cM(X)$, the convex set of Radon probability measures on a compact space $X$. On the other hand, we have simplices whose extreme points form a dense set. In fact, in the metrizable case, there is only one non-trivial such, called the \emph{Poulsen simplex}. It is universal in the sense that any other metrizable simplex can be embedded as a face of the Poulsen simplex. As it turns out, most simplices that appear naturally in ergodic theory and operator algebras tend to be either Bauer or Poulsen.

We undertake a detailed analysis of the affine theories all of whose type spaces belong to one of these classes, i.e., \emph{Bauer theories} and \emph{Poulsen theories}. Their study provides significant insights into the connections between affine and continuous logic, and is carried out in \autoref{part:conn-with-cont}. Let us state here two important consequences of our analysis, proved in \autoref{sec:Poulsen}.

\begin{theorem*}
  Let $T$ be a non-degenerate, affinely complete, simplicial theory. Then $T$ is either a Bauer theory or a Poulsen theory.
\end{theorem*}

\begin{cor*}
  Let $T$ be a simplicial affine theory and let $M$ and $N$ be extremal models of $T$. If $M\equiv^\aff N$, then $M\equiv^\cont N$, that is, affine equivalence implies elementary equivalence in continuous logic.

  More generally, the affine theory of a model $M$ of a simplicial theory determines the distribution of the continuous theories of the models $M_\omega$ appearing in the extremal decomposition of $M$.
\end{cor*}

The last point proves and generalizes a statement conjectured by Farah and Ghasemi \cite{Farah2023p} in the context of tracial von Neumann algebras and continuous logic (which was also confirmed independently, in that setting, by Gao and Jekel \cite{Gao2024pp}).

Below we highlight some of our specific findings and examples concerning Bauer and Poulsen theories.

\subsection*{Bauer theories, and continuous theories with affine reduction} As the extreme types of a Bauer theory form a closed set, its extremal models are closed under ultraproducts and form an elementary class in continuous logic. The extremal structure of a Bauer theory $T$ is thus coded by the continuous logic theory $T_\ext$, defined as the common elementary theory of its extremal models.

On the other hand, let us say that a continuous theory $Q$ has \emph{affine reduction} if every continuous logic formula can be approximated arbitrarily well by affine formulas, modulo $Q$. We have a correspondence between Bauer affine theories and continuous theories with affine reduction (see \autoref{sec:Bauer-theories}).

\begin{theorem*}
  If $T$ is a Bauer theory, then $T_\ext$ has affine reduction. Conversely, if $Q$ is a continuous logic theory with affine reduction, then $Q_\aff$ is a Bauer theory, and the extremal models of $Q_\aff$ are precisely the models of $Q$.

  In either case, letting $Q=T_\ext$ (equivalently, $T=Q_\aff$), we have natural affine homeomorphisms
  \begin{gather*}
    \tS^\aff_x(T) \cong \cM(\tS^\cont_x(Q))
  \end{gather*}
  between the spaces of affine types of $T$ and the Bauer simplices of Radon probability measures on the continuous logic type spaces of $Q$.
\end{theorem*}

If a continuous theory $Q$ has affine reduction, our extremal decomposition theorem then provides a description of all models of the affine part $Q_\aff$: these are precisely the direct integrals of models of $Q$.

Similarly to quantifier elimination, affine reduction is a syntactical property which can be achieved by a definitional expansion of the language. Given a continuous theory $Q$, we let $Q_\Bau = (Q_\Mor)_\aff$ denote the affine part of the Morleyization of $Q$, which we call the \emph{Bauerization} of $Q$.

Theories of random variables taking values in the models of a given classical or continuous logic theory have already been studied in the literature, by the name of Keisler randomizations. They constitute an important source of examples of continuous logic theories, especially because of their preservation properties. We show that the Bauerization of a continuous theory $Q$ is bi-interpretable with the randomization of $Q$ in this sense. As a consequence, we obtain a novel and satisfying description of models of randomized theories. This is done in \autoref{sec:Randomizations}.

\begin{theorem*}
  Let $Q$ be a continuous logic theory. Then every model of the Keisler randomization of $Q$ is a direct integral of models of $Q$, in the appropriate language.
\end{theorem*}

A basic but fundamental example of a Bauer theory is the theory $\PrA$ of probability measure algebras. Indeed, probability algebras are the affine counterpart of the rather trivial structure, from the viewpoint of continuous logic, consisting of just two points. (For a precise statement, $\PrA$ is interdefinable with the Bauerization of the theory of two named points at distance one.) As observed by Bagheri from the very beginning, every affine theory with a non-trivial model $M$ has arbitrarily large models -- for instance, all direct multiples $L^1(\Omega,M)=\int^\oplus_\Omega M\ud\mu(\omega)$. On the other hand, an affine theory may have a bound on the size of its \emph{extremal} models. In \autoref{sec:Bauer-theories}, we prove that $\PrA$ as well as other similar examples, all having a unique, compact extremal model, are indeed Bauer theories, and moreover that they enjoy affine quantifier elimination.

In general, establishing that a concrete continuous theory has affine reduction, particularly via explicit, syntactical means, seems to be a difficult task. In \autoref{sec:criterion-affine-reduction}, we isolate a sufficient condition for affine reduction, in terms of the automorphism group of the direct multiple $L^1(\Omega,M)$, which can be applied to continuous theories having a separable, saturated model $M$. This applies in the following important example (see \autoref{sec:ex:HS}).

\begin{theorem*}
  The continuous theory $\HS_\infty$ of spheres of infinite-dimensional Hilbert spaces has affine reduction.
\end{theorem*}

Given the model-theoretic nature of our proof, we do not know, for instance, how to approximate the continuous logic formula $\ip{x, y}^3$ by affine formulas. If $\HS$ is the continuous logic theory of the spheres of arbitrary Hilbert spaces, our results also imply that the extremal models of $\HS_\aff$ are precisely the spheres of Hilbert spaces.

Discrete, classical logic theories can be seen as continuous theories such that in every model, all basic predicates take only the values $0$ or $1$. Our method also applies to prove the following (see \autoref{sec:ex:classical}).

\begin{theorem*}
  Let $Q$ be a classical, complete theory in a single-sorted language. Then $Q$ has affine reduction.
\end{theorem*}

In particular, for complete, single-sorted classical theories, the space of affine types $\tS^\aff_x(Q_\aff)$ can be identified with the space $\cM(\tS^\cont_x(Q))$ of \emph{Keisler measures}, which has been studied intensively in recent years.

Finally, as already mentioned before, the theory $\TvN$ of tracial von Neumann algebras is affine and simplicial. It is in fact a Bauer theory. Recall that a von Neumann algebra is called a \emph{factor} if it has trivial center.
\begin{theorem*}
  The theory of tracial von Neumann algebras is an affine Bauer theory, and its extremal models are precisely the factors.
\end{theorem*}

Tracial von Neumann algebras were axiomatized in continuous logic in \cite{Farah2014}. We give a self-contained presentation in \autoref{sec:tracial-von-neumann}.

\subsection*{The convex realization property and Poulsen theories}

There are several significant intermediate classes of formulas between affine and continuous logic, starting with \emph{convex} formulas, obtained as maximums of finite collections of affine formulas. A major role in the understanding of the connections between affine and continuous logic is played by those structures that are \emph{existentially closed for convex formulas}, i.e., they have approximate witnesses for any conjunction of affine conditions that can be realized in an affine extension. This property is enjoyed by any direct integral over an atomless probability space.

As it turns out, one can axiomatize in continuous logic the class of structures that are existentially closed for convex formulas. For this we identify an intrinsic elementary condition, which we call the \emph{convex realization property}, and which, roughly speaking, says that for every tuple $\bar a$ in a model and tuple $\bar \lambda$ of positive real numbers with $\sum_i \lambda_i = 1$, there is $b$ in the model which behaves approximately (that is, with respect to finitely many formulas and up to $\eps$) like the element $\bigoplus_i \lambda_i a_i$ in the convex self-combination $\bigoplus_i\lambda_i M$. \emph{A posteriori}, we show that every affine extension of structures with the convex realization property is elementary in continuous logic, which, for instance, implies the following (see \ref{th:DirectIntegralPreservesEquivalence}).

\begin{theorem*}
  Let $(M_\omega:\omega\in\Omega)$ and $(N_\omega:\omega\in\Omega)$ be measurable fields of structures over an atomless probability space. If $M_\omega\equiv^\aff N_\omega$ for almost every $\omega\in\Omega$, then
  \begin{gather*}
    \int_\Omega^\oplus M_\omega\ud\mu(\omega) \equiv^\cont \int_\Omega^\oplus N_\omega\ud\mu(\omega).
  \end{gather*}
\end{theorem*}

Given an affine theory $T$, the continuous theory $T_\CR$ obtained by adding the axiom scheme of the convex realization property can be seen as an \emph{atomless completion} of $T$ in continuous logic. It is indeed the common continuous theory of all atomless direct integrals of models of $T$, and we have the following (see \ref{th:Tnat}).

\begin{theorem*}
  Let $T$ be a complete theory in affine logic. Then $T_\CR$ is a complete continuous logic theory.
\end{theorem*}

In fact, one has the stronger fact that for any affine theory $T$, the affine part maps $\tS^\cont_x(T_\CR) \to \tS^\aff_x(T)$ are affine homeomorphisms.

In the case of simplicial theories, the convex realization property provides us with a characterization of complete Poulsen theories (see \ref{th:Poulsen-complete}).

\begin{theorem*}
  Let $T$ be an affinely complete, simplicial theory. The following are equivalent:
  \begin{enumerate}
  \item $T$ is a Poulsen theory.
  \item Every model of $T$ has the convex realization property.
  \item $T$ is complete as a continuous logic theory.
  \end{enumerate}
\end{theorem*}

The preceding results, and many consequences, are treated in Sections \ref{sec:convex-formulas}-\ref{sec:Poulsen}.

\subsection*{Measure-preserving systems}
Our main examples of Poulsen theories, as well as further interesting examples of Bauer theories, come from ergodic theory. Measure-preserving systems of a countable group $\Gamma$ can be studied in continuous logic by expanding the language and the theory of probability measure algebras with a family of automorphisms, representing the transformations induced by the elements of $\Gamma$. Several model-theoretic results in this framework have been obtained in the literature (see the references in \autoref{sec:PMP}), and also certain applications of model theory to studying weak containment in a strongly ergodic system (see \cite{IbarluciaTsankov}). Indeed, weak containment, strong ergodicity and several other notions of the field have natural interpretations in continuous logic.

On the other hand, an important shortcoming of the model-theoretic treatments of ergodic theory is that \emph{ergodicity} in general cannot be captured by standard notions in continuous logic. Affine logic provides an elegant solution to this problem (see \ref{th:PMPG-simplicial}).

\begin{theorem*}
  Let $\Gamma$ be a countable group and let $\PMP_\Gamma$ be the theory of probability measure-preserving $\Gamma$-systems. Then the following hold:
  \begin{enumerate}
  \item $\PMP_\Gamma$ is a simplicial affine theory.
  \item The extremal models of $\PMP_\Gamma$ are the ergodic systems of $\Gamma$.
  \end{enumerate}
\end{theorem*}

It is interesting to study the extreme completions of this theory, which are simplicial as well. Whether they are Bauer our Poulsen depends on the action and on the group (see \ref{c:PMPG-strongly-ergodic-propT}).

\begin{theorem*}
  Let $\cX$ be an ergodic $\Gamma$-system. If $\cX$ is strongly ergodic, then the affine theory of $\cX$ is a Bauer theory. Otherwise, it is a Poulsen theory.
\end{theorem*}

In a related vein, we have:

\begin{theorem*}
  A group $\Gamma$ has property (T) if and only if $\PMP_\Gamma$ is a Bauer theory.
\end{theorem*}

When the group $\Gamma$ is amenable, we can use some existing results of Giraud~\cite{Giraud2019p} in the continuous logic setting to get much more precise results. An \emph{invariant random subgroup} (or \emph{IRS} for short) of $\Gamma$ is a probability measure on the space $\Sub(\Gamma)$ of subgroups of $\Gamma$ invariant under conjugation. Every probability measure-pre\-serving action $\Gamma \actson (X, \mu)$ gives rise to its \emph{stabilizer IRS} simply via the stabilizer map $X \to \Sub(\Gamma)$, $x \mapsto \Gamma_x$. It turns out that the stabilizer IRS is coded by the affine theory of the action. The following is contained in \ref{th:PMP-hyperfinite}.

\begin{theorem*}
  Let $\Gamma$ be a countable, amenable group, and let $\theta$ be an ergodic IRS of $\Gamma$ of infinite index. Then the theory $\PMP_\theta$ of $\Gamma$-systems with stabilizer IRS equal to $\theta$ is a complete Poulsen theory with affine quantifier elimination in an appropriate language.
\end{theorem*}

Note that when $\theta$ is the trivial IRS (i.e., the Dirac measure on $1_\Gamma$), $\PMP_\theta$ is simply the theory of free actions of $\Gamma$.

\subsection*{Extremal categoricity} Categoricity of a theory at a given cardinal is a fundamental notion in model theory, and affine logic offers new phenomena of this kind.

First of all, by analogy with the classical theory, one may ask when it is the case that all \emph{extremal} models of an affine theory of a fixed density character are isomorphic. In the case of separable models, our affine omitting types theorem yields an extremal version of the continuous logic Ryll-Nardzewski theorem (see \autoref{sec:aleph0-cat}).

\begin{theorem*}
  Let $T$ be a complete, affine theory in a separable language. The following are equivalent:
  \begin{enumerate}
  \item $T$ admits a unique separable extremal model, up to isomorphism.
  \item The metric topology and the logic topology coincide on the extreme types.
  \end{enumerate}
\end{theorem*}

Moreover, if $T$ is extremally $\aleph_0$-categorical and simplicial, then the integral decomposition theorem yields a classification of all separable models of $T$: if $M$ is the unique separable extremal model, then they are precisely the direct multiples $L^1(\Omega,M)$, where $\Omega$ is a standard probability space.

The corresponding question in uncountable cardinals, and the plausibility of an extremal Morley theorem, will not be explored in this paper. On the other hand, the absence of compactness in the class of extremal models introduces an entirely new question: can there be \emph{absolutely categorical} affine theories (i.e., admitting a unique extremal model) whose extremal model is not compact?

A basic example (see \ref{ex:PrA-with-atomless-parameters}) shows that the answer is positive in general. However, we do not know whether a simplicial example exists. A theory with this properties would enjoy some peculiar properties (see \autoref{sec:absolute-categoricity}).

\begin{theorem*}
  Let $T$ be an absolutely categorical, simplicial theory in a separable language, with extremal model $M$. The following hold:
  \begin{enumerate}
  \item Every model of $T$ is of the form $L^1(\Omega, M)$, for a probability space $\Omega$.
  \item If $M$ compact, then $T$ is a Bauer theory. Otherwise, it is a Poulsen theory.
  \end{enumerate}
\end{theorem*}

The study of this broadly open subject is initiated in \autoref{part:Categoricity}.

\subsection*{Acknowledgments}
We would like to thank Ilijas Farah for useful comments on a preliminary version of this paper, and for suggesting that we could obtain \autoref{th:distribution-extremal-cont-theories}.

\subsection*{Addendum}
The completion of the present work took several years. Recently, while the paper was being completed and after some of our results had been announced at conferences, two preprints \cite{Bagheri2024pp,Bagheri2024p} by Bagheri that contain some related results appeared. Notably, he considered the notion of extreme types and extremal models, and obtained several results having non-empty intersection with the contents of \autoref{sec:Definability}, \autoref{sec:ExtremeTypes}, \autoref{sec:cont-theory-extremal-models}, \autoref{prop:ext-aleph0cat-vs-cont-aleph0cat} and \autoref{p:compact-absol-cat} of this paper.


\part{The logic}
\label{part:The-logic}

\section{Preliminaries about compact convex sets}
\label{sec:prel-about-comp}

\subsection{Faces and extreme points}
\label{sec:extreme-points}

We use \cite{Alfsen1971,Phelps2001,Simon2011} as general references for compact convex sets.

Let $E$ be a real vector space.
A subset $X \sub E$ is \emph{convex} if for all $p_1, p_2 \in X$ and all $\lambda \in [0, 1]$, we have that $\lambda p_1 + (1- \lambda)p_2 \in X$.
A \emph{cone} is a subset of $E$ closed under multiplication by non-negative scalars.
A \emph{convex cone} is a cone which is convex.
Equivalently, a convex cone is a subset of $E$ closed under sums and multiplication by non-negative scalars.
A function $f \colon X \to Y$ between convex sets is called \emph{convex} if
\begin{equation*}
  f(\lambda p + (1 - \lambda) q) \leq \lambda f(p) + (1 - \lambda) f(q)\ \ \text{for all}\ \ 0 \leq \lambda \leq 1,
\end{equation*}
\emph{concave} if $-f$ is convex, and \emph{affine} if it is both convex and concave. Note that the collection of convex functions is a cone and that the affine functions form a vector space.

When $E$ is a locally convex topological real vector space, we denote its topological dual by $E^*$.
A compact convex subset of $E$ will be called a \emph{compact convex set}.
Note that for a compact subset $X \subseteq E$ to be convex, it is enough to require that $\half p_1 + \half p_2 \in X$ for all $p_1, p_2 \in X$.
If $X \sub E$ is compact and convex, a continuous, affine function $X \to \R$ is the restriction to $X$ of a function of the type $\varphi  + c$, where $\varphi \in E^*$ and $c \in \R$.
We will denote by $\cA(X)$ the Banach space of continuous, affine functions on $X$ equipped with the supremum norm $\|{\cdot}\|_\infty$.

A non-empty convex subset $F \sub X$ is called a \emph{face} if for every $p_1, p_2 \in X$ and $0<\lambda<1$, if $\lambda p_1 + (1-\lambda) p_2 \in F$, then $p_1,p_2 \in F$ as well.
Note that a face of a face is a face.
A face $F$ is called \emph{exposed} if there exists $\phi \in \cA(X)$ such that $\phi \rest_F = 0$ and $\phi(p) > 0$ for all $p \in X \sminus F$. Note that every exposed face is closed.
A point $p \in X$ is \emph{extreme} if $\set{p}$ is a face, or equivalently, if $X \sminus \set{p}$ is convex.
We will denote by $\cE(X)$ the collection of extreme points of $X$.

If $A \sub X$, we denote by $\clco A$ the \emph{closed convex hull of $A$}, i.e., the intersection of all closed convex subsets of $X$ containing $A$.

We record some basic facts concerning compact convex sets and their extreme points.
Let $X$ denote a compact convex set.

\begin{description}
\item[The Hahn--Banach separation theorem {\cite[Thm.~4.5]{Simon2011}}] If $B_1$ and $B_2$ are closed, convex and disjoint subsets of $X$, then there exists $\varphi \in \cA(X)$ with $\varphi(p) \leq 0$ for all $p \in B_1$ and $\varphi(p) \geq 1$ for all $p \in B_2$. Moreover, if $D \sub \cA(X)$ is a dense subspace (in $\|{\cdot}\|_\infty$), we can find $\varphi \in D$.

\item[The Krein--Milman theorem {\cite[Thm.~8.14]{Simon2011}}] $X = \clco \cE(X)$. In particular, the set $\cE(X)$ is non-empty.

\item[Milman's theorem {\cite[Thm.~9.4]{Simon2011}}] If $K \sub X$ is compact, then $\cE( \clco K ) \subseteq K$.
\end{description}

\begin{ntn}
  If $\varphi \in \cA(X)$, we write $\oset{\varphi > t}$ for the set $\bigl\{ x \in X : \varphi(x) > t \bigr\}$, and similarly for similar conditions.
  We may also write $\oset{\ldots}_\cE = \oset{\ldots} \cap \cE(X)$.
\end{ntn}

A subset of the form $\oset{\varphi > 0}$ for $\varphi \in \cA(X)$ is an \emph{open half-space}.
It follows from the Hahn--Banach theorem and compactness that finite intersections of open half spaces form a basis for the topology on $X$.
Moreover, near extreme points, single open half spaces suffice.

\begin{lemma}
  \label{l:nbhds-extreme}
  If $p \in X$ is an extreme point, then sets of the form $\oset{\varphi > 0}$ with $\varphi \in \cA(X)$ and $\varphi(p) > 0$ form a basis of neighborhoods of $p$ in $X$.
  Consequently, sets of the form $\oset{\varphi > 0}_\cE$ with $\varphi \in \cA(X)$ form a basis for the induced topology on $\cE(X)$.
\end{lemma}
\begin{proof}
  Let $U \ni p$ be open.
  By Milman's theorem, $p \notin \cE\bigl( \clco{(X \sminus U)} \bigr)$, and since it is extreme in $X$, $p \notin \clco{(X \sminus U)}$.
  By Hahn--Banach, there exists $\varphi \in \cA(X)$ with $\varphi(p) = 1$ and $\varphi(q) \leq 0$ for all $q \in \clco{X \sminus U}$.
\end{proof}

\begin{lemma}\label{l:surjective-affine-map}
  Let $X$ and $Y$ be compact convex sets and let $\pi \colon X \to Y$ be a continuous, affine map. If the image of $\pi$ intersects every non-empty half-space $\oset{\varphi > 0}\subseteq Y$ with $\varphi \in \cA(Y)$, then $\pi$ is surjective.
\end{lemma}
\begin{proof}
  By the hypotheses and the previous lemma, the image of $\pi$ contains the extreme points of $Y$. As the image is a compact convex subset, it must be all of $Y$, by Krein--Milman.
\end{proof}

We leave the easy proof of the following lemma to the reader.
\begin{lemma}
  \label{l:extreme-transitivity}
  Let $X$ and $Y$ be compact convex sets and let $\pi \colon X \to Y$ be a continuous, affine, surjective map. Then the following hold:
  \begin{enumerate}
  \item If $F$ is a face of $Y$, then $\pi^{-1}(F)$ is a face of $X$.
  \item\label{l:extreme-transitivity:extreme-points} If $y \in Y$ is extreme, then $\cE(\pi^{-1}(\set{y})) \sub \cE(X)$. In particular, $\pi(\cE(X)) \supseteq \cE(Y)$.
  \item In the special case where $Y \sub \R$, we obtain that an affine function attains its maximum and minimum at extreme points.
  \end{enumerate}
\end{lemma}

\begin{lemma}
  \label{l:extreme-Gdelta}
  Suppose that $X$ is a metrizable compact convex set. Then $\cE(X)$ is a $G_\delta$ set.
\end{lemma}
\begin{proof}
  Let $f \colon X \times X \to X$ denote the map $f(x, y) = \half x + \half y$. Then
  \begin{equation*}
    X \sminus \cE(X) = f \bigl((X \times X) \sminus \Delta(X) \bigr),
  \end{equation*}
  where $\Delta(X)$ denotes the diagonal. The set $(X \times X) \sminus \Delta(X)$ is open, and so, by the metrizability of $X$, also $K_\sigma$. Thus $X \sminus \cE(X)$ is $K_\sigma$ and $\cE(X)$ is $G_\delta$.
\end{proof}

\subsection{Kadison duality}
\label{sec:kadison-duality}

For material in this subsection, we refer to \cite[Ch.~II]{Alfsen1971}.

\begin{defn}
  \label{defn:OrderUnitSpaceSetup}
  Let $A$ be a real, ordered vector space.
  \begin{enumerate}
  \item We let $A^+ = \set{\varphi \in A : \varphi \geq 0}$ denote the \emph{positive cone} of $A$.
  \item We say that $A$ is \emph{Archimedean} if the set $\{n \varphi : n \in \bN\}$ is not bounded for any $\varphi \in A \setminus \{0\}$.
  \item A positive element $\bOne \in A$ is called an \emph{order unit} if for every $\varphi \in A$, there is $n \in \N$ with $\varphi \leq n \bOne$.
  \end{enumerate}
\end{defn}

\begin{defn}
  \label{defn:OrderUnitSpace}
  An \emph{order unit space} is a real, ordered, Archimedean vector space with a distinguished order unit $\bOne$.
  \begin{enumerate}
  \item Assume that $A$ is an order unit space.
    For $\varphi \in A$ we define
    \begin{equation}
      \label{eq:df-norm}
      \|\varphi\| \coloneqq \inf \, \bigl\{ r \in \bR :  -r\bOne \leq \varphi \leq r\bOne \bigr\}.
    \end{equation}
  \item A linear map $T \colon A \to A'$ between order unit spaces is called \emph{unital} if $T(\bOne) = \bOne$ and \emph{positive} if $T(\varphi) \geq 0$ for all $\varphi \geq 0$.
  \item Let $A$ be an order unit space, and consider $\bR$ as an order unit space in the obvious fashion.
    A linear positive unital functional $p\colon A \rightarrow \bR$ is called a \emph{state} of $A$.
  \item The collection of all states of $A$ is denoted by $\tS(A)$.
  \end{enumerate}
\end{defn}

In \autoref{eq:df-norm} we indeed define a norm, making every order unit space a normed space.
A unital map $T\colon A \rightarrow A'$ between order unit spaces is positive if and only if $\nm{T} = 1$ \cite[Prop.~II.1.3]{Alfsen1971}.
In particular, a positive, unital map is automatically continuous, so every state is continuous.

If $X$ is a compact convex set, the space $\cA(X)$ of affine functions on $X$ ordered pointwise with order unit the constant function $1$ is a complete order unit space, and the induced norm is $\|{\cdot}\|_\infty$.
Conversely, if $A$ is an order unit space and $\tS = \tS(A)$, then $\tS$ is a compact convex subset of $A^*$, and if $\varphi \in A$, then $\hat{\varphi} \in \cA(\tS)$, where
\begin{equation*}
  \hat{\varphi}(p) = p(\varphi).
\end{equation*}
It is a theorem of Kadison that the map $\varphi \mapsto \hat{\varphi}$ defines an order-preserving unital embedding $A \hookrightarrow\cA(\tS)$ with dense image \cite[Prop.~II.1.7, Thm.~II.1.8]{Alfsen1971}.
In particular, it is isometric, and we can identify $\cA(\tS)$ with the completion of $A$. We will thus consider $A$ as a subset of $\cA(\tS)$.

The pair of functors $X \mapsto \cA(X)$, $A \mapsto \tS(A)$ defines a duality of categories between compact convex sets with morphisms given by affine, continuous maps and complete order unit spaces with morphisms given by linear, order-preserving, unit-preserving maps.

A useful consequence of the duality is the following analogue of the Stone--Weierstrass theorem. See \cite[\textsection4]{Jellett1968} for a direct proof.

\begin{prop}
\label{prop:Stone-Weierstrass-affine}
Let $X$ be a compact convex set. Let $L\subseteq\cA(X)$ be a linear subspace that contains the constants and separates the points of $X$. Then $L$ is dense in $\cA(X)$.
\end{prop}

\begin{ntn}
  Let $A$ be an order unit space.
  If $\Sigma \subseteq A$, then we let $[\Sigma]_+ \subseteq \tS(A)$ denote the collection of states of $A$ that are positive (or zero) on $\Sigma$.
\end{ntn}

Let us call a \emph{positive cone} of an order unit space $A$ a subset $P \subseteq A$ that is a convex cone and contains $A^+$.

\begin{lemma}
  \label{lem:KadisonHahnBanachImplication}
  Assume that $\Sigma \subseteq A$ and $\varphi \in A$.
  Then the following are equivalent.
  \begin{enumerate}
  \item $\varphi$ belongs to the closed positive cone generated by $\Sigma$.
  \item $\varphi$ is positive on $[\Sigma]_+$.
  \end{enumerate}
\end{lemma}
\begin{proof}
  Since every member of $\Sigma$ is, by definition, positive on $[\Sigma]_+$, so is every member of the generated closed positive cone.
  For the converse, we may assume that $\Sigma$ is a closed positive cone.
  If $\varphi \notin \Sigma$, then by the Hahn--Banach theorem \cite[Thm.~4.5]{Simon2011}, there exists a linear functional $p \in A^*$ that separates $\varphi$ from $\Sigma$ in $A$.

  We have $0 \in \Sigma$ and $p(0) = 0$, so $p(\varphi) \neq 0$, and we may assume that $p(\varphi) < 0$.
  Since $\Sigma$ is a cone, we must have $p\rest_\Sigma \geq 0$.
  In particular, $p$ is positive.
  Therefore, $p(\bOne) > 0$ (or else $p = 0$, which is impossible), and we may assume that $p(\bOne) = 1$.
  In particular, $p$ is a state, and $p \in [\Sigma]_+$, witnessing that $\varphi$ is \emph{not} positive on $[\Sigma]_+$.
\end{proof}

\begin{prop}
  \label{prop:KadisonHahnBanachImplicationDuality}\
  \begin{enumerate}
  \item If $\Sigma \subseteq A$, then $[\Sigma]_+$ is a closed convex subset of $\tS(A)$.
  \item If $X \subseteq \tS(A)$ and $P(X)$ is the collection of $\varphi \in A$ that are positive on $X$, then $P(X)$ is a closed positive cone of $A$.
  \item These two operations define a bijection between closed convex subsets of $\tS(A)$ and closed positive cones of $A$, one being the inverse of the other.
  \end{enumerate}
\end{prop}
\begin{proof}
  The first two assertions are obvious.
  If $\Sigma \subseteq A$ is a closed positive cone, then $P\bigl( [\Sigma]_+\bigr) = \Sigma$ by \autoref{lem:KadisonHahnBanachImplication}.
  Finally, let $C \subseteq \tS(A)$ be convex and closed, and let $\Sigma = P(C)$.
  Then $C \subseteq [\Sigma]_+$ by definition.
  If $p \in [\Sigma]_+ \setminus C$, then there exists $\varphi \in \cA\bigl(\tS(A)\bigr)$ that separates $p$ from $C$, say $\varphi\rest_C \geq 0$ and $\varphi(p) < 0$.
  Since $C$ is compact, and $A$ is dense in $\cA\bigl(\tS(A)\bigr)$, we may assume that $\varphi \in A$.
  But then $\varphi \in P(C) = \Sigma$, so $p \notin [\Sigma]_+$.
  This proves that $C = [\Sigma]_+$, completing the proof.
\end{proof}

\begin{prop}
  \label{prop:OrderUnitSpaceQuotient}
  Let $A$ be an order unit space and let $P \subseteq A$ be a closed positive cone.
  \begin{enumerate}
  \item The relation $\varphi \leq_P \psi$ defined by $\psi - \varphi \in P$ is a preorder relation on $A$.
  \item The intersection $I = P \cap (-P)$ is a vector subspace of $A$, and $\leq_P$ induces an order relation on the vector space quotient $B = A/I$.
  \item Equipped with the induced order and the equivalence class of $\bOne$, $B$ is an order unit space, and the quotient map $\pi\colon A \to B$ is positive and unital.
  \item For $p \in \tS(B)$ let $\iota(p) = p \circ \pi$.
    Then $\iota\colon \tS(B) \cong [P]_+ \subseteq \tS(A)$ is an affine homeomorphism of compact convex sets.
  \end{enumerate}
\end{prop}
\begin{proof}
  The first two items are clear, as is the fact that $\bOne$ is also an order unit on $B$.
  The induced order on $B$ is Archimedean since $P$ is closed, so $B$ is an order unit space.
  The map $\pi$ is positive since $A^+ \subseteq P$, and unital by construction.

  If $p \in \tS(B)$ then $\iota(p) = p \circ \pi \in \tS(A)$, and moreover $\iota(p) \in [P]_+$, essentially by definition.
  Since $\pi$ is surjective, the map $\iota\colon \tS(B) \rightarrow [P]_+$ is injective.
  If $q \in [P]_+$, then it must vanish on $I$, and therefore factor through $B$ as $q = p \circ \pi$.
  Moreover, $p$ is unital (since $q$ is) and positive (since $q$ is positive on $P$), so $p \in \tS(B)$.
  Since $\iota$ is continuous, it is a homeomorphism $\iota\colon \tS(B) \cong [P]_+$.

  Finally, $\iota$ respects the convex structure, by definition.
\end{proof}

In what follows, we may allow ourselves to say that $B = A/P$, and consider $\iota$ to be the identity, so $\tS(A/P) = [P]_+$.

\begin{prop}
  \label{prop:HahnBanachCompactness}
  Assume that $\Sigma \subseteq A$ is closed under addition.
  Then the following are equivalent.
  \begin{enumerate}
  \item $[\Sigma]_+ \neq \emptyset$.
  \item We have $[\varphi]_+ \neq \emptyset$ for all $\varphi \in \Sigma$.
  \item We have $[\varphi + \bOne]_+ \neq \emptyset$ for all $\varphi \in \Sigma$.
  \item We have $-\bOne \notin \Sigma + A^+$.
  \end{enumerate}
\end{prop}
\begin{proof}
  \begin{cycprf}
  \item[\impnnext] Clear.
  \item If $-\bOne \in \Sigma + A^+$, then $-\bOne \geq \varphi$ for some $\varphi \in \Sigma$, and $[2\varphi + \bOne]_+ = \emptyset$.
  \item[\impfirst]
    Observe first that $-r\bOne \notin \Sigma + A^+$ for all $r > 0$.
    Indeed, if not, then $-nr\bOne \in \Sigma + A^+$ for $n$ large enough such that $nr \geq 1$, so $-\bOne \in \Sigma + A^+$.
    Therefore $-\frac{1}{2}\bOne \notin \bQ^+ \Sigma + A^+$, whence it follows that $\|\varphi + \bOne\| \geq \frac{1}{2}$ for all $\varphi \in \bQ^+ \Sigma + A^+$.
    Therefore $-\bOne \notin \overline{\bQ^+ \Sigma + A^+}$, and the latter is the closed convex cone generated by $\Sigma \cup A^+$.
    Therefore $-\bOne$ is not positive on $[\Sigma]_+$, so $[\Sigma]_+ \neq \emptyset$.
  \end{cycprf}
\end{proof}

\subsection{Probability spaces and measure algebras}
\label{sec:measure-algebras}

Throughout the paper, probability measures are always assumed to be countably additive (unless we qualify them explicitly as finitely additive).

Let $(\Omega, \cB, \mu)$ be an arbitrary probability space. Its \emph{measure algebra}, denoted by $\MALG(\Omega, \mu)$ (or, sometimes, just $\MALG(\Omega)$), is defined as the quotient of the Boolean algebra $\cB$ by the ideal of null sets. Then $\MALG(\Omega, \mu)$ is a complete Boolean algebra and the measure $\mu$ is well-defined on it. It also has the structure of a complete metric space: the distance between two elements is the measure of their symmetric difference.

If $\theta \colon \Omega' \to \Omega$ is a measure-preserving map between probability spaces (i.e., $\theta$ is measurable and $\mu'(\theta^{-1}(B))=\mu(B)$ for all $B\in\cB$), then one has a dual embedding $\theta^* \colon \MALG(\Omega,\mu) \to \MALG(\Omega',\mu')$ defined by $\theta^*([B]_\mu) = [\theta^{-1}(B)]_{\mu'}$ for all $B\in\cB$.

We will say that a set $A \in \cB$ (or the corresponding element $[A] \in \MALG(\Omega, \mu)$) is an \emph{atom} if $\mu(A) > 0$ and for every $B \sub A$, $B \in \cB$, we have that $\mu(B) = \mu(A)$ or $\mu(B) = 0$. A probability space (or its measure algebra) is \emph{atomless} if it has no atoms. In the operator algebra literature, atomless probability spaces are often called \emph{diffuse}.

If $M$ is the measure algebra of a probability space $(\Omega, \mu)$, we may denote the associated function spaces $L^\infty(\Omega, \mu)$ and $L^1(\Omega, \mu)$ simply by $L^\infty(M)$ and $L^1(M)$, respectively. To see that they do not depend on the choice of $\Omega$, note that $L^\infty(\Omega, \mu)$ can be constructed directly from $M$ by taking the completion of the collection of simple functions with respect to the essential supremum norm, and similarly for $L^1$. The following well-known example of a compact convex set will be relevant in \autoref{sec:bauer-compact-extremal}.

\begin{example}
  \label{ex:Linfty-convex-set}
  Let $M$ be a measure algebra and let $B=\set{f\in L^\infty(M) : 0\leq f\leq 1}$. Equipped with the weak$^*$ topology of $L^\infty(M)$, $B$ is a compact convex set, and its extreme points are precisely the characteristic functions $\chi_b$ associated to the elements $b\in M$. If $M$ is atomless, then the extreme points of $B$ form a dense set.

  More generally, given $n\in\N$, let $B_n = \set{f \in (L^\infty(M))^{2^n} : f_\eps \geq 0, \sum_{\eps \in 2^n} f_\eps = 1 }$. With the product of the weak$^*$ topology, $B_n$ is a compact convex set. Its extreme points are the tuples of the form $(\chi_{b_\eps})_{\eps\in 2^n}$ such that for some tuple $b\in M^n$ we have $b_\eps=\bigcap_{i<n}b^{\eps_i}$ for every $\eps\in 2^n$, where $b^0=b$ and $b^1=\neg b$. The extreme points are dense in $B_n$ if and only if $M$ is atomless.
\end{example}

If $(\Omega,\cB,\mu)$ is a probability space, we will denote by $\cB_\mu$ be the completion of $\cB$ with respect to $\mu$, i.e., the $\sigma$-algebra generated by $\cB$ and the $\mu$-null subsets of $\Omega$. The sets in $\cB_\mu$ are the \emph{$\mu$-measurable} subsets of $\Omega$. Similarly, a function $f\colon\Omega\to \R$ is \emph{$\mu$-measurable} if it differs from a $\cB$-measurable function on a $\mu$-null subset of $\Omega$.

Given a non-null $\mu$-measurable set $A$, we will denote by $\mu_A$ the \emph{conditional measure} on $A$, defined on $\cB$ by
\begin{equation*}
  \mu_A(B) = \frac{1}{\mu(A)} \mu(A \cap B).
\end{equation*}
Of course, if $A$ has full measure then $\mu_A=\mu$. This construction can be extended to certain non-measurable sets, as follows.

\begin{defn}
  \label{df:measure-concentr-general}
  Let $(\Omega,\cB,\mu)$ be a probability space and let $A \sub \Omega$ be an arbitrary subset.
  We will say that \emph{$\mu$ concentrates on $A$} if for every $B\in \cB$ disjoint from $A$ we have that $\mu(B) = 0$.
\end{defn}

\begin{remark}
  It is possible that a measure $\mu$ concentrates simultaneously on a set $A$ and on its complement. See, for instance, \autoref{ex:random-set}.
\end{remark}

\begin{lemma}
  \label{l:measure-concentration}
  Suppose that $\mu$ concentrates on $A$. Then the formula
  \begin{equation*}
    \mu_A(B \cap A) = \mu(B)
  \end{equation*}
  determines a well-defined probability measure on the trace $\sigma$-algebra
  $$\cB_A=\set{B\cap A : B\in\cB}.$$
  If $(\Omega,\cB,\mu)$ is a complete probability space, then so is $(A,\cB_A,\mu_A)$.

  The map $B\mapsto B\cap A$ induces an isomorphism $\MALG(\Omega, \mu) \cong \MALG(A, \mu_A)$, and for every $\mu$-measurable and integrable function $f \colon \Omega \to \R$ we have:
  \begin{equation*}
    \int_A f\rest_A \, \ud \mu_A = \int_\Omega f \, \ud \mu.
  \end{equation*}
\end{lemma}
\begin{proof}
  If $B\cap A=B'\cap A$ then $(B\triangle B')\cap A = \emptyset$, so $\mu(B\triangle B')=0$ and $\mu_A$ is well-defined. The rest is easy.
\end{proof}

Finally, we recall that a probability space $(\Omega, \cB, \mu)$ is called \emph{standard} if it is isomorphic to one where $\Omega$ is a Polish space, $\mu$ is a Borel probability measure, and $\cB$ is the completion of the Borel $\sigma$-algebra with respect to~$\mu$.

\subsection{Choquet theory}
\label{sec:choquet-theory}

The theorems of Choquet and Bishop--de Leeuw are abstract representation theorems for points in a compact convex set as barycenters of measures concentrated on the set of extreme points. They are the basis of a large number of integral decomposition theorems in analysis. See \cite{Phelps2001}, \cite[Chs. 10, 11]{Simon2011}, and \cite{Alfsen1971}.

If $Y$ is a compact Hausdorff space, let $C(Y)$ denote the vector lattice of real-valued continuous functions on $Y$.
We will denote by $\cM(Y)\coloneqq \tS(C(Y))$ the compact convex set of states on $C(Y)$.
The elements of $\cM(Y)$ are called \emph{Radon probability measures} on $Y$, and can be identified with the regular Borel probability measures on $Y$ (regularity means that $\mu(A) =\sup\set{\mu(K):K\subseteq A,\ K\text{ compact}} =\inf\set{\mu(U):A\subseteq U,\ U\text{ open}}$ for every measurable set $A$).
Recall that the $\sigma$-algebra of \emph{Baire subsets} of $Y$ is the one generated by all compact, $G_\delta$ sets, or equivalently, by all continuous real-valued functions (or, more generally, by all continuous functions to metrizable spaces).
Every probability measure on the Baire sets extends to a unique regular Borel probability measures.
Therefore, the Radon measures on $Y$ may also be identified with the Baire probability measures.
If $Y$ is metrizable, the Baire sets coincide with the Borel sets.

For any bounded Borel function $f$ on $Y$, we will denote by $\mu(f)$ the value of the integral $\int_Y f \ud \mu$. If $\pi\colon Y\to Z$ is a continuous map to another compact Hausdorff space $Z$, then the pushforward map $\pi_*\colon \cM(Y)\to \cM(Z)$, given by $(\pi_*\mu)(f)=\mu(f\circ\pi)$, is continuous and affine (and surjective if $\pi$ is).

Now let $X$ be a compact convex set and let $\mu \in \cM(X)$. The \emph{barycenter} of $\mu$ is the unique point $R(\mu) \in X$ satisfying
\begin{equation*}
  f(R(\mu)) = \mu(f) \quad \text{ for all } f \in \cA(X).
\end{equation*}
The map $R \colon \cM(X) \to X$ is continuous and affine. If $\pi\colon X\to Y$ is a continuous, affine map to another compact convex set $Y$, then $R(\pi_*(\mu)) = \pi(R(\mu))$.

\begin{lemma}\label{l:barycenter-face-measure-one}
  Let $X$ be a compact convex set and $F\subseteq X$ be a closed face. If $\mu\in\cM(X)$ and $R(\mu)\in F$, then $\mu(F)=1$.
\end{lemma}
\begin{proof}
  If $\mu(F) < 1$, by regularity, there exists a compact set $K \sub X$ disjoint from $F$ with $\mu(K) > 0$. Then $\clco(K) \cap F = \emptyset$, for if not, any extreme point of $\clco(K) \cap F$ is also an extreme point of $\clco(K)$, so in $K$, by Milman's theorem, contradicting $K \cap F = \emptyset$. Now the conditional measure $\mu_K$ satisfies $R(\mu_K)\in\clco(K)$, and the representation $R(\mu) = \mu(K) R(\mu_K) + \mu(X \sminus K) R(\mu_{X \sminus K})$ contradicts the fact that $F$ is a face.
\end{proof}

We denote by $\cC(X)$ the convex cone of continuous, convex functions on $X$ and we define the \emph{Choquet partial order} $\preceq$ on $\cM(X)$ by
\begin{equation*}
  \mu \preceq \nu \iff \mu(f) \leq \nu(f) \quad \text{ for all } f \in \cC(X).
\end{equation*}
A measure $\mu \in \cM(X)$ is called a \emph{boundary measure} if it is a maximal element for $\preceq$.
Intuitively, if $\mu \preceq \nu$, then $\mu$ and $\nu$ have the same barycenter and $\nu$ is concentrated ``closer'' to the extreme boundary than $\mu$.
In particular, boundary measures are concentrated on $\cE(X)$ in an appropriate sense.

We note several basic, well-known properties of the Choquet order.
\begin{lemma}
  \label{l:choquet-order}
  Let $X$ be a compact convex set. Then the following hold:
  \begin{enumerate}
  \item \label{i:choquet-order:1} $\preceq$ is a partial order and it is closed as a subset of $\cM(X)^2$.
  \item \label{i:choquet-order:2} If $\mu \preceq \nu$, then $R(\mu) = R(\nu)$.
  \item \label{i:choquet-order:4} For any $x \in X$, there exists a boundary measure $\mu \in \cM(X)$ with $\delta_x \preceq \mu$.
  \end{enumerate}
\end{lemma}
\begin{proof}
  \ref{i:choquet-order:1} This is clear from the definition.

  \ref{i:choquet-order:2} If $f$ is continuous and affine, then both  $f$ and $-f$ are convex and by definition, $\mu(f) \leq \nu(f)$ and $\mu(-f) \leq \nu(-f)$. Thus $\mu(f) = \nu(f)$ for every continuous, affine $f$ and this implies that $R(\mu) = R(\nu)$.

  \ref{i:choquet-order:4} This follows from \ref{i:choquet-order:1} and an application of Zorn's lemma.
\end{proof}

\begin{defn}
  \label{defn:ConcaveEnvelope}
  If $f$ is a bounded function on $X$, we define $\hat{f}$, the \emph{concave envelope} of $f$, by
  \begin{equation*}
    \hat{f}(x) = \inf \set{h(x) : h \in \cA(X) \And h \geq f}.
  \end{equation*}
\end{defn}

The following are some basic properties of the concave envelope that we will need.
\begin{lemma}
  \label{l:concave-envelope}
  Let $X$ be a compact convex set and let $f$ be a bounded function on $X$. Then the following hold:
  \begin{enumerate}
  \item \label{i:concave-env:1} $\hat f$ is concave, bounded, and upper semi-continuous.
  \item \label{i:concave-env:2} If $f$ is concave, upper semi-continuous, then $\hat f = f$.
  \item \label{i:concave-env:3} If $f$ is concave, upper semi-continuous, then $\mu(f) \leq f(R(\mu))$.
  \end{enumerate}
\end{lemma}
\begin{proof}
  \ref{i:concave-env:1}, \ref{i:concave-env:2}. See \cite[\textsection3]{Phelps2001}.

  \ref{i:concave-env:3} This is just a version of Jensen's inequality. Using \ref{i:concave-env:2}, we have that:
  \begin{equation*}
    \begin{split}
      \mu(f) & \leq \inf \, \bigl\{ \mu(h) : h \in \cA(X), \, h \geq f \bigr\} \\
             & = \inf \, \bigl\{ h(R(\mu)) : h \in \cA(X), \, h \geq f \bigr\}
               = \hat f\bigl( R(\mu) \bigr) = f\bigl( R(\mu) \bigr).
               \qedhere
    \end{split}
  \end{equation*}
\end{proof}

When considering Radon measures, \autoref{df:measure-concentr-general} will always be applied with respect to the algebra of Baire sets.

\begin{defn}
  \label{df:measure-concentr}
  Let $X$ be a compact space, $\mu$ a Radon probability measure on $X$ and $A \sub X$ an arbitrary subset.
  We will say that \emph{$\mu$ concentrates on $A$} if for every Baire set $B$ disjoint from $A$, we have that $\mu(B) = 0$. If $A$ is Baire, this simply means $\mu(A) = 1$.
\end{defn}

\begin{lemma}
  \label{l:concentrates-implies-vanishes-Gdelta}
  Let $X$ be a compact space, $A \sub X$, and let $\mu$ be a Radon probability measure on $X$ that concentrates on $A$. Then for every $G_\delta$ set $B \sub X$ disjoint from $A$, we have that $\mu(B) = 0$.
\end{lemma}
\begin{proof}
  Suppose that $\mu(B) = r > 0$. By regularity, there exists a compact $K \sub B$ with $\mu(K) > r/2$. Write $B = \bigcap_n U_n$ with each $U_n$ open. By Urysohn's lemma, there exist $f_n \in C(X)$ with $f_n \rest_K = 0$ and $f_n \rest_{X \sminus U_n} = 1$. Let $O_n = f_n^{-1}((-\infty, 1/2))$. Then $K \sub O_n \sub U_n$ and each $O_n$ is open and Baire. Thus $\bigcap_n O_n$ is Baire, disjoint from $A$, and with measure $\geq r/2$, contradicting the assumptions.
\end{proof}

The fundamental result about boundary measures is given by the following theorem (see \cite[Ch.~4]{Phelps2001}).
\begin{theorem}[Choquet--Bishop--de Leeuw]
  \label{th:Bishop-dL}
  Let $X$ be a compact convex set.
  Then every boundary measure $\mu \in \cM(X)$ concentrates on $\cE(X)$.
  In particular, if $X$ is metrizable, then $\mu(\cE(X)) = 1$.
\end{theorem}

The following characterization of boundary measures is also important (see \cite[Thm.~10.5]{Simon2011} or \cite[Prop.~10.3]{Phelps2001}).
\begin{prop}[Mokobodzki]
  \label{p:Mokobodzki}
  Let $X$ be a compact convex set and let $\mu \in \cM(X)$. Then the following are equivalent:
  \begin{enumerate}
  \item $\mu$ is a boundary measure.
  \item For every $f \in C(X)$, we have $\mu(f) = \mu(\hat f)$.
  \end{enumerate}
  In particular, for every $x \in \cE(X)$, $f(x) = \hat f(x)$.
\end{prop}

\begin{lemma}
  \label{l:boundary-measures}
  Let $X$ be a compact convex set. The following hold:
  \begin{enumerate}
  \item \label{i:lbm:convex} The set of boundary measures on $X$ is convex.
  \item \label{i:lbm:conditional} If $\mu$ is a boundary measure on $X$ and $A \sub X$ is a Borel set with $\mu(A) > 0$, then the conditional measure $\mu_A$ is also a boundary measure.
  \item \label{i:lbm:face} If $Y \sub X$ is a closed face and $\mu$ is a boundary measure on $Y$, then it is also a boundary measure on $X$.
  \end{enumerate}
\end{lemma}
\begin{proof}
  \ref{i:lbm:convex} This follows from \autoref{p:Mokobodzki}.

  \ref{i:lbm:conditional} Let $f \in  C(X)$. As $\mu$ is a boundary measure, we have that $\mu(f) = \mu(\hat f)$. Since $f \leq \hat f$, this implies that $\mu(\set{x : f(x) = \hat f(x)}) = 1$, so $\mu_A(\set{x : f(x) = \hat f(x)}) = 1$ and thus, $\mu_A(f) = \mu_A(\hat f)$.

  \ref{i:lbm:face} Let $\mu$ be a boundary measure on $Y$ and suppose that $\nu$ is a measure on $X$ with $\mu \preceq \nu$. Then, in particular, $R(\nu) = R(\mu) \in Y$, so $\nu(Y) = 1$ by \autoref{l:barycenter-face-measure-one}. Hence $\nu = \mu$, by maximality of $\mu$.
\end{proof}

\subsection{Simplices}
\label{sec:simplices}

A compact convex set $X$ is called a \emph{simplex} if every two positive functionals in $\cA(X)^*$ have a least upper bound (equivalently, $\cA(X)^*$ is a vector lattice).

\begin{theorem}
  \label{thm:SimplexRiesz}
  Let $X$ be a compact convex set and let $A_0 \subseteq \cA(X)$ be dense.
  The following are equivalent:
  \begin{enumerate}
  \item $X$ is a simplex.
  \item
    \label{item:SimplexWeakRieszInterpolation}
    (Weak Riesz interpolation property) Whenever $(\varphi_1\vee\varphi_2) + \varepsilon \leq \psi_1\wedge\psi_2$ for some $\varphi_1,\varphi_2,\psi_1,\psi_2 \in A_0$ and $\varepsilon>0$, there exists $\chi \in A_0$ such that $\varphi_1\vee\varphi_2 \leq \chi \leq \psi_1\wedge\psi_2$.
    \item (Riesz interpolation property) Whenever $\varphi_1\vee\varphi_2\leq \psi_1\wedge\psi_2$ for $\varphi_1,\varphi_2,\psi_1,\psi_2 \in \cA(X)$, there exists $\chi\in \cA(X)$ such that $\varphi_1\vee\varphi_2\leq \chi\leq \psi_1\wedge\psi_2$.
  \item (Riesz decomposition property) Whenever $\varphi_i,\psi_j\in \cA(X)$ are positive and satisfy $\sum_{i<n}\varphi_i = \sum_{j<m}\psi_j$, there exist positive $\chi_{ij}\in \cA(X)$ such that $\varphi_i = \sum_{j<m}\chi_{ij}$ and $\psi_j = \sum_{i<n}\chi_{ij}$ for all $i<n$ and $j<m$.
  \end{enumerate}
\end{theorem}
\begin{proof}
  See \cite[Ch.~II, \textsection3]{Alfsen1971}, in particular, Thm.~II.3.11 and Prop.~II.3.1.
\end{proof}

\begin{lemma}
  \label{l:inverse-limit-simplices}
  Let $(X_\alpha : \alpha \in A)$ be an inverse system of simplices with connecting maps that are continuous, surjective, and affine. Then the inverse limit $\varprojlim X_\alpha$ is also a simplex.
\end{lemma}
\begin{proof}
  This is the contents of \cite[Thm.~2]{Jellett1968}, but also follows immediately from \autoref{thm:SimplexRiesz} using criterion \autoref{item:SimplexWeakRieszInterpolation}.
\end{proof}

The following is the fundamental result about simplices (see \cite[\textsection10]{Phelps2001}).

\begin{theorem}[Choquet--Meyer]
  \label{th:Choquet-Meyer}
  The following are equivalent for a compact convex set $X$:
  \begin{enumerate}
  \item $X$ is a simplex.
  \item For every $x \in X$, there exists a \emph{unique} boundary measure $\mu \in \cM(X)$ with $\delta_x \preceq \mu$.
  \end{enumerate}
\end{theorem}

A \emph{Bauer simplex} is a simplex $X$ such that $\cE(X)$ is closed. It is not difficult to see that for any compact space $Y$, the compact convex set $\cM(Y)$ is a Bauer simplex (with $\cE(\cM(Y)) = Y$) and conversely, for any Bauer simplex $X$, the barycenter map $R \colon \cM(\cE(X)) \to X$ is an isomorphism. There is also a characterization of Bauer simplices in terms of the order unit space $\cA(X)$ (see \cite[Thm.~II.4.1]{Alfsen1971}).

\begin{theorem}[Bauer]
  \label{th:Bauer}
  The following are equivalent for a compact convex set $X$:
  \begin{enumerate}
  \item $X$ is a Bauer simplex.
  \item $\cA(X)$ is a vector lattice.
  \end{enumerate}
  Moreover, in this case, the restriction map $\cA(X)\to C(\cE(X))$ is a vector lattice isomorphism.
\end{theorem}

The \emph{Poulsen simplex} is the unique, up to affine homeomorphism, non-trivial metrizable simplex whose extremal points are dense \cite{Poulsen1961, Lindenstrauss1978}. The following lemma, together with \autoref{l:inverse-limit-simplices}, implies that a countable inverse limit of Poulsen simplices is Poulsen.
\begin{lemma}
  \label{l:inv-limit-dense-extreme}
  Let $(X_\alpha : \alpha \in A)$ be an inverse system of compact convex sets with connecting maps that are continuous, surjective, and affine. Suppose that $\cE(X_\alpha)$ is dense in $X_\alpha$ for every $\alpha$. Then the inverse limit $\varprojlim X_\alpha$ has the same property.
\end{lemma}
\begin{proof}
  A basis of open sets of the inverse limit is given by preimages of basic open subsets of the $X_\alpha$ by the projection maps $\pi_\alpha$. Let $\pi_\alpha^{-1}(U)$ be such an open set with $U \sub X_\alpha$ non-empty, open. Let $p \in \cE(X_\alpha) \cap U$. Then any extreme point of $\pi_\alpha^{-1}(\set{p})$ belongs to $\cE(\varprojlim X_\alpha)$ by \autoref{l:extreme-transitivity}.
\end{proof}


\section{Affine logic}
\label{sec:affine-logic}

For the reader already familiar with continuous logic \cite{BenYaacov2010, BenYaacov2008}, we may define affine logic through the following modifications:
\begin{itemize}
\item The set of connectives is restricted from all continuous functions $\R^n \to \R$ to the \emph{affine} ones (i.e., functions of the form $x \mapsto f(x) + b$, where $f$ is linear and $b \in \R$).
\item The continuity moduli for symbols are required to be continuous convex functions $\varepsilon \mapsto \delta(\varepsilon)$.
\end{itemize}

The convexity requirement for moduli is needed since we want the class of structures to be closed under convex combinations, and more generally, direct integrals (see \autoref{sec:direct-integrals}).
On the other hand, one can show that continuous logic with convex moduli and continuous logic with arbitrary ones have the same expressive power.
Thus, we may say that affine logic is the fragment of continuous logic (with convex moduli) generated by affine connectives and quantifiers, whence its name.

For a full presentation, let us define a \emph{signature} $\cL$ (in a single sort), in either affine or ordinary continuous logic, to consist of the following information:
\begin{itemize}
\item Two disjoint sets of symbols, referred to as \emph{predicate symbols} and \emph{function symbols}.
\item For each symbol $s$, its \emph{arity} $n_s \in \bN$.
\item For each predicate symbol $P$, a compact \emph{value interval} $[a_P,b_P] \subseteq \bR$.
\item For each symbol $s$, a \emph{continuity modulus} $\delta_s$.
  This is a convex function $\delta_s\colon [0,\infty) \rightarrow [0,\infty)$ such that $\delta_s(\varepsilon) = 0$ if and only if $\varepsilon = 0$ (such a function is automatically continuous).
\item We always assume that $\cL$ includes a distinguished binary predicate symbol $d$ with $a_d=0$ and $\delta_d = \id$, called the \emph{distance symbol}.
\end{itemize}

This can be adapted to a signature in multiple sorts.
A full definition will be tedious and add very little to the exposition, so we leave any such adaptations to the reader.

We fix the convention that finite products of metric spaces are equipped with the sum metric:
\begin{gather}
  \label{eq:TupleDistance}
  d(a,b) = \sum_i d(a_i,b_i).
\end{gather}

An \emph{$\cL$-structure} consists of a non-empty, complete metric space $(M,d)$, together with \emph{interpretations} of the symbols:
\begin{itemize}
\item Each $n$-ary function symbol $F$ is interpreted by a function $F^M\colon M^n \rightarrow M$ which respects the modulus $\delta_F$:
  \begin{gather*}
    \label{eq:ContinuityModulusUgly}
    d(a,b) < \delta_F(\varepsilon) \quad \Longrightarrow \quad d\bigl( F^M(a), F^M(b) \bigr) \leq \varepsilon.
  \end{gather*}
  Equivalently (given our hypotheses for $\delta_F$),
  \begin{gather*}
    \label{eq:ContinuityModulusNice}
    \delta_F\Bigl( d\bigl( F^M(a), F^M(b) \bigr) \Bigr) \leq d(a,b).
  \end{gather*}
\item Each $n$-ary predicate  symbol $P$ is interpreted by a function $P^M\colon M^n \rightarrow [a_P,b_P]$ which respects the modulus $\delta_P$:
  \begin{gather*}
    \delta_P \Bigl( \bigl| P^M(a) - P^M(b) \bigr| \Bigr) \leq d(a,b).
  \end{gather*}
\item The distance symbol is interpreted by the metric (so the diameter of $M$ is always less than $b_d$).
\end{itemize}

Affine formulas are defined as in continuous logic, with connectives restricted to affine ones:
\begin{itemize}
\item Terms and atomic formulas are defined as in continuous logic (or classical logic, for that matter).
\item Every affine combination, with real coefficients, of a finite family of formulas, is a formula.
  This can be achieved by introducing the constant $\bOne$, multiplication by a real constant $r$, and addition, as connectives (of arities zero, one and two, respectively).
\item As in continuous logic, we close the collection of formulas under the quantifiers $\sup$ and $\inf$.
\end{itemize}

In particular, every affine logic formula is a continuous logic formula, and if we add $\vee$ or $\wedge$ as connectives, then we obtain full continuous logic in the signature $\cL$.
On the other hand, the \emph{forced limit} connective of continuous logic (see \cite[\textsection3.2]{BenYaacov2010}) is very non-affine, and no affine analogue of it exists.

When $x = (x_i : i \in I)$ is a family of distinct variables, we denote by $\cL^\aff_x \subseteq \cL^\cont_x$ the collections of affine formulas and continuous formulas with free variables in $x$, respectively.
Such a formula (affine or continuous) may be denoted by $\varphi$, or by $\varphi(x)$ if we wish to make the free variables explicit.
When the exact identity of the variables is not very important, we may denote the same sets by $\cL^\aff_I \subseteq \cL^\cont_I$, the sets of formulas with free variables indexed in $I$.
Usually, this will be used when $I = n = \{0,\ldots,n-1\}$.

The interpretation of a formula $\varphi(x)$, with free variables $x$, in a structure $M$, is a function $\varphi^M\colon M^x \rightarrow \bR$, defined in the obvious manner by induction on the syntactic definition of $\varphi$, in either continuous or affine logic.
In either logic, one can deduce from the syntactic construction of a formula a value interval and a modulus of continuity that its interpretations respect.

\begin{defn}
  \label{defn:Substructure}
  Let $M$ be an $\cL$-structure.
  \begin{enumerate}
  \item
    A \emph{substructure} of $M$ is a structure $N$ whose underlying set is a subset of that of $M$, and all interpretations of symbols in $N$ are the restrictions of their interpretations in $M$.
    We denote this by $N \subseteq M$, observing that the structure $N$ is entirely determined by its underlying set.
  \item
    An \emph{affine substructure} of $M$ is a substructure $N \subseteq M$ such that, in addition, $\varphi^M(a) = \varphi^N(a)$ for every affine formula $\varphi \in \cL^\aff_x$ and $a \in M^x \subseteq N^x$.
    We denote this by $N \preceq^\aff M$. (See \autoref{rem:affine-embedding} about the choice of terminology.)
  \end{enumerate}
\end{defn}

\begin{prop}[Tarski--Vaught test]
  \label{p:Tarski-Vaught}
  Let $M$ be a structure and $A \sub M$.
  Then $\cl{A} \preceq^\aff M$ if and only if for all affine formulas $\varphi(x, y)$ with $x$ a single variable, and for a dense set of $a \in A^{y}$,
  \begin{equation*}
    \big( \sup_x \varphi(x, a) \big)^M = \sup_{b \in A} \varphi^M(b, a).
  \end{equation*}
\end{prop}
\begin{proof}
  The usual proof by induction on formulas.
\end{proof}

\begin{defn}
  \label{defn:DensityCharacter}
  For a metric (or topological) space $X$, we define its \emph{density character} $\fd(X)$ as the least cardinal of a dense subset, if it is infinite, or else $\aleph_0$.
\end{defn}

\begin{prop}[Downward Löwenheim--Skolem]
  \label{prop:DownwardLS}
  Let $M$ be an $\cL$-structure, and let $A \subseteq M$ be a subset.
  Assume that $\fd(M) \geq |\cL|$.
  Then $M$ admits an affine substructure $N \preceq^\aff M$ that contains $A$, such that $\fd(N) = \fd(A) + |\cL|$.
\end{prop}
\begin{proof}
  One can repeat the same proof as in continuous logic, or just apply Downward Löwenheim--Skolem from continuous logic.
\end{proof}

Given a family of structures $(M_i : i \in I)$ indexed by a directed set $(I,<)$ with the property that $M_i\subseteq M_j$ for every $i<j$, we denote by $\cl{\bigcup_{i \in I} M_i}$ the unique structure on the completion of the union containing each $M_j$ as a substructure.

\begin{prop}[Affine chains]
  \label{p:affine-chains}
  Let $(I, <)$ be a directed set and let $(M_i : i \in I)$ be structures such that $M_i \preceq^\aff M_j$ for $i < j$. Then $M_j \preceq^\aff \cl{\bigcup_{i \in I} M_i}$ for every $j \in I$.
\end{prop}
\begin{proof}
  As in continuous logic.
\end{proof}


\section{Types, theories and compactness}
\label{sec:Type}

Recall that $\cL^\aff_x$ denotes the collection of all affine formulas in the language $\cL$, with variables in $x$.
Given two formulas $\varphi,\psi \in \cL^\aff_x$, we shall write $\varphi \leq \psi$ if $\varphi^M \leq \psi^M$ for every $\cL$-structure $M$.
In particular, if $\varphi \leq \psi \leq \varphi$, then we shall identify $\varphi$ and $\psi$.
Together with the constant formula $\bOne$, this makes $\cL^\aff_x$ into an order unit space, equipped with the norm
\begin{gather*}
  \|\varphi\| = \inf \, \bigl\{ r \in \bR :  -r\bOne \leq \varphi \leq r\bOne \bigr\}.
\end{gather*}

\begin{defn}
  \label{defn:Type}
  Let $M$ be an $\cL$-structure and $a \in M^x$.
  The \emph{affine type} of $a$, denoted $\tp^\aff(a)$, is the function $\cL^\aff_x \rightarrow \bR$ defined by
  \begin{gather*}
    \varphi \mapsto \varphi^M(a).
  \end{gather*}
\end{defn}

\begin{defn}
  \label{defn:TypeSpace}
  By an \emph{affine type} in the variables $x$ we mean a state on $\cL^\aff_x$.
  Consequently, the \emph{space of affine types} in $x$ is the state space of $\cL^\aff_x$:
  \begin{gather*}
    \tS^\aff_x(\cL) \coloneqq \tS(\cL^\aff_x),
  \end{gather*}
  which is a compact convex set.
  When $p \in \tS^\aff_x(\cL)$ and $\varphi \in \cL^\aff_x$, we may denote $p(\varphi)$ by $\varphi(p)$ or $\varphi^p$.
\end{defn}

When $x = (x_i : i \in I)$ we may denote $\cL^\aff_x$ by $\cL^\aff_I$, and members of $\tS^\aff_I(\cL) = \tS(\cL^\aff_I)$ are called \emph{$I$-types}.

It is easy to check that if $a \in M^x$, then $\tp^\aff(a)$ is a state on $\cL^\aff_x$, i.e.,
\begin{gather*}
  \tp^\aff(a) \in \tS^\aff_x(\cL).
\end{gather*}
For a converse of this observation, we require the following construction (a much more general version of which will be treated in \autoref{sec:direct-integrals}).

\begin{defn}
  \label{defn:TypeRestrictionMap}
  We define the \emph{variable restriction} map
  \begin{gather*}
    \pi_x\colon \tS^\aff_{xy}(\cL) \rightarrow \tS^\aff_x(\cL)
  \end{gather*}
  as the dual of the inclusion $\cL^\aff_x \subseteq \cL^\aff_{xy}$.
\end{defn}

This map is continuous and affine.
If $p = \tp^\aff(a,b) \in \tS^\aff_{xy}(\cL)$, then $\pi_x(p) = \tp^\aff(a)$.
Given any $q = \tS^\aff_x(T)$, we have $q = \pi_x(p)$ if and only if $p(\varphi) = q(\varphi)$ for every affine formula $\varphi(x)$ (which we may also write as $\varphi(x,y)$, where $y$ is a tuple of \emph{dummy variables}).

\begin{lemma}
  \label{lemma:NaiveConvexLos}
  Let $M$ and $N$ be two $\cL$-structures.
  Define a new structure $K = \frac{1}{2}M \oplus \frac{1}{2}N$, whose underlying set is $M \times N$ as follows:
  \begin{itemize}
  \item Function symbols are interpreted coordinate-wise.
    That is to say that if $f$ is an $n$-ary function symbol, $a \in M^n$, $b \in N^n$, and $c = \bigl( (a_i,b_i) : i < n\bigr) \in K^n$, then
    \begin{gather*}
      f^K(c) = \bigl( f^M(a), f^N(b)\bigr).
    \end{gather*}
  \item If $P$ is an $n$-ary predicate symbol, and $a$, $b$ and $c$ are as above, then
    \begin{gather*}
      P^K(c) = \half P^M(a) + \half P^N(b).
    \end{gather*}
  \end{itemize}
  Then $K$ is an $\cL$-structure (in particular, the bounds and continuity moduli of symbols are respected).
  Moreover, for every formula $\varphi \in \cL^\aff_n$ and $a$, $b$ and $c$ as above:
  \begin{gather*}
    \varphi^K(c) = \half\varphi^M(a) + \half\varphi^N(b).
  \end{gather*}
\end{lemma}
\begin{proof}
  It is easy to check that $K = \frac{1}{2}M \oplus \frac{1}{2}N$ is indeed an $\cL$-structure.
  The moreover part is proved by induction on the construction of $\varphi$.
  More precisely, the collection of formulas that satisfy the moreover part:
  \begin{itemize}
  \item contains all atomic formulas, essentially by construction,
  \item is closed under affine combinations, and
  \item is closed under quantifiers (just combine witnesses from $M$ and from $N$ to get witnesses in $K$).
    \qedhere
  \end{itemize}
\end{proof}

\begin{prop}
  \label{prop:Type}
  The space of affine types $\tS^\aff_x(\cL)$ consists exactly of all types $\tp^\aff(a)$, where $a \in M^x$ for some $\cL$-structure $M$.
\end{prop}
\begin{proof}
  Let $C$ be the collection of types of the form $\tp^\aff(a)$, where $a \in M^x$ for some $M$.
  We have already observed that $C \subseteq \tS^\aff_x(\cL)$.
  Since every affine formula is also a formula of continuous logic, and using the compactness theorem for continuous logic, the set $C$ is closed.
  In addition, it is closed under the operation $(p,q) \mapsto \half p + \half q$, by \autoref{lemma:NaiveConvexLos}.
  It is therefore a closed convex subset.

  If $C \subsetneq \tS^\aff_x(\cL)$, then by Hahn--Banach, there exists an affine formula $\varphi \in \cL^\aff_x$ that is positive at every member of $C$, but strictly negative at some $p \in \tS^\aff(\cL)$.
  But then $\varphi^M(a) \geq 0$ for every structure $M$ and $a \in M^x$, so $\varphi \in (\cL^\aff_x)^+$.
  Since $p$ is positive, $p(\varphi) = \varphi(p) \geq 0$, a contradiction.
  Therefore, $C = \tS^\aff_x(\cL)$.
\end{proof}

We recall that $(\pi_x)_*$ denotes the pushforward map $\cM\big(\tS^\aff_{xy}(\cL)\big) \rightarrow \cM\big(\tS^\aff_x(\cL)\big)$ associated to the variable restriction map, which is continuous and affine. We recall as well that $R$ denotes the barycenter map, and that $R\big((\pi_x)_*\nu)\big) = \pi_x\big(R(\nu)\big)$.

\begin{lemma}
  \label{lemma:TypeConvexCombinationLift}
  Let $p \in \tS^\aff_{xy}(\cL)$ and $q = \pi_x(p) \in \tS^\aff_x(\cL)$.
  \begin{enumerate}
  \item
    \label{item:TypeConvexCombinationLiftMain}
    If $\mu$ is a probability measure on $\tS^\aff_x(\cL)$ with $R(\mu) = q$, then there exists a probability measure $\nu$ on $\tS^\aff_{xy}(\cL)$ such that $R(\nu) = p$ and $(\pi_x)_* \nu = \mu$.
  \item
    \label{item:TypeConvexCombinationLiftBoundary}
    If $\mu$ is a boundary measure, then $\nu$ can be taken a boundary measure.
  \item
    \label{item:TypeConvexCombinationLiftFinite}
    If $q = \sum_{i<n} \lambda_i q_i$, then there exist $p_i$ such that $q_i = \pi_x(p_i)$ and $p = \sum_{i<n} \lambda_i p_i$.
  \end{enumerate}
\end{lemma}
\begin{proof}
  For \autoref{item:TypeConvexCombinationLiftMain}, let
  \begin{gather*}
    C_1 = \bigl\{ \nu \in \cM\bigl( \tS^\aff_{xy}(\cL) \bigr) : (\pi_x)_* \nu = \mu \bigr\},
    \qquad
    C_2 = \bigl\{ R(\nu) : \nu \in C_1 \bigr\} \subseteq \tS^\aff_{xy}(\cL).
  \end{gather*}
  Then $C_1$ is a compact convex set, and therefore so is $C_2$.

  If $p \in C_2$, then we are done, so assume that $p \notin C_2$.
  Since $C_2$ is compact convex, by Hahn--Banach there exists an affine formula $\varphi(x,y)$ such that $\varphi(p) > 0$ and $\varphi(p') \leq 0$ for all $p' \in C_2$.
  Let $\psi(x) = \sup_y \varphi(x,y)$ and define
  \begin{gather*}
    X = \bigl\{ p' \in \tS^\aff_{xy}(\cL) : \varphi(p') = \psi(p') \bigr\}.
  \end{gather*}
  In other words, $X$ consists of all $\tp^\aff(a,b)$ such that $\varphi(a,b)$ is greatest possible given $a$.
  Then $X$ is a compact and convex subset of $\tS^\aff_{xy}(\cL)$, and by compactness, $\pi_x(X) = \tS^\aff_x(\cL)$.
  The map $(\pi_x)_*\colon\cM(X)\to \cM(\tS^\aff_x(\cL))$ is therefore surjective and, in particular, there exists $\nu_0 \in C_1$ that is supported on $X$.

  Then $p_0 = R(\nu_0) \in C_2 \cap X$ and $\pi_x(p_0) = q$, so
  \begin{gather*}
    \psi(q) = \psi(p_0) = \varphi(p_0) \leq 0 < \varphi(p).
  \end{gather*}
  Therefore, $\pi_x(p) \neq q$.

  For \autoref{item:TypeConvexCombinationLiftBoundary}, we can replace $\nu$ with a boundary measure $\nu' \succeq \nu$.
  Then $(\pi_x)_*(\nu') \succeq \mu$ and since $\mu$ is maximal, we have equality.

  For \autoref{item:TypeConvexCombinationLiftFinite}, we may assume that the $q_i$ are distinct. We then apply \autoref{item:TypeConvexCombinationLiftMain} to $\mu = \sum \lambda_i \delta_{q_i}$, to obtain a measure $\nu$.
  Then $\nu(A_i) = \lambda_i$, where $A_i = \pi_x^{-1}(q_i)$, and we can choose $p_i$ to be the barycenter of $\nu_{A_i}$.
\end{proof}

\begin{defn}
  \label{dfn:Condition}
  An \emph{affine condition} in the variables $x$ is an expression of the form $\varphi \geq \psi$, where $\varphi,\psi \in \cL^\aff_x$.

  If $M$ is a structure and $a \in M^x$, then $(M,a)$ is a \emph{model} of the condition $\varphi \geq \psi$, in symbols $(M,a) \models \varphi \geq \psi$, if $\varphi^M(a) \geq \psi^M(a)$.
  Similarly, $(M,a)$ is a model of a set of conditions $\Sigma$, in symbols $(M,a) \models \Sigma$, if it is a model of each condition in $\Sigma$.
  When $M$ is known from the context, we may omit it and write $a \models \varphi \geq \psi$ or $a \models \Sigma$.

  A family of conditions is \emph{consistent} if it admits a model.
\end{defn}

When $\Phi \subseteq \cL^\aff_x$, we denote by $\Phi \geq 0$ the collection of conditions $\{\varphi \geq 0 : \varphi \in \Phi\}$.
Since every set of conditions can be written in this form, we may freely restrict our attention to such.

\begin{defn}
  Let $\Phi, \Psi \subseteq \cL^\aff_x$.
  If every model of $\Phi \geq 0$ is also a model of $\Psi \geq 0$, then we say that $\Psi \geq 0$ is a \emph{consequence} of $\Phi \geq 0$, or that $\Phi \geq 0$ \emph{entails} $\Psi \geq 0$, in symbols
  \begin{gather*}
    \Phi \geq 0 \models \Psi \geq 0.
  \end{gather*}
  Two sets of conditions that have exactly the same models (i.e., that imply one another) are called \emph{equivalent}.
\end{defn}

Note that $\Phi \geq 0 \models \Psi \geq 0$ if and only if $\Phi \geq 0 \models \psi \geq 0$ separately for every $\psi \in \Psi$.

\begin{theorem}
  \label{th:CompactnessImplication}
  Let $\Phi \subseteq \cL^\aff_x$ and $\psi \in \cL^\aff_x$.
  Then the following are equivalent:
  \begin{enumerate}
  \item $\Phi \geq 0 \models \psi \geq 0$.
  \item The formula $\psi$ belongs to the closed convex cone generated by $\Phi \cup (\cL^\aff_x)^+$.
  \item For every $\varepsilon > 0$, there exist $n$ and (not necessarily distinct) $\varphi_i \in \Phi$ for $i < n$ such that
    \begin{gather}
      \label{eq:CompactnessImplication}
      \sum_{i<n} \varphi_i + 1 \geq 0 \  \models \  \psi + \varepsilon \geq 0.
    \end{gather}
  \item For every $\varepsilon > 0$ there exists $\Phi_0 \subseteq \Phi$ finite and $\delta > 0$ such that $\Phi_0 + \delta \geq 0 \models \psi + \varepsilon \geq 0$.
  \end{enumerate}
\end{theorem}
\begin{proof}
  \begin{cycprf}
  \item Given \autoref{prop:Type}, this is a special case of \autoref{lem:KadisonHahnBanachImplication}.
  \item The hypothesis implies that there exist formulas $\psi' \geq 0$ and $\psi'' = \frac{1}{k} \sum_{i<n} \varphi_i$, where $k\geq 1$ and $\varphi_i \in \Phi$, such that $\|\psi - \psi' - \psi''\| < \varepsilon/2$ (by repeating each $\phi_i$ several times, $\psi''$ can approximate an arbitrary linear combination with positive rational coefficients of elements of $\Phi$).
    In particular, $\psi \geq \psi' + \psi'' - \varepsilon/2 \geq \psi'' - \varepsilon/2$.
    Choose $\ell \in \N$ such that $1/(k\ell) < \eps/2$.
    Then
    \begin{gather*}
      \psi + \varepsilon
      \geq \frac{1}{k} \sum_{i<n} \varphi_i + \frac{\varepsilon}{2}
      \geq \frac{1}{k\ell} \Big( \sum_{i<n} \ell\varphi_i + 1 \Big).
    \end{gather*}
  \item Given \autoref{eq:CompactnessImplication}, we can take $\Phi_0 = \{ \varphi_i : i < n \}$ and $\delta = 1/n$ (or any $\delta>0$ if $n=0$).
  \item[\impfirst] Clear.
  \end{cycprf}
\end{proof}

\begin{ntn}
  \label{ntn:TypeImplication}
  If $p \in \tS^\aff_x(\cL)$ is a type in $x$, then $p$ \emph{satisfies} the condition $\varphi \geq \psi$, in symbols $p \models \varphi \geq \psi$, if $\varphi(p) \geq \psi(p)$.
  Similarly for a set of conditions.

  We sometimes also denote the relation $p = \tp^\aff(a)$ by $(M,a) \models p$, or $a \models p$.
\end{ntn}

\begin{remark}
  \label{rmk:TypeImplication}
  A type $p \in \tS^\aff_x(\cL)$ is determined by the collection of conditions $\bigl\{ \varphi \geq 0 : \varphi \in \cL^\aff_x, \, \varphi(p) \geq 0\bigr\}$.
  If we identify the two objects, then \autoref{ntn:TypeImplication} merely consists of special cases of previously introduced notations.
\end{remark}

The following is a form of the compactness theorem for affine logic, established in \cite{Bagheri2014a,Bagheri2014}.

\begin{theorem}[Bagheri's compactness theorem for affine logic]
  \label{th:HahnBanachCompactness}
  Let $\Phi \subseteq \cL^\aff_x$ be closed under addition.
  Then the following are equivalent:
  \begin{enumerate}
  \item The set of conditions $\Phi \geq 0$ is consistent.
  \item For every $\varphi \in \Phi$, the condition $\varphi \geq 0$ is consistent.
  \item For every $\varphi \in \Phi$, the condition $\varphi + \bOne \geq 0$ is consistent.
  \item We have $-\bOne \notin \Phi + (\cL^\aff_x)^+$.
  \end{enumerate}
\end{theorem}
\begin{proof}
  Given \autoref{prop:Type}, this is just \autoref{prop:HahnBanachCompactness}.
\end{proof}

\begin{cor}
  \label{cor:RelativeHahnBanachCompactness}
  Let $\Sigma$ be a collection of affine conditions and let $\Phi \subseteq \cL^\aff_x$ be closed under addition.
  Then the following are equivalent:
  \begin{enumerate}
  \item \label{i:cor:RelativeHahnBanachCompactness:1} $\Sigma \cup (\Phi \geq 0)$ is consistent.
  \item \label{i:cor:RelativeHahnBanachCompactness:2} For every $\varphi \in \Phi$, the set $\Sigma \cup \{\varphi \geq 0\}$ is consistent.
  \item \label{i:cor:RelativeHahnBanachCompactness:3} For every $\varphi \in \Phi$, the set $\Sigma \cup \{\varphi + \bOne \geq 0\}$ is consistent.
  \end{enumerate}
\end{cor}
\begin{proof}
  Top to bottom is immediate.
  For \autoref{i:cor:RelativeHahnBanachCompactness:3} $\Rightarrow$ \autoref{i:cor:RelativeHahnBanachCompactness:1}, we can rewrite $\Sigma$ as $\Psi \geq 0$ and apply \autoref{th:HahnBanachCompactness} to the closure of $\Phi \cup \Psi$ under addition.
\end{proof}

An \emph{affine sentence} is a formula with no free variables, i.e., a member of $\cL^\aff_0$.
A collection of conditions with no free variables is a \emph{theory}.
Sometimes we identify two equivalent theories -- this is really a matter of taste.
By \autoref{prop:KadisonHahnBanachImplicationDuality}, there exists a bijection between (equivalence classes of) theories and compact convex subsets of $\tS^\aff_0(\cL)$.
In particular, the inconsistent theory corresponds to the empty set, and maximal consistent theories to singletons.
Consequently, we identify a maximal consistent theory with the unique member of $\tS^\aff_0(\cL)$ that satisfies it, and call either one a \emph{complete affine $\cL$-theory}.

If $M$ is an $\cL$-structure, then the affine type of the empty tuple in $M$ is called the \emph{affine theory of $M$}, denoted $\Th^\aff(M)$.
By the previous paragraph, this is a complete affine theory.
More generally, if $\cM$ is a class of $\cL$-structures, we define $\Th^\aff(\cM)$, the affine theory of $\cM$, as the collection of all affine conditions (without variables) that hold in $\cM$, i.e., as the intersection of all affine theories of members of $\cM$.

\begin{defn}
  \label{dfn:AffineEverything}
  \
  \begin{enumerate}
  \item We say that two $\cL$-structures $M$ and $N$ are \emph{affinely equivalent}, in symbols $M \equiv^\aff N$, if $\Th^\aff(M) = \Th^\aff(N)$, i.e., if they are models of the exact same affine conditions without variables.
  \item A partial map $\theta\colon M \dashrightarrow N$ with domain $A \subseteq M$ is \emph{affine} if $\varphi^M(a) = \varphi^N(\theta a)$ for every affine $\cL$-formula $\varphi(x)$ and all $a \in A^x$.
    A total affine map $\theta\colon M \rightarrow N$ is called an \emph{affine embedding} (indeed, every affine map is an isometric embedding of its domain).
  \item Let $A \subseteq M$ be a subset.
    We define $\cL(A)$ to consist of $\cL$ augmented by a new constant symbol for each $a \in A$.
    The structure $M$ is implicitly expanded to an $\cL(A)$-structure by interpreting each $a \in A$ as itself.

    We call the affine theory of $M$ in this augmented language the \emph{affine diagram} of $A$ (in $M$).
    It will be denoted $D^\aff_A$, considering that $M$ will always be known from context.
  \end{enumerate}
\end{defn}

\begin{remark}
  \label{rem:affine-embedding}
  One might want to call these \emph{affinely elementarily equivalent} and \emph{affine elementary embedding}, respectively.
  We chose the shorter version since there can be no other meaning of \emph{affinely equivalent} and \emph{affine embedding}.
  Indeed, any map that preserves a set of formulas also preserves all their continuous combinations.
  The difference is that for ``affine'' we also consider quantification over affine combinations, while ``elementary'' means that quantification over any continuous combination is allowed.
\end{remark}

Note that if $M \preceq^\aff N$ and $a \in M^x$, then the affine types of $a$ in $M$ and in $N$ are the same.
Similarly, if $A \subseteq M \preceq^\aff N$, then the affine diagrams of $A$ in $M$ and in $N$ are the same.

The affine diagram of $A$ in $M$ carries the same information as the affine type in $M$ of any fixed enumeration of $A$.
In fact, since affine formulas are uniformly continuous, it carries the same information as the affine type of any enumeration of a dense subset of $A$.

Finally, $M \preceq^\aff N$ if and only if $M$ has the same affine diagram in $M$ and in $N$, i.e., if $M$ and $N$ are affinely equivalent as $\cL(M)$-structures.

\begin{prop}[Affine joint embedding]
  \label{prop:AffineJEP}
  Let $M$ and $N$ be two $\cL$-structures.
  Then $M \equiv^\aff N$ if and only if they admit affine embeddings into a third structure $K$.
\end{prop}
\begin{proof}
  One direction is clear.
  For the other, we may assume that $M \cap N = \emptyset$.
  Observe that a structure $K$ admits an affine embedding $M \hookrightarrow K$ if and only if it can be expanded to $\cL(M)$ to be a model of $D^\aff_M$.
  Therefore, it will suffice to show that the set of conditions $D^\aff_M \cup D^\aff_N$ is consistent.

  By \autoref{cor:RelativeHahnBanachCompactness}, it will suffice to show that for every $\cL(N)$-sentence $\varphi$, if $\varphi^N \geq 0$, then $D^\aff_M \cup \{\varphi + \bOne \geq 0\}$ is consistent.
  Indeed, we may rewrite $\varphi$ as $\psi(b)$, where $b \in N^x$ is the tuple of constants that appear in $\varphi$, and $\psi \in \cL^\aff_x$.
  Now, $\psi^N(b) \geq 0$ implies that $(\sup_x \psi)^N \geq 0$.
  But $M \equiv^\aff N$, so $(\sup_x \psi)^M \geq 0$.
  In particular, there exists $a \in M^x$ such that $\psi^M(a) \geq -1$.
  Interpreting the constant symbols $b$ in $M$ as $a$ (and other constants in $N$ arbitrarily), we obtain an expansion $M'$ of $M$ that models $D^\aff_M \cup \{\varphi + \bOne \geq 0\}$, completing the proof.
\end{proof}

\begin{cor}[Affine amalgamation]
  \label{cor:AffineAmalgamation}
  Let $M$ and $N$ be $\cL$-structures, and let $\theta\colon M \dashrightarrow N$ be a partial map.
  Then $\theta$ is affine if and only if $M$ and $N$ admit affine embeddings into a structure $K$, say $f\colon M \hookrightarrow K$ and $g\colon N \hookrightarrow K$, such that $f$ and $g \circ \theta$ agree on $\dom \theta$.
\end{cor}
\begin{proof}
  Let $A = \dom(\theta)$.
  Expand $M$ to $\cL(A)$ as usual, and $N$ by interpreting $a \in A$ as $\theta(a)$.
  The resulting $\cL(A)$-structures are affinely equivalent if and only if $\theta$ is affine, so we may apply \autoref{prop:AffineJEP}.
\end{proof}

Similarly to continuous logic, affine type spaces also carry a topometric structure. The compact topology we have already seen: it is given by the weak$^*$ topology inherited from $\cL^\aff_x$, or simply pointwise convergence on formulas. To distinguish it from the metric topology, which we are about to define, we call it the \emph{logic topology} and we denote it by $\tau$.
\begin{defn}
  \label{defn:AffineTypeDistance}
  Recall that the distance between finite tuples in a structure is given by the sum of the distances between their coordinates, as per \autoref{eq:TupleDistance}.
  Define the \emph{distance} between two types $p, q \in \tS^\aff_n(\cL)$ by
  \begin{gather*}
    \partial(p, q) \coloneqq \inf \, \bigl\{ d(a,b) : M \ \text{a structure}, \ a,b \in M^n, \ a \models p, \ b \models q \bigr\},
  \end{gather*}
  where the infimum of $\emptyset$ is $+\infty$.
\end{defn}

\begin{prop} The type distance $\dtp$ enjoys the following properties:
  \label{prop:AffineTypeDistance}
  \begin{enumerate}
  \item
    \label{item:AffineTypeDistanceAttained}
    If $\partial(p,q) < \infty$, then the infimum is attained.
  \item
    \label{item:AffineTypeDistanceGeneralizedDistance}
    The function $\partial$ is a generalized metric on $\tS^\aff_n(\cL)$ (i.e., it is a metric that is also allowed to take the value $\infty$).
  \item
    \label{item:AffineTypeDistanceFinite}
    If $p = \tp^\aff(a)$ and $q = \tp^\aff(b)$, where $a \in M^n$ and $b \in N^n$, then $\partial(p,q) < \infty$ if and only if $M \equiv^\aff N$.
  \item
    \label{item:AffineTypeDistanceTopometric}
    The $\dtp$-topology refines $\tau$ and $\dtp$ is $(\tau \times \tau)$-lower semi-continuous as a function $\tS^\aff_n(\cL)^2 \rightarrow [0,\infty]$.
  \item
    \label{item:AffineTypeDistanceConvex}
    The function $\partial\colon \tS^\aff_n(\cL)^2 \rightarrow [0,\infty]$ is convex.
  \item
    \label{item:AffineTypeDistanceVariableRestriction}
    The variable restriction maps $\pi_n\colon \tS^\aff_{n+m}(\cL) \to \tS^\aff_n(\cL)$ are $\partial$-contractive.
  \end{enumerate}
\end{prop}
\begin{proof}
  We prove \autoref{item:AffineTypeDistanceAttained} exactly as for continuous logic, by applying the compactness theorem for continuous logic to the set of conditions $p(x) \cup q(y) \cup \bigl\{ d(x,y) \leq s : s > r \bigr\}$.
  (We can also apply \autoref{cor:RelativeHahnBanachCompactness}: since $\partial(p,q) = r$, the set $p(x) \cup q(y) \cup \bigl\{ n \bigl(r-d(x,y)\bigr) + \bOne \geq 0 \bigr\}$ is consistent for all $n$.)

  It follows that $\partial(p,q) = 0$ if and only if $p = q$.
  Symmetry is also clear.
  Towards proving the triangle inequality, assume that $\partial(p_0,p_1) = r$ and $\partial(p_1,p_2) = s$, both finite.
  By \ref{item:AffineTypeDistanceAttained}, we may witness the first by $a_0,a_1 \in M^n$ (so $a_i \models p_i$ and $d(a_0,a_1) = r$), and the second similarly by $b_1,b_2 \in N^n$.
  Then $a_1$ and $b_1$ have the same type, $p_1$, so the partial map sending $a_1 \mapsto b_1$ is affine.
  By \autoref{cor:AffineAmalgamation}, we can find $K$ and $c_0,c_1,c_2 \in K^n$ that witness both distances.
  Then $d(c_0,c_2) \leq r+s$, and therefore $\partial(p_0,p_2) \leq r+s$, proving \autoref{item:AffineTypeDistanceGeneralizedDistance}.

  If $M \equiv^\aff N$, then $\partial(p,q) < \infty$ directly by \autoref{prop:AffineJEP}.
  Conversely, assume that $\partial(p,q) < \infty$.
  Then $p$ and $q$ can both be realized in some structure $K$, and it follows that $M \equiv^\aff K \equiv^\aff N$, proving \autoref{item:AffineTypeDistanceFinite}.

  Item \autoref{item:AffineTypeDistanceTopometric} is proved exactly as in continuous logic.
  Indeed, the $\dtp$-topology refines $\tau$ since every affine formula is uniformly continuous.
  For $r \geq 0$, let $C_r = \bigl\{ (p,q) : \partial(p,q) \leq r \bigr\}$ and $D_r = \bigl\{ q(x,y) \in \tS^\aff_{2n}(\cL) : q \models d(x,y) \leq r \bigr\}$.
  Then $C_r = (\pi_x \times \pi_y)(D_r)$, where $\pi_x, \pi_y\colon \tS^\aff_{2n}(\cL) \rightarrow \tS^\aff_n(\cL)$ are the two projection maps.
  Since $D_r$ is compact and $\pi_x \times \pi_y$ is continuous, $C_r$ is compact, so $\partial$ is $\tau$-lower semi-continuous.

  For \autoref{item:AffineTypeDistanceConvex}, assume that $\partial(p,q) < r$ and $\partial(p',q') < r'$.
  Let $p'' = \half p + \half p'$ and $q'' = \half q + \half q'$.
  Since $\partial$ is lower semi-continuous, it will suffice to prove that $\partial(p'',q'') < \half r + \half r'$.
  In particular, we may assume that $r,r' < \infty$.
  Choose $M,a,b$ that witness $\partial(p,q) < r$ (namely, $a,b \in M^n$, $a \models p$, $b \models q$, and $d(a,b) < r$), and similarly $M',a',b'$ for $\partial(p',q') < r'$.
  Let $N = \frac{1}{2}M \oplus \frac{1}{2}M'$, as in \autoref{lemma:NaiveConvexLos}, and let $a'' = (a,a') \in N$ and $b'' = (b,b')$.
  Then $N,a'',b''$ witness that $\partial(p'',q'') < \half r + \half r'$, completing the proof.

  The last point, \autoref{item:AffineTypeDistanceVariableRestriction}, is clear.
\end{proof}

\begin{remark}
  \label{rem:type-metric-inf-tuples}
  In \autoref{defn:AffineTypeDistance}, we only define the metric $\dtp$ for type spaces over a finite tuple of variables. For a general type space $\tS^\aff_I(\cL)$, we can define the \emph{$\dtp$-uniformity}, which is simply the inverse limit uniformity given by the representation $\tS^\aff_I(\cL) = \varprojlim_{J \sub I} \tS^\aff_J(\cL)$, where the inverse limit is taken over all finite $J \sub I$ and each $\tS^\aff_J$ is equipped with its $\dtp$-metric. We will only need this when $I = \N$, in which case the uniformity is metrizable. One compatible metric can be obtained as follows. Equip $M^\N$ with the metric given by $d(a, b) = \sum_i 2^{-i} d(a_i, b_i)$ and define $\dtp$ on $\tS^\aff_\N(\cL)$ as in \autoref{defn:AffineTypeDistance}. Then all items but \autoref{item:AffineTypeDistanceVariableRestriction} of \autoref{prop:AffineTypeDistance} continue to hold.
\end{remark}

Let $T$ be a theory.
For each tuple of variables $x$, we can equip $\cL^\aff_x$ with the relation of preorder modulo $T$, namely
\begin{equation}
  \label{eq:affine-order}
  \varphi \leq_T \psi \quad \Longleftrightarrow \quad T \models \varphi \leq \psi.
\end{equation}
Say that $\varphi \equiv_T \psi$ if $\varphi \geq_T \psi$ and $\varphi \leq_T \psi$.
We define $\cL^\aff_x(T) \coloneqq \cL^\aff_x / {\equiv}_T$, which, equipped with the induced order and the equivalence class of $\bOne$, is again an order unit space.
Its state space will be denoted
\begin{gather*}
  \tS^\aff_x(T) \coloneqq \tS\bigl( \cL^\aff_x(T) \bigr),
\end{gather*}
and will be called the \emph{space of types of $T$} in the variables $x$.

This construction of $\cL^\aff_x(T)$ is the same as the one in \autoref{prop:OrderUnitSpaceQuotient} with respect to the closed positive cone $P_x(T) = \{\varphi \in \cL^\aff_x : T \models \varphi \geq 0 \}$: i.e., the quotient of $\cL^\aff_x$ by the ideal $P_x(T) \cap -P_x(T)$.
In particular, composition with the quotient map $\cL^\aff_x \to \cL^\aff_x(T)$ allows us to identify the space of types of $T$ with the space of types that imply $T$ (which is just $\bigl[P_x(T)\bigr]_+$ in the notations of \autoref{prop:OrderUnitSpaceQuotient}):
\begin{gather*}
  \tS^\aff_x(T)
  = \bigl\{ p \in \tS^\aff_x(\cL) : p \models T \bigr\}
  = \bigl\{ \tp^\aff(a) : a \in M^x, \, M \models T \bigr\}.
\end{gather*}

\begin{lemma}
  \label{l:realized-convex-dense}
  Let $T$ be a complete affine theory and let $M \models T$.
  Then
  \begin{gather*}
    \tS^\aff_x(T) = \cco \set{\tp^\aff(a) : a \in M^x}.
  \end{gather*}
\end{lemma}
\begin{proof}
  Suppose not.
  Then by Hahn--Banach, there exists $p \in \tS^\aff_x(T)$ and a formula $\varphi$ such that $\varphi(p) = 0$ and $\varphi(a) \geq 1$ for all $a \in M^x$.
  Let $\psi = \inf_x \varphi(x)$.
  Then $\psi^M \geq 1$ but $\psi^N \leq 0$ for any model realizing $p$.
  This contradicts the completeness of $T$.
\end{proof}

\begin{defn}
  \label{defn:LanguageDensityCharacter}
  If $T$ is an affine theory, we define the \emph{density character of the language} modulo $T$, as
  \begin{equation*}
    \density_T(\cL) = \sup_n \density\bigl( \cL^\aff_n(T) \bigr).
  \end{equation*}
  The theory $T$ is said to have a \emph{separable language} if $\density_T(\cL) = \aleph_0$.
\end{defn}

If $T$ has a separable language, then $\tS^\aff_I(T)$ is metrizable for any countable~$I$.

\begin{remark}
  \label{remark:LanguageDensityCharacter}
  Let $M \models T$, and let $A \subseteq M$ be a subset.
  Assume that $\fd(M) \geq \fd_T(\cL)$.
  Then $M$ admits an affine substructure $N \preceq^\aff M$ that contains $A$ such that $\fd(N) = \fd(A) + \fd_T(\cL)$.
  Indeed, modulo $T$, we can replace $\cL$ with a sublanguage of cardinal $\fd_T(\cL)$, and apply \autoref{prop:DownwardLS}.
\end{remark}

Observe that if $T$ is an affine theory, then $\pi_x\colon \tS^\aff_{xy}(\cL) \rightarrow \tS^\aff_x(\cL)$ restricts to a surjective, continuous, affine map $\tS^\aff_{xy}(T) \rightarrow \tS^\aff_x(T)$, which we also denote by $\pi_x$.

\begin{prop}
  \label{prop:VariableRestrictionBoundary}
  Let $T$ be an affine theory, and let $x$ and $y$ denote tuples of variables.
  \begin{enumerate}
  \item
    \label{item:VariableRestrictionHat}
    For every bounded function $f \colon \tS^\aff_x(T) \to \bR$ we have $\widehat{f \circ \pi_x} = \hat{f} \circ \pi_x$ on $\tS^\aff_{xy}(T)$, where the hat denotes the concave envelope, in the sense of \autoref{defn:ConcaveEnvelope}.
  \item
    \label{item:VariableRestrictionBoundary}
    If $\mu$ is a boundary measure on $\tS^\aff_{xy}(T)$, then $(\pi_x)_* \mu$ is a boundary measure on $\tS^\aff_x(T)$.
  \end{enumerate}
\end{prop}
\begin{proof}
  For \autoref{item:VariableRestrictionHat}, the inequality $\hat{f} \circ \pi_x \geq \widehat{f \circ \pi_x}$ is easy.
  For the opposite inequality, consider an affine formula $\varphi(x, y)$ such that $\varphi \geq f \circ \pi_x$, and let $\psi(x) = \inf_y \varphi(x,y)$.
  Then $\psi \geq f$, and so $\psi \geq \hat{f}$.
  Therefore, $\varphi \geq \psi \circ \pi_x \geq \hat{f} \circ \pi_x$.
  Since formulas are uniformly dense in the space of all affine continuous functions, we conclude that $\widehat{f \circ \pi_x} \geq \hat{f} \circ \pi_x$.

  For \autoref{item:VariableRestrictionBoundary}, let $\nu = (\pi_x)_* \mu$.
  We will use the characterization given by \autoref{p:Mokobodzki}.
  Let $f \colon \tS^\aff_x(T) \to \bR$ be a continuous function.
  Since $\mu$ is a boundary measure, we have
  \begin{equation*}
    \nu(f) = \mu(f \circ \pi_x) = \mu(\widehat{f \circ \pi_x}) = \mu(\hat f \circ \pi_x) = \nu(\hat f).
  \end{equation*}
  Since $f$ was arbitrary, $\nu$ is a boundary measure.
\end{proof}

\begin{defn}[Types with parameters]
  \label{defn:TypeWithParameters}
  Let $M$ be a structure and $A \subseteq M$.
  \begin{enumerate}
  \item We define the \emph{space of affine types with parameters in $A$} in the variables $x$ as the space of types of the theory $D^\aff_A$, namely, of the affine diagram of $A$, in the language $\cL(A)$:
    \begin{gather*}
      \tS^\aff_x(A) \coloneqq \tS^\aff_x(D^\aff_A).
    \end{gather*}
  \item If $N \succeq^\aff M$ and $a \in N^x$, then the \emph{affine type of $a$ over $A$}, denoted $\tp^\aff(a/A)$, is the affine type of $a$ in the language $\cL(A)$.
    In other words, it is the function that associates to every formula $\varphi(x) \in \cL(A)^\aff_x$ the value $\varphi^N(a)$.
    Equivalently, it associates to every $\varphi(x,b)$ the value $\varphi^N(a,b)$, where $\varphi(x,y) \in \cL^\aff_{xy}$ and $b \in A^y$.
  \end{enumerate}
\end{defn}

The diagram $D^\aff_A$, and therefore the type space $\tS^\aff_x(A)$, depend implicitly on the ambient structure $M$.
However, if $N \succeq^\aff M$, then $A$ has the same affine diagram in both.
Therefore, as long as we restrict our attention to affine extensions $N \succeq^\aff M$, the space $\tS^\aff_x(A)$ is well-defined.
Since $N \models D^\aff_A$, for every $a \in N^x$ we have that $\tp^\aff(a/A) \in \tS^\aff_x(A)$.

\begin{prop}
  \label{prop:TypeWithParameters}
  Let $M$ be a structure, and $A \subseteq M$.
  Then
  \begin{gather*}
    \tS^\aff_x(A)
    =
    \bigl\{ \tp^\aff(a/A) : a \in N^x, \ N \succeq^\aff M \bigr\}.
  \end{gather*}
\end{prop}
\begin{proof}
  The inclusion $\supseteq$ has been observed above.
  For the converse, assume that $p \in \tS^\aff_x(A)$.
  Then there exists an $\cL(A)$-structure, call it $K$, that is a model of $D^\aff_A$, and a tuple $a \in K^x$, such that $p = \tp^\aff_x(a)$.
  Then $M$ and $K$ are affinely equivalent in the language $\cL(A)$.
  By \autoref{cor:AffineAmalgamation}, we can find an affine extension $N \succeq^\aff M$, together with an embedding $\theta\colon K \hookrightarrow N$ that is affine in the language $\cL(A)$.
  Let $b = \theta(a) \in N^x$.
  Then
  \begin{gather*}
    p = \tp^\aff_{\cL(A)}(a) = \tp^\aff_{\cL(A)}(b) = \tp^\aff(b/A),
  \end{gather*}
  as desired.
\end{proof}

\begin{defn}
  \label{defn:TypeOfModel}
  Let $x = (x_i : i \in I)$ and let $y$ denote a single variable.
  We say that a type $p \in \tS^\aff_x(\cL)$ has the \emph{Tarski--Vaught} property, or that it \emph{enumerates a model}, if for every affine formula $\varphi(x, y)$ we have
  \begin{equation*}
    \bigl( {\sup}_y \ \varphi(x, y) \bigr)^p = \sup_{i \in I} \varphi(x, x_i)^p.
  \end{equation*}
  The collection of all Tarski--Vaught types in $x$ will be denoted by $\tS^{\fM,\aff}_x(\cL)$.
  When $T$ is a theory, we define $\tS^{\fM,\aff}_x(T) \coloneqq \tS^{\fM,\aff}_x(\cL) \cap \tS^\aff_x(T)$.
\end{defn}

\begin{remark}
  \label{remark:TypeOfModel}
  Let $M$ be an $\cL$-structure, $a \in M^I$, $p = \tp^\aff(a)$ and $A = \{a_i : i \in I\}$.
  Then the following are equivalent:
  \begin{enumerate}
  \item $p \in \tS^{\fM,\aff}_I(\cL)$.
  \item The set $A$ satisfies the \emph{Tarski--Vaught test} of \autoref{p:Tarski-Vaught}.
  \item We have $\overline{A} \preceq^\aff M$.
  \end{enumerate}
  In particular, if $p \in \tS^{\fM,\aff}_I(T)$, then every realization of $p$ is dense in a model of $T$, whence the terminology.
\end{remark}

Because of measurability issues, when either the language or $I$ is uncountable, we may require a more fine-tuned variant.

\begin{ntn}
  \label{ntn:TypeOfModelSingleFormula}
  Let $x = (x_i : i \in I)$ and let $y$ denote a single variable.
  Given a formula $\varphi(x,y)$ and $J \subseteq I$, let us denote by $X_{\varphi,J} \subseteq \tS^\aff_I(\cL)$ the collection of types $p$ that satisfy
  \begin{gather*}
    \bigl( {\sup}_y \ \varphi(x, y) \bigr)^p = {\sup}_{i \in J} \varphi(x, x_i)^p.
  \end{gather*}
\end{ntn}

\begin{lemma}
  \label{lem:TypeOfModelSingleFormula}
  If $J \subseteq I$ is countable, then $X_{\varphi,J}$ is a Baire $G_\delta$ subset of $\tS^\aff_I(\cL)$.
  In addition, for every $p \in \tS^{\fM,\aff}_I(\cL)$ and every formula $\varphi(x, y)$, there exists a countable $I_{\varphi} \sub I$ such that $p \in X_{\varphi, I_{\varphi}}$.
\end{lemma}
\begin{proof}
  A type belongs to $X_{\varphi,J}$ if and only if for every $\varepsilon > 0$, there exists $i \in J$ such that $\bigl( {\sup}_y \ \varphi(x, y) \bigr)^q < \varphi(x, x_i)^q + \varepsilon$.
  The last condition is open and $F_\sigma$; a countable union of such is open Baire; and a countable intersection of those is $G_\delta$ and Baire.

  The second assertion holds by definition.
\end{proof}

\begin{prop}
  \label{prop:TypeOfModelPreMean}
  Let $\mu$ be a Radon probability measure on $\tS^\aff_I(\cL)$ such that $p = R(\mu) \in \tS^{\fM,\aff}_I(\cL)$.
  \begin{enumerate}
  \item
    \label{item:TypeOfModelPreMeanSingleFormula}
    Let $\varphi(x,y)$ be a formula, $x$ indexed by $I$ and $y$ a singleton, and let $I_\varphi \subseteq I$ be countable such that $p \in X_{\varphi,I_\varphi}$.
    Then $\mu(X_{\varphi,I_\varphi}) = 1$.
    In particular, all atoms of $\mu$ belong to $\bigcap_\varphi X_{\varphi,I_\varphi} \subseteq \tS^{\fM,\aff}_I(\cL)$.
  \item
    \label{item:TypeOfModelPreMeanCountable}
    If $\cL$ and $I$ are countable, then $\tS^{\fM,\aff}_I(\cL)$ is $G_\delta$ in $\tS^\aff_I(\cL)$, and $\mu\bigl( \tS^{\fM,\aff}_I(\cL) \bigr) = 1$.
  \end{enumerate}
\end{prop}
\begin{proof}
  For \autoref{item:TypeOfModelPreMeanSingleFormula}, assume towards a contradiction that $\delta = 1 - \mu(X_{\varphi,I_\varphi}) > 0$.
  Then there exists $\varepsilon > 0$ such that
  \begin{gather*}
    {\sup}_y \, \varphi(x,y) \geq {\sup}_{i \in I_\varphi} \, \varphi(x,x_i) + \varepsilon
  \end{gather*}
  with probability at least $\delta / 2$.
  On the other hand, by definition,
  \begin{gather*}
    {\sup}_y \, \varphi(x,y) \geq {\sup}_{i \in I_\varphi} \, \varphi(x,x_i)
  \end{gather*}
  holds throughout.
  Therefore,
  \begin{align*}
    \left( {\sup}_y \, \varphi(x,y)\right)^p
    & = \int \left( {\sup}_y \, \varphi(x,y) \right)^q \, \ud \mu(q)
    \\
    & \geq \varepsilon \delta / 2 + \int {\sup}_{i \in I_\varphi} \, \varphi(x,x_i)^q \, \ud \mu(q)
    \\
    & \geq \varepsilon\delta / 2 + {\sup}_{i \in I_\varphi} \, \int \varphi(x,x_i)^q \, \ud \mu(q)
    \\
    & = \varepsilon\delta / 2 + {\sup}_{i \in I_\varphi} \, \varphi(x,x_i)^p.
  \end{align*}
  This contradicts our hypothesis that $p \in X_{\varphi,I_\varphi}$.

  Assertion \autoref{item:TypeOfModelPreMeanCountable} is an immediate consequence of \autoref{item:TypeOfModelPreMeanSingleFormula} and \autoref{lem:TypeOfModelSingleFormula}.
\end{proof}

\begin{remark}
  \label{remark:TypeOfModelPreMeanUncountable}
  When $I$ is uncountable and $|I| \geq |\cL|$, one can show that every non-empty Baire set $X \subseteq \tS^\aff_I(\cL)$ meets $\tS^{\fM,\aff}_I(\cL)$.
  Therefore, every Radon probability measure $\mu$ on $\tS^\aff_I(\cL)$ concentrates on $\tS^{\fM,\aff}_I(\cL)$ in the sense of \autoref{df:measure-concentr}. See also \autoref{cor:TypeOfExtremalModelVeryDense} for a similar statement.
\end{remark}

Assume that $X$ is a compact convex set and $x \in X$.
Let $F$ be the intersection of all faces of $X$ that contain $x$.
Then $F$ is the smallest face containing $x$, and is called the face \emph{generated} by $x$.
Equivalently, $F$ consists of all $y \in X$ such that $x = \lambda y + (1-\lambda) z$ for some $z \in X$ and $0 < \lambda \leq 1$ (see \cite[II.5]{Alfsen1971}).

\begin{cor}
  \label{cor:FrecciaRossa}
  Let $T$ be an affine theory.
  Then the set $\tS^{\fM,\aff}_I(T)$ is a union of faces of $\tS^\aff_I(T)$.
  Equivalently, if $p \in \tS^{\fM,\aff}_I(T)$, then the face generated by $p$ is contained in $\tS^{\fM,\aff}_I(T)$.
\end{cor}
\begin{proof}
  Let $q$ belong to the face generated by $p$ in $\tS^\aff_I(T)$.
  In other words, assume that $p = \lambda q + (1-\lambda) q'$, where $0 < \lambda \leq 1$ and $q,q' \in \tS^\aff_I(T)$.
  Let $\mu = \lambda \delta_q + (1-\lambda) \delta_{q'}$, viewed as a Radon measure, so $p = R(\mu)$.
  By \autoref{prop:TypeOfModelPreMean}\autoref{item:TypeOfModelPreMeanSingleFormula}, $q \in \tS^{\fM,\aff}_I(\cL)$, and thus $q \in \tS^{\fM,\aff}_I(T)$. (In fact, $T$ plays absolutely no role here.)
\end{proof}

We end this section with a few remarks regarding the quantifier-free fragment in affine logic.
Given a language $\cL$, we let $\cL^{\qf,\aff}_x$ be the set of quantifier-free affine $\cL$-formulas in the variables $x$. Let $T$ be an affine theory which will be fixed for the remainder of the section.
We let
\begin{equation*}
  \tS^\qf_x(T) \coloneqq \tS({\cL^{\qf,\aff}_x/\equiv_T})
\end{equation*}
be the set of \emph{quantifier-free types of $T$} in the variables $x$.
It is a compact convex set.
Since quantifier-free types are determined by the values they give to atomic formulas, $\tS^\qf_x(T)$ (for an affine theory $T$) agrees with the space of quantifier-free types of $T$ in the sense of continuous logic (but the convex structure only applies to values given to affine quantifier-free formulas).

We have a natural projection
\begin{gather}
  \label{eq:TypePorjectionQF}
  \rho^\qf_x \colon \tS^\aff_x(T) \to \tS^\qf_x(T),
\end{gather}
which is continuous, affine and surjective.

\begin{defn}
  \label{df:quantifier-elim}
  We will say that $T$ has \emph{affine quantifier elimination} if for every (finite) $x$, the map $\rho^\qf_x$ is injective.
\end{defn}

\begin{remark}
  \label{rmk:AffineQE}
  By \autoref{prop:Stone-Weierstrass-affine}, the map $\rho^\qf_x$ is injective if and only if every affine formula in the variables $x$ is a uniform limit (modulo $T$) of quantifier-free affine formulas.
\end{remark}

The following was observed by Bagheri and Safari in \cite[Prop.~4.10]{Bagheri2014a}, with a different argument.

\begin{lemma}
  \label{lem:Continuous-vs-AffineQE}
  If $T$ is an affine theory that eliminates quantifiers in the sense of full continuous logic, then it also has affine quantifier elimination.
\end{lemma}
\begin{proof}
  We have a sequence of maps
  \begin{gather*}
    \tS^\cont_x(T) \rightarrow \tS^\aff_x(T) \overset{\rho^\qf_x}{\longrightarrow} \tS^\qf_x(T),
  \end{gather*}
  where $\tS^\cont_x(T)$ is the usual type space in continuous logic (see \autoref{sec:affine-part-continuous-theory}). The first is surjective, and the composition is injective by hypothesis, so $\rho^\qf_x$ is injective.
\end{proof}

\begin{lemma}\label{l:rho-injective-Smod}
  For any tuple of variables $x$, the restriction of $\rho^\qf_x$ to $\tS^{\fM,\aff}_x(T)$ is injective.
\end{lemma}
\begin{proof}
  If $p,q$ enumerate models and $\rho^\qf_x(p)=\rho^\qf_x(q)$, then the corresponding models are isomorphic, so $p=q$.
\end{proof}


\section{Definability}
\label{sec:Definability}

Throughout this section, we consider definable predicates and definable sets \emph{without parameters}, inside a family of structures, or relative to an affine theory.
Various cases of definability with parameters are subsumed in the parameter-free version.
For example, in order to consider definability with a parameter that may vary, we expand the language with constants representing the parameter, and consider definability in structures in this expanded language.
Similarly, if we wish to consider definability over a fixed parameter set $A \subseteq M$, we may work in the expanded language $\cL_A$, and relative to the affine diagram $D^\aff_A$.

\begin{defn}
  \label{defn:DefinablePredicate}
  An \emph{affine definable predicate} on a structure $M$, in the variables $x$, is a function $\psi\colon M^x \rightarrow \bR$ that is a uniform limit of (interpretations of) affine formulas in $M$.
  As for formulas, we may make the variables explicit using the notation $\psi(x)$.

  More generally, let $\cM$ be a family of structures, and for each $M \in \cM$, let $\psi^M\colon M^x \rightarrow \bR$ be a function.
  Then the family $\psi = (\psi^M: M \in \cM)$ is an affine definable predicate in $\cM$ if there exists a sequence of affine formulas $(\varphi_n)$ in the variables $x$ such that $\varphi_n^M \rightarrow \psi^M$ uniformly for every $M \in \cM$, at a rate that does not depend on $M$.
  When $\cM$ is the class of all models of an affine theory $T$, we say that $\psi$ is an affine definable predicate of $T$.

  When we wish to insist on the fact that a family of predicates is defined by the same sequence of formulas, especially when these depend on a parameter that may vary, we may refer to it as \emph{uniformly definable}.
\end{defn}

\begin{remark}
  \label{remark:DefinablePredicate}
  An affine definable predicate $\psi(x)$ in a family $\cM$ is the limit of a sequence of formulas $(\varphi_n)$ whose interpretations are uniformly Cauchy in $\cM$, namely, such that $|\varphi_n^M - \varphi_m^M| \leq \varepsilon$ for every $M \in \cM$ and $n,m \geq N_\varepsilon$, where $N_\varepsilon$ only depends on $\varepsilon$.
  Let $T = \Th^\aff(\cM)$.
  Then $T$ implies the family of conditions
  \begin{gather*}
    \inf_x \ \bigl( \varphi_n(x) - \varphi_m(x) + \varepsilon \bigr) \geq 0, \qquad n,m \geq N_\varepsilon.
  \end{gather*}
  It follows that the sequence $(\varphi_n)$ converges uniformly in every model of $T$.
  By a similar argument, any other sequence that converges uniformly to $\psi$ in $\cM$ converges to the same limit as $(\varphi_n)$ in every model of $T$.

  Therefore, any affine definable predicate in $\cM$ extends in a unique fashion to all models of $T = \Th^\aff(\cM)$.
\end{remark}

\begin{lemma}
  \label{lemma:DefinablePredicateTypes}
  Let $T$ be an affine theory.
  Every affine definable predicate $\psi(x)$ on models of $T$ factors uniquely as $\hat{\psi} \circ \tp^\aff$, where $\hat{\psi}\colon \tS^\aff_x(T) \rightarrow \bR$ is affine and continuous.
  Conversely, every such function on $\tS^\aff_x(T)$ arises from a unique affine definable predicate.
\end{lemma}
\begin{proof}
  Assume that $\varphi_n(x)$ are affine formulas such that $\varphi_n \rightarrow \psi$ uniformly.
  Each formula factors as $\hat{\varphi}_n \circ \tp^\aff$, where $\hat{\varphi}_n\colon \tS^\aff_x(T) \rightarrow \bR$ is continuous and affine.
  Therefore, $\psi = \hat{\psi} \circ \tp^\aff$, where $\hat{\varphi}_n \rightarrow \hat{\psi}$ uniformly on $\tS^\aff_x(T)$, and $\hat{\psi}$ is continuous and affine.
  The converse holds since the (functions defined by) formulas on $\tS^\aff_x(T)$ are dense among all continuous affine functions.
\end{proof}

When no ambiguity may arise, we are not going to be very careful about the distinction between an affine definable predicate $\psi$ (on a family of structures, or on all models of $T$), its interpretations $\psi^M$ in the various structures, and the corresponding continuous affine function $\hat{\psi}$ on $\tS^\aff_x(T)$, denoting all by the same letter $\psi$.

Affine definable predicates can be treated, for essentially all intents and purposes, as affine formulas.
Being uniform limits of affine formulas, they are bounded and uniformly continuous.
Clearly, we may apply affine connectives to definable predicates, and obtain new ones.
Similarly, if $\varphi_n(x,y) \rightarrow \psi(x,y)$ uniformly, then $\inf_x \varphi_n(x,y) \rightarrow \inf_x \psi(x,y)$ at the same rate, so we may also apply quantifiers to affine definable predicates.

Finally, we may simply \emph{name} an affine definable predicate $\psi(x)$ in the language.
More precisely, assume that $\psi(x)$ is an affine definable predicate in a theory $T$, in a language $\cL$.
Let $\cL'$ consist of $\cL$ together with a new predicate symbol $P_\psi$, with the appropriate arity, bound, and continuity modulus.
Let $T'$ consist of $T$ together with affine axioms stating that $P_\psi$ is to be interpreted as $\psi$.
Then $T'$ is an \emph{affine definitional expansion} of $T$, and there is no real difference between the class of models of $T$ (in $\cL$) and that of $T'$ (in $\cL'$).

One difference between definable predicates and formulas is that definable predicates may depend (essentially) on countably many variables rather than just finitely many. We note that in that case, we can also quantify over countable tuples of variables, i.e., if $\psi(x, y)$ is a definable predicate then $\inf_y \psi(x, y)$ is also a definable predicate without restrictions on the lengths of $x$ and $y$. Indeed, if $y_1, y_2, \ldots$ enumerate all of the $y$-variables that $\psi$ depends on and $\psi_n = \inf_{y_1, \ldots, y_n} \psi(x, y)$, then $\psi_n \to \inf_y \psi(x, y)$ uniformly.

\begin{defn}
  \label{defn:DefinableSet}
  Let $M$ be a structure.
  A non-empty closed subset $D \subseteq M^I$ is called \emph{affinely definable} if we can quantify over it in affine logic.
  By that we mean that if $x = (x_i : i \in I)$, and $\varphi(x,y)$ is an affine definable predicate on $M$ then the expression $\inf_{x \in D} \, \varphi(x,y)$ is again an affine definable predicate in $M$ (equivalently, with $\sup$).

  More generally, let $\cM$ be a family of structures, and let $D = (D^M : M \in \cM)$ be a family of closed non-empty subsets $D^M \subseteq M^I$.
  Then it is a \emph{uniformly affinely definable} family if for every affine definable predicate $\varphi(x,y)$ in $\cM$, the expression $\inf_{x \in D} \, \varphi(x,y)$, interpreting $D$ as $D^M$ in each $M \in \cM$, coincides with an affine definable predicate in $\cM$.
\end{defn}

If $\varphi_n(x,y)$ are affine formulas such that $\varphi_n(x,y) \rightarrow \varphi(x,y)$ uniformly in $\cM$, then $\inf_{x \in D} \varphi_n(x,y) \rightarrow \inf_{x \in D} \varphi(x,y)$ at the same rate.
It follows that we may require $\varphi$ to be an affine formula without changing \autoref{defn:DefinableSet}.
Since an affine formula only depends on finitely many variables, a set $D$ of $I$-tuples is affinely definable if and only if its projection to every finite $I_0 \subseteq I$ is affinely definable, and similarly for families of sets.
Therefore, we are going to be interested mostly in affinely definable sets of finite tuples.

For a criterion for definability that is easier to verify, we need a definable metric on $I$-tuples, i.e., a metric given by an affine definable predicate.
For finite $I$, we usually take $d(x,y) = \sum_{i \in I} d(x_i,y_i)$, and for $I = \bN$, we can use $d(x,y) = \sum_{i} 2^{-i} d(x_i,y_i)$.

If $(X, d)$ is a metric space and $Y \sub X$ is non-empty, for $x \in X$, we denote $d(x, Y) \coloneqq \inf_{y \in Y} d(x, y)$.

\begin{prop}
  \label{prop:DefinableSet}
  Let $D \subseteq M^I$ be a non-empty closed set, and let $d$ be any definable distance predicate on $I$-tuples.
  Then the following are equivalent:
  \begin{enumerate}
  \item
    The set $D$ is affinely definable.
  \item
    The function $d(\cdot, D)$ is an affine definable predicate in $M$.
  \item
    For every $\varepsilon > 0$, there exists an affine formula $\psi(x)$ such that $\psi(x) \geq 0$ in $M$, $\psi(x) < \varepsilon$ if $x \in D$, and $\psi(x) < 1$ implies $d(x,D) < \varepsilon$ in $M$.
  \end{enumerate}
  Similarly, for a uniformly affinely definable family $(D^M : M \in \cM)$.
\end{prop}
\begin{proof}
  \begin{cycprf*}
  \item
    Since $d(x,D) = \inf_{y \in D} \, d(x,y)$.
  \item[\impnext]
    Immediate.
  \item[\impfirst]
    Let $\varphi(x, y)$ be an affine formula, and let $\rho(y)$ denote the expression $\inf_{x \in D} \, \varphi(x,y)$. Our goal is to show that $\rho$ is an affine definable predicate on $M$.
    Let $R$ be large enough that $|\varphi(x,y)| \leq R$ always.

    Let $\varepsilon > 0$.
    There exists $\delta > 0$ such that $d(x,x') < \delta$ implies $|\varphi(x,y) - \varphi(x',y)| < \varepsilon$.
    There exists an affine formula $\psi(x) \geq 0$ such that $\psi(x) \leq \varepsilon/2R$ on $D$, and $\psi(x) < 1$ implies $d(x,D) < \delta$.
    Define
    \begin{gather*}
      \chi(y) = \inf_x \, \bigl( \varphi(x,y) + 2R \psi(x) \bigr).
    \end{gather*}

    On the one hand, considering $x \in D$, we have $\chi(y) \leq \rho(y) + \varepsilon$.
    For the opposite inequality, let $a \in M^x$ and $b \in M^y$.
    If $\psi(a) \geq 1$, then $\varphi(a,b) + 2R \psi(a) \geq \varphi(a',b) \geq \rho(b)$ for any $a' \in D$.
    On the other hand, if $\psi(a) < 1$, then there exists $a' \in D$ such that $d(a,a') < \delta$, so $\varphi(a,b) + 2R \psi(a) + \varepsilon \geq \varphi(a',b) \geq \rho(b)$.
    Therefore $\rho(y) \leq \chi(y) + \varepsilon$.

    We conclude that $\rho$ is a uniform limit of affine formulas.
  \end{cycprf*}
  The same argument works uniformly in a family of structures $\cM$.
\end{proof}

Intuitively, the possibility to quantify over a definable set $D$ makes it similar to a sort.
Moreover, just as an affine definable predicate can be named in the language by a new predicate symbol, an affinely definable set $D$ can be named by a new sort.
This is more involved, but in a nutshell, we add a sort to $\cL$, and axioms to $T$ asserting that the new sort is to be in isometric bijection with $D$ (using function symbols that we should also add).
We leave the technical details to the reader.

\begin{remark}
  \label{remark:DefinableSetInfinitary}
  When $I = \bN$, we obtain two equivalent criteria for definability.
  On the one hand, $D$ is affinely definable if $d(x, D)$ is an affine definable predicate, where $d(x,y) = \sum_k 2^{-k} d(x_k,y_k)$, say.
  On the other hand, it is affinely definable if for every $n$, its projection to the first $n$ coordinates is affinely definable, i.e., if $d_n(x,D)$ is an affine definable predicate, where $d_n(x,y) = \sum_{k<n} d(x_k,y_k)$.

  When $I$ is uncountable, the ``definable uniform structure'' on $M^I$ requires an uncountable family of definable pseudometrics, so the first criterion no longer makes sense, but the second one still holds.
\end{remark}

If $\psi(x)$ is a definable predicate and $M$ is a structure, the \emph{zero-set} of $\psi$ in $M$ is $(\psi^M)^{-1}(\set{0})$.
\begin{lemma}
  \label{lem:DefinableSetDistancePredicate}
  Let $(D^M : M \in \cM)$ be a uniformly affinely definable family of sets, and let $\psi(x)$ be the affine definable predicate $d(x,D)$.
  Then the following hold in $\cM$:
  \begin{gather}
    \label{eq:DefinableSetDistancePredicatePositiveLipschitz}
    \inf_x \, \psi(x) \geq 0,
    \qquad
    \inf_{x,y} \, \bigl( d(x,y) - \psi(x) + \psi(y) \bigr) \geq 0, \\
    \label{eq:DefinableSetDistancePredicateRealized}
    \qinf_x \sup_y \, \bigl( \psi(x) - 2 \psi(y) - d(x,y) \bigr) \geq 0.
  \end{gather}
  Conversely, assume that $\psi$ is an affine definable predicate in $\cM$, and that it satisfies \autoref{eq:DefinableSetDistancePredicatePositiveLipschitz} and \autoref{eq:DefinableSetDistancePredicateRealized}.
  Then the zero-sets of $\psi$ form a uniformly affinely definable family of sets in all models of $\Th^\aff(\cM)$, to which the distance is given by $\psi$.
\end{lemma}
\begin{proof}
  For the first assertion, any function defining the distance to a non-empty set must satisfy \autoref{eq:DefinableSetDistancePredicatePositiveLipschitz} and \autoref{eq:DefinableSetDistancePredicateRealized}: for the latter, choose $y \in D$ such that $d(x,y)$ is close to $d(x,D)$.

  For the converse, recall first that if $\psi$ is an affine definable predicate in $\cM$, then it is one in all models of the theory $T=\Th^\aff(\cM)$ (see \autoref{remark:DefinablePredicate}).
  The property $\inf_x \, \psi(x) \geq 0$ can then be expressed by the family of affine conditions $\inf_x \, \varphi(x) + \varepsilon \geq 0$, for every affine formula $\varphi$ and $\varepsilon > 0$ such that $|\varphi-\psi| \leq \varepsilon$ in $\cM$, and similarly for the other conditions (replacing $\varepsilon$ with $2\varepsilon$ and $3\varepsilon$, respectively).
  In particular, if \autoref{eq:DefinableSetDistancePredicatePositiveLipschitz} and \autoref{eq:DefinableSetDistancePredicateRealized} hold in $\cM$, then they hold in all models of $T$.

  Now let $M \models T$ and let $D$ be the zero-set of $\psi$ in $M$.
  By \autoref{eq:DefinableSetDistancePredicateRealized}, for every $a \in M^I$ and $\varepsilon > 0$, there exists $a' \in M^I$ such that $\psi(a) - 2 \psi(a') - d(a,a') > -\varepsilon$.
  Using \autoref{eq:DefinableSetDistancePredicatePositiveLipschitz}, we obtain $\psi(a') < \psi(a) - \psi(a') - d(a,a') + \varepsilon \leq \varepsilon$, and $d(a,a') < \psi(a) - 2\psi(a') + \varepsilon \leq \psi(a) + \varepsilon$.

  Iterating this, we construct a sequence $(a_k)$, starting with $a_0 = a$, such that
  \begin{gather*}
    \psi(a_{k+1}) < \varepsilon/2^k, \qquad d(a_k,a_{k+1}) < \psi(a_k) + \varepsilon/2^k.
  \end{gather*}
  It is Cauchy, say with limit $b \in M^I$.
  This limit must satisfy $\psi(b) = 0$ (in particular, $D$ is non-empty), and $d(a,b) < \psi(a) + \sum_k \psi(a_{k+1}) + \sum_k \varepsilon/2^k < \psi(a) + 4\varepsilon$.

  We conclude that $d(a,D) \leq \psi(a)$.
  The converse inequality follows from \autoref{eq:DefinableSetDistancePredicatePositiveLipschitz}, so $\psi$ is indeed the distance to its zero-set, in every model $M \models T$.
\end{proof}

Another way to describe the notion of definable sets is to consider the distance, in the sense of \autoref{defn:AffineTypeDistance}, to a set of types.

We can extend the definitions of convex and concave functions to functions which are allowed to take the value $\infty$ in the obvious way. Note that a concave, upper semi-continuous function that takes the value $\infty$ somewhere must be the constant $\infty$.
\begin{prop}
  \label{prop:DefinableSetTypes}
  Let $T$ be an affine theory, let $I$ be countable, and let $X \subseteq \tS^\aff_I(T)$ be non-empty.
  For $p \in \tS^\aff_I(T)$, let $\partial^X(p) = \partial(p, X) \in [0,\infty]$.
  Then the following are equivalent:
  \begin{enumerate}
  \item The set $X$ is $\partial$-closed and $\partial^X$ is continuous and affine, i.e., an affine definable predicate.
  \item The set $X$ is $\tau$-compact and convex and $\partial^X$ is $\tau$-upper semi-continuous and concave.
  \item For $M \models T$, define $D^M = \bigl\{ a \in M^I : \tp^\aff(a) \in X \bigr\}$.
    Then the family $D = (D^M : M \models T)$ is uniformly affinely definable, and $X$ is the zero-set in $\tS^\aff_I(T)$ of the definable predicate $d(x,D)$.
  \item
    There exists a family of structures $\cM$ such that $T = \Th^\aff(\cM)$, and a uniformly affinely definable family of sets $D = (D^M \colon M \in \cM)$ such that $X$ is the zero-set in $\tS^\aff_I(T)$ of the definable predicate $d(x,D)$.
  \end{enumerate}
  Moreover, if these conditions hold, the affine definable predicates $\partial^X$ and $d(x,D)$ must agree, and $X$ is an exposed face of $\tS^\aff_I(T)$.
\end{prop}
\begin{proof}
  \begin{cycprf}
  \item
    Immediate.
  \item
    Since $\partial$ is $\tau$-lower semi-continuous and $X$ is $\tau$-compact, $\partial^X$ is $\tau$-lower semi-continuous.
    Since $\partial$ is a convex function and $X$ a convex set, $\partial^X$ is convex.
    Therefore $\partial^X$ is continuous and affine.
    Since $\partial^X = 0$ on $X$ and $\partial^X > 0$ elsewhere, $X$ is an exposed face.
    Since $X \neq \emptyset$, $\dtp^X$ is not the constant $\infty$, so $\dtp^X < \infty$ everywhere.

    Let $M \models T$ and $a,b \in M^I$, with respective types $p,q$.
    Then $\partial^X(a) = \partial^X(p) \geq 0$, and $\partial^X(a) = \partial^X(p) \leq \partial^X(q) + \partial(p,q) \leq \partial^X(b) + d(a,b)$, so \autoref{eq:DefinableSetDistancePredicatePositiveLipschitz} holds in $M$.
    Since $X$ is $\tau$-compact and $\partial$ is $\tau$-lower semi-continuous, there exists $q' \in X$ such that $\partial^X(p) = \partial(p,q')$.
    So there exists $N \succeq^\aff M$ and $c \in N^I$ such that $\tp^\aff(c) = q'$ and $d(a,c) = \partial(p,q') = \partial^X(a)$.
    Since $\partial^X(c) = 0$, it witnesses that $\sup_y \, \bigl( \partial^X(a) - 2 \partial^X(y) - d(a,y) \bigr) \geq 0$ (in $N$ and therefore in $M$), so \autoref{eq:DefinableSetDistancePredicateRealized} holds as well.
    By \autoref{lem:DefinableSetDistancePredicate} and by the definition of $D^M$, the family $D = (D^M : M \models T)$ is uniformly affinely definable and $d(x,D) = \partial^X(x)$.
  \item
    Immediate.
  \item[\impfirst]
    By \autoref{lem:DefinableSetDistancePredicate}, the family $D$ can be extended from $\cM$ to all models of $T$. In particular, $D^M$ is non-empty for every $M\models T$, and $\partial^X(p)<\infty$ for every $p\in\tS^\aff_I(T)$.

    Let $p \in \tS^\aff_I(T)$.
    By compactness of $X$, there exists $q \in X$ such that $\partial(p,q) = \partial^X(p)$.
    Therefore, there exists $M \models T$ and $a,b \in M^I$ such that $\tp^\aff(a) = p$, $\tp^\aff(b) = q$, and $d(a,b) = \partial(p,q)$.
    Then $b \in D$, so $d(p,D) \leq d(a,b) = \partial^X(p)$.
    Assume, on the other hand, that $d(p,D) < r$.
    Then there exists $c \in D^M$ such that $d(a,c) < r$ and $\tp^\aff(c) \in X$.
    It follows that $\partial^X(p) \leq d(a,c) < r$ and since $r$ was arbitrary, $\partial^X(p) \leq d(p,D)$.

    Therefore $\partial^X$ is the definable predicate $d(x,D)$.
    Since $X$ is its zero-set, it is $\tau$-closed and \textit{a fortiori} $\dtp$-closed.
  \end{cycprf}
\end{proof}

A special case of the above is when $D = \set{a}$ is a singleton.
This is the case if and only if $\psi(x) = d(x,D)$ satisfies the conditions of \autoref{lem:DefinableSetDistancePredicate} and, in addition, $\sup_{x,y \in D} \, d(x,y) = 0$.
The distance to a singleton can be equivalently characterized as a function $\psi$ that is $1$-Lipschitz, satisfies $\psi(x) + \psi(y) \geq d(x,y)$ (a function satisfying these two conditions is sometimes called a \emph{Kat\v{e}tov function}), and $\inf_x \psi(x) = 0$.
If $D = \set{a}$ is an affinely definable set, then we say that $a$ is \emph{affinely definable}.
In this case, if $\psi(x,y)$ is any affine definable predicate, then $\psi(a,y) = \sup_{x \in \set{a}} \, \psi(x,y) = \inf_{x \in \set{a}} \, \psi(x,y)$ is an affine definable predicate.

\begin{defn}
  Let $M$ be a structure.
  A function $f\colon M^I \rightarrow M$ is said to be \emph{affinely definable} if $d\bigl( f(x), y \bigr)$ is an affine definable predicate.

  A family of functions on a family of structures is \emph{uniformly affinely definable} if all its members are defined using a common affine definable predicate.
\end{defn}

It is easy to check that a definable function is always uniformly continuous.
Note that $f$ is affinely definable in $M$ if and only if the family of singletons $(\set{f(a)} : a \in M^I)$ is uniformly affinely definable from the parameter $a$.
It follows that if $\psi(x,y)$ is an affine definable predicate, with $x$ a singleton, then $\psi\bigl( f(z), y \bigr)$ is an affine definable predicate as well.
Similarly, for a family of uniformly affinely definable functions on a class of structures.
Finally, an affinely definable function can be named by a new function symbol, following a procedure similar to that for affine definable predicates.


\section{Extreme  types and extremal models}
\label{sec:ExtremeTypes}

Let us fix an affine theory $T$ and variables $x$.
Recall that the type space $\tS^\aff_x(T)$, being the state space of $\cL^\aff_x(T)$, is a compact convex set.

\begin{defn}
  \label{dfn:ExtremeType}
  An extreme point $p \in \tS^\aff_x(T)$ is called an \emph{extreme type} of $T$ (in $x$).
  The collection of extreme types of $T$ in $x$ will be denoted by $\cE_x(T)$.

  If $M$ is a structure, and $A \subseteq M$, then we define $\cE_x(A) = \cE_x(D^\aff_A)$, namely, the collection of extreme types in $\tS^\aff_x(A) = \tS^\aff_x(D^\aff_A)$.
\end{defn}

We start by observing that extreme types over a complete theory are approximately realizable (i.e., finitely satisfiable up to small error) in every model.

\begin{lemma}\label{l:realized-extreme-dense}
  Let $T$ be a complete affine theory.
  Then for every model $M\models T$ and $A\subseteq M$, we have $\cE_x(A)\subseteq \cl{\set{\tp^\aff(a/A) : a\in M^x}}$.
\end{lemma}
\begin{proof}
  Follows from \autoref{l:realized-convex-dense} and Milman's theorem.
\end{proof}

\begin{remark}
  \label{rem:ExtremeRestrictionTwoTheories}
  Consider two languages $\cL_1 \subseteq \cL_2$ and two theories in these languages, $T_1 \subseteq T_2$ (i.e., $T_2 \models T_1$: every consequence of $T_1$ is also a consequence of $T_2$).
  Let $p_1 \in \cE_x(T_1) \subseteq \tS^\aff_x(\cL_1)$, and let $C \subseteq \tS^\aff_x(T_2) \subseteq \tS^\aff_x(\cL_2)$ be the collection of extensions of $p$.
  Then either $C = \emptyset$ (i.e., $p_1$ is inconsistent with $T_2$), or else, $C$ is a closed face of $\tS^\aff_x(T_2)$.
  In the latter case, $p_1$ extends to some $p_2 \in \cE_x(T_2)$ (namely, any extreme point of $C$).
  Special cases of this observation include:
  \begin{itemize}
  \item
    If $\cL_1 = \cL_2 = \cL$, then $\cE_x(T_2) \subseteq \tS^\aff_x(T_2) \subseteq \tS^\aff_x(T_1)$, but an extreme type of $T_2$ need not be an extreme type of $T_1$.
    On the other hand, we do have $\cE_x(T_1) \cap \tS^\aff_x(T_2) \subseteq \cE_x(T_2)$ (if an extreme type of $T_1$ is consistent with $T_2$, then it is \textit{a fortiori} an extreme type there).
  \item
    Let us continue with the assumption that $\cL_1 = \cL_2$.
    Say that $T_2$ is an \emph{extremal extension} of $T_1$ if $\tS^\aff_0(T_2)$ is a face of $\tS^\aff_0(T_1)$.
    This is equivalent to $\tS^\aff_x(T_2)$ being a closed face of $\tS^\aff_x(T_1)$ for any variables $x$, and implies the inclusion $\cE_x(T_2) \subseteq \cE_x(T_1)$.
    Since a face of a face is a face, the extremal extension relation is transitive.
  \item
    Consider now the case of distinct languages, but assume that $T_1$ consists exactly of the $\cL_1$-consequences of $T_2$.
    Then every $p_1 \in \cE_x(T_1)$ extends to some $p_2 \in \cE_x(T_2)$.
  \item
    If $M$ is a structure and $A \subseteq B \subseteq M$, then $D^\aff_A \subseteq D^\aff_B$ is an example of the last situation, so every $p \in \cE_x(A)$ can be extended to some $q \in \cE_x(B)$.
  \end{itemize}
\end{remark}

For the following few results, we fix an affine theory $T$, and tuples of variables $x$ and $y$.
As usual, $\pi_x\colon \tS^\aff_{xy}(T) \rightarrow \tS^\aff_x(T)$ is the variable restriction map.
Recall that it is affine, $\tau$-continuous and $\dtp$-contractive.

\begin{prop}
  \label{prop:ExtremeTypeTwoSteps}
  Let $M \models T$ and $a,b \in M$ be tuples.
  Then $\tp^\aff(a,b) \in \cE_{xy}(T)$ if and only if $\tp^\aff(a) \in \cE_x(T)$ and $\tp^\aff(b/a) \in \cE_y(a)$.
\end{prop}
\begin{proof}
  Let $p(x,y) = \tp^\aff(a,b)$, $q(x) = \pi_x(p) = \tp^\aff(a)$, and $r(y) = p(a,y) = \tp^\aff(b/a)$.

  In one direction, assume that $p$ is extreme.
  By \autoref{lemma:TypeConvexCombinationLift}\autoref{item:TypeConvexCombinationLiftFinite}, any expression of $q$ as a proper convex combination of types in $\tS^\aff_x(T) \subseteq \tS^\aff_x(\cL)$ lifts to one for $p$ in $\tS^\aff_{xy}(T) \subseteq \tS^\aff_{xy}(\cL)$, so $q$ must be extreme.
  (Alternatively, we could apply \autoref{prop:VariableRestrictionBoundary}\autoref{item:VariableRestrictionBoundary}.)
  Also, we can identify $\tS^\aff_y(a)$ with the fiber of $\pi_x$ over $\tp^\aff(a)$.
  This identifies $p$ with $r$, so the latter must be extreme as well.

  Conversely, assume that both $q$ and $r$ are extreme, and $p = \half p_1 + \half p_2$.
  Then $q = \half \pi_x(p_1) + \half \pi_x(p_2)$, so $\pi_x(p_i) = q = \tp^\aff(a)$.
  But then $r = p(a,y) = \half p_1(a,y) + \half p_2(a,y)$ in $\tS^\aff_y(a)$, so $p_i(a,y) = p(a,y)$.
  Therefore, $p_i = p$ and the proof is complete.
\end{proof}

\begin{cor}
  \label{cor:ExtremeTypeVariableRestrictionSurjectiveOpen}
  The variable restriction map $\pi_x$ restricts to a continuous, surjective, open map
  \begin{gather*}
    \pi^\cE_x \colon \cE_{xy}(T) \to \cE_x(T).
  \end{gather*}
\end{cor}
\begin{proof}
  By \autoref{prop:ExtremeTypeTwoSteps}, $\pi_x\bigl( \cE_{xy}(T) \bigr) = \cE_x(T)$.
  For a formula $\varphi(x,y)$, let $\psi(x) = \sup_y \varphi(x,y)$.
  Then the set $\pi_x\bigl( \oset{\varphi > 0} \bigr) = \oset{\psi > 0}$ is open.
  Together with \autoref{l:nbhds-extreme}, this implies that $\pi^\cE_x$ is open.
\end{proof}

\begin{prop}
  \label{prop:ExtremeTypesMetricallyClosed}
  Let $T$ be an affine theory and let $x$ be a countable tuple of variables.
  \begin{enumerate}
  \item \label{item:ExtremeTypeDistanceRealizedExtreme}
    Let $y$ be a tuple of variables with $|x| = |y|$ and let $p \in \cE_x(T), q \in \cE_y(T)$. Then
    \begin{gather*}
      \partial(p,q) = \min \bigl\{ d(x,y)^r : r \in \cE_{xy}(T), \pi_x(r) = p, \pi_y(r) = q \bigr\}.
    \end{gather*}

  \item
    \label{item:ExtremeTypeMetricallyClosedFinite}
    The set $\cE_x(T)$ is closed in $\bigl( \tS^\aff_x(T), \partial \bigr)$.
  \end{enumerate}
\end{prop}
\begin{proof}
  \autoref{item:ExtremeTypeDistanceRealizedExtreme} Assume that $\partial(p,q) < \infty$.
  Then the set
  \begin{gather*}
    \bigl\{ r \in \tS^\aff_{xy}(T) : \pi_x(r) = p,\ \pi_y(r) = q,\ d(x,y)^r = \partial(p,q) \bigr\}
  \end{gather*}
  is easily checked to be a closed face, and therefore contains a point of $\cE_{xy}(T)$.

  \autoref{item:ExtremeTypeMetricallyClosedFinite}
  Let $p_n \to^\dtp p$ with $p_n \in \cE_x(T)$. Let $y_1, y_2, \ldots$ be variables with $|y_n| = |x|$ for each $n$. Let
  \begin{equation*}
    F = \set[\big]{q \in \tS^\aff_{x\bar y}(T) : \pi_{y_n}(q) = p_n,\ d(x, y_n)^q = \dtp(p, p_n) \text{ for all } n}.
  \end{equation*}
  $F$ is clearly closed, convex, and non-empty by \autoref{cor:AffineAmalgamation}, and we claim that it is a face. Suppose that $q \in F$ and $q = \lambda q_1 + (1-\lambda) q_2$ with $0<\lambda<1$. Then $\pi_{y_n}(q_i) = p_n$ for all $n$ and for $i=1,2$. Also,
  \begin{equation*}
    \dtp(\pi_x(q_1), p_n) \leq d(x, y_n)^{q_1} \leq \frac{1}{\lambda} \dtp(p, p_n) \to 0,
  \end{equation*}
  and similarly for $q_2$.
  This implies that $\pi_x(q_i) = p$ and therefore $d(x, y_n)^{q_i} \geq \dtp(p, p_n)$. We conclude that $d(x, y_n)^{q_i} = \dtp(p, p_n)$ for all $n$ and for $i=1,2$ and thus $q_1, q_2 \in F$. Now any extreme point $q$ of $F$ is also an extreme point of $\tS^\aff_{x\bar y}$ and $p = \pi_x(q)$ is extreme by \autoref{prop:ExtremeTypeTwoSteps}.
\end{proof}

\begin{lemma}
  \label{l:ExtremeTypesInfinite}
  Let $T$ be an affine theory, let $x$ be a tuple of variables, and let $p \in \tS^\aff_x(T)$. Then $p$ is extreme if and only if $\pi_{x'}(p)$ is extreme for every finite $x' \subseteq x$.
  In other words,
  \begin{gather*}
    \cE_x(T) = \varprojlim_{x' \subseteq x \ \text{finite}} \cE_{x'}(T).
  \end{gather*}
\end{lemma}
\begin{proof}
  One direction follows from \autoref{prop:ExtremeTypeTwoSteps}.
  For the other, assume that $p = \half p_1 + \half p_2$, and let $\varphi \in \cL^\aff_x$.
  Then only a finite tuple $x' \subseteq x$ actually occurs in $\varphi$.
  In addition, $\pi_{x'}(p_i) = \pi_{x'}(p)$ by hypothesis.
  Therefore $\varphi(p_i) = \varphi(p)$, and since $\varphi$ was arbitrary, $p_1 = p_2 = p$.
\end{proof}

\begin{defn}
  \label{defn:ExtremalModel}
  A structure $M$ is an \emph{extremal model} of $T$ if it is a model of $T$ that only realizes extreme types.
  We also define
  \begin{gather*}
    \cE^\fM_x(T) = \tS^{\fM,\aff}_x(T) \cap \cE_x(T) = \tS^{\fM,\aff}_x(\cL) \cap \cE_x(T).
  \end{gather*}
\end{defn}

Clearly, every affine submodel of an extremal model is extremal.
By \autoref{prop:ExtremeTypesMetricallyClosed} and \autoref{l:ExtremeTypesInfinite}, if $M_0 \subseteq M$ is a dense subset, then $M$ is an extremal model of $T$ if and only if $\tp^\aff(a) \in \cE_n(T)$ for every $n \in \bN$ and every $a \in M_0^n$.
Consequently, we have the following extension of \autoref{remark:TypeOfModel}.

\begin{remark}
  \label{remark:TypeOfExtremalModel}
  Let $M \models T$, $a \in M^I$, $p = \tp^\aff(a)$ and $A = \{a_i : i \in I\}$.
  Then the following are equivalent:
  \begin{enumerate}
  \item $p \in \cE^\fM_I(\cL)$.
  \item The set $A$ is dense in an extremal affine submodel of $M$.
  \end{enumerate}
\end{remark}

\begin{theorem}
  \label{th:ExtremalModelTypes}
  Let $T$ be an affine theory.
  Then $\cE_x(T)$ is the collection of types that are realized in extremal models of $T$.
\end{theorem}
\begin{proof}
  One direction holds by definition.
  For the other, it will suffice to show that for every $p \in \cE_x(T)$ there exist $y \supseteq x$ and $q \in \cE^\fM_y(T)$ such that $\pi_x(q) = p$.

  Before showing that, let us consider a single formula $\varphi(x,z)$, with $z$ a singleton.
  Let $C_1 = \pi_x^{-1}(p) \subseteq \tS^\aff_{xz}(T)$.
  Since $p$ is extreme, $C_1$ must be a closed face.
  Let
  \begin{gather*}
    t = \bigl( \qsup_z \varphi(x,z) \bigr)^p = \sup \, \bigl\{ \varphi(x,z)^q : q \in C_1 \bigr\}.
  \end{gather*}
  Then the set
  \begin{gather*}
    C_2 = \bigl\{ q \in C_1 : \varphi(x,z)^q = t \bigr\}
  \end{gather*}
  is a closed face of $C_1$.
  Finally, let $q$ be any extreme point of $C_2$.
  Then $q$ is an extreme type, $\pi_x(q) = p$, and $\varphi(x,z)^q = \bigl( {\sup}_{z}\, \varphi(x,z) \bigr)^p = \bigl( {\sup}_z \, \varphi(x,z) \bigr)^q$.

  We can now construct a type $q \in \cE_y(T)$ by transfinite induction.
  At successor stages we introduce a new variable and proceed as above; and at limit steps, we take the union, observing that this remains extreme by \autoref{l:ExtremeTypesInfinite}.
  With the proper bookkeeping, $q$ satisfies the Tarski--Vaught property, as desired.
\end{proof}

\begin{cor}\label{c:extremal-models-existence}
  Every affine theory $T$ admits an extremal model of density character at most $\fd_T(\cL)$.
\end{cor}
\begin{proof}
  It admits an extremal model $M$ by \autoref{th:ExtremalModelTypes}.
  By \autoref{prop:DownwardLS} (and \autoref{remark:LanguageDensityCharacter}), we can find $N \preceq^\aff M$ of the desired density character.
\end{proof}

\begin{cor}
  \label{cor:ExtremeTypeDistance}
  The distance between two extreme types of $T$, if it is finite, is attained in an extremal model of $T$.
\end{cor}
\begin{proof}
  By \autoref{prop:ExtremeTypesMetricallyClosed}\autoref{item:ExtremeTypeDistanceRealizedExtreme} and \autoref{th:ExtremalModelTypes}.
\end{proof}

\begin{cor}
  \label{cor:TypeOfExtremalModelVeryDense}
  Assume that $\kappa \geq \fd_T(\cL)$ and that $\kappa$ is uncountable.
  Then every Baire set $X \subseteq \tS^\aff_\kappa(T)$ that meets $\cE_\kappa(T)$ also meets $\cE^\fM_\kappa(T)$.
\end{cor}
\begin{proof}
  Let $X \subseteq \tS^\aff_\kappa(T)$ and $p \in X \cap \cE_\kappa(T)$.
  If $X$ is a Baire set, then it is defined using countably many formulas, which depend on countably many variables $x$.
  Let $p_0 = \pi_{x}(p) \in \cE_{x}(T)$.
  By \autoref{th:ExtremalModelTypes} and \autoref{prop:DownwardLS}, we can realize $p_0$ in an extremal model $M \models T$ of density character at most $\kappa$.
  We can choose a dense sequence $a \in M^\kappa$ such that $a \rest_x \models p_0$.
  Then $\tp^\aff(a) \in X \cap \cE^\fM_\kappa(T)$.
\end{proof}

\begin{theorem}
  \label{th:Bishop-dL-ForModels}
  Let $\kappa \geq \fd_T(\cL)$.
  Let $\mu$ be a boundary measure on $\tS^\aff_\kappa(T)$ such that $R(\mu) \in \tS^{\fM,\aff}_\kappa(T)$.
  Then $\mu$ concentrates on $\cE^\fM_\kappa(T)$ (in the sense of \autoref{df:measure-concentr}).
\end{theorem}
\begin{proof}
  If $\kappa$ is countable, then so is $\fd_T(\cL)$.
  In this case, $\tS^\aff_\kappa(T)$ is metrizable, and both $\cE_\kappa(T)$ and $\tS^{\fM,\aff}_\kappa(T)$ are Baire of measure $1$ by \autoref{th:Bishop-dL} and \autoref{prop:TypeOfModelPreMean}\autoref{item:TypeOfModelPreMeanCountable}.
  Therefore, $\cE^\fM_\kappa(T)$ is Baire of measure $1$.

  If $\kappa$ is uncountable, then our assertion follows from \autoref{th:Bishop-dL} and \autoref{cor:TypeOfExtremalModelVeryDense} (and we do not use our hypothesis that $R(\mu) \in \tS^{\fM,\aff}_\kappa(T)$).
\end{proof}

\begin{prop}[Extremal joint embedding and amalgamation]
  \label{prop:ExtremalJEP}
  Let $T$ be a complete affine theory, and let $M$ and $N$ be two extremal models of $T$.
  Then they admit affine embeddings into a third extremal model of $T$.

  More generally, assume that $\theta\colon M \dashrightarrow N$ is a partial affine map.
  Then $M$ and $N$ admit affine embeddings into an extremal model $K$ of $T$, say $f\colon M \hookrightarrow K$ and $g\colon N \hookrightarrow K$, such that $f$ and $g \circ \theta$ agree on $\dom \theta$.
\end{prop}
\begin{proof}
  For the first assertion, let $p_0 = \tp^\aff(M) \in \cE^\fM_{I_0}(T)$, for some $I_0$-enumeration of $M$ (or of a dense subset), and similarly $p_1 = \tp^\aff(N) \in \cE^\fM_{I_1}(T)$.
  Let $J = I_0 \sqcup I_1$ be the disjoint union, and let
  \begin{gather*}
    C = \bigl\{ q \in \tS^\aff_{J} : \pi_{I_0}(q) = p_0, \ \pi_{I_1}(q) = p_1 \bigr\}.
  \end{gather*}
  Then $C$ is non-empty by \autoref{prop:AffineJEP}, and it is a closed face as the intersection of two such.
  Therefore, it contains an extreme type.
  This can be realized in an extremal model $K \models T$, whence the two desired embeddings.

  For the second assertion, let $A = \dom \theta \subseteq M$.
  Then $M$ is \textit{a fortiori} an extremal model of $D^\aff_A$, and so is $N$ if we identify $A$ with $\theta(A)$.
  Let $a$ enumerate a dense subset of $A$.
  If $b$ realizes an extreme type of $D^\aff_A$, then $\tp^\aff(a,b)$ and $\tp^\aff(a)$ are extreme by \autoref{prop:ExtremeTypeTwoSteps}.
  Therefore, an extremal model of $D^\aff_A$ is an extremal model of $T$.
  The second assertion now reduces to the first (compare with the proof of \autoref{cor:AffineAmalgamation}).
\end{proof}

\begin{prop}[Extremal affine chains]
  \label{p:affine-chains-extremal}
  If $(M_i : i \in I)$ is an affine chain, as in \autoref{p:affine-chains}, and each $M_i$ is extremal, then so is $\cl{\bigcup_{i \in I} M_i}$.
\end{prop}
\begin{proof}
  Follows from \autoref{prop:ExtremeTypesMetricallyClosed}.
\end{proof}

Bagheri observed that one cannot define algebraic closure in affine logic exactly as in continuous logic.
Indeed, compactness of affinely definable sets is not preserved under affine extensions.
However, the following result implies that compactness in \emph{extremal} models can be coded in affine logic.

\begin{prop}\label{p:compact-def-sets}
  Let $T$ be a complete affine theory and let $D$ be a uniformly affinely definable set in models of $T$. Then the following are equivalent:
  \begin{enumerate}
  \item $D^M$ is compact in some model $M\models T$.
  \item $D^M$ is compact in every extremal model $M\models T$.
  \end{enumerate}
  Moreover, if $M\preceq^\aff N$ is an extension of extremal models of $T$ and $D^M$ is compact, then $D^M = D^N$.
\end{prop}
\begin{proof}
  By \autoref{th:ExtremalModelTypes}, extremal models exist, so one of the implications is clear.
  Conversely, suppose there is an extremal model $M$ with $D^M$ non-compact, and therefore not totally bounded.
  Then there exists an $\varepsilon > 0$ and an infinite sequence $(a_i) \subseteq D^M$ such that $d(a_i,a_j) > \varepsilon$ for all $i \neq j$.
  The theory $T$ is complete, $D$ definable, and $p_n = \tp^\aff(a_1, \dots, a_n)$ is extreme for each $n$.
  Therefore, by \autoref{l:realized-extreme-dense}, $p_n$ is approximately finitely satisfiable in $D^M$ for every model $M$ of $T$.
  It follows that for no model $M$ of $T$ is $D^M$ compact.

  For the moreover part, assume $M\preceq^\aff N$ are extremal models with $D^M$ compact, and let $b\in D^N$.
  For $\varepsilon > 0$, let $B(a_i,\varepsilon)$, for $1 \leq i \leq n$, cover $D^M$, with $a_i \in D^M$.
  The type $\tp^\aff(b/a_1,\dots, a_n)$ is extreme by \autoref{prop:ExtremeTypeTwoSteps} and is therefore approximately realized in $M$ by \autoref{l:realized-extreme-dense} (applied to the theory of $(M,a_1\dots a_n)$).
  Therefore $d(b,D^M) < \varepsilon$, which is enough.
\end{proof}


\section{Homogeneity and saturation}
\label{sec:Homogeneity-and-saturation}

The notions of homogeneity and saturation and the interplay between them are a classical subject in model theory. In this section, we discuss the appropriate counterparts in affine logic. There are two variants which will be of interest: saturation with respect to all affine types and saturation of extremal models with respect to extreme types. The case of approximate $\aleph_0$-saturation and homogeneity is more delicate and we treat it at the end of the section. The arguments that we present are not specific to affine logic and work with any reasonable notion of type (even without compactness). This is the reason why all results below are valid both in the general affine and the extremal setting.

\begin{defn}
  \label{dfn:Homogeneous}
  Let $\kappa$ be an infinite cardinal and let $M$ be a structure.
  \begin{enumerate}
  \item We say that $M$ is \emph{affinely $\kappa$-homogeneous} if for every partial affine map $f\colon M \dashrightarrow M$, if $|\dom f| < \kappa$, then $f$ extends to an automorphism of $M$.
    If, in addition, $\kappa = \fd(M)$, then we say that $M$ is \emph{affinely homogeneous}.
  \item We say that $M$ is \emph{weakly affinely $\kappa$-homogeneous} if for every partial affine map $f\colon M \dashrightarrow M$ and $a \in M$, if $|\dom f| < \kappa$, then $f$ can be extended to $\{a\} \cup \dom f$.
  \end{enumerate}
\end{defn}

It is easy to see that the condition $|\dom f| < \kappa$ can be replaced with $\fd(\dom f) < \kappa$.
By an easy back-and-forth argument, if $M$ is weakly affinely $\fd(M)$-homoge\-neous, then it is affinely homogeneous.

\begin{lemma}
  \label{lem:HomogeneousTypes}
  Let $M$ and $N$ be structures, and let $\kappa$ be an infinite cardinal.
  \begin{enumerate}
  \item
    \label{item:HomogeneousTypesUniversal}
    Assume that $M$ is weakly affinely $\kappa$-homogeneous and that for every $n \in \N$ and every $p \in \tS^\aff_n(\cL)$, if $p$ is realized in $N$, then it is also realized in $M$.
    Then the same is true for every $p \in \tS^\aff_\lambda(\cL)$, for $\lambda \leq \kappa$.
    In particular, if $\fd(N) \leq \kappa$, then $N$ admits an affine embedding into $M$.

  \item
    \label{item:HomogeneousTypesIsomorphic}
    Assume that both $N$ and $M$ are affinely homogeneous of the same density character and that they realize the same affine $n$-types for every $n \in \bN$.
    Then $N \cong M$.
  \end{enumerate}
\end{lemma}
\begin{proof}
  We prove \autoref{item:HomogeneousTypesUniversal} by induction on $\lambda$. By hypothesis, we may assume that $\lambda$ is infinite.
  For each $p(x) \in \tS^\aff_\lambda(T)$ realized in $N$, we inductively construct a sequence $(a_\alpha : \alpha \leq \lambda)$ in $M$ such that $(a_i : i < \alpha)$ realizes $p\rest_{x_{<\alpha}}$ for each $\alpha \leq \lambda$. For limit $\alpha$ there is nothing to do. For a successor $\alpha$ with $\alpha < \lambda \leq \kappa$, we first realize $p\rest_{x < \alpha}$ somewhere in $M$ (which we can do by the inductive hypothesis because $|\alpha| < \lambda$) and then use the weak $\kappa$-homogeneity of $M$ to construct $a_\alpha$.

  A back-and-forth variant of this argument proves \autoref{item:HomogeneousTypesIsomorphic}.
\end{proof}

\begin{defn}
  \label{dfn:Saturated}
  Let $\kappa$ be an infinite cardinal and $M$ a structure.
  \begin{enumerate}
  \item We say that $M$ is \emph{affinely $\kappa$-saturated} if every affine type $p\in\tS^\aff_1(A)$ over $A\subseteq M$ with $|A|<\kappa$ is realized in $M$.
    If, in addition, $\kappa = \fd(M)$, then we say that $M$ is \emph{affinely saturated}.
  \item We say that $M$ is \emph{extremally $\kappa$-saturated} if it is an extremal model (of some affine theory $T$, or equivalently, of $\Th^\aff(M)$), and every extremal type $p \in \cE_1(A)$ over $A\subseteq M$ with $|A|<\kappa$ is realized in $M$.
    If $\kappa = \fd(M)$, then we say that $M$ is \emph{extremally saturated}.
  \end{enumerate}
\end{defn}

\begin{lemma}
  \label{lem:SaturatedByDense}
  Let $\kappa$ be an uncountable cardinal and let $M$ be a structure.
  If $M_0 \subseteq M$ is dense and $M$ realizes every type in $\tS^\aff_1(A)$ for every $A\subseteq M_0$ with $|A| < \kappa$, then $M$ is affinely $\kappa$-saturated.
  Moreover, it realizes every affine type over a subset $A \subseteq M$ such that $\fd(A) < \kappa$.

  The same holds, \textit{mutatis mutandis}, for extremal saturation.
\end{lemma}
\begin{proof}
  If $A \subseteq M$ and $\fd(A) < \kappa$, then there exists $B \subseteq M_0$, $|B| < \kappa$ such that $A \subseteq \overline{B}$. Now the claim follows from the fact that the restriction maps $\tS^\aff_x(\cl{B}) \to \tS^\aff_x(B)$ and $\cE_x(\cl{B}) \to \cE_x(B)$ are bijective.
\end{proof}

\begin{lemma}
  \label{lem:ParameterChange}
  Assume that $f\colon M \dashrightarrow N$ is a partial affine map, with domain $A\subseteq M$ and image $B\subseteq N$.
  Then the induced map $f_*\colon \tS^\aff_x(A) \rightarrow \tS^\aff_x(B)$ defined by
  \begin{equation*}
    \phi(x, f(a))^{f_*(p)} = \phi(x, a)^p,
  \end{equation*}
  for all affine formulas $\phi(x, y)$ and $a \in A^y$, is an affine homeomorphism.
\end{lemma}
\begin{proof}
  Follows from the definition of affine map.
\end{proof}

\begin{prop}
  \label{prp:SaturatedHomogeneous}
  Let $\kappa$ be an infinite cardinal and let $M$ be a structure.
  Let $T = \Th^\aff(M)$.
  Then $M$ is affinely $\kappa$-saturated if and only if it is weakly affinely $\kappa$-homogeneous and realizes every affine type $p \in \tS^\aff_n(T)$ for every $n \in \bN$.
  Similarly, \textit{mutatis mutandis}, for extremal saturation and extremal types.

  Consequently, two affinely equivalent affinely (or extremally) saturated models of the same density character are isomorphic.
\end{prop}
\begin{proof}
  Using \autoref{lem:ParameterChange} in one direction and \autoref{lem:HomogeneousTypes}\autoref{item:HomogeneousTypesUniversal} in the other.

  The last assertion follows from \autoref{lem:HomogeneousTypes}\autoref{item:HomogeneousTypesIsomorphic}.
\end{proof}

\begin{lemma}
  \label{lem:SaturatedUniversal}
  Let $\kappa$ be an infinite cardinal and let $M$ be a structure.
  If $M$ is affinely $\kappa$-saturated then it realizes every type in $\tS^\aff_\kappa(A)$ for every $A\subseteq M$, $\density(A)<\kappa$.

  Similarly, \textit{mutatis mutandis}, for extremal saturation.
\end{lemma}
\begin{proof}
  By a standard inductive construction.
\end{proof}

\begin{prop}
  \label{prp:SaturatedExistence}
  Let $T$ be an affine theory and let $M$ be a model of $T$.
  Let $\kappa,\lambda$ be cardinals such that $\lambda = \lambda^{<\kappa}$, $\fd_T(\cL) < \kappa$, and $\fd(M) \leq \lambda$.
  Then $M$ admits an affinely $\kappa$-saturated affine extension of density character at most $\lambda$.
  In particular, if $\fd_T(\cL) \leq \mu$, then every model of $T$ such that $\fd(M) \leq 2^\mu$ admits an affinely $\mu^+$-saturated affine extension of density character at most $2^\mu$.

  Similarly, \textit{mutatis mutandis}, for extremal saturation and extremal $M$.
\end{prop}
\begin{proof}
  This is fairly standard.
  We construct an affine chain of models $(M_i : i \leq \lambda)$, where $\fd(M_i) \leq \lambda$.
  At limit stages we take the (completion of the) union.
  At successor stage $i+1$ we choose a dense subset $M'_i \subseteq M_i$ of cardinal $\leq \lambda$ and realize in $M_{i+1}$ all types over subsets of $M'_i$ of cardinal $< \kappa$.
  By \autoref{lem:SaturatedByDense} and the fact that $\lambda$ has cofinality at least $\kappa$, the final model $M_\lambda$ is $\kappa$-saturated.

  Similarly in the extremal case.
\end{proof}

Countable saturation or homogeneity in affine logic (or, for that matter, in continuous logic) are somewhat less robust than their uncountable counterparts (or countable counterparts in classical, discrete logic). For this reason, we need the following approximate notions.

\begin{defn}
  \label{dfn:ApproximatelyHomogeneous}
  Let $M$ be a structure.
  \begin{enumerate}
  \item We say that $M$ is \emph{approximately affinely $\aleph_0$-homogeneous} if for every $a,b \in M^n$ and $\varepsilon > 0$, if $\tp^\aff(a) = \tp^\aff(b)$, then there exists $g \in \Aut(M)$ such that $d(g \cdot a,b) < \varepsilon$.
    If, in addition, $\fd(M) = \aleph_0$, then $M$ is \emph{approximately affinely homogeneous}.
  \item We say that $M$ is \emph{weakly approximately affinely $\aleph_0$-homogeneous} if for every $a,b \in M^n$, $c \in M$ and $\varepsilon > 0$, if $\tp^\aff(a) = \tp^\aff(b)$, then there exist $b' \in M^n$ and $d \in M$ such that $d(b',b) < \varepsilon$ and $\tp^\aff(ac) = \tp^\aff(b'd)$.
  \end{enumerate}
\end{defn}

We link the two notions of approximate homogeneity using a Baire category method inspired by the use of joinings in ergodic theory.

\begin{lemma}
  \label{lem:ApproximatelyHomogeneousExtension}
  Assume that $M$ is weakly approximately affinely $\aleph_0$-homogeneous and $p(x) \in \tS^\aff_\bN(\cL)$ is such that all its restrictions to finitely many variables are realized in $M$.
  Let $n \in \N$ and suppose also that $a \in M^n$ realizes the restriction of $p$ to $x_{<n}$.
  Then for every $\varepsilon > 0$ there exists $\hat{a} \in M^\bN$ such that $d(a,\hat{a}_{<n}) < \varepsilon$ and $\hat{a}$ realizes $p$.
  In particular, $p$ is realized in $M$.
\end{lemma}
\begin{proof}
  By an iterated application of the definition, we construct $b^k = (b_i^k : i < n+k)$ such that $b^0 = a$, $b^k$ realizes the restriction of $p$ to $x_{<n+k}$ and $d(b^k_i, b^{k+1}_i) < \varepsilon / (n 2^{k+1})$ for $i < n+k$.
  Then $\hat{a}_i$ is the limit of the Cauchy sequence $(b^k_i : k \in \N)$ (which is defined for $k > i-n$).
\end{proof}

\begin{defn}
  \label{dfn:Joining}
  Let $M$ be weakly approximately affinely $\aleph_0$-homogeneous, and fix $p(x),q(y) \in \tS^\aff_\bN(\cL)$ whose restrictions to finitely many variables are realized in $M$.

  We define the set of \emph{$M$-joinings} of $p$ and $q$, denoted $\fJ = \fJ^M_{p,q}$, to consist of all types $s(x,y)$ such that
  \begin{itemize}
  \item $\pi_x(s) = p$ and $\pi_y(s) = q$, and
  \item for every $n \in \bN$, the restriction $s_n = \pi_{x_{<n}y_{<n}}(s)$ is realized in $M$.
  \end{itemize}
  For $n \in \bN$ and $s,t \in \fJ$, define
  \begin{gather*}
    \partial^M_n(s,t)
    = \inf \, \bigl\{ d(ab,cd) : a,b,c,d \in M^n, \ \tp^\aff(ab) = s_n, \ \tp^\aff(cd) = t_n \bigr\},
    \\
    \partial^M(s,t) = \qsup_n (2^{-n} \wedge \partial_n(s,t)).
  \end{gather*}
  Finally, we define
  \begin{gather*}
    \fJ_R = \fJ \cap \bigcap_{j,\varepsilon} \, \bigcup_i \bigl\llbracket d(x_i,y_j) < \varepsilon \bigr\rrbracket.
  \end{gather*}
  We define $\fJ_L$ similarly, inverting the roles of $i$ and $j$.
\end{defn}

The letters $R$ and $L$ stand for \emph{right-} and \emph{left-inclusion}.
Indeed, if $s = \tp^\aff(a,b) \in \fJ$, then $\overline{\{a_i : i \in \bN\}} \supseteq \overline{\{b_j : j \in \bN\}}$ if and only if $s \in \fJ_R$, and the other way round for $\fJ_L$.

\begin{lemma}
  \label{lem:Joining}
  Let $M$, $p$, $q$, $\fJ = \fJ^M_{p,q}$, $\partial_n = \partial^M_n$ and $\partial = \partial^M$ be as in \autoref{dfn:Joining}.
  \begin{enumerate}
  \item
    \label{item:JoiningAge}
    The definitions of $\fJ$, $\partial_n$ and $\partial$ only depend on the collection of types in finitely many variables realized in $M$.
  \item
    \label{item:JoiningCompleteMetric}
    Each $\partial_n$ defines a pseudometric on $\fJ$ and $\partial$ defines a complete metric on $\fJ$.
    If $M$ is separable, then so is $(\fJ,\partial)$.
  \item
    \label{item:JoiningComeager}
    Assume that $p$ is realized by some tuple $a \in M^\N$ which is dense in $M$.
    Then $\fJ_R$ is comeager in $(\fJ,\partial)$.
  \end{enumerate}
\end{lemma}
\begin{proof}
  Item \autoref{item:JoiningAge} is immediate.
  For \autoref{item:JoiningCompleteMetric}, observe first (as in \autoref{lem:ApproximatelyHomogeneousExtension}) that if $\partial_n(s,t) < \alpha$, then for every $a,b \in M^n$ such that $\tp^\aff(ab) = s_n$ there exist $c,d \in M^n$ such that $d(ab,cd) < \alpha$ and $\tp^\aff(cd) = t_n$.
  It follows that each $\partial_n$ is a pseudometric.
  If $\partial_n(s,t) = 0$ for all $n$, then $s = t$, so $\partial$ is a metric.
  Transfer of separability is clear.

  For completeness, assume that $\partial_n(s^n,s^{n+1}) < 2^{-n}$ for all $n$.
  For each $n$, choose $a^n,b^n \in M^n$ such that $\tp^\aff(a^nb^n) = s^n$.
  By \autoref{lem:ApproximatelyHomogeneousExtension}, we may assume that $d(a^nb^n,a^{n+1}_{<n}b^{n+1}_{<n}) < 2^{-n}$.
  For each $k \in \bN$, the sequence $(a^{n+k+1}_k : n \in \N)$ is Cauchy and converges to some $c_k \in M$.
  Similarly, $b^{n+k+1}_k \rightarrow d_k$.
  Now, $\tp^\aff(c) = p$, $\tp^\aff(d) = q$, $t = \tp^\aff(c,d) \in \fJ$, and $\partial_n(s^m,t) < 2^{m+1}$ for $m \geq n$.
  It follows that $(\fJ,\partial)$ is complete.

  For \autoref{item:JoiningComeager}, as $\llbracket d(x_i,y_j) < \varepsilon \rrbracket$ is open, by the Baire category theorem, all we need to show is that $\bigcup_i \, \llbracket d(x_i,y_j) < \varepsilon \rrbracket$ is dense in $\fJ$ for every pair $j,\varepsilon$.
  Let $U \subseteq \fJ$ be open, non-empty.
  Then there exist $s \in U$, $n \in \bN$ and $\delta > 0$ such that
  \begin{gather*}
    \bigl\{ t \in \fJ : \partial_n(s,t) < \delta \bigr\} \subseteq U.
  \end{gather*}
  We may assume that $j < n$ and $\delta < \varepsilon$.
  Choose $c,d \in M^n$ such that $\tp^\aff(c,d) = s_n$.
  By \autoref{lem:ApproximatelyHomogeneousExtension}, we may assume that $d(a_{<n},c) < \delta/2$.
  Again by \autoref{lem:ApproximatelyHomogeneousExtension}, there exists $b \in M^\bN$ realizing $q$ such that $d(b_{<n},d) < \delta/2$.
  Let $t = \tp^\aff(a,b)$.
  Then $t \in \fJ$ and $\partial_n(s,t) < \delta$, so $t \in U$.
  Since $a$ is dense in $M$, there exists $i$ such that $d(a_i,d_j) < \delta/2$, so $d(a_i,b_j) < \delta < \varepsilon$, completing the proof.
\end{proof}

\begin{theorem}
  \label{thm:ApproximatelyHomogeneousTypes}
  Every separable weakly approximately affinely $\aleph_0$-homogeneous structure is approximately affinely homogeneous.
  Any two such structures that realize the same affine types in finitely many variables are isomorphic.
\end{theorem}
\begin{proof}
  It will suffice to show that if $M$ and $N$ are separable and weakly approximately affinely $\aleph_0$-homogeneous and $a \in M^n$ and $b \in N^n$ have the same affine type, then there exists an isomorphism $f\colon M \to N$ sending $a$ arbitrarily close to $b$.
  Extend $a$ and $b$ to dense tuples $c \in M^\bN$ and $d \in N^\bN$ and let $p = \tp^\aff(c)$ and $q = \tp^\aff(d)$.
  Since $M$ and $N$ realize the same affine types in finitely many variables, we have that $\fJ = \fJ^M_{p,q} = \fJ^N_{p,q}$ and the metrics $\dtp^M$ and $\dtp^N$ coincide on $\fJ$.
  Therefore, by \autoref{lem:Joining}, both $\fJ_R$ and $\fJ_L$ are comeager in $\fJ$.
  Let
  \begin{gather*}
    U = \fJ \cap \left\llbracket d(x_{<n},y_{<n}) < \varepsilon \right\rrbracket.
  \end{gather*}
  Then $U$ is open in $\fJ$.
  By \autoref{lem:ApproximatelyHomogeneousExtension}, we may realize $\tp^\aff(d)$ by a sequence $\tilde{d} \in M^\bN$ such that $d(a,\tilde{d}_{<n}) < \varepsilon$, and then $\tp^\aff(c,\tilde{d}) \in U$.

  Therefore, by completeness of $\fJ$ and the Baire category theorem, there exists $s \in U \cap \fJ_L \cap \fJ_R$.
  Let $c',d'$ realize $s$ in some larger structure.
  Sending $c \mapsto c'$ defines an affine embedding of $M$, and similarly for $d$, $d'$ and $N$.
  Since $s \in \fJ_R \cap \fJ_L$, we have that $\overline{\{c'_i : i \in \bN\}} = \overline{\{d'_j : j \in \bN\}}$.
  We obtain an isomorphism between $M$ and $N$ and since $s \in U$, it sends $a$ to a tuple at distance at most $\varepsilon$ from $b$.
\end{proof}

For approximate saturation, one needs to allow the parameters to be moved a little.
Therefore, we shall refer to parameters explicitly: if $a \in M^n$, then a type over $a$ will be denoted $p(x,a) \in \tS^\aff_x(a)$, where $p(x,y) \in \tS^\aff_{xy}(T)$.

\begin{defn}
  \label{dfn:ApproximatelySaturated}
  Let $M$ be a structure.
  \begin{enumerate}
  \item We say that $M$ is \emph{approximately affinely $\aleph_0$-saturated} if for every finite $a \in M^n$, every affine type $p(x,a) \in \tS^\aff_1(a)$, and every $\varepsilon > 0$, there exists $b \in M^n$ such that $\tp^\aff(b) = \tp^\aff(a)$, $d(a,b) < \varepsilon$, and $p(x,b)$ is realized in $M$.
    If, in addition, $\fd(M) = \aleph_0$, then we say that $M$ is \emph{approximately affinely saturated}.
  \item We say that $M$ is \emph{approximately extremally $\aleph_0$-saturated} if it is an extremal model (of some affine theory $T$) and for every finite $a \in M^n$, every extreme type $p(x,a) \in \cE_1(a)$, and every $\varepsilon > 0$, there exists $b \in M^n$ such that $\tp^\aff(a) = \tp^\aff(b)$, $d(a,b) < \varepsilon$, and $p(x,b)$ is realized in $M$ (note that this implies that $p(x,y) \in \cE_{n+1}(T)$, by \autoref{prop:ExtremeTypeTwoSteps}).
    If, in addition, $\fd(M) = \aleph_0$, then we say that $M$ is \emph{approximately extremally saturated}.
  \end{enumerate}
\end{defn}

The following lemma is the approximate analogue of \autoref{lem:SaturatedByDense}, but the statement is a little more complex.
In addition to restricting our attention to types $p(x,a)$ where $a$ belongs to some dense set, in \autoref{item:ApproximatelySaturatedByDenseMoveType} we allow $p(x,y)$ to be moved, while fixing $a$ (which is not the same as moving $p(x,a)$!), so it is not immediately comparable with approximate saturation, while \autoref{item:ApproximatelySaturatedByDenseMoveTypeAndParams} generalizes both, allowing the parameters \emph{and} the type to move. This lemma will be used in \autoref{sec:criterion-affine-reduction}.

It will be convenient to extend the notation of distance between types: if $p \in \tS^\aff_n(\cL)$ and $a \in M^n$, then by $\partial(p,a)$ we mean $\partial\bigl( p, \tp^\aff(a) \bigr)$.
Similarly, by $\partial(a,b)$ we mean $\partial\bigl( \tp^\aff(a), \tp^\aff(b)\bigr)$ (which is not the same as $d(a,b)$).

\begin{lemma}
  \label{lem:ApproximatelySaturatedByDense}
  Let $M$ be a structure and let $M_0 \subseteq M$ be dense.
  Then the following are equivalent:
  \begin{enumerate}
  \item
    \label{item:ApproximatelySaturatedByDenseBase}
    The structure $M$ is approximately affinely $\aleph_0$-saturated.
  \item
    \label{item:ApproximatelySaturatedByDenseMoveTypeAndParams}
    For every finite tuple $a \in M_0^n$, affine type $p(x,a) \in \tS^\aff_1(a)$, and $\varepsilon > 0$, there exist $b \in M^n$ and $c \in M$ such that $d(a,b) < \varepsilon$ and $\partial(p,cb) < \varepsilon$.
  \item
    \label{item:ApproximatelySaturatedByDenseMoveType}
    For every finite tuple $a \in M_0^n$, affine type $p(x,a) \in \tS^\aff_1(a)$, and $\varepsilon > 0$, there exists $c \in M_0$ such that $\partial(p,ca) < \varepsilon$.
  \item
    Any of the above, for $m$-types rather than $1$-types.
  \end{enumerate}
  Similarly, \textit{mutatis mutandis}, for extremal saturation.
\end{lemma}
\begin{proof}
  \begin{cycprf*}
  \item By definition.
  \item Let $b,c$ be as in the hypothesis.
    Choose $c' \in M_0$ such that $d(c',c) < \varepsilon$.
    Then $\partial(c'a,p) \leq d(c'a,cb) + \partial(cb,p) < 3\varepsilon$, which is good enough.
  \item
    Let us first prove that \autoref{item:ApproximatelySaturatedByDenseMoveType} holds for $p(x,a) \in \tS^\aff_m(a)$, for any $m \in \bN$.
    For $m = 0$ this is vacuous.
    Let us now assume the statement for $m$ and prove it for $m+1$.
    Let $p(x,y,a) \in \tS^\aff_{m+1}(a)$ (so $p(x,y,z) \in \tS^\aff_{n+m+1}(\cL)$).
    Let $p_0(x,a) \in \tS^\aff_m(a)$ be its restriction to $x$.
    By the induction hypothesis, there exists $b \in M_0^m$ such that $\partial(p_0,ba) < \varepsilon$.
    In an affine extension, we may find $b',c',a'$ such that $\tp^\aff(b'c'a') = p$, and $d(b'a',ba) < \varepsilon$.
    Then $\partial(p,bc'a) \leq d(b'a',ba) < \varepsilon$.
    By hypothesis, there exists $c \in M_0$ such that $\partial(bc'a,bca) < \varepsilon$.
    Thus $bc \in M_0^{m+1}$ and $\partial(p,bca) < 2\varepsilon$, which is good enough.

    Next, we claim that if $a \in M_0^n$, $p \in \tS^\aff_n(\cL)$, and $\partial(a,p) < \varepsilon$, then there exists $b \in M^n$ realizing $p$ such that $d(a,b) < 3\varepsilon$.
    Indeed, let us construct a Cauchy sequence $(a^k : k \in \N)$ such that $a^k \in M_0^n$ and $\partial(a^k,p) < \varepsilon/2^k$.
    We start with $a^0 = a$.
    Given $a^k$, we may choose in an affine extension a tuple $b^k$ realizing $p$ such that $d(a^k,b^k) < \varepsilon/2^k$.
    By our previous claim, there exists $a^{k+1} \in M_0^n$ such that $\partial(a^{k+1}a^k,b^ka^k) < \varepsilon/2^{k+1}$.
    In particular, $\partial(a^{k+1},p) < \varepsilon/2^{k+1}$ and $d(a^{k+1},a^k) \leq d(b^k,a^k) + \varepsilon/2^{k+1} < 3\varepsilon/2^{k+1}$.
    The sequence $(a^k)$ converges to a limit $b \in M^n$ which is as desired.

    Now let us prove that \autoref{item:ApproximatelySaturatedByDenseBase} holds for arbitrary $m$, from which \autoref{item:ApproximatelySaturatedByDenseMoveTypeAndParams} follows as before.
    Let $a \in M^n$, $p(x,a) \in \tS^\aff_m(a)$, and $\varepsilon > 0$ be given.
    Let $c$ realize $p(x,a)$ in some affine extension.
    Choose $a' \in M_0^n$ such that $d(a,a') < \varepsilon$.
    By \autoref{item:ApproximatelySaturatedByDenseMoveType} for $m$-types, there exists $c' \in M_0^m$ such that $\partial(ca',c'a') < \varepsilon$.
    Then $\partial(c'a',p) \leq \partial(c'a',ca') + d(a',a) < 2\varepsilon$.
    By our claim, there exist $a'' \in M^n$ and $c'' \in M^m$ that realize $p(x,y)$, such that $d(a''c'',a'c') \leq 6\varepsilon$.
    In particular, $d(a'',a) < 7\varepsilon$, which is good enough.
  \item[\impfirst] Immediate.
  \end{cycprf*}
  The same argument works if we restrict to extreme types and extremal models.
\end{proof}

\begin{prop}
  \label{prp:ApproximatelySaturatedHomogeneous}
  Let $M$ be a structure and let $T = \Th^\aff(M)$.
  Then $M$ is approximately affinely $\aleph_0$-saturated if and only if it is weakly approximately affinely $\aleph_0$-homogeneous and realizes every affine type $p \in \tS^\aff_n(T)$ for every $n \in \bN$.

  Similarly, \textit{mutatis mutandis}, for extremal saturation and extreme types.
\end{prop}
\begin{proof}
  Immediate, using \autoref{lem:ParameterChange} in one direction (and there is no need for \autoref{lem:HomogeneousTypes} in the other).
\end{proof}

Combining \autoref{thm:ApproximatelyHomogeneousTypes} with \autoref{prp:ApproximatelySaturatedHomogeneous} yields the following two corollaries.
\begin{cor}
  \label{cor:ApproximatelySaturatedHomogeneousSeparable}
  Let $M$ be a separable structure.
  If $M$ is approximately affinely or extremally saturated, then it is approximately affinely homogeneous.
\end{cor}

\begin{cor}
  \label{cor:ApproximatelySaturatedUnique}
  Let $T$ be a complete affine theory.
  Then any two separable approximately affinely saturated models of $T$ are isomorphic, as are any two separable approximately extremally saturated models.
\end{cor}

\section{Omitting types}
\label{sec:omitting-types}

Every type space $\tS^\aff_x(T)$ or $\cE_x(T)$, for a countable tuple of variables $x$, carries two structures: the logic topology $\tau$, given by its construction as a state space, and a metric structure, given by the metric $\partial$ of \autoref{defn:AffineTypeDistance} (and \autoref{rem:type-metric-inf-tuples}). The interplay between them allows to have an appropriate notion of isolated types and an omitting types theorem. The situation is similar to that of continuous logic if one restricts to extreme types but the proofs are somewhat more involved.

We recall that $\cE^{\fM}_\N(T)$ denotes the set of extreme types of countable tuples that enumerate a model of $T$. For what follows, we will need the following slight refinement. We define $\cE^{\fM^*}_\bN(T)$ as the collection of all $p \in \cE^\fM_\bN(T)$ such that if $p = \tp^\aff(a)$, then $\overline{\{a_n : n \in \bN\}} = \overline{\{a_n : n \geq k\}}$ for every $k \in \bN$. In words, every tail of a realization of $p$ is dense in the model enumerated by the realization. This is only relevant if the model has isolated points, and then it means that every isolated point is hit by the tuple infinitely many times.

\begin{lemma}
  \label{l:enmod-Gdelta}
  Suppose that $T$ is an affine theory in a separable language.
  Then $\cE_\bN(T)$ is a Polish space, and $\cE^\fM_\bN(T)$ and $\cE^{\fM^*}_\bN(T)$ are dense $G_\delta$ subsets of $\cE_\bN(T)$.
\end{lemma}
\begin{proof}
  We may assume that the language is countable.
  Then $\tS^\aff_\bN(T)$ is metrizable, so \autoref{l:extreme-Gdelta} applies and $\cE_\bN(T)$ is a Polish space.
  For a formula $\varphi(x,y)$, with $x = (x_i : i \in \bN)$ and $y$ a singleton, let $U_{\varphi,i} \subseteq \cE_\bN(T)$ be the subset defined by the open condition $\varphi(x,x_i) > \sup_y \varphi(x,y) - 1$, and let $U_\varphi = \bigcup_i U_{\varphi,i}$.

  In order to show that $U_\varphi$ is dense, consider a non-empty open set $V \subseteq \cE_\bN(T)$, say with $p \in U$.
  By \autoref{l:nbhds-extreme}, we may assume that $U$ is defined by an open condition $\psi(x) > 0$, for some affine formula $\psi$.
  We may assume that only variables in $x' = (x_i : i < n)$ actually occur in $\varphi$ or $\psi$ (plus the new variable $y$), and let $p' = \pi_{x'}(p) \in \cE_n(T)$.
  By the same argument as for \autoref{th:ExtremalModelTypes}, we can extend $p'$ to $p'' \in \cE_\bN(T)$ that satisfies $\varphi(x,x_n) = \sup_y \varphi(x,y)$.
  Then $p'' \in U_\varphi \cap V$.

  Since $\cE^\fM_\bN(T)$ is the (countable) intersection of all $U_\varphi$, by the Baire category theorem, it is a dense $G_\delta$ subset of $\cE_\bN(T)$.

  For $\cE^{\fM^*}_\bN(T)$, proceed as above, but with $U'_{\varphi,k} = \bigcup_{i \geq k} U_{\varphi,i}$.
  By the same argument, each is open and dense, and their intersection over all $\varphi$ and $k$ is $\cE^{\fM^*}_\bN(T)$.
\end{proof}

\begin{theorem}[Omitting types]
  \label{th:omitting-types}
  Suppose that $T$ is an affine theory in a separable language and let for every $n$, $\Xi_n \sub \cE_n(T)$ be $\tau$-meager and $\dtp$-open.
  Then there is a separable, extremal model of $T$ that omits all $\Xi_n$.
\end{theorem}
\begin{proof}
  For an injective tuple $\sigma \in \N^n$, let $\pi_\sigma \colon \cE_\bN(T) \to \cE_n(T)$ denote the corresponding projection.
  Let
  \begin{equation*}
    \Xi = \bigcup_\sigma \pi_\sigma^{-1}(\Xi_{|\sigma|}),
  \end{equation*}
  the union being taken over all such $\sigma$.
  Each $\pi_\sigma$ is $\tau$-continuous and open, so $\Xi$ is $\tau$-meager.
  By \autoref{l:enmod-Gdelta}, there exists $p \in \cE^{\fM^*}_\bN(T) \sminus \Xi$.
  By \autoref{remark:TypeOfExtremalModel}, there exists a separable extremal model $M \models T$, as well as $a \in M^\bN$, such that $p = \tp^\aff(a)$ and every tail of $(a_i : i \in \bN)$ is dense in $M$.

  Assume now that some $b \in M^n$ realizes a type $q \in \Xi_n$.
  Since $\Xi_n$ is $\dtp$-open, we have $B_\partial(q,\varepsilon) \subseteq \Xi_n$ for some $\varepsilon > 0$.
  By tail-density of $a$, we can find a strictly increasing map $\sigma\colon n \rightarrow \bN$ such that $d(b_i,a_{\sigma(i)}) < \varepsilon/n$.
  Then $\partial\bigl(\pi_\sigma(p),q\bigr) < \varepsilon$, so $\pi_\sigma(p) \in \Xi_n$, contradicting our construction.
  Therefore, $M$ omits every type in $\Xi_n$, for every $n$.
\end{proof}

\begin{defn}
  \label{defn:TopometricIsolatedPoint}
  A \emph{topometric space} is a triplet $(X,\tau,\partial)$, where $X$ is a set of points, $\tau$ is a topology on $X$, and $\partial$ is a generalized metric (i.e., possibly infinite) such that the $\partial$-topology refines $\tau$ and $\dtp$ is $\tau$-lower semi-continuous.

  A point $x \in X$ is \emph{isolated} in the topometric space $X$ if $\tau$ and the $\partial$-topology agree at $x$: i.e., if every $\partial$-ball around $x$ in $X$ is a $\tau$-neighborhood of $x$.
  It is \emph{weakly isolated} if every $\partial$-ball around $x$ in $X$ has non-empty $\tau$-interior.
\end{defn}

Observe that \autoref{prop:AffineTypeDistance} asserts that if $x$ is a countable tuple of variables, $\tS^\aff_x(\cL)$ is a topometric space, and if $T$ is complete, then the metric $\dtp$ on $\tS^\aff_x(T)$ is always finite.
This induces a topometric space structure on $\cE_x(T)$.
Therefore, for $p \in \cE_x(T)$, we may ask whether it is (weakly) isolated in $\tS^\aff_x(T)$, or in $\cE_x(T)$.

\begin{lemma}
  \label{lemma:IsolatedExtremeType}
  Let $T$ be a complete affine theory and let $x$ be a countable tuple of variables.
  For a type $p \in \cE_x(T)$, the following are equivalent:
  \begin{enumerate}
  \item
    \label{item:IsolatedExtremeTypeDefinable}
    The function $\partial(p,\cdot)\colon \tS^\aff_x(T) \rightarrow \bR^+$ is continuous and affine.
  \item The type $p$ is isolated in $\tS^\aff_x(T)$.
  \item The type $p$ is isolated in $\cE_x(T)$.
  \item The type $p$ is weakly isolated in $\cE_x(T)$.
  \item The function $\partial(p,\cdot)\colon \cE_x(T) \rightarrow \bR^+$ is continuous at $p$.
  \end{enumerate}
\end{lemma}
\begin{proof}
  \begin{cycprf}
  \item[\impnnnext] Immediate.
  \item[\impfirst]
    Let $\varepsilon > 0$.
    By assumption and \autoref{l:nbhds-extreme}, there exist an affine formula $\varphi(x)$ and a type $p'$ such that $p' \in \llbracket \varphi < 1 \rrbracket_\cE \subseteq B_\partial(p,\varepsilon)$.
    By post-composing with an affine function $\R \to \R$, we may assume that $\min \varphi = 0$.
    Then $\phi^{-1}(0)$ is a face of $\tS^\aff_x(T)$ and by replacing $p'$ by an extreme point of that face, we may further assume that $\varphi(p') = 0$.
    Let $R$ be the bound on the distance of $n$-tuples in models of $T$.
    Define
    \begin{gather*}
      \psi(x) = \inf_y \, \bigl( d(x,y) + R \varphi(y) \bigr).
    \end{gather*}
    Let $q \in \tS^\aff_x(T)$.
    On the one hand,
    \begin{gather*}
      \psi(q)
      \leq \partial(q,p') + R \varphi(p')
      \leq \partial(p,q) + \partial(p,p')
      < \partial(p,q) + \varepsilon.
    \end{gather*}
    On the other hand, let $M \models T$ realize $q$, say by $a \in M^n$.
    Let $b \in M^n$ be arbitrary, and $q' = \tp^\aff(b)$.
    Let $E_+ = \llbracket \varphi \geq 1 \rrbracket_\cE$ and $E_- = \llbracket \varphi < 1 \rrbracket_\cE$.
    Since $q' \in \clco(E_+ \cup E_-)$, we can express it as $(1-\lambda) q'_- + \lambda q'_+$, where $q'_\pm \in \clco(E_\pm)$.
    We have $E_- \subseteq B_\partial(p,\varepsilon)$ by hypothesis, and since $\partial$ is convex and lower semi-continuous, we have $\partial(p,q'_-) \leq \varepsilon$.
    Therefore
    \begin{equation*}
      \begin{split}
        \partial(p,q)
        &\leq \partial(p,q') + \partial(q',q) \\
        &\leq (1-\lambda) \partial(p,q'_-) + \lambda \partial(p,q'_+) + d(a,b) \\
        &\leq \varepsilon + \lambda R + d(a,b).
      \end{split}
    \end{equation*}
    In addition,
    \begin{gather*}
      \varphi(b) = \varphi(q') = (1-\lambda) \varphi(q'_-) + \lambda \varphi(q'_+) \geq \lambda.
    \end{gather*}
    Therefore,
    \begin{gather*}
      \partial(p,q) \leq d(a,b) + R \varphi(b) + \varepsilon.
    \end{gather*}
    Since $b$ was arbitrary, $\partial(p,q) \leq \psi(q) + \varepsilon$.

    We have shown that $|\partial(p,q) - \psi(q)| \leq \varepsilon$ for all $q \in \tS^\aff_x(T)$, which is enough.
  \item[\impnumnum{1}{5}{3}] Immediate.
  \end{cycprf}
\end{proof}

From now on, we may simply refer to a type $p \in \cE_x(T)$ as being \emph{isolated}, without specifying if this is in $\tS^\aff_x(T)$ or in $\cE_x(T)$.
Note that condition \autoref{item:IsolatedExtremeTypeDefinable} of \autoref{lemma:IsolatedExtremeType} for an arbitrary $p \in \tS^\aff_x(T)$ implies that $p$ is exposed and therefore, extreme.

\begin{prop}
  \label{prop:IsolatedExtremeTypeRealizedIFF}
  Let $T$ be a complete affine theory in a separable language, let $x$ be countable, and let $p \in \tS^\aff_x(T)$.
  Then the following are equivalent:
  \begin{enumerate}
  \item The type $p$ is extreme and isolated.
  \item The type $p$ is realized in every model of $T$, and moreover, its sets of realizations form a uniformly definable family of sets.
  \item The type $p$ is realized in every extremal model of $T$.
  \end{enumerate}
\end{prop}
\begin{proof}
  \begin{cycprf}
  \item By \autoref{lemma:IsolatedExtremeType} and \autoref{prop:DefinableSetTypes} applied to $X = \{p\}$, the sets of realizations of $p$ in models of $T$ form a uniformly definable family of sets, which are, in particular, non-empty.
  \item Immediate.
  \item[\impfirst]
    Since $p$ is realized in extremal models, it is extreme.
    Assume that it is not isolated.
    By \autoref{lemma:IsolatedExtremeType}, some $\dtp$-neighborhood $\Xi = B_\partial(p,r) \cap \cE_x(T)$ has empty $\tau$-interior in $\cE_x(T)$.
    On the one hand, $\Xi$ is $\dtp$-open in $\cE_x(T)$.
    On the other hand, $\Xi = \bigcup_{s < r} \overline{B}_\partial(p,s) \cap \cE_x(T)$, where $\overline{B}_\partial(p,s)$ denotes the closed ball.
    It follows from lower semi-continuity of $\dtp$ that each closed $\dtp$-ball is also $\tau$-closed, so $\overline{B}_\partial(p,s) \cap \cE_x(T)$ is closed and nowhere dense.
    Therefore $\Xi$ is $\tau$-meager, and can be omitted in an extremal separable model by \autoref{th:omitting-types}.
  \end{cycprf}
\end{proof}

\begin{remark}
  \label{remark:IsolatedExtremeTypeInfinitary}
  Similarly to definable sets (cf.\ \autoref{remark:DefinableSetInfinitary}), we have two equivalent criteria for a type in a countably infinite tuple of variables to be isolated: one is given by \autoref{lemma:IsolatedExtremeType} and the other is that every restriction to finitely many variables is isolated. We also make the convention to use the second definition of isolated for a type over an arbitrary tuple of variables.
\end{remark}

\begin{prop}
  \label{p:isolated-gdelta-closed}
  Let $T$ be an affine theory and let $x$ be a countable tuple of variables.
  Then the set of isolated types in $\cE_x(T)$ is $G_\delta$ in $\tau$ and $\dtp$-closed.
\end{prop}
\begin{proof}
  Let $I_x(T) \subseteq \cE_x(T)$ denote the collection of isolated types, and let
  \begin{gather*}
    I'_x = \bigcap_k \bigcup_{p \in \cE_x(T)} \Bigl( B_\partial(p,2^{-k}) \cap \cE_x(T) \Bigr)^\circ,
  \end{gather*}
  where $\cdot^\circ$ denotes the $\tau$-interior in $\cE_x(T)$. Note that $I'_x$ is $G_\delta$ in $\tau$.
  We have
  \begin{equation*}
    I_x(T) \subseteq I'_x \subseteq \overline{I'_x}^\partial \subseteq I_x(T)
  \end{equation*}
  by \autoref{lemma:IsolatedExtremeType}, so equality holds throughout. To see the last inclusion, let $q \in \cl[\dtp]{I'_x}$ and fix $r > 0$ in order to show that $B_\dtp(q, r)$ has non-empty $\tau$-interior. Let $q' \in B_\dtp(q, r/2) \cap I'_x$. Then there exists $p \in \cE_x(T)$ such that $\dtp(p, q') < r/4$ and $B_\dtp(p, r/4)$ has non-empty $\tau$-interior in $\cE_x(T)$. The last ball is contained in $B_\dtp(q, r)$, so we are done.
\end{proof}

A model $M$ of a complete theory $T$ is called \emph{atomic} if it only realizes isolated, extreme $n$-types, for every $n$.
In particular, an atomic model is extremal, and following \autoref{remark:IsolatedExtremeTypeInfinitary}, types of arbitrary tuples in an atomic model are isolated as well.

\begin{prop}
  \label{p:existence-atomic}
  Let $T$ be a complete affine theory in a separable language.
  Then $T$ admits a (separable) extremal atomic model if and only if the isolated types are dense in $\cE_n(T)$ for all $n$.
\end{prop}
\begin{proof}
  Let $I_n(T)$ denote the set of isolated types in $\cE_n(T)$.
  If $M$ is atomic, of any density character, then every type in $I_n(T)$ is realized in $M$ by \autoref{prop:IsolatedExtremeTypeRealizedIFF}, and therefore, by \autoref{l:realized-extreme-dense}, $I_n(T)$ is dense in $\cE_n(T)$.

  Conversely, assume that $I_n(T)$ is dense in $\cE_n(T)$ for all $n$, and let $\Xi_n$ be the complement.
  By \autoref{p:isolated-gdelta-closed}, $\Xi_n$ is $\tau$-meager, and $\dtp$-open.
  By \autoref{th:omitting-types}, $T$ admits a separable extremal model that omits $\Xi_n$ for all $n$.
  Such a model is atomic.
\end{proof}

\begin{theorem}
  \label{th:uniqueness-atomic-model}
  Let $T$ be a complete, affine theory in a separable language.
  \begin{enumerate}
  \item
    A model $M \models T$ is separable and atomic if and only if it is \emph{prime}, that is to say that it admits an affine embedding in every model of $T$.
  \item
    Such a model, if it exists, is unique up to isomorphism, and is approximately affinely homogeneous.
  \end{enumerate}
\end{theorem}
\begin{proof}
  Assume first that $M$ is prime.
  Since $T$ admits separable extremal models, $M$ must be separable.
  In addition, every type realized in $M$ must be realized in every model of $T$, so by \autoref{prop:IsolatedExtremeTypeRealizedIFF}, $M$ is atomic.

  Assume now that $M$ is separable and atomic.
  Let $a = (a_i : i \in \bN)$ enumerate a dense subset of $M$.
  Then $p = \tp^\aff(a) \in \cE^\fM_\bN(T)$ is isolated in $\cE_\bN(T)$.
  By \autoref{prop:IsolatedExtremeTypeRealizedIFF}, if $N$ is any other model of $T$, then $p$ is realized in $N$, say by $a'$.
  Now, $M \cong \overline{\{a'_i : i \in \bN\}} \preceq^\aff N$, so $M$ admits an affine embedding in $N$.
  Thus $M$ is prime.

  Let us show that every atomic model $M$ is weakly approximately affinely $\aleph_0$-homogeneous (as in \autoref{dfn:ApproximatelyHomogeneous}).
  Let $a,b \in M^n$ have the same affine type $p(x)$ and let $c \in M$.
  The type $q(x,y) = \tp^\aff(a,c)$ is isolated, so we may quantify over its realizations and
  \begin{gather*}
    \varphi(x) = \inf_{(x',y) \models q} d(x,x')
  \end{gather*}
  is an affine definable predicate.
  We have $\varphi(a) = 0$ (as witnessed by $a,c$), so $\varphi(b) = 0$ as well, and there exist $b',d$ realizing $q$ such that $d(b,b')$ is as small as desired, proving our homogeneity claim.
  Since any two atomic models of $T$ realize exactly the isolated $n$-types, we may now conclude the proof using \autoref{thm:ApproximatelyHomogeneousTypes}.
\end{proof}


\part{Direct integrals and extremal decomposition}
\label{part:Direct-integrals}

\section{Direct integrals}
\label{sec:direct-integrals}

As affine logic is a fragment of continuous logic, one can use ultraproducts to construct models from existing ones and \Los's theorem is valid as usual.
The main novelty is that now one can also construct \emph{direct integrals} of a collection of models indexed by a probability space.
Such constructions have been considered in functional analysis and representation theory (most notably direct integrals of Hilbert spaces, from where we borrow the terminology) but only for separable structures in a separable language, because of measurability issues.

The construction we describe in this section works in complete generality, with no restriction on either the probability space or the density character of the language or the structures. The measurability issues are overcome by imposing a uniformly measurable presentation of the models as direct limits of countable pieces, as we explain below.

\begin{defn}
  \label{df:field}
  Let $\Omega$ be a set.
  \begin{enumerate}
  \item A family of $\cL$-structures $M_\Omega = (M_\omega : \omega \in \Omega)$ will be called a \emph{field of structures}.
  \item A \emph{section} of a field $M_\Omega$ is an element of the product $\prod_{\omega \in \Omega} M_\omega$, or, in other words, a function $f \colon \Omega \to \coprod_{\omega \in \Omega} M_\omega$ such that $f(\omega) \in M_\omega$ for all $\omega$.
  \item A family of sections $e_I = (e_i : i \in I)$ is a \emph{pointwise enumeration} of $M_\Omega$ if $\bigl\{ e_i(\omega) : i \in I \bigr\}$ is dense in $M_\omega$ for every $\omega \in \Omega$.
  \item Given a pointwise enumeration and $I_0 \subseteq I$, we define
    \begin{gather*}
      M_{\omega,I_0} \coloneqq \cl{\bigl\{ e_i(\omega) : i \in I_0 \bigr\}} \subseteq M_\omega.
    \end{gather*}
  \end{enumerate}
\end{defn}

If $M$ and $N$ are $\cL$-structures and $\cL_0 \sub \cL$, we write $M \preceq^\aff_{\cL_0} N$ if $M$ is an affine substructure of $N$ with respect to $\cL_0$-formulas. Similarly, we write $M \preceq^\cont_{\cL_0} N$ if $M$ is an elementary substructure of $M$ with respect to $\cL_0$-formulas, in the sense of continuous logic. (The few definitions and results of the present section concerning continuous logic will be relevant in \autoref{part:conn-with-cont}.)

\begin{defn}
  \label{defn:MeasurableField}
  Let $(\Omega, \cB, \mu)$ be a probability space, $M_\Omega$ a field of $\cL$-structures, and $e_I$ a pointwise enumeration of $M_\Omega$.
  For a countable $\cL_0 \subseteq \cL$, we denote by
  $$\fI^\aff(M_\Omega,e_I,\cL_0) \coloneqq \set{I_0\in \cP_{\aleph_0}(I) : \text{for $\mu$-a.e.\ } \omega \in \Omega,\ M_{\omega,I_0} \preceq^\aff_{\cL_0} M_\omega}$$
  the collection of all countable subsets $I_0 \subseteq I$ such that $M_{\omega,I_0} \preceq^\aff_{\cL_0} M_\omega$ almost surely. Similarly,
  $$\fI^\cont(M_\Omega,e_I,\cL_0) \coloneqq \set{I_0\in \cP_{\aleph_0}(I) : \text{for $\mu$-a.e.\ } \omega \in \Omega,\ M_{\omega,I_0} \preceq^\cont_{\cL_0} M_\omega}.$$

  We say that $(M_\Omega,e_I)$ is a \emph{measurable field of structures} if the following hold:
  \begin{enumerate}
  \item
    \label{item:MeasurableFieldPredicate}
    For every $n$-ary predicate symbol $P$ (including the distance symbol) and $n$-tuple $\bar{e}$ from $e_I$, the function
    \begin{equation*}
      \omega \mapsto P^{M_\omega}\bigl( \bar{e}(\omega) \bigr)
    \end{equation*}
    is measurable.
  \item
    \label{item:MeasurableFieldFunction}
    For every $n$-ary function symbol $F$ and $n$-tuple $\bar{e}$ from $e_I$, and every $i \in I$, the function
    \begin{equation*}
      \omega \mapsto
      d^{M_\omega}\bigl(F^{M_\omega}( \bar{e}(\omega) ), e_i(\omega) \bigr)
    \end{equation*}
    is measurable.
  \item
    \label{item:MeasurableFieldCofinal}
    For every finite $\cL_0 \subseteq \cL$, the collection $\fI^\aff(M_\Omega,e_I,\cL_0)$ is cofinal for inclusion in the set $\cP_{\aleph_0}(I)$ of all countable subsets of $I$.
  \end{enumerate}
  If, in addition, condition \autoref{item:MeasurableFieldCofinal} holds for the collection $\fI^\cont(M_\Omega,e_I,\cL_0)$, we say that $(M_\Omega,e_I)$ is an \emph{elementarily measurable field of structures}.
\end{defn}

Note that if $I_0 \in \fI^\aff(M_\Omega,e_I,\cL_0)$, then $M_{\omega,I_0}$ is almost surely an $\cL_0$-structure, that is to say that it is closed under all function symbols in $\cL_0$.
When $\cL_0$ consists of all symbols appearing in a formula $\varphi$, we may write $\fI^\aff(M_\Omega,e_I,\varphi)$ instead of $\fI^\aff(M_\Omega,e_I,\cL_0)$.

It follows from the affine chain theorem that $\fI^\aff(M_\Omega,e_I,\cL_0)$ is always closed under unions of countable chains.
Thus, condition \autoref{item:MeasurableFieldCofinal} implies that $\fI^\aff(M_\Omega,e_I,\cL_0)$ is cofinal and closed under unions of countable chains for every countable $\cL_0 \subseteq \cL$ (see \cite[Thm.~8.22]{Jech2003}).

Finally, note that if $I$ is countable, then condition \autoref{item:MeasurableFieldCofinal} is automatically satisfied, simply because $I \in \fI^\aff(M_\Omega,e_I,\cL_0)$ for every $\cL_0 \sub \cL$.

Similar remarks and conventions apply to the collections $\fI^\cont(M_\Omega,e_I,\cL_0)$.

\begin{defn}
  \label{defn:DirectIntegralMeasurableSection}
  Let $(M_\Omega,e_I)$ be a measurable field of structures, and let $f$ be a section.
  \begin{enumerate}
  \item We say that $f$ is \emph{basic} if it is a member of the family $e_I$.
  \item We say that $f$ is \emph{simple} if  it is of the form $f(\omega) = e_{k(\omega)}(\omega)$, where $k \colon \Omega \rightarrow I$ is measurable with finite image.
  \item Let $I_0 \subseteq I$ be countable.
    We say that $f$ is \emph{$I_0$-measurable} if $f(\omega) \in M_{\omega,I_0}$ almost surely, and for every $i \in I_0$, the function $\omega \mapsto d^{M_\omega}\bigl( f(\omega), e_i(\omega) \bigr)$ is measurable.
  \item We say that $f$ is \emph{measurable} if it is $I_0$-measurable for some countable $I_0 \subseteq I$.
  \end{enumerate}
  We identify two measurable sections that are equal almost surely.
  The collection of all measurable sections, up to this identification, will be denoted by $M_{\Omega,I}$.
  Similarly, the collection of all $I_0$-measurable sections will be denoted $M_{\Omega,I_0}$.
\end{defn}

Note that the set $M_{\Omega, I}$ depends on the pointwise enumeration $e_I$ even though it is absent from the notation.

\begin{lemma}
  \label{lemma:DirectIntegralMeasurableDistance}
  If $f,g \in M_{\Omega,I}$, then the function $\omega \mapsto d^{M_\omega}\bigl( f(\omega), g(\omega) \bigr)$ is measurable.
  Consequently, if $f \in M_{\Omega,I}$, then $f \in M_{\Omega,I_0}$ for every large enough countable $I_0 \subseteq I$ (indeed, as soon as $f(\omega) \in M_{\omega,I_0}$ almost surely).
\end{lemma}
\begin{proof}
  Say $f \in M_{\Omega,I_0}$  and $g \in M_{\Omega,I_1}$ for countable $I_0,I_1\subseteq I$.
  We may assume that $I_0$ and $I_1$ are well-ordered and define
  \begin{gather*}
    i(k, \omega) = \min \, \bigl\{i \in I_0 : d^{M_\omega}\bigl( f(\omega), e_i(\omega) \bigr) < 1/k\bigr\},
    \\
    j(k, \omega) = \min \, \bigl\{j \in I_1 : d^{M_\omega}\bigl( g(\omega), e_j(\omega) \bigr) < 1/k\bigr\}.
  \end{gather*}
  Then
  \begin{equation*}
    d^{M_\omega}\bigl( f(\omega), g(\omega) \bigr) = \lim_{k \to \infty} d^{M_\omega}\bigl( e_{i(k, \omega)}(\omega), e_{j(k, \omega)}(\omega) \bigr)
  \end{equation*}
  and it is measurable as the pointwise limit of measurable functions.

  The second assertion follows.
\end{proof}

\begin{lemma}
  \label{lemma:DirectIntegralMeasurableSection}
  Let $(M_\Omega,e_I)$ be a measurable field of structures.
  \begin{enumerate}
  \item
    \label{item:DirectIntegralMeasurableSectionSimple}
    Every basic or simple section is measurable.
  \item
    \label{item:DirectIntegralMeasurableSectionLimit}
    Every pointwise limit (almost surely) of a sequence of measurable sections is measurable.
  \item
    \label{item:DirectIntegralMeasurableSectionLimitSimple}
    Conversely, every measurable section is a pointwise limit (almost surely) of a sequence of simple sections.
    Moreover, if $f \in M_{\Omega,I_0}$ for some countable $I_0 \subseteq I$, then it is a limit of simple sections in $M_{\Omega,I_0}$.
  \end{enumerate}
\end{lemma}
\begin{proof}
  Item \autoref{item:DirectIntegralMeasurableSectionSimple} is clear.

  For \autoref{item:DirectIntegralMeasurableSectionLimit}, assume that $f_n(\omega) \rightarrow f(\omega)$ pointwise almost surely.
  Then we may find $I_0 \subseteq I$ countable such that $f_n \in M_{\Omega,I_0}$ for all $n$.
  It is then easy to see that $f \in M_{\Omega,I_0}$ as well.

  For \autoref{item:DirectIntegralMeasurableSectionLimitSimple} let $f$ be a measurable section, say $f \in M_{\Omega,I_0}$, and enumerate $I_0 = \{i_j : j \in \bN\}$.
  For $n\in\bN$ and $\epsilon > 0$, let $E_{n,\epsilon}\subseteq\Omega$ be the set of $\omega$ for which there is $j<n$ such that $d\bigl( f(\omega), e_{i_j}(\omega) \bigr) < \epsilon$. Then for every $k\in\bN$ there is $n(k)$ such that $\mu\big(\Omega\setminus E_{n(k),1/k}\big) < 1/k^2$. For $\omega\in E_{n(k),1/k}$, we define $s_k(\omega)$ to be equal to $e_{i_j}(\omega)$ where $j < n(k)$ is least such that $d\bigl( f(\omega), e_{i_j}(\omega) \bigr) < 1/k$, and for $\omega\in \Omega\setminus E_{n(k),1/k}$, we let $s_k(\omega) = e_{i_0}(\omega)$.
  Then $s_k$ is a simple section in $M_{\Omega,I_0}$ for each $k$, and $s_k \rightarrow f$ pointwise on the set $\bigcup_m\bigcap_{k\geq m}E_{n(k),1/k}$, which has measure 1 by Borel--Cantelli.
\end{proof}

\begin{lemma}
  \label{lemma:DirectIntegralMeasurableFunctionSymbol}
  Let $(\Omega, \cB, \mu)$ be a probability space and let $(M_\Omega,e_I)$ be a measurable field of structures.
  Let $F \in \cL$ be a function symbol and $\bar{f}$ a tuple in $M_{\Omega,I}$ of the appropriate length.
  Then the section $\omega \mapsto F^{M_\omega}\bigl(\bar f(\omega)\bigr)$ is measurable.
  Consequently, the same holds for every $\cL$-term.
\end{lemma}
\begin{proof}
  Choose $I_0 \in \fI^\aff(M_\Omega,e_I,\set{F})$ such that $\bar{f}$ is a tuple in $M_{\Omega,I_0}$.
  Then $F^{M_\omega}\bigl( \bar{f}(\omega) \bigr)$ belongs to $M_{\omega,I_0}$ almost surely.

  We now need to show that for $i \in I_0$, the function $d^{M_\omega}\bigl( F^{M_\omega}\bigl( \bar{f}(\omega) \bigr), e_i(\omega) \bigr)$ is measurable.
  When $\bar{f}$ consists of basic sections, this holds by definition, and the case where $\bar{f}$ consists of simple sections follows.
  For the general case, we may express the members of $\bar{f}$ as pointwise limits of simple sections in $M_{\Omega,I_0}$.
  Since $F^{M_\omega}$ is continuous, the desired function is a pointwise limit of measurable functions and therefore measurable.
\end{proof}

\begin{lemma}
  \label{lemma:DirectIntegralMeasurableLanguage}
  Let $(M_\Omega,e_I)$ be a measurable field of structures.
  Let $\varphi(\bar x)$ be an affine formula in $n$ variables and $\bar{f}$ an $n$-tuple in $M_{\Omega,I}$.
  Then the function $\omega \mapsto \varphi^{M_\omega}\bigl( \bar{f}(\omega) \bigr)$ is $\mu$-measurable.

  If moreover the field $(M_\Omega,e_I)$ is elementarily measurable, then the same holds for every continuous logic formula $\varphi(\bar x)$.
\end{lemma}
\begin{proof}
  Consider first the case where $\varphi$ is a predicate symbol $P$.
  Then our assertion holds by definition when $\bar{f}$ consists of basic sections.
  The case of simple sections follows, and since $P^{M_\omega}$ is continuous, so does the case of general measurable sections.

  We now proceed by induction on the structure of $\varphi$.
  By \autoref{lemma:DirectIntegralMeasurableFunctionSymbol}, the case where $\varphi$ is an arbitrary atomic formula holds, and the case of affine connectives is immediate.
  For quantifiers, let $\varphi(\bar x) = \sup_y \psi(\bar x, y)$ and suppose by induction that the function $\omega \mapsto \psi^{M_\omega}(\bar{f}(\omega), g(\omega))$ is $\mu$-measurable for every measurable $g$, and in particular, for every basic $g = e_i$.
  Choose $I_0 \in \fI^\aff(M_\Omega,e_I,\varphi)$ so all members of $\bar{f}$ belong to $M_{\Omega,I_0}$.
  Then almost surely,
  \begin{equation*}
    \begin{split}
      \varphi^{M_\omega}\bigl( \bar{f}(\omega) \bigr)
      &= \varphi^{M_{\omega,I_0}}\bigl( \bar{f}(\omega) \bigr) \\
      &= \sup_{i \in I_0} \psi\bigl( \bar{f}(\omega), e_i(\omega) \bigr)^{M_{\omega,I_0}} \\
      &= \sup_{i \in I_0} \psi\bigl( \bar{f}(\omega), e_i(\omega) \bigr)^{M_\omega},
    \end{split}
  \end{equation*}
  and this is a $\mu$-measurable function as a countable supremum of $\mu$-measurable functions.

  The same argument works for continuous formulas, if the field of structures is elementarily measurable.
\end{proof}

We define a natural metric on $M_{\Omega,I}$ by
\begin{equation*}
  d(f, g) = \int_\Omega d^{M_\omega}\bigl( f(\omega), g(\omega) \bigr) \ud \mu(\omega).
\end{equation*}

\begin{prop}
  \label{p:DirectIntegralMetricStructure}
  Let $(M_\Omega,e_I)$ be a measurable field of structures.
  \begin{enumerate}
  \item
    \label{item:DirectIntegralMetricStructureCauchy}
    If $(f_n)$ is a Cauchy sequence in $(M_{\Omega,I},d)$, then there exist a subsequence $(f_{n_k})$ and $g \in M_{\Omega,I}$ such that $f_{n_k} \rightarrow g$ pointwise almost surely.
  \item
    \label{item:DirectIntegralMetricStructurePointwise}
    If $(f_n)$ is a sequence of measurable sections and $f_n \rightarrow f$ pointwise almost surely, then $d(f_n,f) \rightarrow 0$.
    Conversely, if $d(f_n,f) \rightarrow 0$, then $f_{n_k} \rightarrow f$ pointwise almost surely for some subsequence.
  \item
    \label{item:DirectIntegralMetricStructureComplete}
    The metric space $(M_{\Omega,I},d)$ is complete and the collection of simple sections is dense in $(M_{\Omega,I},d)$.
  \item
    \label{item:DirectIntegralMetricStructureStandard}
    If $\Omega$ is a standard probability space and $I$ is countable, then $M_{\Omega,I}$ is separable.
  \end{enumerate}
\end{prop}
\begin{proof}
  For item \autoref{item:DirectIntegralMetricStructureCauchy}, choose $I_0$ large enough so $f_n \in M_{\Omega,I_0}$ for all $n$.
  Since the distance is measurable, we may find a subsequence $(f_{n_k})$ such that $\bigl( f_{n_k}(\omega) \bigr)$ is almost surely Cauchy in the complete space $M_\omega$.
  Let $g(\omega) = \lim_k f_{n_k}(\omega)$.
  Then $g \in M_{\Omega,I_0}$ and it is almost surely the pointwise limit of $(f_{n_k})$.

  Item \autoref{item:DirectIntegralMetricStructurePointwise} holds by the dominated convergence theorem, and its converse part follows from \autoref{item:DirectIntegralMetricStructureCauchy}.

  Item \autoref{item:DirectIntegralMetricStructureComplete} follows from the previous items together with \autoref{lemma:DirectIntegralMeasurableSection}~\autoref{item:DirectIntegralMeasurableSectionLimitSimple}.

  Item \autoref{item:DirectIntegralMetricStructureStandard} follows from \autoref{item:DirectIntegralMetricStructureComplete} and the fact that the measure algebra of a standard probability space is separable.
\end{proof}

We may now make $M_{\Omega,I}$ into an $\cL$-structure.
A function symbol $F \in \cL$ is interpreted pointwise, that is,
\begin{equation}
  \label{eq:def-functions}
  F^{M_{\Omega,I}}(\bar f)(\omega) = F^{M_\omega}\bigl( \bar f(\omega) \bigr).
\end{equation}
By \autoref{lemma:DirectIntegralMeasurableFunctionSymbol}, this is a measurable section.

A predicate symbol $P$ is treated in the same fashion as the metric, namely
\begin{equation}
  \label{eq:def-predicates}
  P^{M_{\Omega,I}}(\bar f) = \int_\Omega P^{M_\omega}\bigl( \bar f(\omega) \bigr) \ud \mu(\omega).
\end{equation}
It is clear that $P^{M_{\Omega,I}}$ respects the appropriate bound for $P$.

In both cases, an easy calculation using the convexity of the continuity moduli and Jensen's inequality shows that the interpretation respects the required continuity modulus.

\begin{defn}
  Let $(\Omega, \mu)$ be a probability space and let $(M_\Omega,e_I)$ be a measurable field of structures.
  The structure constructed above, namely, the collection $M_{\Omega,I}$ of all measurable sections (up to almost sure equality), equipped with the interpretation of the language, is called the \emph{direct integral} of $(M_\Omega,e_I)$, and denoted
  \begin{equation*}
    M_{\Omega,I} = \int^\oplus_\Omega M_\omega \ud \mu(\omega).
  \end{equation*}
\end{defn}

\begin{lemma}
  \label{lem:LosQuantifier}
  Let $(\Omega, \mu)$ be a probability space, let $(M_\Omega,e_I)$ be a measurable field of structures.
  Then for every affine formula $\psi(x,y)$ (with singleton $y$), every $\eps>0$ and every tuple $\bar f \in M_{\Omega,I}^x$, there is a measurable section $g\in M_{\Omega,I}$ such that
  \begin{gather*}
    \psi^{M_\omega}\bigl( \bar{f}(\omega), g(\omega) \bigr) > \sup_{y \in M_\omega} \psi^{M_\omega}\bigl( \bar f(\omega), y \bigr) - \eps
  \end{gather*}
  for almost every $\omega\in\Omega$. Consequently:
  \begin{gather*}
    \sup_{g \in M_{\Omega,I}} \int_\Omega \psi^{M_\omega}\bigl( \bar f(\omega), g(\omega) \bigr) \ud \mu(\omega)
    = \int_\Omega \sup_{y \in M_\omega} \psi^{M_\omega}\bigl( \bar f(\omega), y \bigr) \ud \mu(\omega).
  \end{gather*}

  If moreover the field $(M_\Omega,e_I)$ is elementarily measurable, then the same holds for every continuous logic formula $\psi(\bar x,y)$.
\end{lemma}
\begin{proof}
  Let $s(\omega) = \sup_{y \in M_\omega} \psi^{M_\omega}\bigl( \bar f(\omega), y \bigr)$. This is a $\mu$-measurable function, by \autoref{lemma:DirectIntegralMeasurableLanguage}.
  Choose $I_0 \in \fI^\aff(M_\Omega,e_I,\psi)$ such that $\bar f$ is a tuple from $M_{\Omega,I_0}$.
  Then almost surely $M_{\omega,I_0} \preceq^\aff_{\cL_0} M_\omega$, so
  \begin{equation*}
    s(\omega)
    = \sup_{y \in M_{\omega,I_0}} \psi^{M_\omega}\bigl( \bar f(\omega), y \bigr)
    = \sup_{i \in I_0} \psi^{M_\omega}\bigl( \bar f(\omega), e_i(\omega) \bigr).
  \end{equation*}
  As in the proof of \autoref{lemma:DirectIntegralMeasurableDistance}, we may assume that $I_0$ is well-ordered, and define
  \begin{gather*}
    i(\omega) = \min \, \bigl\{i \in I_0 : \psi^{M_\omega}\bigl( \bar{f}(\omega), e_i(\omega) \bigr) > s(\omega) - \eps\bigr\}.
  \end{gather*}
  Then $g(\omega) = e_{i(\omega)}(\omega)$ is a.e.\ equal to a measurable section, and $\psi^{M_\omega}\bigl( \bar{f}(\omega), g(\omega) \bigr) > s(\omega)-\eps$ almost surely. The rest follows.

  The case of elementarily measurable fields and continuous formulas holds with the same proof.
\end{proof}

\begin{theorem}[\Los's theorem for direct integrals]
  \label{th:Los}
  Let $(M_\Omega,e_I)$ be a measurable field of $\cL$-structures over a probability space $(\Omega, \mu)$.
  Then for every affine $\cL$-formula $\varphi$ and tuple $\bar f$ from $M_{\Omega,I}$, the function $\omega \mapsto \varphi^{M_\omega}\bigl( \bar f(\omega) \bigr)$ is $\mu$-measurable and
  \begin{equation}
    \label{eq:Los}
    \varphi^{M_{\Omega,I}}(\bar f) = \int_\Omega \varphi^{M_\omega}\bigl( \bar f(\omega) \bigr) \ud \mu(\omega).
  \end{equation}
\end{theorem}
\begin{proof}
  The measurability part is just \autoref{lemma:DirectIntegralMeasurableLanguage}.
  For \autoref{eq:Los}, we proceed by induction, where the cases of atomic formulas and affine connectives are easy, and quantifiers are dealt with in \autoref{lem:LosQuantifier}.
\end{proof}

A Feferman--Vaught-like theorem for continuous logic formulas in direct integrals of elementarily measurable fields will be discussed in \autoref{sec:continuity-direct-integral}.

It is natural to ask how much the direct integral construction depends on the choice of the basic sections $e_I$.

\begin{defn}
  \label{df:equiv-fields}
  Let $e_I$, $e'_J$ be two pointwise enumerations of a field of structures $M_\Omega$ such that $(M_\Omega,e_I)$ and $(M_\Omega,e'_J)$ are measurable fields. We say that $e_I$ and $e'_J$ are \emph{equivalent} if the functions
  \begin{equation*}
    \omega \mapsto d^{M_\omega}\bigl( e_i(\omega), e_j'(\omega) \bigr)
  \end{equation*}
  are measurable for all $i \in I$ and $j \in J$.
\end{defn}

\begin{remark}
  \label{rem:equivalent-fields}
  It follows from \autoref{lemma:DirectIntegralMeasurableDistance} that equivalence of pointwise enumerations is an equivalence relation.
  If $e_I$ and $e'_J$ are equivalent, then a section is measurable with respect to $e_I$ if and only if it is measurable with respect to $e'_J$.
  Consequently, $(M_\Omega,e_I)$ and $(M_\Omega,e'_J)$ have the same direct integral.
\end{remark}

\begin{ntn}
  \label{ntn:DirectIntegralConvexCombination}
  A particularly simple case of the direct integral construction is when $\Omega$ is countable and all singletons are measurable.
  Then all functions defined on $\Omega$ are measurable, and all possible choices of a pointwise enumeration $e_I$ are equivalent.
  In this case, we say that the measurable field is \emph{atomic} and call the direct integral $\int^{\oplus}_\Omega M_\omega \ud \mu(\omega)$ a \emph{convex combination} of the family $M_\Omega$. For these direct integrals we may use the notation
  $$\bigoplus_\omega \lambda_\omega M_\omega,$$
  where $\lambda_\omega = \mu(\set{\omega})$. This generalizes the construction of \autoref{lemma:NaiveConvexLos}. We may denote the elements (and tuples) of the convex combination of the family $M_\Omega$ by expressions of the form $\bigoplus\lambda_\omega a_\omega$, with $a_\omega\in M_\omega$. For example, the elements of $\half M\oplus\half N$ will be written $\half a\oplus\half b$, with $a\in M$ and $b\in N$.
\end{ntn}

Another important special case is when the field of structures is constant.

\begin{defn}
  \label{dfn:DirectMultiple}
  Let $(\Omega,\mu)$ be a probability space, and assume that $M_\omega = M$ for all $\omega\in\Omega$.
  Let $I \subseteq M$ be dense, and let $e_I$ be the collection of constant sections $e_i(\omega) = i$.
  We call the measurable field $(M_\Omega,e_I)$ constructed in this fashion a \emph{constant measurable field}.

  We call the corresponding direct integral the \emph{direct multiple} of $M$ by $(\Omega,\mu)$, and denote it by
  $$L^1(\Omega,\mu,M) \coloneqq \int^\oplus_\Omega M\ud \mu(\omega),$$
  or simply by $L^1(\Omega,M)$ if the measure is understood. It is clear that it does not depend on the choice of $I$.
\end{defn}

Note that by \Los's theorem for direct integrals, we have $M\preceq^\aff L^1(\Omega,M)$.

\begin{remark}
  \label{rem:L1-only-depends-MALG}
  The structure $L^1(\Omega, \mu, M)$ only depends on $\MALG(\Omega,\mu)$ and on $M$. This is because it can be viewed as the completion of the simple measurable sections, and each one of these is determined by a finite measurable partition of $\Omega$ up to measure zero and the corresponding elements of $M$.
\end{remark}

\begin{remark}
  \label{rem:L1OmegaM}
  If $(\Omega, \cB, \mu)$ is a measure space and $(M, d)$ is a metric space, one may wish to consider \emph{Borel measurable} functions $f \colon \Omega \to M$, i.e., those such that $f^{-1}(B) \in \cB$ for every Borel set $B \sub M$. This is a very strong condition for non-separable $M$. Indeed, by a theorem of Marczewski and Sikorski~\cite{Marczewski1948}, if $f \colon \Omega \to M$ is Borel measurable and the density character of $M$ is not greater than a \emph{real-valued measurable cardinal}, then $f$ is a.e.\ equal to a function with separable image. The existence of a real-valued measurable cardinal is equiconsistent with the existence of a measurable cardinal, which is a large cardinal hypothesis strictly stronger than the consistency of $\mathsf{ZFC}$ (see \cite[Thm.~22.1]{Jech2003}).

  On the other hand, we can reformulate \autoref{dfn:DirectMultiple} as follows:
  \begin{equation*}
    L^1(\Omega, \mu, M) = \set{f \colon \Omega \to M : \fd(\Im f)=\aleph_0 \text{ and $f$ is Borel measurable}} / =_{\text{a.e.}}.
  \end{equation*}
  The function symbols in $L^1(\Omega, \mu, M)$ are interpreted pointwise and the predicates by integrating.

  In view of the remark above, if $\fd(M)$ is small or simply there are no real-valued measurable cardinals, we can omit the condition $\fd(\Im f) = \aleph_0$ from the definition above.
\end{remark}

\begin{remark}
  \label{rmk:DirectMultipleDenseRealizedTypes}
  If $\Omega$ is an atomless probability space and $N = L^1(\Omega,M)$, then the collection of $n$-types realized in $N$ is closed under convex combinations.
  Letting $T = \Th^\aff(M) = \Th^\aff(N)$, it follows from \autoref{l:realized-convex-dense} that the collection of types realized in $N$ is dense in $\tS^\aff_n(T)$.
\end{remark}

\begin{remark}
  \label{rmk:elem-measurable-fields-convex-comb-direct-mult}
  We have introduced two notions of measurability for fields of structures: measurable and elementarily measurable fields.
  In some distinguished cases, a measurable field is automatically elementarily measurable.
  These include, of course, fields with a countable pointwise enumeration, but also atomic fields (i.e., those underlying convex combinations of structures) and constant fields (those underlying direct multiples).
  Other important cases will be discussed later.
\end{remark}

In \autoref{sec:DirectIntegralExtremalModels} we are going to see that every model of $T$ embeds affinely in a direct integral of extremal models.
On the other hand, if we wish to embed an extremal model in a direct integral, we have the following ``rigidity'' phenomenon.

\begin{prop}
  \label{prp:DirectIntegralExtremalRigid}
  Let $(M_\Omega,\mu)$ be a measurable field of models of an affine theory $T$, let $N \models T$ be an extremal model, and assume that $N \preceq^\aff \int_\Omega^\oplus M_\omega\ud\mu(\omega)$.
  Then for every affine formula $\phi(x)$, $a \in N^x$, and almost every $\omega \in \Omega$,
  \begin{equation*}
    \phi^N(a) = \phi^{M_\omega}(a(\omega)).
  \end{equation*}
  In particular, if $T$ has a separable language and $N$ is separable, then the map $a \mapsto a(\omega)$ is an affine embedding $N \to M_\omega$ for almost every $\omega$.
\end{prop}
\begin{proof}
  Fix $a \in N^x$ and let $p = \tp^\aff(a)$. Consider the map $\xi \colon \Omega \to \tS^\aff_x(T)$ defined by $\xi(\omega) = \tp^\aff(a(\omega))$, which is measurable with respect to the Baire algebra on $\tS^\aff_x(T)$. It follows from \autoref{th:Los} that $R(\xi_*\mu) = p$ and as $p$ is extreme, $\xi_*\mu$ must be the Dirac measure at $p$. Therefore, for every affine formula $\phi(x)$, the set $\set{\omega \in \Omega : \phi(p) \neq \phi^{M_\omega}(a(\omega))}$ is $\mu$-null.
\end{proof}

We finish this section with isolating a non-degeneracy condition for a measurable field that will be particularly relevant for proving the uniqueness of the integral decomposition in certain situations (see \autoref{sec:ExtremalDecomposition}).
What we want to exclude, for example, is the situation where every $M_\omega$ is a single point and the measure space $\Omega$ is non-trivial.

\begin{remark}
  \label{rmk:LosForTypes}
  Let $(M_\Omega,e_I)$ be a measurable field of $\cL$-structures over a probability space $(\Omega, \cB, \mu)$, and let $M_{\Omega,I}$ be its direct integral.
  Let $f = (f_i : i < \kappa)$ be a sequence of measurable sections, and let $\theta\colon \Omega \rightarrow \tS^\aff_\kappa(\cL)$ be the map $\omega \mapsto \tp^\aff\bigl( f(\omega) \bigr)$.
  If $\varphi(x)$ is an affine formula in $x = (x_i : i < \kappa)$, then $\varphi \circ \theta\colon \Omega \rightarrow \bR$ is $\mu$-measurable by \autoref{lemma:DirectIntegralMeasurableLanguage}, so $\theta$ is measurable with respect to the $\sigma$-algebra of Baire sets in $\tS^\aff_\kappa(\cL)$ and the $\sigma$-algebra $\cB_\mu$ of $\mu$-measurable sets in $\Omega$.

  Then \autoref{eq:Los} can be restated by saying that $R(\theta_* \mu) = \tp^\aff(f)$ in $M_{\Omega,I}$, and the map $\theta\colon (\Omega,\mu) \rightarrow \bigl( \tS^\aff_\kappa(\cL), \theta_* \mu \bigr)$ induces an embedding of measure algebras
  \begin{gather}
    \label{eq:LosForTypes}
    \theta^*\colon \MALG\bigl( \tS^\aff_\kappa(\cL), \theta_* \mu \bigr) \hookrightarrow \MALG(\Omega,\mu).
  \end{gather}
\end{remark}

\begin{defn}
  \label{dfn:NonDenegerateField}
  In the setting of \autoref{rmk:LosForTypes}, assume moreover that the sequence $f$ is dense in the direct integral $M_{\Omega,I}$. We say then that the measurable field $(M_\Omega,e_I)$ is \emph{non-degenerate} if the embedding $\theta^*$ of \autoref{eq:LosForTypes} is surjective.
\end{defn}

Notice that this does not depend on the choice of the dense sequence $f$.
Indeed, we may always extend $f$ to an enumeration $g = (g_\alpha : \alpha < \lambda)$ of \emph{all} measurable sections.
We obtain a sequence of maps
\begin{gather*}
  (\Omega,\mu) \overset{\tilde{\theta}}{\longrightarrow} \bigl(\tS^\aff_\lambda(\cL), \tilde{\theta}_*\mu \bigr) \overset{\pi}{\longrightarrow} \bigl(\tS^\aff_\kappa(\cL), \theta_*\mu \bigr),
\end{gather*}
where $\pi$ is the restriction to the first $\kappa$ variables, and $\pi \circ \tilde{\theta} = \theta$.

It follows from \autoref{p:DirectIntegralMetricStructure} that for every $\alpha < \lambda$ there exist $\beta_{\alpha,n} < \kappa$ such that $f_{\beta_{\alpha,n}} \rightarrow g_\alpha$ almost surely.
If $\varphi$ is an affine formula in $\lambda$ variables, then we may always express it as
\begin{gather*}
  \varphi(x_\alpha : \alpha < \lambda) = \varphi_0(x_{\alpha_t} : t <k).
\end{gather*}
Let
\begin{gather*}
  \psi_n(x_\beta : \beta < \kappa) = \varphi_0(x_{\beta_{\alpha_t,n}} : t < k).
\end{gather*}
Then, almost surely,
\begin{gather*}
  \psi_n^{M_\omega}\bigl(f(\omega)\bigr) \rightarrow \varphi^{M_\omega}\bigl( g(\omega) \bigr).
\end{gather*}
In other words, as a measurable function on $\bigl(\tS^\aff_\lambda(\cL), \tilde{\theta}_*\mu \bigr)$, we may factor $\varphi$ almost surely through $\pi$.
Since affine formulas generate the Baire $\sigma$-algebra, it follows that $\pi^*$ is an isomorphism in the dual diagram
\begin{gather*}
  \MALG(\Omega,\mu) \overset{\tilde{\theta}^*}{\longleftarrow} \MALG\bigl(\tS^\aff_\lambda(\cL), \tilde{\theta}_*\mu \bigr) \overset{\pi^*}{\longleftarrow} \MALG\bigl(\tS^\aff_\kappa(\cL), \theta_*\mu \bigr),
\end{gather*}
which is what we wanted to show.

Degenerate measurable fields are pathological, do not arise naturally, and are fairly easy to avoid.
Indeed, given such a field, we can always render it non-degenerate, without changing its direct integral, by reducing the $\sigma$-algebra of $\Omega$ to the collection of $\theta$-preimages of Baire sets.
In addition, the following provides a sufficient criterion for non-degeneracy that holds in essentially all interesting cases.

\begin{defn}
  \label{dfn:NonDegenerateTheory}
  We will say that a theory (affine or continuous) is \emph{non-degenerate} if all its models have at least two elements in at least one sort.
\end{defn}

Usually, this will be used for theories in a single sort, in which case, a compactness argument yields a strictly positive lower bound $\delta > 0$ on the diameter of the sort.
In a language with several (possibly infinitely many) sorts, a similar argument yields $\delta > 0$ and a finite family of sorts such that in each model, at least one of those has diameter $\geq \delta$.
In particular, the product of these sorts, equipped with the sum distance, has diameter $\geq \delta$ in all models.

\begin{lemma}
  \label{lem:NonDenegerateFieldTwoElements}
  Assume that $(M_\Omega,e_I)$ is a measurable field of models of some non-degenerate theory.
  Then $(M_\Omega,e_I)$ is non-degenerate.
\end{lemma}
\begin{proof}
  We may assume that $f$ contains every measurable section.
  For simplicity, let us assume that a single sort (rather than a finite product of sorts) witnesses non-degeneracy.
  Let $A \subseteq \Omega$ be measurable and let $\varepsilon > 0$.
  By hypothesis, the formula $\sup_{x,y} d(x,y)$ (with $x$ and $y$ in the witnessing sort) is strictly positive in $M_\omega$, for almost every $\omega\in \Omega$.
  Thus, using \autoref{lem:LosQuantifier} and cutting and gluing measurable sections, we can find $i,j<\kappa$ such that $f_i(\omega) = f_j(\omega)$ for every $\omega \in A$ and $f_i(\omega) \neq f_j(\omega)$ for every $\omega\in\Omega \setminus A$.
  Let
  \begin{gather*}
    C = \bigl\{ p \in \tS^\aff_{x^\ell}(\cL) : d^p(x_i, x_j) = 0 \bigr\}.
  \end{gather*}
  Then $C$ is Baire and $\theta^*(C) = \theta^{-1}(C) = A$, as desired.
\end{proof}


\section{Comparison with Bagheri's ultrameans}
\label{sec:comp-with-bagheri}

The direct integral is somewhat reminiscent of, yet quite different from, the ultraproduct construction.
Indeed, the former allows for a $[0,1]$-valued measure on an algebra of measurable sets, but requires countable additivity and various measurability conditions, which only make sense in the presence of a pointwise enumeration $e_I$.
The latter, by contrast, requires the measure to be $\{0,1\}$-valued and defined on all subsets (that is, to be an ultrafilter), but forgoes all additional requirements.

While one might want to look for a common generalization, we propose instead to view these two constructions as distinct building blocks, which can be joined as separate steps in more complicated constructions.
Let us exemplify this by analyzing yet another construction of a similar flavor, the ultramean proposed by Bagheri~\cite{Bagheri2010}.

Let $J$ be a set and let $M_J = (M_j : j \in J)$ be a field of structures.
Let $\nu$ be a finitely additive probability measure defined on the algebra of all subsets of $J$.
Bagheri defines the \emph{ultramean} of $M_J$, with respect to $\nu$, by interpreting  function and predicate symbols on $\prod_{j \in J} M_j$ as per \autoref{eq:def-functions} and \autoref{eq:def-predicates}, and passing to the quotient by the zero-distance relation.
He also proves a theorem analogous to \autoref{th:Los}.

We claim that Bagheri's ultramean can be constructed in two steps as a direct integral of ultraproducts
\begin{equation}
  \label{eq:UltrameanDecomposition}
  \int^\oplus_\Omega M_\omega \ud \mu(\omega),
\end{equation}
where $\Omega = \beta J$ is the collection of all ultrafilters on $J$ (i.e., the Stone--\v{C}ech compactification of $J$), and $M_\omega = \prod M_j / \omega$ is the ultraproduct with respect to $\omega \in \Omega$.

By standard measure construction theorems, $\Omega$ admits a unique Radon probability measure $\mu$ that satisfies $\mu(\widehat{A}) = \nu(A)$, where $A \subseteq J$ and $\widehat{A} = \{\omega \in \Omega : A \in \omega\}$.
Integration of functions on the finitely additive probability space $(J,\nu)$ (all functions being measurable) can be reduced to the integration in the more usual sense of continuous functions on $(\Omega,\mu)$.
It is possible to do essentially the same for the entire ultramean construction.

Let $I = \prod_j M_j$ be the set of all sections of $M_J$.
Then each $M_\omega$ is a quotient of $I$, so let $e_i(\omega) \in M_\omega$ be the class of $i \in I$ in $M_\omega$.
This makes $e_i$ a section of the field $M_\Omega = (M_\omega:\omega\in\Omega)$, and $e_I = (e_i : i \in I)$ is a pointwise enumeration of $M_\Omega$ (in fact, for every $\omega$, it enumerates all of $M_\omega$).
If $\omega_j$ is the principal ultrafilter concentrated at $j \in J$, then (up to canonical identifications) we have $M_j = M_{\omega_j}$ and $i_j = e_i(\omega_j)$, so let us denote the latter by $e_i(j)$.
The measurability requirements of \autoref{defn:MeasurableField} are satisfied by $(M_\Omega,e_I)$ -- all the functions are in fact continuous.

For the last property of \autoref{defn:MeasurableField}, let $\cL_0 \subseteq \cL$ be finite (or countable, for that matter), and let $I_0 \subseteq I$ be countable.
We may perform a ``simultaneous downward Löwenheim--Skolem construction'', and find a countable $I_1$ such that $I_0 \subseteq I_1 \subseteq I$, and for every continuous logic $\cL_0$-formula $\varphi(\bar x,y)$, and every tuple $\bar{e}$ in $e_{I_1}$, there exists $i \in I_1$ such that, simultaneously for all $j \in J$:
\begin{gather*}
  \varphi^{M_j}\bigl( \bar{e}(j), e_i(j) \bigr) \leq \Bigl( \inf_y \varphi\bigl( \bar{e}(j), y \bigr) \Bigr)^{M_j} + 1.
\end{gather*}
By \Los's theorem for ultraproducts, the same holds for all $\omega \in \Omega$ in place of $j$, so $M_{\omega,I_1} \preceq^\cont_{\cL_0} M_\omega$.
Therefore $I_0 \subseteq I_1 \in \fI^\cont(M_\Omega,e_I,\cL_0)$, showing that $(M_\Omega,e_I)$ is indeed an (elementarily) measurable field.

Finally, we claim that Bagheri's ultramean of $(M_J,\nu)$, which we denote by $\cl{M_J}$, can be canonically identified with the direct integral of ultraproducts \autoref{eq:UltrameanDecomposition}, which we denote as usual by $M_{\Omega,I}$.
Indeed, each member of $\cl{M_J}$ is a class $[i]$ for some $i \in I = \prod_j M_j$, modulo the zero-distance equivalence relation.
Essentially by definition, we have $d^{\cl{M_J}}\bigl( [i], [i'] \bigr) = d^{M_{\Omega,I}}(e_i,e_{i'})$.
This gives rise to a canonical isometric inclusion map $\cl{M_J} \subseteq M_{\Omega,I}$, sending $[i] \mapsto e_i$, and again, essentially by definition, this is an embedding of $\cL$-structures.
Let us now consider a simple section of $M_{\Omega,I}$.
Up to almost sure equality, we may assume that it is defined by clopen sets, namely, $f(\omega) = e_{i_k}(\omega)$ when $A_k \in \omega$, where $\{A_k : k < m\}$ partitions $J$.
Recall that each $i_k$ is a member of $\prod_j M_j$, and define $i(j) = f(j) = f(\omega_j)$, where we identify $j \in J$ with the principal ultrafilter $\omega_j$.
Then $i$ agrees with $i_k$ on $A_k$, so $e_i$ agrees with $f$ on $\widehat{A}_k$, for each $k$ -- and therefore, on $\Omega$.
In other words, every simple section is of the form $e_i$, and therefore belongs to the image of $\cl{M_J}$.
By density of the simple sections, $\cl{M_J} = M_{\Omega,I}$, as desired.
Bagheri's \Los theorem for ultrameans now follows from the corresponding theorems for ultraproducts and direct integrals.


\section{A direct integral of extremal models}
\label{sec:DirectIntegralExtremalModels}

For this section, let us fix an affine theory $T$, a cardinal $\kappa \geq \fd_T(\cL)$, and a boundary measure $\mu$ on $\tS^\aff_\kappa(T)$ such that $p = R(\mu) \in \tS^{\fM,\aff}_\kappa(T)$.
It will be convenient to fix a realization $a = (a_i : i < \kappa)$ of $p$.
As per \autoref{remark:TypeOfModel}, this can be done inside a model $N \models T$ in which $a$ enumerates a dense set, and this essentially determines $N$ and $a$ (up to a unique isomorphism).

Similarly, we define
\begin{equation*}
  \Omega = \cE^\fM_\kappa(T) = \cE_\kappa(T) \cap \tS^\fM_\kappa(T),
\end{equation*}
and realize each type $\omega \in \Omega = \cE^\fM_\kappa(T)$ as $e(\omega)$ inside an extremal model $M_\omega = \overline{\bigl\{ e_i(\omega) : i < \kappa \bigr\}}$, again in a unique fashion (\autoref{remark:TypeOfExtremalModel}).

Let $\cB$ denote the $\sigma$-algebra of Baire subsets of $\tS^\aff_\kappa(T)$ (or its completion with respect to $\mu$), and let
\begin{equation*}
  \cB_\Omega = \{B \cap \Omega : B \in \cB\}.
\end{equation*}
By \autoref{th:Bishop-dL-ForModels}, $\mu$ concentrates on $\Omega$.
Consequently, by \autoref{l:measure-concentration},
\begin{gather*}
  \mu_\Omega(B \cap \Omega) = \mu(B)
\end{gather*}
defines a (complete) measure on the $\sigma$-algebra $\cB_\Omega$.

\begin{lemma}
  \label{lem:IntegralDecompositionMeasurableField}
  The family $M_\Omega = \bigl(M_\omega : \omega \in \Omega\bigr)$ is a field of extremal models of $T$, and $e = (e_i : i < \kappa)$ is a pointwise enumeration of $M_\Omega$, in the sense of \autoref{df:field}.
  Moreover, the pair $(M_\Omega,e)$ is a measurable field of models of $T$ relative to the measure $\mu_\Omega$, in the sense of \autoref{defn:MeasurableField}, and it is non-degenerate in the sense of \autoref{dfn:NonDenegerateField}.
\end{lemma}
\begin{proof}
  The first assertion is immediate.
  As for \autoref{defn:MeasurableField}, the functions in conditions \autoref{item:MeasurableFieldPredicate} and \autoref{item:MeasurableFieldFunction} are continuous, and in particular, measurable.
  Indeed, for example, if $\bar e = (e_{i_j} : j<n)$ and $\bar x = (x_{i_j} : j<n)$, then $P^{M_\omega}\bigl( \bar{e}(\omega) \bigr) = P(\bar x)^\omega$, and the map $\omega\mapsto P(\bar x)^\omega$ is continuous by definition of the topology on $\Omega\subseteq \tS^\aff_\kappa(T)$.

  Consider a finite sublanguage $\cL_0 \subseteq \cL$ and a countable subset $I_0 \subseteq \kappa$.
  By standard arguments (essentially, the Downward Löwenheim--Skolem theorem), there exists a larger countable set $I_0 \subseteq I_1 \subseteq \kappa$ such that $\overline{\{a_i: i \in I_1\}} \preceq^\aff_{\cL_0} N$.

  Let $\Phi$ denote the collection of all affine $\cL_0$-formulas $\varphi(x',y)$, where $x' = (x_i : i \in I_1)$ and $y$ is a singleton.
  This is a separable set, so let $\Phi_0$ be a countable dense subset.
  Following \autoref{ntn:TypeOfModelSingleFormula} we have $p \in \bigcap_{\varphi \in \Phi_0} X_{\varphi,I_1}$.
  By \autoref{prop:TypeOfModelPreMean}~\autoref{item:TypeOfModelPreMeanSingleFormula}, $\mu(X_{\varphi,I_1}) = 1$ for all $\varphi \in \Phi_0$, and therefore
  \begin{gather*}
    \mu_\Omega\Bigl( \bigcap_{\varphi \in \Phi_0} X_{\varphi,I_1} \cap \Omega \Bigr) = 1.
  \end{gather*}
  Thus $I_1 \in \fI^\aff(M_\Omega,e,\cL_0)$, and \autoref{item:MeasurableFieldCofinal} holds as well.

  Finally, let $f_\lambda$ be any dense sequence of measurable sections that starts with $e$.
  Then we have maps
  \begin{gather*}
    \Omega \overset{\theta}{\longrightarrow} \tS^\aff_\lambda(T) \overset{\pi}{\longrightarrow} \tS^\aff_\kappa(T),
  \end{gather*}
  where $\theta$ is as per \autoref{rmk:LosForTypes} and $\pi$ is the variable restriction map.
  Dually, these induce
  \begin{gather*}
    \MALG(\Omega,\mu_\Omega) \overset{\theta^*}{\longleftarrow} \MALG\bigl( \tS^\aff_\lambda(T), \theta_* \mu_\Omega \bigr) \overset{\pi^*}{\longleftarrow} \MALG\bigl(\tS^\aff_\kappa(T), \pi_*\theta_* \mu_\Omega \bigr).
  \end{gather*}
  By construction, $\pi_* \theta_* \mu_\Omega = \mu$, and the composition $\theta^* \circ \pi^*$ is surjective.
  Therefore, so is $\theta^*$.
\end{proof}

We may therefore define
\begin{gather*}
  M_\mu = \int_\Omega^\oplus M_\omega \ud \mu_\Omega(\omega).
\end{gather*}
By \autoref{th:Los} and \autoref{l:measure-concentration}, we have for any affine formula $\varphi(x)$:
\begin{align*}
  \varphi^{M_\mu}(e)
  & = \int_\Omega \varphi^{M_\omega}\bigl( e(\omega) \bigr) \ud \mu_\Omega(\omega) \\
  & = \int_{\tS^\aff_\kappa(T)} \varphi(q) \ud \mu(q)
    = \varphi(p).
\end{align*}
Therefore, $e$ is a realization of $p$, and can be identified with $a$.
Since $a$ is dense in $N$, this makes it a submodel $N \preceq^\aff M_\mu$.
We have thus proved the following, which will be the base for the proof of the extremal decomposition theorem in \autoref{sec:ExtremalDecomposition}:

\begin{prop}
  \label{prop:DirectIntegralExtremalModelsConstruction}
  Let $T$ be an affine theory and $N \models T$.
  Let $a = (a_i : i < \kappa)$ enumerate a dense subset of $N$, and let $\mu$ be a boundary measure on $\tS^\aff_\kappa(T)$ with barycenter $\tp^\aff(a)$.
  Then $\mu$ concentrates on $\Omega = \cE^\fM_\kappa(T)$.
  This gives rise to a non-degenerate measurable field $(M_\Omega,e)$ of extremal models of $T$, where $e(\omega)$ realizes $\omega$, and $N$ admits a unique affine embedding in $M_\mu = \int_\Omega^\oplus M_\omega \ud \mu_\Omega(\omega)$ that sends $a \mapsto e$.
\end{prop}

\begin{cor}
  Every model $N$ of an affine theory $T$ admits an affine embedding into a direct integral of a measurable field of extremal models $(M_\Omega,e)$.
  If $N$ is of density character $\leq \kappa$, then every $M_\omega$ can be taken of density character $\leq \kappa$, and the pointwise enumeration $e$ to be indexed by $\kappa$.
  If $N$ is separable, and $T$ has a separable language, then the measure space can be taken to be standard.
\end{cor}


\section{Simplicial theories}
\label{sec:SimplicialTheory}

Simplices (see \autoref{sec:simplices}) form a important class of compact convex sets.
The corresponding class of affine theories will play a similarly distinguished role.

\begin{defn}
  \label{df:simplicial-theory}
  Let $T$ be an affine theory. We will say that $T$ is \emph{simplicial} if $\tS^\aff_x(T)$ is a simplex for every finite tuple $x$ .
\end{defn}

If $T$ is simplicial, then $\tS^\aff_x(T)$ is a simplex for any $x$, being an inverse limit of simplices (\autoref{l:inverse-limit-simplices}).

\autoref{thm:SimplexRiesz} gives the following.

\begin{prop}
  Let $T$ be an affine theory.
  The following are equivalent:
  \begin{enumerate}
  \item $T$ is simplicial.
  \item Whenever $\varphi_1,\varphi_2,\psi_1,\psi_2$ are affine formulas such that $(\varphi_1\vee\varphi_2) +\varepsilon \leq_T \psi_1\wedge\psi_2$ (i.e., $\varphi_i + \varepsilon \leq \psi_j$ for all $i,j = 1,2$) for some $\varepsilon>0$, there exists an affine formula $\chi$ such that $\varphi_1\vee\varphi_2 \leq_T \chi\leq_T \psi_1\wedge\psi_2$.
  \end{enumerate}
\end{prop}

\begin{cor}
  \label{cor:SimplicialSubLanguageClub}
  Let $T$ be a simplicial affine theory in a language $\cL$.
  Then the collection of countable sublanguages $\cL_0 \subseteq \cL$ such that $T\rest_{\cL_0}$ is simplicial is cofinal and closed under unions of countable chains within the set $\cP_{\aleph_0}(\cL)$ of countable sublanguages of $\cL$.
\end{cor}

If $T \sub T_0$ are affine theories, we will say that $T_0$ is a \emph{face of $T$} if $\tS^\aff_0(T_0)$ is a face of $\tS^\aff_0(T)$.
\begin{prop}
  \label{p:face-simplicial}
  Let $T$ be an affine, simplicial $\cL$-theory.
  \begin{enumerate}
  \item If $T_0$ is a face of $T$, then $T_0$ is simplicial.
  \item\label{i:p:face-simplicial-constants} Let $\cL'\supseteq\cL$ be an expansion by constants. Then $T$ remains simplicial as an $\cL'$-theory.
  \end{enumerate}
  In particular, every extreme completion of $T$ is simplicial, and if $M$ is an extremal model of $T$ and $A\subseteq M$, then the type spaces $\tS^\aff_x(A)$ are simplices.
\end{prop}
\begin{proof}
  Let $\pi \colon \tS^\aff_x(T) \to \tS^\aff_0(T)$ denote the natural projection.
  Then $\tS^\aff_x(T_0) = \pi^{-1}\bigl( \tS^\aff_0(T_0) \bigr)$, so $\tS^\aff_x(T_0)$ is a closed face of the simplex $\tS^\aff_x(T)$ and therefore also a simplex. For the second point, note that if $\cL'$ expands $\cL$ with a tuple of distinct constants $c$, say of the sort of a variable $y$, then $\tS_x^{\aff,\cL'}(T)\cong \tS_{xy}^{\aff,\cL}(T)$ as convex sets. For the last assertion, note that the affine diagram $D^\aff_A$ is an extreme completion of $T$ in the language $\cL_A$.
\end{proof}

Describing the convex sets $\tS^\aff_x(T)$ and checking that they are simplices can be difficult in concrete examples. We conclude this section with a sufficient condition for establishing that $T$ is simplicial based on information about the quantifier-free type spaces $\tS^\qf_x(T)$.
This criterion will be applied in Sections \ref{sec:PMP} and \ref{sec:tracial-von-neumann}.

Recall that $\rho^\qf_\bN\colon \tS^\aff_\bN(T) \rightarrow \tS^\qf_\bN(T)$ is the map sending an affine type to its quantifier-free part.
As for affine types, let $\cE^\qf_x(T)$ denote the set of extreme points of $\tS^\qf_x(T)$, namely, of extreme quantifier-free types.

\begin{prop}
  \label{prop:SimplicialCriterionQF}
  Let $T$ be an affine theory in a separable language with the following properties:
  \begin{enumerate}
  \item $\tS^\qf_n(T)$ is a simplex for every $n\in\N$.
  \item For every $p\in\cE^\fM_\bN(T)$, we have $\rho^\qf_\bN(p)\in\cE^\qf_\bN(T)$.
  \end{enumerate}
  Then $T$ is a simplicial theory.
\end{prop}
\begin{proof}
  Let $p \in \tS^\aff_n(T)$ and let $\mu_1$ and $\mu_2$ be two measures on $\cE_n(T)$ (equivalently, boundary measures on $\tS^\aff_n(T)$) with barycenter $p$.
  Let $\pi \colon \tS^\aff_\bN(T) \to \tS^\aff_n(T)$ be the natural projection and let $q \in \tS^{\fM,\aff}_\bN(T)$ be such that $\pi(q) = p$.
  By \autoref{lemma:TypeConvexCombinationLift}, there exist measures $\nu_1$ and $\nu_2$ on $\cE_\bN(T)$ with $\pi_*\nu_i = \mu_i$ and $R(\nu_i) = q$.
  By \autoref{th:Bishop-dL-ForModels}, we have $\nu_i(\cE^\fM_\bN(T)) = 1$.

  If we let $\lambda_i = (\rho^\qf_\bN)_* \nu_i$, then $(\rho^\qf_\bN)^{-1}(\cE_\bN^\qf(T))\supseteq \cE^\fM_\bN(T)$ by hypothesis, so $\lambda_i$ is a boundary measure on $\tS^\qf_\bN(T)$ with barycenter $\rho^\qf_\bN(q)$.
  The convex set $\tS^\qf_\bN(T)$ is the inverse limit of the simplices $\tS^\qf_n(T)$, and hence a simplex by \autoref{l:inverse-limit-simplices}.
  It follows that $\lambda_1 = \lambda_2$.

  The space $\cE^\fM_\bN(T)$ is a Polish space on which $\rho^\qf_\bN$ is injective by \autoref{l:rho-injective-Smod}.
  Therefore, $\rho^\qf_\bN$ sends Borel subsets of $\cE^\fM_\bN(T)$ to Borel subsets of $\cE^\qf_\bN(T)$.
  Since $\lambda_1 = \lambda_2$, we must have $\nu_1 = \nu_2$, and therefore $\mu_1 = \mu_2$.
  We conclude that $\tS^\aff_n(T)$ is a simplex.
\end{proof}

We also observe that a converse of the second condition of \autoref{prop:SimplicialCriterionQF} always holds.

\begin{lemma}
  \label{lem:ExtremeTypeProjectionQF}
  If $p\in \tS^{\fM,\aff}_x(T)$ and $\rho^\qf_x(p)\in \cE^\qf_x(T)$, then $p\in \cE^\fM_x(T)$.
\end{lemma}
\begin{proof}
  Let $p$ be as in the statement.
  Suppose we have $p_1, p_2 \in \tS^\aff_x(T)$ with $p = \half p_1 + \half p_2$.
  Then $\rho^\qf_x(p) = \half \rho^\qf_x(p_1) + \half \rho^\qf_x(p_2)$ and therefore $\rho^\qf_x(p) = \rho^\qf_x(p_1) = \rho^\qf_x(p_2)$.
  By \autoref{prop:TypeOfModelPreMean}, $p_1, p_2 \in \tS^{\fM,\aff}_x(T)$.
  By \autoref{l:rho-injective-Smod}, $p = p_1 = p_2$, and $p$ is an extreme type.
\end{proof}

\section{Extremal decomposition for simplicial theories}
\label{sec:ExtremalDecomposition}

In this section, we prove one of the main results of the paper: a general decomposition theorem for models simplicial theories, as direct integrals of extremal models.
We also prove the uniqueness of this decomposition, up to a reasonable notion of equivalence.

If $p \in \tS^\aff_x(T)$, then by the Choquet--Meyer theorem (\autoref{th:Choquet-Meyer}), there exists a unique boundary measure with barycenter $p$, which will be denoted $\mu_p$.
We observe that the map $\tS^\aff_x(T) \to \cM(\tS^\aff_x(T))$, $p \mapsto \mu_p$ is affine but it is in general not continuous.

\begin{lemma}
  \label{l:mu-p-under-projections}
  Let $T$ be a simplicial theory and let $\pi \colon \tS^\aff_{xy}(T) \to \tS^\aff_x(T)$ denote the variable restriction map.
  Then for every $p \in \tS^\aff_{xy}(T)$, we have
  \begin{equation*}
    \pi_* \mu_p = \mu_{\pi(p)}.
  \end{equation*}
\end{lemma}
\begin{proof}
  By \autoref{prop:VariableRestrictionBoundary}, $\pi_* \mu_p$ is a boundary measure.
  Its barycenter is clearly $\pi(p)$, whence the identity $\pi_* \mu_p = \mu_{\pi(p)}$.
\end{proof}

For the following discussion and lemma, let us fix a simplicial theory $T$, as well as a model $N \models T$.
Let $\kappa \geq \fd(N) + \fd_T(\cL)$ and let $a = (a_i : i < \kappa)$ enumerate a dense subset of $N$. Let $x$ be a tuple of variables of length $\kappa$, let $p = \tp^\aff(a) \in \tS^{\fM,\aff}_x(T)$, and let $\mu_p$ be the unique boundary measure on $\tS^\aff_x(T)$ with barycenter~$p$.

We are now in the setting of \autoref{sec:DirectIntegralExtremalModels}.
In particular, we have a probability space $(\Omega,\cB_\Omega,\mu_\Omega)$, where $\Omega = \cE^\fM_x(T)$, and a measurable field of extremal models $(M_\omega,e)$, where $e$ and $x$ are indexed by $\kappa$, with direct integral $M$ such that $N$ embeds in $M$ via the identification of $a$ with $e$.
Our goal is to show that this embedding is in fact an isomorphism.

Let $y$ be a single variable, and $z$ an arbitrary tuple.
The condition $d(y,x_j) = 0$ defines a closed face $F_z^j \subseteq \tS^\aff_{xyz}(T)$, which is affinely homeomorphic to $\tS^\aff_{xz}(T)$ via the variable restriction map $\pi_{xz}$.
Let $\theta^j \colon \tS^\aff_{xz}(T) \to F_z^j$ be the affine homeomorphism which is the inverse of $\pi_{xz} \rest_{F_z^j}$ (the argument of $\theta^j$ determines $z$, so no ambiguity can arise).

Consider now a simple section $b \in M$ with the goal to show that $b \in N$.
Let $J \sub \kappa$ be finite and let $\bigsqcup_{j \in J} A_j$ be a measurable partition of $\Omega$ such that $b(\omega) = e_{j}(\omega)$ for $\omega \in A_j$.
For every $j \in J$, we let $A_j' \sub \tS^\aff_x(T)$ be a Baire set such that $A_j = A_j' \cap \Omega$, and we may assume that  $(A_j' : j \in J)$ forms a partition of $\tS^\aff_x(T)$.
We may now glue the family of $\theta^j$ along this partition, defining $\theta^b \colon \tS^\aff_{xz}(T) \to \tS^\aff_{xyz}(T)$ by
\begin{equation*}
  \theta^b(q) = \theta^j(q) \quad \text{if} \quad \pi_x(q) \in A'_j.
\end{equation*}
Since the sets $A'_j$ are Baire, and each $\theta^j$ is continuous, the map $\theta^b$ is measurable with respect to the $\sigma$-algebras of Baire sets.

In order to simplify notation, if $a$ is a tuple from a model of $T$, we write $\mu_a$ for $\mu_{\tp^\aff(a)}$.
In particular, $\mu_e = \mu_p$.

The following is the main lemma.
\begin{lemma}
  \label{lem:IntegralDecompositionTheta}
  Let $b$ and $\theta^b$ be as above and let $c$ be a $z$-tuple in some affine extension of $M$.
  Then $\theta^b_* \mu_{ec} = \mu_{ebc}$ as measures on the Baire $\sigma$-algebra of $\tS^\aff_{xyz}(T)$.
\end{lemma}
\begin{proof}
  For simplicity of notation, let us write $\tS_x$ for $\tS^\aff_x(T)$, and similarly for other variables.
  We are going to consider measures and subsets along the following diagram
  \begin{equation*}
    \begin{tikzcd}
      \tS_{xz} \arrow[r, "\theta^b", yshift=0.6ex] \arrow[d, "\pi_x"] &
      \tS_{xyz} \arrow[d, "\pi_{xy}"] \arrow[l, "\pi_{xz}", yshift=-0.6ex] \\
      \tS_x \arrow[r, "\theta^b", yshift=0.6ex] &
      \tS_{xy} \arrow[l, "\pi_x", yshift=-0.6ex].
    \end{tikzcd}
  \end{equation*}
  The map $\theta^b \colon \tS_x \to \tS_{xy}$ at the bottom is defined analogously to the one at the top by taking $z$ to be the empty tuple.

  First of all, we claim that
  \begin{gather}
    \label{eq:IntegralDecompositionThetaMuEB}
    (\pi_{xy})_* \mu_{ebc} = \mu_{eb} = \theta^b_* \mu_e.
  \end{gather}
  Indeed, the first equality is just \autoref{l:mu-p-under-projections}, so let us prove the second.
  Recall that given a probability measure $\nu$ and a non-null measurable set $A$, we denote by $\nu_A$ the conditional probability measure on $A$, defined by $\nu_A(B) = \nu(B \cap A) / \nu(A)$.
  Then, by construction,
  \begin{gather*}
    \theta^b_* \mu_e = \sum_{j \in J} \mu_e(A'_j) \theta^j_* \big((\mu_e)_{A'_j}\big).
  \end{gather*}
  By \autoref{l:boundary-measures}, $\theta^b_* \mu_e$ is a boundary measure.
  If $\varphi(x, y)$ is an affine formula, then
  \begin{align*}
    \int_{\tS_{xy}} \varphi \ud \theta^b_* \mu_e
    & = \sum_{j \in J} \int_{A_j'} \varphi\bigl(\theta^j(q)\bigr) \ud \mu_e(q) \\
    & = \sum_{j \in J} \int_{A_j} \varphi\bigl( e(\omega), e_j(\omega) \bigr) \ud \mu_\Omega(\omega) \\
    & = \int_\Omega \varphi\bigl( e(\omega), b(\omega) \bigr) \ud \mu_\Omega(\omega)
      = \varphi(e, b).
  \end{align*}
  In other words, $R(\theta^b_* \mu_e) = \tp^\aff(eb)$, and the second equality of \autoref{eq:IntegralDecompositionThetaMuEB} holds by the definition of $\mu_{eb}$.

  Let
  \begin{gather*}
    B_0 = \theta^b( \tS_x ) = \bigl\{ q \in \tS_{xy} : q \models d(y,x_j)=0 \text{ if } \pi_x(q) \in A_j' \bigr\}
    \subseteq \tS_{xy}, \\
    B = \pi_{xy}^{-1}(B_0) \subseteq \tS_{xyz}
  \end{gather*}
  and note that both $B_0$ and $B$ are Baire.
  On the one hand, by \autoref{eq:IntegralDecompositionThetaMuEB},
  \begin{gather*}
    \mu_{ebc}(B) = \mu_{eb}(B_0) = \mu_e( \tS_x ) = 1.
  \end{gather*}
  On the other, $B = \theta^b( \tS_{xz} )$ and $\theta^b \circ \pi_{xz}\rest_B = \id_B$.
  Therefore
  \begin{gather*}
    \theta^b_* \mu_{ec} = \theta^b_* (\pi_{xz})_* \mu_{ebc} = \mu_{ebc},
  \end{gather*}
  completing the proof.
\end{proof}

\begin{theorem}[Extremal decomposition]
  \label{th:integral-decomposition}
  Let $T$ be a simplicial theory, let $N \models T$, and let $\kappa \geq \fd(N) + \fd_T(\cL)$.
  Then there exists a probability space $(\Omega, \cB_\Omega, \mu_\Omega)$ and a non-degenerate measurable field of extremal models $(M_\Omega,e)$, such that $\density\bigl( M_\omega \bigr) \leq \kappa$ for all $\omega$, $e$ is indexed by $\kappa$, and
  \begin{equation*}
    N \cong \int_\Omega^\oplus M_\omega \ud \mu_\Omega(\omega).
  \end{equation*}
  In the case where $T$ has a separable language and $N$ is separable, we can take $\Omega$ to be a standard probability space.
\end{theorem}
\begin{proof}
  Let $(\Omega,\cB_\Omega,\mu_\Omega)$ and $N \preceq^\aff M_\mu = \int_\Omega^\oplus M_\omega \ud \mu_\Omega(\omega)$ be as per \autoref{prop:DirectIntegralExtremalModelsConstruction}.
  In order to complete the proof, we need to show that $N = M_\mu$, and for this it will suffice to show that every simple section $b \in M_\mu$ belongs to $N$.

  Note that for every $c$ in an affine extension of $M_\mu$, $\tp^\aff(ec)$ determines $\tp^\aff(ebc)$.
  Indeed, since the theory is simplicial, this is equivalent to saying that $\mu_{ec}$ determines $\mu_{ebc}$, which holds by \autoref{lem:IntegralDecompositionTheta}.

  In other words, the natural map $\tS^\aff_{z}(eb) \to \tS^\aff_z(e)$ is an affine homeomorphism. In particular, the formula $d(b, z)$ is equivalent, modulo $\Th^\aff(M_\mu,eb)$, to a definable predicate $\psi(z)$ in the language of $(M_\mu,e)$. But then $\big(\inf_z \psi(z)\big)^{M_\mu} = 0$ and as the inclusion $N \sub M_\mu$ is affine, $\big(\inf_z \psi(z)\big)^N = 0$. Thus $b \in N$.

  For the last statement of the theorem, if $T$ is in a separable language and $N$ is separable, then $x$ can be taken countable.
  In this case, $\tS^\aff_x(T)$ is metrizable and $\Omega = \cE^\fM_x(T)$ is a Polish space by \autoref{l:extreme-Gdelta} and \autoref{prop:TypeOfModelPreMean}.
  Then $\cB_\Omega$ is (the completion of) the $\sigma$-algebra of Borel subsets of $\Omega$, so $(\Omega, \cB_\Omega, \mu_\Omega)$ is standard.
\end{proof}

\begin{remark}
We will show in \autoref{c:simplicial-ext-elementarily-measurable} that, under mild assumptions on the simplicial theory $T$, every measurable field of extremal models of $T$ is elementarily measurable.
\end{remark}

\begin{remark}
  The probability space $(\Omega,\cB_\Omega,\mu_\Omega)$ defined in the previous proof is a subset of a compact Hausdorff space, endowed with the trace of its Baire $\sigma$-algebra. We may also note that the functions $\omega\mapsto \varphi^{M_\omega}(\bar e(\omega))$, where $\varphi$ is an affine formula and $\bar e$ is a subtuple of the pointwise enumeration $e_I$, are continuous. However, in general the topology and the measure may not interact well, in that the latter need not extend to a regular Borel probability measure on $\Omega$.

  On the other hand, in the particular case where the type spaces of $T$ are \emph{Bauer simplices} (a situation that will be studied in detail in \autoref{sec:Bauer-theories}), the space $(\Omega,\cB_\Omega,\mu_\Omega)$ becomes a compact Hausdorff space equipped with a Radon probability measure.
\end{remark}

We can also state a decomposition theorem for affine extensions.

\begin{theorem}\label{th:extremal-decomposition-for-pairs}
  Let $M\preceq^\aff N$ be an affine extension of models of a simplicial theory $T$.
  Then there exist probability spaces $(\Omega,\mu)$, $(\Xi,\nu)$, a measure-preserving map $\pi\colon\Xi\to \Omega$, and non-degenerate extremal decompositions $(M_\omega : \omega\in\Omega)$ of $M$ and $(N_\xi : \xi\in\Xi)$ of $N$ such that $M_{\pi(\xi)}\preceq^\aff N_\xi$ for every $\xi\in\Xi$, $f\circ\pi$ is a measurable section in $N_\Xi$ for every measurable section $f$ in $M_\Omega$, and the inclusion $M\preceq^\aff N$ corresponds to the map
  \begin{gather*}
    \int^\oplus_\Omega M_\omega\ud\mu(\omega) \to \int^\oplus_\Xi N_\xi\ud\nu(\xi), \quad f\mapsto f\circ\pi.
  \end{gather*}
\end{theorem}
\begin{proof}
  We sketch the idea. In the construction of \autoref{th:integral-decomposition}, we may start with a dense tuple $a=(a_i:i<\kappa)$ of $N$ such that some initial segment $a_{<\lambda}$ is a dense enumeration of $M$. Note that if $p=\tp^\aff(a)$ and $q=\tp^\aff(a_{<\lambda})$, then $(\pi_\lambda)_*\mu_p = \mu_q$ by \autoref{l:mu-p-under-projections}. Let $\Omega = \cE^\fM_\lambda(T)$ and define $\mu=(\mu_q)_\Omega$ as in \autoref{l:measure-concentration}. By an easy adaptation of \autoref{th:Bishop-dL-ForModels} (and of \autoref{cor:TypeOfExtremalModelVeryDense}), which we leave to the reader, the measure $\mu_p$ concentrates on the set
  \begin{gather*}
    \Xi = \cE^\fM_\kappa(T) \cap \pi_\lambda^{-1}(\Omega).
  \end{gather*}
  We then define $\nu=(\mu_p)_\Xi$ as in \autoref{l:measure-concentration}.
  Letting $\pi = \pi_\lambda\rest_{\Xi} \colon \Xi \to \Omega$, it is clear that $\pi_*\nu = \mu$, i.e., $\pi$ is measure-preserving.
  We may form measurable fields $\bigl((M_\omega:\omega\in \Omega),e\bigr)$ and $\big((N_\xi:\xi\in\Xi),e'\bigr)$ essentially as in \autoref{sec:DirectIntegralExtremalModels}, so that $M_{\pi(\xi)}\preceq^\aff N_\xi$ for every $\xi\in\Xi$, and $e\circ\pi = e'_{<\lambda}$.
  If $M_\mu$ and $N_\nu$ are the respective direct integrals, then the embeddings $M\to M_\mu$ and $N\to N_\nu$ that send $a_{<\lambda}\mapsto e$ and $a\mapsto e'$ are surjective, as in \autoref{th:integral-decomposition}. In particular, $e$ is dense in $M_\mu$ and since $e\circ\pi = e'_{<\lambda}$, it follows that $f\circ\pi$ is a measurable section in $(N_\Xi,e')$ whenever $f$ is a measurable section in $(M_\Omega,e)$, and that the inclusion $M\preceq^\aff N$ corresponds to the map $M_\mu\to N_\nu$, $f\mapsto f\circ\pi$, as desired.
\end{proof}

Now that we have established the existence of an extremal decomposition, let us study its uniqueness.
The naïve statement that one might expect is that if we have an isomorphism
\begin{equation}
  \label{eq:isomorphism-integrals}
  \int_\Omega^\oplus M_\omega \ud \mu(\omega) \cong \int_\Xi^\oplus N_\xi \ud \nu(\xi),
\end{equation}
with $M_\omega$ and $N_\xi$ extremal models of a simplicial theory $T$, then the fields $M_\Omega$ and $N_\Xi$ are isomorphic in the following sense.

\begin{defn}
  \label{dfn:MeasurableFieldPointwiseEmbedding}
  Let $(\Omega, \cB, \mu)$ and $(\Xi,\cC,\nu)$ be probability spaces, let $M_\Omega$ and $N_\Xi$ be measurable fields of $\cL$-structures.
  A \emph{pointwise affine embedding} of the measurable field $M_\Omega$ in $N_\Xi$ is a pair $(s,S)$, where
  \begin{itemize}
  \item $s \colon \Xi_0 \rightarrow \Omega$ is a measure-preserving map defined on a full-measure subset $\Xi_0 \subseteq \Xi$,
  \item $S=(S_\xi : \xi\in\Xi_0)$, where for each $\xi \in \Xi_0$, $S_\xi\colon M_{s(\xi)} \rightarrow N_\xi$ is an affine embedding, and
  \item for every measurable section $f$ of $M_\Omega$, the section $f^{(s,S)}(\xi) = S_\xi\bigl( f (s(\xi)) \bigr)$ is measurable for $N_\Xi$.
  \end{itemize}
  If moreover $s\colon \Xi_0 \rightarrow \Omega_0$ is an isomorphism of measure spaces (i.e., a bi-measurable measure-preserving bijection), where $\Omega_0 \subseteq \Omega$ is necessarily of full measure, and each $S_\xi$ is an isomorphism of $\cL$-structures,
  then $(s,S)$ is a \emph{pointwise isomorphism} of measurable fields. (Note that in that case, the inverse pair $(s^{-1},S^{-1})$ with $S^{-1}=\bigl((S_{s^{-1}(\omega)})^{-1}:\omega\in\Omega_0\bigr)$ also satisfies the third condition above with respect to measurable sections of $N_\Xi$, and so is a pointwise affine embedding.)
\end{defn}

\begin{remark}
  In \autoref{th:extremal-decomposition-for-pairs}, the pair $(\pi,S)$ where each $S_\xi$ is the inclusion map $M_{s(\xi)}\rightarrow N_\xi$, is a pointwise affine embedding.
\end{remark}

In the case where $\Omega$ and $\Xi$ are standard and everything is separable, or when the direct integrals are convex combinations, one indeed obtains isomorphism in this sense.
However, in the general case, there are two obstacles for this.
\begin{itemize}
\item First, very different probability spaces can have the same measure algebra, but a direct multiple $L^1(\Omega, \mu, M)$, for example, depends only on $M$ and $\MALG(\Omega, \mu)$ (cf.\ \autoref{rem:L1-only-depends-MALG}).
  It follows that, at the level of the probability spaces, the most one can expect if \autoref{eq:isomorphism-integrals} holds is an isomorphism between the measure algebras of $\Omega$ and $\Xi$.
\item Second, when the structures $M_\omega$ and $N_\xi$ are not separable, isomorphism is no longer a measurable condition.
  This will be illustrated in \autoref{ex:random-set}.
\end{itemize}

What we do recover in full generality is the weaker property introduced below.

Let $\Omega$ and $\Xi$ be two probability spaces. If $\fs\colon \MALG(\Omega) \rightarrow \MALG(\Xi)$ is an embedding of measure algebras, we will denote by the same letter the corresponding isometric embedding of normed $\bR$-algebras
\begin{gather*}
  \fs\colon L^\infty(\Omega, \bR)\to L^\infty(\Xi,\bR),
\end{gather*}
defined on simple functions by $\fs\bigl(\sum_{i<n}r_i\chi_{A_i}\bigr) = \sum_{i<n}r_i\chi_{\fs(A_i)}$.

Recall that by \autoref{lemma:DirectIntegralMeasurableLanguage}, if $M_\Omega$ is a measurable field, $\varphi(x)$ is an affine formula, and $f$ is an $x$-tuple of measurable sections, then the function
\begin{gather*}
  \oset{\varphi(f)} \colon \omega \mapsto \varphi^{M_\omega}\bigl( f(\omega) \bigr)
\end{gather*}
belongs to $L^\infty(\Omega)$.
Similarly, if $M_\Omega$ is an elementarily measurable field and $\varphi$ is a continuous logic formula, then $\oset{\varphi(f)}\in L^\infty(\Omega)$ as well.

\begin{defn}
  \label{dfn:MeasurableFieldEmbedding}
  In the setting of \autoref{dfn:MeasurableFieldPointwiseEmbedding}, let $M$ and $N$ be the respective direct integrals of $M_\Omega$ and $N_\Xi$.
  An \emph{affine embedding} of the measurable field $M_\Omega$ in $N_\Xi$ is a pair $(\sigma,\fs)$, where $\sigma\colon M \rightarrow N$ is a map, $\fs \colon\MALG(\Omega) \to \MALG(\Xi)$ is an embedding of measure algebras, and for every affine formula $\varphi(x)$ and every tuple $f\in M^x$, we have $\fs\oset{\varphi(f)} = \oset{\varphi(\sigma f)}$.

  If $\sigma$ and $\fs$ are bijective, then $(\sigma,\fs)$ is an \emph{isomorphism} of measurable fields.
\end{defn}

\begin{remark}
  \label{rmk:MeasurableFieldPointwiseEmbedding}
  Suppose that $(s,S)$ is a pointwise affine embedding of $M_\Omega$ in $N_\Xi$.
  For $f \in M$, a measurable section of $M_\Omega$, define $\sigma(f) = f^{(s,S)}$ and for $A \in \MALG(\Omega)$ define $\fs(A) = s^{-1}(A)$.
  Then it is easy to see that everything is well-defined (up to null measure), and $(\sigma,\fs)$ is an affine embedding of measurable fields.
  Similarly for a (pointwise) isomorphism.
\end{remark}

\begin{remark}
  \label{rmk:MeasurableFieldEmbedding}
  Suppose that $(\sigma,\fs)$ is an affine embedding of fields.
  Then by \autoref{th:Los}, for every affine formula $\varphi(x)$ and $f \in M^x$:
  \begin{gather*}
    \varphi(f)^M
    = \int_\Omega \oset{\varphi(f)} \ud \mu
    = \int_\Xi \fs\oset{\varphi(f)} \ud \nu
    = \int_\Xi \oset{\varphi(\sigma f)} \ud \nu
    = \varphi(\sigma f)^N.
  \end{gather*}
  Therefore $\sigma\colon M \rightarrow N$ is an affine embedding.
  If $(\sigma,\fs)$ is an isomorphism of fields, then $\sigma$ is an $\cL$-isomorphism.
\end{remark}

\begin{remark}\label{rmk:MeasurableFieldEmbedding-BijectiveMap}
  In the setting of \autoref{dfn:MeasurableFieldEmbedding}, if $\fs$ is an isomorphism of measure algebras and $\sigma$ is a bijective map, then to see that $(\sigma,\fs)$ is an isomorphism of measurable fields it suffices to check that it is an \emph{embedding}, in the sense that $\fs\oset{P(f)} = \oset{P(\sigma f)}$ and $\sigma F^M(f) = F^N(\sigma f)$ for every $f\in M^x$, every predicate symbol $P$, and every function symbol $F$ from $\cL$. Indeed, this follows by induction using \autoref{lemma:DirectIntegralMeasurableFunctionSymbol} and \autoref{lem:LosQuantifier}.

  If $(\sigma,\fs)$ is an isomorphism of measurable fields and, in addition, the fields $M_\Omega$ and $N_\Xi$ are elementarily measurable, then the condition $\fs\oset{\varphi(f)} = \oset{\varphi(\sigma f)}$ holds for every continuous logic formula $\varphi$ as well, by the same inductive argument.
\end{remark}

\begin{remark}
  \label{rmk:MeasurableFieldEmbeddingStandard}
  Assume that $(\sigma,\fs)$ is an affine embedding.
  Assume moreover that $\Omega$ and $\Xi$ are standard probability spaces, that the implicit pointwise enumerations of the fields $M_\Omega$ and $N_\Xi$ are countable, and that the structures of the fields are models of some affine theory $T$ with separable language.
  In particular, there exists a measure-preserving map $s\colon \Xi_0\to \Omega$ defined on a full-measure subset $\Xi_0\subseteq\Xi$ such that $\fs(A) = s^{-1}(A)$ for every measurable set $A\subseteq\Omega$. Moreover, the direct integrals $M$ and $N$ are separable, as is $\tS^\aff_\bN(T)$.
  Let $e = (e_n : n \in \bN)$ enumerate a dense subset of $M$, and take $f = \sigma(e) \in N^\bN$.
  Let $\theta^\Omega\colon \Omega \rightarrow \tS^\aff_\bN(T)$ and $\theta^\Xi\colon \Xi \rightarrow \tS^\aff_\bN(T)$ be the corresponding map as in \autoref{rmk:LosForTypes}.

  Let $\varphi(x)$ be an affine formula in $x = (x_n : n \in \bN)$, inducing an affine function $\tS^\aff_\bN(T) \rightarrow \bR$ that we still denote by $\varphi$.
  Then, almost surely, the following measurable functions agree on $\Xi$:
  \begin{gather*}
    \varphi \circ \theta^\Xi
    = \oset{\varphi(f)}
    = \fs\oset{\varphi(e)}
    = \oset{\varphi(e)} \circ s
    = \varphi \circ \theta^\Omega \circ s.
  \end{gather*}
  It follows that $\theta^\Xi = \theta^\Omega \circ s$ almost surely.
  In other words, possibly reducing $\Xi_0$, we may assume that $\theta^\Xi = \theta^\Omega \circ s$.
  We may also assume that $e \circ s(\xi)$ is dense in $M_{s(\xi)}$ for every $\xi \in \Xi_0$.
  This means that for every $\xi \in \Xi_0$, if we send $e \circ s(\xi) \mapsto f(\xi)$, this defines an affine embedding $S_\xi\colon M_{s(\xi)} \rightarrow N_\xi$.
  Thus, in the standard, separable case, every affine embedding $(\sigma,\fs)$ comes from a pointwise affine embedding $(s,S)$.

  A similar argument works for isomorphisms.
\end{remark}

\begin{remark}
  \label{rmk:MeasurableFieldEmbeddingAtomic}
  If $(\sigma,\fs)$ is an affine embedding and the probability spaces $\Omega$ and $\Xi$ are countable sets with points of positive measure, then the argument of the previous remark works as well, without any separability assumptions, because all equalities hold pointwise rather than almost surely. Thus $(\sigma,\fs)$ comes from a pointwise affine embedding $(s,S)$, where $s\colon \Xi\to \Omega$ is a measure-preserving surjection (or a bijection, if $(\sigma,\fs)$ is an isomorphism).
\end{remark}

\begin{lemma}
  \label{lem:ExtremalFieldBoundaryMeasure}
  Let $(\Omega, \cB, \mu)$ be a probability space, let $M_\Omega$ be a measurable field of extremal models of an affine theory $T$, and let $f \in M^x$ be a tuple of sections.
  Let $\theta\colon \Omega \rightarrow \tS^\aff_x(T)$ send $\omega \mapsto \tp^\aff\bigl( f(\omega) \bigr)$, which is measurable by \autoref{rmk:LosForTypes}.
  Then $\theta_* \mu$ is a boundary measure on $\tS^\aff_x(T)$.
\end{lemma}
\begin{proof}
  By the last part of \autoref{p:Mokobodzki}, for every $h\in C\bigl(\tS^\aff_x(T)\bigr)$ and $\omega\in\Omega$ we have $h \circ \theta(\omega) = \hat{h} \circ \theta(\omega)$, since $M_\omega$ is an extremal model.
  Therefore
  \begin{gather*}
    \theta_* \mu(h)
    = \int_\Omega h \circ \theta \ud \mu
    = \int_\Omega \hat{h} \circ \theta \ud \mu
    = \theta_* \mu(\hat{h}).
  \end{gather*}
  By another application of \autoref{p:Mokobodzki}, $\theta_* \mu(h)$ is a boundary measure.
\end{proof}

\begin{theorem}[Uniqueness]
  \label{th:uniqueness-decomposition}
  Let $T$ be a simplicial theory.
  Let $(\Omega, \cB, \mu)$ and $(\Xi,\cC,\nu)$ be probability spaces, and let $M_\Omega$ and $N_\Xi$ be measurable fields of extremal models of $T$, with $M_\Omega$ non-degenerate.
  Let $M$ and $N$ be their respective direct integrals, and suppose that $\sigma\colon M \rightarrow N$ is an affine embedding.

  Then there is a measure algebra embedding $\fs \colon\MALG(\Omega) \to \MALG(\Xi)$ such that $(\sigma,\fs)$ is an affine embedding of the measurable field $M_\Omega$ in $N_\Xi$.
  If $\sigma$ is bijective and $N_\Xi$ is non-degenerate as well, then $(\sigma,\fs)$ is an isomorphism of fields.
\end{theorem}
\begin{proof}
  Let $e = (e_\alpha : \alpha < \kappa)$ enumerate the measurable sections of $M_\Omega$.
  Let $f = (f_\alpha: \alpha < \kappa)$ be a sequence of measurable sections in $N_\Xi$ such that $f_\alpha = \sigma(e_\alpha)$ (up to null measure) for every $\alpha < \kappa$.
  As in \autoref{rmk:LosForTypes}, let $\theta^\Omega\colon \Omega \rightarrow \tS^\aff_\kappa(\cL)$ be the map $\omega \mapsto \tp^\aff\bigl( e(\omega) \bigr)$, and similarly for $\theta^\Xi$ and $f$.
  Since $\sigma$ is an affine embedding, we have $\tp^\aff(e) = \tp^\aff(f)$, call it $q$.
  The corresponding image measures $\theta^\Omega_* \mu$ and $\theta^\Xi_* \nu$ are boundary measures on $\tS^\aff_\kappa(T)$ by \autoref{lem:ExtremalFieldBoundaryMeasure}, with common barycenter $q$.
  Since $T$ is simplicial, these two measures must coincide.
  Let us denote the measure algebra of this measure by
  \begin{gather*}
    \cA
    =
    \MALG\bigl( \tS^\aff_\kappa(T), \theta^\Omega_* \mu \bigr)
    =
    \MALG\bigl( \tS^\aff_\kappa(T), \theta^\Xi_* \nu \bigr).
  \end{gather*}
  The measure-preserving maps $\theta^\Omega$ and $\theta^\Xi$ induce embeddings
  \begin{gather*}
    \iota_\Omega\colon \cA \rightarrow \MALG(\Omega),
    \qquad
    \iota_\Xi\colon \cA \rightarrow \MALG(\Xi).
  \end{gather*}
  Moreover, $\iota_\Omega$ is an isomorphism, by non-degeneracy of $M_\Omega$.
  Let $\fs = \iota_\Xi \circ \iota_\Omega^{-1}$.

  Let $\varphi(x)$ be an affine formula and $g \in M^x$.
  Then $g$ is a sub-tuple of $e$ (possibly with some repetition).
  Let $y = (y_\alpha : \alpha < \kappa)$, and let $\psi(y)$ be the formula obtained from $\varphi(x)$ by substituting $y_\alpha$ for $x_i$ in $\varphi$ when $e_\alpha = g_i$.
  Recalling that we identify the formula $\psi$ with the corresponding continuous affine function $\tS^\aff_\kappa(T) \rightarrow \bR$, we have
  \begin{gather*}
    \bigl\llbracket \varphi(g) \bigr\rrbracket
    = \bigl\llbracket \psi(e) \bigr\rrbracket
    = \iota_\Omega\psi,
    \qquad
    \bigl\llbracket \varphi(\sigma g) \bigr\rrbracket
    = \bigl\llbracket \psi(f) \bigr\rrbracket
    = \iota_\Xi \psi.
  \end{gather*}
  Then
  \begin{gather*}
    \fs\bigl\llbracket \varphi(g) \bigr\rrbracket
    = \fs \iota_\Omega \psi = \iota_\Xi \psi = \bigl\llbracket \varphi(\sigma g) \bigr\rrbracket.
  \end{gather*}
  Therefore, $(\sigma,\fs)$ is an affine embedding of fields.
  If $\sigma$ is bijective and $N_\Xi$ is non-degenerate as well, then $f$ enumerates the sections of $N_\Xi$ and $\fs$ is an isomorphism of measure algebras, so $(\sigma,\fs)$ is a field isomorphism.
\end{proof}

\begin{cor}
  Let $M_\Omega$ and $N_\Xi$ be measurable fields over standard probability spaces of separable extremal models (with countable pointwise enumerations) of a non-degenerate simplicial theory in a separable language. If their direct integrals are isomorphic, then $M_\Omega$ and $N_\Xi$ are pointwise isomorphic.
\end{cor}
\begin{proof}
  By \autoref{th:uniqueness-decomposition} and \autoref{rmk:MeasurableFieldEmbeddingStandard}.
\end{proof}

\begin{cor}
  Let $M=\bigoplus\lambda_{i\in I} M_i$ and $N=\bigoplus_{j\in J}\mu_j N_j$ be convex combinations of extremal models of a non-degenerate simplicial theory in a separable language. If $M\cong N$, then there is a bijection $s\colon J\to I$ and isomorphisms $M_{s(j)}\cong N_j$ for every $j\in J$.
\end{cor}
\begin{proof}
  By \autoref{th:uniqueness-decomposition} and \autoref{rmk:MeasurableFieldEmbeddingAtomic}.
\end{proof}

We end by recording a simple fact that will be used in \autoref{ex:random-set}.

\begin{lemma}
  \label{l:concentration-direct-integral}
  Let $(M_\Omega,e_I)$ be a measurable field of $\cL$-structures over a probability space $(\Omega, \cB, \mu)$ and let $\Omega'\sub\Omega$ be an arbitrary set on which $\mu$ concentrates, as per \autoref{df:measure-concentr-general}. Consider the field $M_{\Omega'}=(M_\omega : \omega \in\Omega')$ and the tuple $e'_I = e_I\rest_{\Omega'}$ consisting of the restrictions of the sections in $e_I$ to $\Omega'$.

  Then $(M_{\Omega'},e'_I)$ is a measurable field of structures over the induced probability space $(\Omega',\cB_{\Omega'},\mu_{\Omega'})$, and the map $f \mapsto f\rest_{\Omega'}$ defines an isomorphism
  \begin{gather*}
    \int^\oplus_\Omega M_\omega \ud \mu(\omega) \cong \int^\oplus_{\Omega '}M_\omega \ud \mu_{\Omega'}(\omega).
  \end{gather*}
\end{lemma}
\begin{proof}
  Everything follows from \autoref{l:measure-concentration}.
\end{proof}

\begin{question}
  \label{q:isomorphic-L1}
  Let $T$ be a simplicial theory in a separable language and let $\Omega$ be a standard probability space. Can there be non-isomorphic extremal models $M$ and $N$ of $T$ such that $L^1(\Omega,M)\cong L^1(\Omega,N)$?
\end{question}

In \cite{OzawaMathOverflow}, Ozawa asks the following question: if $\Omega$ is an atomless standard probability space and $M$ and $N$ are von Neumann factors such that $L^\infty(\Omega) \bar \otimes M \cong L^\infty(\Omega) \bar \otimes N$, does it follow that $M \cong N$? In view of our results in \autoref{sec:tracial-von-neumann}, his question in the case where $M$ and $N$ are II$_1$ factors is a special case of \autoref{q:isomorphic-L1}.


\part{Connections with continuous logic}
\label{part:conn-with-cont}

\section{The affine part of a continuous theory}
\label{sec:affine-part-continuous-theory}

Let $\cL$ be a fixed signature with convex continuity moduli.
Let $\cL^\cont_x$ denote the collection of $\cL$-formulas in continuous logic with free variable $x$. We may take as continuous connectives all affine connectives together with the absolute value $|\cdot|$, from which one can readily define the connectives $\land$ and $\lor$ (interpreted as $\min$ and $\max$, respectively). Alternatively, we could also take multiplication as an additional connective; this leads to an equivalent set of continuous formulas, up to uniform approximations.

If $Q$ is a theory in continuous logic (or a \emph{continuous theory} for short), then we can define the quasi-order $\leq_Q$ on $\cL^\cont_x$ similarly to \autoref{eq:affine-order}:
\begin{equation*}
  \varphi \leq_Q \psi \iff Q \models \varphi \leq \psi,
\end{equation*}
and the relation $\varphi\equiv_Q \psi$ if $\varphi\geq_Q \psi$ and $\varphi\leq_Q \psi$. Then $\cL^\cont_x(Q) \coloneqq \cL^\cont_x/{\equiv_Q}$ is naturally an ordered unit space equipped additionally with the operations $\land$ and $\lor$, which make it into a vector lattice.
We may define the type space in the usual sense of continuous logic as
\begin{multline*}
  \tS^\cont_x(Q) \coloneqq \set{p \in \cL^\cont_x(Q)^* : p \text{ is a state and } \\
    p(|\varphi|) = |p(\varphi)| \text{ for all } \varphi \in \cL^\cont_x(Q)}.
\end{multline*}
As usual, the \emph{logic topology} $\tau$ is just the weak$^*$ topology. It follows from the definition that the natural map $\cL^\cont_x(Q) \hookrightarrow C(\tS^\cont_x(Q))$ is a norm-preserving vector lattice embedding and it follows from the Stone--Weierstrass theorem that its image is dense.

\begin{remark}
  By the Riesz representation theorem, the compact convex set of states of $\cL^\cont_x(Q)$ can be identified with the set of Radon probability measures $\cM(\tS^{\cont}_x(Q))$, whose extreme points, the Dirac measures, we identify with $\tS^{\cont}_x(Q)$.
\end{remark}

\begin{defn}
  For a continuous theory $Q$, we define its \emph{affine part} $Q_\aff$ as the collection of all affine consequences of $Q$.
\end{defn}

Consider a continuous theory $Q$ and a formula $\varphi \in \cL^\aff_x \sub \cL^\cont_x$.
Then
\begin{gather*}
  \varphi \geq_{Q_\aff} 0 \iff \varphi \geq_Q 0.
\end{gather*}
Therefore, the inclusion $\cL^\aff_x \subseteq \cL^\cont_x$ induces an injective, unital, linear, positive map
\begin{equation}
  \label{eq:iota-Fx-to-Lx}
  \iota_x \colon \cL^\aff_x(Q_\aff) \to \cL^\cont_x(Q).
\end{equation}
Dually, we obtain a surjective, continuous, affine map
\begin{equation}
  \label{eq:IotaStar}
  \iota_x^* \colon \cM(\tS^{\cont}_x(Q)) \to \tS^\aff_x(Q_\aff).
\end{equation}
Its restriction to $\tS^{\cont}_x(Q)$ (i.e., to the Dirac measures) will be denoted by
\begin{equation}
  \label{eq:RhoAff}
  \rho^\aff_x \colon \tS^{\cont}_x(Q) \to \tS^\aff_x(Q_\aff).
\end{equation}
When clear from context, we shall omit the index $x$ from $\rho^\aff_x$ and $\iota_x$.
Conversely, we may define $\rho^\aff$ directly as the affine part of a continuous type, and then $\iota^* = R \circ \rho^\aff_*$, where $R$ denotes the barycenter map.

\begin{lemma}
  \label{l:image-rho-aff}
  The image of $\rho^\aff$ contains $\cE_x(Q_\aff)$. Moreover, $\rho^\aff$ is $\tau$-continuous and it is a $\dtp$-contraction (where the distance between types in continuous logic is defined in the same manner as in \autoref{defn:AffineTypeDistance}).
\end{lemma}
\begin{proof}
  The first statement follows from \autoref{l:extreme-transitivity} applied to $\iota_x^*$. The second is obvious.
\end{proof}

The following simple result will be fundamental for many arguments.

\begin{prop}
  \label{p:ext-models-QAff}
  Let $Q$ be a continuous logic theory.
  Then every extremal model of $Q_\aff$ embeds affinely into a model of $Q$.
\end{prop}
\begin{proof}
  Given an extremal model $M\models Q_\aff$, let $p\in\cE_x(Q_\aff)$ be the type of an enumeration of $M$.
  As $\cE_x(Q_\aff)\subseteq \rho^\aff(\tS^{\cont}_x(Q))$, the type $p$ is realized in a model $N\models Q$. The realization of $p$ is thus an affine copy of $M$ in $N$.
\end{proof}

We have already shown that every model of an affine theory $T$ embeds affinely in a direct integral of extremal models of $T$, so one may expect that a model of $Q_\aff$ should embed affinely in a direct integral of models of $Q$.
This is indeed true, but is not an immediate consequence of the two results, and we prove it later, as \autoref{th:general-models-Qaff}.


\section{The continuous theory of extremal models}
\label{sec:cont-theory-extremal-models}

\begin{defn}
  \label{df:Text}
  Let $T$ be an affine theory. We will denote by $T_\ext$ the common continuous logic theory of the extremal models of $T$.
\end{defn}

The following holds in general.

\begin{lemma}
  \label{l:T=Textaff}
  Let $T$ be an affine theory. Then $(T_\ext)_\aff \equiv T$.
\end{lemma}
\begin{proof}
  The inclusion $T \sub (T_\ext)_\aff$ follows from the definitions. In order to show that $(T_\ext)_\aff = T$, we can show that every completion of $T$ is consistent with $(T_\ext)_\aff$, and it is enough to check this for the extreme completions (i.e., extreme points of $\tS_0^\aff(T)$). Now if $T_0$ is an extreme completion of $T$, every extreme type of $T_0$ is also an extreme type of $T$, by \autoref{l:extreme-transitivity}; hence any extremal model of $T_0$ is also an extremal model of $T$, and therefore a model of $(T_\ext)_\aff$.
\end{proof}

Thus, if $T$ is any affine theory, letting $Q = T_\ext$ in \autoref{eq:IotaStar} we obtain a map
\begin{gather*}
  \iota_x^* \colon \cM(\tS^\cont_x(T_\ext)) \to \tS^\aff_x(T),
\end{gather*}
and its restriction
\begin{gather*}
  \rho^\aff_x \colon \tS^\cont_x(T_\ext) \to \tS^\aff_x(T).
\end{gather*}

\begin{prop}
  \label{p:iota-ScText}
  Let $T$ be an affine theory.
  Then $\rho^\aff(\tS^\cont_x(T_\ext)) = \cl{\cE_x(T)}$.
\end{prop}
\begin{proof}
  As $\tS^\cont_x(T_\ext) = \cE(\cM(\tS^\cont_x(T_\ext)))$ and $\iota_x^*$ is surjective, we obtain the inclusion $\supseteq$ from \autoref{l:extreme-transitivity}~\ref{l:extreme-transitivity:extreme-points}.
  For the other inclusion, suppose that $p \notin \cl{\cE_x(T)}$. Then there are affine formulas $\varphi_1(x), \ldots, \varphi_n(x)$ such that
  \begin{equation*}
    p \in \bigcap_i \oset{\varphi_i > 0} \quad \And \quad \bigcap_i \oset{\varphi_i > 0} \cap \cE_x(T) = \emptyset.
  \end{equation*}
  In particular, this implies that $T_\ext \models \sup_x \bigwedge_i \varphi_i(x) \leq 0$. This means that for any $q \in \rho^\aff(\tS^\cont_x(T_\ext))$, we have that $\bigwedge_i \varphi_i(q) \leq 0$.
  Thus $p \notin \rho^\aff(\tS^\cont_x(T_\ext))$.
\end{proof}

\begin{cor}\label{c:closed-extreme-types}
  Let $T$ be an affine theory. The following are equivalent:
  \begin{enumerate}
  \item The extreme type spaces $\cE_x(T)$ are closed.
  \item The extremal models of $T$ form an elementary class in continuous logic.
  \end{enumerate}
\end{cor}
\begin{proof}
  The extremal models form an elementary class precisely if every model of $T_\ext$ is an extremal model of $T$. If this holds, then $\rho^\aff(\tS^\cont_x(T_\ext)) = \cE_x(T)$, and hence the latter is closed, by the previous proposition. Conversely, if the subspaces $\cE_x(T)$ are closed, then by the previous proposition, the affine type of any tuple of a model of $T_\ext$ is extreme; hence every model of $T_\ext$ is an extremal model of $T$.
\end{proof}

\begin{theorem}\label{th:closed-extreme-types-Text}
  Let $T$ be an affine theory such that the extreme type spaces $\cE_x(T)$ are closed. Let $M, N$ be extremal models of $T$. The following hold:
  \begin{enumerate}
  \item If $M\preceq^\aff N$, then $M\preceq^\cont N$.
  \item If $M\equiv^\aff N$, then $M\equiv^\cont N$.
  \end{enumerate}
  In particular, if $T$ is complete as an affine theory, then $T_\ext$ is complete as a continuous logic theory.
\end{theorem}
\begin{proof}
  For the first part, consider $D^\cont_M$, the continuous logic elementary diagram of $M$. If $M\preceq^\aff N$, then $N$ is a model of $D_M^\aff = (D^\cont_M)_\aff$. Moreover, if $N$ is extremal and $a$ is any $x$-tuple from $N$, then $\tp^\aff(a/M) \in \tS^\aff_x(D_M^\aff)$ is extreme, by \autoref{prop:ExtremeTypeTwoSteps}. So $N$ is an extremal model of $D^\aff_M$.

  By \autoref{p:ext-models-QAff}, there is an extension $N\preceq^\aff M_1$ such that $M_1\models D^\cont_M$, i.e., $M\preceq^\cont M_1$. By \autoref{c:closed-extreme-types}, this implies that $M_1$ is an extremal model of $T$. We thus have an extension $N\preceq^\aff M_1$ of extremal models of $T$, and we can iterate the previous step to produce a chain
  $$M_0\preceq^\aff N_0\preceq^\aff M_1\preceq^\aff N_1\preceq^\aff \cdots$$
  with $M_0=M$, $N_0=N$ such that $M_n\preceq^\cont M_{n+1}$ and $N_n\preceq^\cont N_{n+1}$ for all $n$. By considering the limit of this chain, which is a common continuous elementary extension of $M_0$ and $N_0$, we conclude that $M\preceq^\cont N$.

  The second part follows from the first and the joint embedding property for extremal models, \autoref{prop:ExtremalJEP}.
\end{proof}

\begin{cor}
  \label{c:ext-closed-elementarily-measurable-fields}
  Let $T$ be an affine theory in a separable language and assume that the extreme type spaces $\cE_x(T)$ are closed. Then every measurable field of extremal models of $T$ is elementarily measurable.
\end{cor}

It will follow from later results (see \autoref{cor:general-simplicial-theories}) that the conclusions of \autoref{th:closed-extreme-types-Text} also hold for simplicial theories without the assumption that the spaces $\cE_x(T)$ are closed.

\begin{question}\label{q:aff-extension-of-extremal-models-cont}
  Does \autoref{th:closed-extreme-types-Text} hold for arbitrary affine theories?
\end{question}

\begin{cor}\label{cor:E(T)-closed-iota-homeo}
  Let $T$ be an affine theory such that the extreme type spaces $\cE_x(T)$ are closed. Then for any variable $x$, the map
  $$\rho^\aff\colon \tS^\cont_x(T_\ext)\to \cE_x(T) \subseteq \tS^\aff_x(T)$$
  is a $\tau$-homeomorphism and a $\dtp$-isometry.
\end{cor}
\begin{proof}
  By the hypothesis and \autoref{p:iota-ScText}, the image of $\rho^\aff$ is exactly $\cE_x(T)$.
  For injectivity, it suffices to show that if $M, N$ are extremal models of $T$ and $a,b$ are $x$-tuples from $M$ and $N$, respectively, such that $(M,a)\equiv^\aff (N,b)$, then $(M,a)\equiv^\cont (N,b)$. This follows from \autoref{th:closed-extreme-types-Text} applied to the affine diagram $D^\aff_a = \Th^\aff(M,a)$. To verify the hypothesis of the theorem, note that each type space $\tS^\aff_y(a)$ is affinely homeomorphic to the convex subset $C=\pi_x^{-1}(\tp^\aff(a))\subseteq \tS^\aff_{xy}(T)$, which is a face because $\tp^\aff(a)$ is extreme. Hence $\cE_y(a)\cong C\cap\cE_{xy}(T)$, and thus the extreme type spaces of $D^\aff_a$ are closed. We note as well that $(M,a)$ and $(N,b)$ are extremal models of $D^\aff_a$, by \autoref{prop:ExtremeTypeTwoSteps}. By continuity and compactness, we deduce that $\rho^\aff$ is a $\tau$-homeomorphism.

  We have in general that $\dtp(\rho^\aff(p),\rho^\aff(q)) \leq \dtp(p,q)$ for all $p,q\in\tS^\aff_x(T_\ext)$. For the converse, suppose that $\dtp(\rho^\aff(p),\rho^\aff(q)) = r < \infty$. It follows from \autoref{cor:ExtremeTypeDistance} that there exists $s \in \cE^\aff_{xx'}(T)$ with $\pi_x(s) = p$, $\pi_{x'}(s) = q$ and $d^s(x, x') = r$. Using the surjectivity of $\rho_{xx'}^\aff$, let $t \in \tS_{xx'}^\cont(T_\ext)$ be such that $\rho_{xx'}^\aff(t) = s$. Then $d^t(x, x') = r$, and by the injectivity of $\rho_x^\aff$, $\pi_x(t) = p$ and $\pi_{x'}(t) = q$, witnessing that $\dtp(p, q) \leq r$.
\end{proof}


\section{Affine reduction and Bauer theories}
\label{sec:Bauer-theories}

In this section, we will show how every continuous logic theory $Q$ can be encoded by an affine theory $Q_\Bau$ in a way that the extremal models of $Q_\Bau$ correspond to the models of $Q$.
\begin{defn}
  \label{df:affine-reduction}
  We will say that a continuous logic theory $Q$ has \emph{affine reduction} if the map $\iota_x \colon \cL^\aff_x(Q_\aff) \to \cL^\cont_x(Q)$ defined in \autoref{eq:iota-Fx-to-Lx} has a dense image for every $x$, i.e., if every continuous logic formula can be approximated, modulo $Q$, by affine formulas.
\end{defn}

\begin{lemma}
  \label{l:AffineReduction}
  The following are equivalent for a continuous logic theory $Q$.
  \begin{enumerate}
  \item $Q$ has affine reduction.
  \item \label{i:l:aff-el:2} Every quantifier-free formula is a uniform limit modulo $Q$ of affine formulas.
  \item \label{i:l:aff-el:3} For any affine quantifier-free formula $\varphi$, the formula $|\varphi|$ is a uniform limit modulo $Q$ of affine formulas.
  \item \label{i:l:aff-el:4} For any two affine formulas $\varphi$ and $\psi$, the formula $\varphi \vee \psi$ is a uniform limit modulo $Q$ of affine formulas.
  \end{enumerate}
\end{lemma}
\begin{proof}
  \begin{cycprf}
  \item[\impnnext] Immediate.
  \item The case where $\varphi$ and $\psi$ are quantifier-free is immediate, as $\varphi\vee\psi = \half(\varphi+\psi+|\varphi-\psi|)$. For the general case, we observe that every affine formula can be written in prenex form
    \begin{equation*}
      \qsup_{x_1} \qinf_{x_2} \cdots \qsup_{x_n} \theta(x_1, \ldots, x_n)
    \end{equation*}
    with $\theta$ affine, quantifier-free, and use that quantifiers commute with $(\cdot\vee\psi)$ and with $(\theta\vee\cdot)$, up to renaming variables.
  \item[\impfirst] Our hypothesis implies that the closure of $\cL^\aff_x(Q_\aff)$ in $C(\tS^{\cont}_x(Q))$ is closed under the lattice operations. Now an easy induction on formulas shows that it must contain $\cL^\cont_x(Q)$.
  \end{cycprf}
\end{proof}

\begin{lemma}\label{l:affine-reduction-cont-tarski-vaught}
  Let $Q$ be a continuous logic theory with affine reduction. If $N\models Q$ and $M\preceq^\aff N$, then $M\preceq^\cont N$.
\end{lemma}
\begin{proof}
  Follows from the Tarski--Vaught test for continuous logic.
\end{proof}

\begin{remark}
  A theory $Q$ has affine reduction if and only if the maps $\iota_x^* \colon \cM(\tS^{\cont}_x(Q)) \to \tS^\aff_x(Q_\aff)$ defined in \autoref{eq:IotaStar} are injective.
  In that case, they are affine homeomorphisms of compact convex sets. This gives a complete description of the type spaces of the affine theory $Q_\aff$.
\end{remark}

\begin{defn}
  \label{df:Bauer-theory}
  We will say that an affine theory $T$ is a \emph{Bauer theory} if for every finite tuple $x$, the type space $\tS^\aff_x(T)$ is a Bauer simplex (i.e., $\tS^\aff_x(T)$ is a simplex and $\cE_x(T)$ is closed).
\end{defn}

In particular, Bauer theories satisfy all the conclusions of \autoref{c:closed-extreme-types} and \autoref{th:closed-extreme-types-Text}. The following two theorems are the main results of this section.

\begin{theorem}\label{th:Bauer-correspondence-Q}
  Let $Q$ be a continuous logic theory with affine reduction. Then:
  \begin{enumerate}
  \item $Q_\aff$ is a Bauer theory.
  \item The extremal models of $Q_\aff$ are precisely the models of $Q$. In particular, $(Q_\aff)_\ext \equiv Q$.
  \item\label{i:models-of-Q-aff} Every model of $Q_\aff$ is a direct integral of models of $Q$.
  \end{enumerate}
\end{theorem}
\begin{proof}
  Since $Q$ has affine reduction, the maps $\iota_x^*$ defined in \autoref{eq:IotaStar} are isomorphisms of compact convex sets. Therefore $\tS_x^\aff(Q_\aff)\cong \cM(\tS_x^\cont(Q))$, which are Bauer simplices.

  For the second assertion, if $M$ is an extremal model of $Q_\aff$, then it embeds affinely into a model $N\models Q$, by \autoref{p:ext-models-QAff}.
  By \autoref{l:affine-reduction-cont-tarski-vaught}, $M\preceq^\cont N$, and thus $M\models Q$. Conversely, as the maps $\iota_x^*$ are isomorphisms of convex sets, the affine type of any tuple realized in a model of $Q$ is extreme. Hence all models of $Q$ are extremal models of $Q_\aff$.

  The last assertion then follows from the extremal decomposition theorem for simplicial theories, \autoref{th:integral-decomposition}.
\end{proof}

Conversely, we have the following.

\begin{theorem}\label{th:Bauer-correspondence-T}
  Let $T$ be a Bauer theory. Then:
  \begin{enumerate}
  \item $T_\ext$ has affine reduction.
  \item The models of $T_\ext$ are precisely the extremal models of $T$.
  \end{enumerate}
\end{theorem}
\begin{proof}
  We check that $T_\ext$ has affine reduction using \autoref{l:AffineReduction}~\ref{i:l:aff-el:4}. Let $\varphi,\psi\in \cL^\aff_x$. Let $\chi$ be the least upper bound of $\varphi$ and $\psi$ in $\cA(\tS^\aff_x(T))$. It exists by \autoref{th:Bauer}, and is given at $p \in \cE_x(T)$ by
  \begin{equation*}
    \chi(p) = \varphi(p) \vee \psi(p).
  \end{equation*}
  For any $\eps > 0$ there exists an affine formula $\theta$ such that $\theta \leq \chi \leq \theta + \eps$ on $\tS^\aff_x(T)$.
  For any extremal model $M \models T$ and $a \in M^x$, we have $\tp^\aff(a) \in \cE_x(T)$, so
  \begin{equation*}
    \varphi(a)\vee\psi(a) = \chi(a).
  \end{equation*}
  Therefore $T_\ext \models \theta \leq \varphi\vee\psi \leq \theta + \eps$, as desired.

  The fact that the models of $T_\ext$ are precisely the extremal models of $T$ is just \autoref{c:closed-extreme-types}. (One may also use that $(T_\ext)_\aff \equiv T$ and the preceding theorem.)
\end{proof}

\begin{cor}
  \label{c:bauer-elementarily-measurable-fields}
  Every measurable field of extremal models of a Bauer $\cL$-theory $T$ is elementarily measurable.
\end{cor}
\begin{proof}
  If $T$ has a separable language, this is a particular case of \autoref{c:ext-closed-elementarily-measurable-fields}, but here we can bypass this assumption. Indeed, since $T_\ext$ has affine reduction, for every finite (or countable) sublanguage $\cL_0\subseteq\cL$ there exists an intermediate countable language $\cL_0\subseteq\cL_1 \subseteq\cL$ such that $T_\ext\rest_{\cL_1}$ still has affine reduction. Thus, given a measurable field $(M_\Omega,e_I)$ of extremal models of $T$, we have that $\fI^\cont(M_\Omega,e_I,\cL_0)\supseteq \fI^\cont(M_\Omega,e_I,\cL_1) = \fI^\aff(M_\Omega,e_I,\cL_1)$ by \autoref{l:affine-reduction-cont-tarski-vaught}. As the latter collection is cofinal, so is the former.
\end{proof}

Like quantifier elimination, affine reduction is a syntactical property that can always be obtained in a definitional expansion of the theory. One easy way to achieve it (together with quantifier elimination) is by Morleyization.
Formally, we define the \emph{Morleyization language} to consist of all function symbols of $\cL$, together with a predicate symbol $P_\varphi(x)$ for each $\varphi \in \cL^\cont_x$, with the same arity, bound and convex uniform continuity modulus as $\varphi$:
\begin{equation*}
  \cL_\Mor \coloneqq \bigl\{ F : F\in\cL\text{ function symbol} \bigr\}
  \cup
  \bigl\{ P_\varphi:\varphi\in \cL^\cont_x, \, \text{ finite tuple } x \bigr\}.
\end{equation*}
Identifying a predicate symbol $\rho \in \cL$ with $P_\rho \in \cL_\Mor$, we may view $\cL$ as a sublanguage of $\cL_\Mor$.

The \emph{Morleyization} of a continuous $\cL$-theory $Q$ is the continuous $\cL_\Mor$-theory $Q_\Mor$ that consists of $Q$ together with the axioms asserting that $P_\varphi(x)=\varphi(x)$ for each $\varphi$, which are easily expressible in continuous logic (and if $\varphi$ is affine, then even in affine logic).
Clearly, $Q_\Mor$ has affine reduction. Note also that $\density_{Q_\Mor}(\cL_\Mor) = \density_Q(\cL)$, where the density character of the language of a continuous theory is defined in the same way as in \autoref{defn:LanguageDensityCharacter} with $\cL_n^\cont$ in the place of $\cL_n^\aff$

\begin{remark}
  \label{rmk:AffineMinGivesMax}
  A more minimalistic approach would be to name $P_\varphi$ for quantifier-free formulas only.
  In fact, it is enough to name $P_\varphi$ for every $\varphi$ of the form $\bigwedge_i \sum_j \varphi_{ij}$, where $\varphi_{ij}$ are atomic.
  Indeed, using the identities
  \begin{gather*}
    \bigwedge_i \varphi_i + \bigwedge_j \psi_j = \bigwedge_{i, j} \, (\varphi_i + \psi_j),
    \\
    (\varphi - \psi) \wedge (\sigma - \tau)
    = (\varphi + \tau) \wedge (\psi + \sigma) - (\psi + \tau),
  \end{gather*}
  one checks that the closure of this set under affine combinations is closed under $\wedge$, and thus also under $\vee$.
  It is therefore dense among all quantifier-free formulas.
\end{remark}

\begin{defn}
  We define the \emph{Bauerization} of $Q$ as the affine theory $Q_\Bau\coloneqq (Q_\Mor)_\aff$.
\end{defn}

\begin{remark}\label{rmk:Bau}
  By \autoref{th:Bauer-correspondence-Q}, $Q_\Bau$ is a Bauer theory, the extremal models of $Q_\Bau$ are precisely the models of $Q_\Mor$, and more generally, the models of $Q_\Bau$ are precisely the direct integrals of models of $Q_\Mor$.
\end{remark}

Of course, as $Q_\Mor$ is a definitional expansion of $Q$, the natural restriction maps $\eta_x\colon \tS^\cont_x(Q_\Mor) \to \tS^\cont_x(Q)$ are homeomorphisms.
Similarly, as $Q_\Bau\models Q_\aff$, we can consider the restriction maps
\begin{equation*}
  \eta_x^\aff \colon \tS^\aff_x(Q_\Bau) \to \tS^\aff_x(Q_\aff),
\end{equation*}
which are continuous and affine.

\begin{lemma}\label{l:interdef-affine-reduction-Bau}
  The following are equivalent:
  \begin{enumerate}
  \item The continuous theory $Q$ has affine reduction.
  \item The affine theory $Q_\Bau$ is an affine definitional expansion of $Q_\aff$ (i.e., $Q_\Bau$ consists of $Q_\aff$ together with affine axioms defining the new predicates in terms of affine $\cL$-formulas).
  \item The maps $\eta^\aff_x$ are affine homeomorphisms.
  \end{enumerate}
\end{lemma}
\begin{proof}
  \begin{cycprf}
  \item
    For every $\varphi\in\cL_x^\cont$ there exist $\psi_k\in\cL_x^\aff$ such that $\psi_k \rightarrow \varphi$ uniformly in models of $Q$.
    Since $\psi_k$ are affine, $\psi_k = P_{\psi_k} \rightarrow P_\varphi$ in models of $Q_\Bau$, at a rate independent of the model.
    Thus, $Q_\Bau$ is an affine definitional expansion of $Q_\aff$.
  \item Clear.
  \item[\impfirst]
    Since $Q_\Bau$ is a Bauer theory and has the same type spaces as $Q_\aff$, the latter is a Bauer theory as well, and its extremal models are exactly the reducts of those of $Q_\Bau$.
    In other words, the extremal models of $Q_\aff$ are the models of $Q$, so $Q = (Q_\aff)_\ext$.
    By \autoref{th:Bauer-correspondence-T}, $Q$ has affine reduction.
  \end{cycprf}
\end{proof}

In the general case, we still have the following.

\begin{prop}\label{p:eta-surjective}
  For any theory $Q$, the maps $\eta_x^\aff$ are surjective.
\end{prop}
\begin{proof}
  By \autoref{p:ext-models-QAff}, every extreme type of $Q_\aff$ is realized in a model of $Q$ (equivalently, of $Q_\Mor$) and therefore belongs to the image of $\eta_x^\aff$.
  By Krein--Milman, $\eta_x^\aff$ is surjective.
\end{proof}

Applying \autoref{p:eta-surjective} to $x = \emptyset$ yields the following.
\begin{cor}
  The set of affine $\cL$-consequences of $Q_\Bau$ is precisely $Q_\aff$.
\end{cor}

Note that for any continuous $\cL$-theory $Q$, we have that $\fd_Q(\cL) = \fd_{Q_\aff}(\cL)$. This is true because $\fd_Q(\cL)$ is determined by the density character of atomic formulas and because the sentences asserting that a certain set of atomic formulas is dense are affine. In particular, $\fd_{Q_\Bau}(\cL_\Mor) = \fd_Q(\cL)$.
\begin{theorem}
  \label{th:general-models-Qaff}
  Let $Q$ be a continuous theory.
  Then for every model $M$ of $Q_\aff$, there exists a probability space $(\Omega, \mu)$ and an elementarily measurable field $(M_\omega : \omega \in \Omega)$ of models of $Q$ with $\density(M_\omega)\leq\density(M)+\density_Q(\cL)$ such that $M$ embeds affinely in  $\int^\oplus_{\Omega}M_\omega\ud\mu$.
  If $Q$ has a separable language and $M$ is separable, then $(\Omega,\mu)$ can be chosen to be a standard probability space.
\end{theorem}
\begin{proof}
  By \autoref{p:eta-surjective}, every model $M\models Q_\aff$ embeds affinely into (the $\cL$-reduct of) a model of $Q_\Bau$.
  Indeed, if $p\in\tS^\aff_x(Q_\aff)$ is the type of an enumeration of $M$, then the proposition implies that $p$ is realized in a model $N$ of $Q_\Bau$, and such a realization is an affine image of $M$ in $N\rest_{\cL}$.
  By Löwenheim--Skolem and the remarks above, we may moreover take $N$ with $\density(N)\leq\density(M) + \density_Q(\cL)$.

  Now as per \autoref{rmk:Bau}, $N\cong\int^\oplus_{\Omega}(M_\omega)_\Mor\ud\mu$ for some measurable field of $\cL_\Mor$-structures such that $M_\omega\models Q$ for all $\omega$. Moreover, by \autoref{c:bauer-elementarily-measurable-fields}, the field is elementarily measurable. Then so is the field of $\cL$-structures $(M_\omega : \omega\in\Omega)$, with the same pointwise enumeration, and $M$ embeds affinely into the direct integral $\int^\oplus_{\Omega}M_\omega\ud\mu$.
  Recalling the precise statement of \autoref{th:integral-decomposition}, we may choose $M_\omega$ with $\density(M_\omega)\leq \density(N) + \density_Q(\cL) \leq \density(M) + \density_Q(\cL)$, and if $Q$ has a separable language and $M$ is separable, we can take $\Omega$ to be a standard probability space.
\end{proof}


\section{Convex formulas}
\label{sec:convex-formulas}

In between the sets of affine and continuous formulas, several natural sets of formulas also deserve special attention.
We discuss here some of them that will be relevant later in the paper.

\begin{defn}
  Let $\psi$ be a continuous $\cL$-formula.
  \begin{enumerate}
  \item
    We say that $\psi$ is \emph{convex} (respectively, \emph{concave}) if it can be written in the form  $\bigvee_{i<n}\varphi_i$ (respectively, $\bigwedge_{i<n}\varphi_i$) for affine formulas $\varphi_i$.
  \item
    We say that $\psi$ is a \emph{delta-convex formula} if it is the difference of two convex (equivalently, concave) formulas.
  \item
    We say that $\psi$ is \emph{inf-convex} (respectively, \emph{sup-concave}) if it can be written as $\inf_x \bigvee_{i<n}\varphi_i$ (respectively, $\sup_x \bigwedge_{i<n}\varphi_i$) for affine formulas $\varphi_i$ and a tuple of variables $x$.
  \item
    We say that $\psi$ is \emph{sup-delta-convex} (resp., \emph{inf-delta-convex}) if it can be written as $\sup_x \varphi$ (resp., $\inf_x \varphi$) where $\varphi$ is a delta-convex formula and $x$ is a tuple of variables.
  \end{enumerate}
\end{defn}

A convex formula $\varphi(x)$ induces a continuous convex function $\tS_x^\aff(T) \to \R$, which we still denote by $\varphi$.
Conversely, every continuous convex function is a uniform limit of functions of this form (see \cite[Cor.~I.1.3]{Alfsen1971}).

Every delta-convex formula can be written in the form
\begin{equation}
  \label{eq:delta-convex-form}
  \bigvee_i \bigwedge_j \phi_i + \psi_j = \bigwedge_j \bigvee_i \phi_i + \psi_j
\end{equation}
with $\phi_i$ and $\psi_j$ affine.
Conversely, by the same argument as in \autoref{rmk:AffineMinGivesMax}, the collection of delta-convex formulas forms a vector lattice.
Therefore, it coincides with the lattice generated by affine formulas.
In particular, the collection of delta-convex formulas defines a family of continuous functions on $\tS^\aff_x(T)$ that is uniformly dense in the collection of all continuous functions.

Note also that the collections of sup-delta-convex and inf-delta-convex formulas are closed under $+$, $\vee$, $\wedge$, and multiplication by positive scalars. If $\phi$ is sup-delta-convex, then $-\phi$ is inf-delta-convex and vice versa.

\begin{defn}
  Let $Q$ be a theory in continuous logic.
  \begin{enumerate}
  \item We say that $Q$ is \emph{universal-delta-convex ($\forall\Delta$)} if it can be axiomatized by conditions of the form $\psi \geq 0$ where $\psi$ is an inf-delta-convex sentence (equivalently, of the form $\psi \leq 0$ where $\psi$ is a sup-delta-convex sentence).
  \item We say that $Q$ is \emph{existential-delta-convex ($\exists\Delta$)} if it can be axiomatized by conditions of the form $\psi \geq 0$ where $\psi$ is a sup-delta-convex sentence.
  \item We denote by $Q_{\forall\Delta}$ its \emph{$\forall\Delta$ part}, i.e., the set of $\forall\Delta$ conditions implied by $Q$.
    We denote by $Q_{\exists\Delta}$ its \emph{$\exists\Delta$ part}.
  \end{enumerate}
\end{defn}

\begin{remark}
  \label{rmk:CounterpartAffineQF}
  Since delta-convex formulas are dense among all continuous combinations of affine formulas, many of the notions and results we present here are counterparts of standard notions and results regarding quantifier-free, and single-quantifier formulas in continuous logic.
  For example, a $\forall\Delta$ theory is analogous to a universal continuous theory (but is not a special case of that, since, if we named all affine formulas by atomic predicates, then we would also need the ``affine Morleyization'' axioms, which are not universal).
  Similarly, the next lemma is analogous to ``$M \models Q_\forall$ if and only if $M$ embeds in a model of $Q$'', and so on.
\end{remark}

\begin{lemma}
  \label{l:Quda}
  Let $Q$ be a continuous logic theory and let $M$ be a structure.
  Then $M\models Q_{\forall\Delta}$ if and only if there exists an extension $M\preceq^\aff N$ with $N\models Q$.
\end{lemma}
\begin{proof}
  One direction is immediate.
  For the other, we need to show that $D^\aff_M \cup Q$ is satisfiable, where $D^\aff_M$ is the affine diagram of $M$ (see \autoref{dfn:AffineEverything}).
  Otherwise, by the compactness theorem for continuous logic, there exist finitely many conditions $\varphi_i(a)\geq 0$ from $D^\aff_M$ and some $\varepsilon>0$ such that
  $$Q\models \sup_x \bigwedge_{i<n}\varphi_i(x) + \varepsilon \leq 0.$$
  Since the latter is a $\forall\Delta$ condition and $M\models Q_{\forall\Delta}$, this is a contradiction.
\end{proof}

\begin{prop}\label{p:universal-concave-theories}
  Let $Q$ be a continuous logic theory. The following are equivalent:
  \begin{enumerate}
  \item $Q$ is a $\forall\Delta$ theory.
  \item $Q$ is preserved by affine substructures, i.e., if $M\preceq^\aff N$ and $N\models Q$, then $M\models Q$.
  \end{enumerate}
\end{prop}
\begin{proof}
  \begin{cycprf}
  \item Since affine extensions respect delta-convex formulas.
  \item[\impfirst] Follows from \autoref{l:Quda}.
  \end{cycprf}
\end{proof}

Combining this with \autoref{p:ext-models-QAff}, we obtain the following corollary which will be useful in some examples.

\begin{cor}\label{c:uda-ext-models}
  Let $Q$ be a $\forall\Delta$ theory.
  Then every extremal model of $Q_\aff$ is a model of~$Q$.
\end{cor}

\begin{prop}
  \label{p:universal-lattice-aff-formulas}
  Let $Q$ be a continuous logic theory and $\varphi(x)$ be a continuous logic formula.
  The following are equivalent:
  \begin{enumerate}
  \item $\varphi(x)$ can be approximated modulo $Q$ by sup-delta-convex formulas.
  \item For every extension $M \preceq^\aff N$ of models of $Q$ and every tuple $a \in M^x$, we have $\varphi^M(a) \leq \varphi^N(a)$.
  \end{enumerate}
\end{prop}
\begin{proof}
  \begin{cycprf}
  \item Easy, as in the previous proposition.
  \item[\impfirst] Up to subtracting a constant, we may assume that $\varphi\leq 0$. We consider the set
    \begin{equation*}
      \Delta = \set{\psi(x) : Q \models \psi(x) \leq \varphi(x) \And \psi \text{ is sup-delta-convex} }.
    \end{equation*}
    Let $\eps > 0$.
    We claim that the set of conditions
    \begin{equation*}
      \set{ \psi + \eps \leq \varphi(x) : \psi(x) \in \Delta }
    \end{equation*}
    is inconsistent with $Q$.
    Otherwise, there is a model $M \models Q$ and a tuple $a\in M^x$ with $\psi^M(a) + \eps \leq \varphi^M(a)$ for every $\psi \in \Delta$. Let $D^\aff_M$ be the affine diagram of $M$, and let $r=\varphi^M(a)\leq 0$. We have that
    \begin{equation}
      \label{eqn:Diag}
      Q \cup D^\aff_M \models r \leq \varphi(a).
    \end{equation}
    Indeed, any $N \models Q \cup D^\aff_M$ is an affine extension of $M$, hence by hypothesis, $r = \varphi^M(a) \leq \varphi^N(a)$.

    By compactness, \autoref{eqn:Diag} implies that there are affine formulas $\varphi_i(x,y)$, a tuple $b \in M^y$, and $\delta > 0$ such that $\varphi_i^M(a,b) = 0$ for each $i<n$ and
    \begin{equation*}
      Q \cup \set{\varphi_i(a,b) \leq \delta}_{i<n} \models r -\eps/2 \leq \varphi(a).
    \end{equation*}
    Hence, if we let $m$ be any negative lower bound for $\varphi$, we have:
    \begin{equation*}
      Q \models (r - \eps/2) \wedge \bigwedge_{i<n}(m/\delta)\varphi_i(x,y) \leq \varphi(x).
    \end{equation*}
    If we consider the concave formula $\psi(x, y) = (r + \eps/2) \wedge \bigwedge_{i<n}(m/\delta)\varphi_i(x,y)$, we obtain that
    \begin{equation*}
      Q \models \sup_y \psi(x, y) \leq \varphi(x),
    \end{equation*}
    that is, $\sup_y \psi(x, y) \in \Delta$. Thus, on the one hand, $M \models \sup_y \psi(a, y) + \eps \leq \varphi(a)$, but on the other,
    \begin{equation*}
      \bigl(\sup_y \psi(a, y)\bigr)^M \geq (r - \eps/2) \wedge \bigwedge_{i<n} (m/\delta) \varphi_i^M(a,b) = \varphi(a) - \eps/2,
    \end{equation*}
    a contradiction.

    We have established that $Q \cup \{\psi(x) + \eps \leq \varphi(x)\}_{\psi \in \Delta}$ is inconsistent. As $\Delta$ is closed under maxima, by compactness, there is $\psi \in \Delta$ such that $Q \models \psi(x) + \eps \geq \varphi(x)$. On the other hand, as $\psi \in \Delta$, $Q \models \psi(x) \leq \varphi(x)$.
    As $\eps$ was arbitrary, we conclude that $\varphi$ can be approximated by sup-delta-convex formulas modulo $Q$.
  \end{cycprf}
\end{proof}

\begin{defn}
  A \emph{q-convex} condition is one of the form
  $$\sup_{x_0}\inf_{y_0}\sup_{x_1}\cdots\inf_{y_m}\bigvee_{i<n}\varphi_i(x,y)\leq 0,$$
  where $m,n$ are arbitrary and the $\varphi_i$ are affine formulas.
  A continuous logic theory is \emph{q-convex} if it can be axiomatized by q-convex conditions.
\end{defn}

\begin{prop}\label{p:qconv-direct-integral-preservation}
  Every q-convex theory is preserved by direct integrals of elementarily measurable fields. That is, if $Q$ is q-convex and $(M_\Omega,e_I)$ is any elementarily measurable field of models of $Q$, then $\int_\Omega^\oplus M_\omega\ud\mu \models Q$.
\end{prop}
\begin{proof}
  Let $\psi\leq 0$ be a q-convex condition satisfied by every structure $M_\omega$. Say $\psi=\sup_{x_0}\cdots\inf_{y_m}\bigvee_{i<n}\varphi_i(x,y)$ for some affine formulas $\varphi_i$. Letting $M=\int_\Omega^\oplus M_\omega\ud\mu$, we have, using \autoref{th:Los} and \autoref{lem:LosQuantifier}:
  \begin{align*}
    \psi^M & = \sup_{f_0\in M}\cdots\inf_{g_m\in M}\bigvee_{i<n}\int_\Omega \varphi_i^{M_\omega}(f_0(\omega),\dots,g_m(\omega))\ud\mu(\omega) \\
           & \leq \sup_{f_0\in M}\cdots\inf_{g_m\in M}\int_\Omega \bigvee_{i<n}\varphi_i^{M_\omega}(f_0(\omega),\dots,g_m(\omega))\ud\mu(\omega) \\
           & = \int_\Omega \big(\sup_{x_0}\cdots\inf_{y_m}\bigvee_{i<n}\varphi_i(x,y)\big)^{M_\omega}\ud\mu(\omega) \leq 0,
  \end{align*}
  as desired.
\end{proof}
A converse of \autoref{p:qconv-direct-integral-preservation} for complete theories will be proved in \autoref{th:direct-integral-preservation}.

\begin{cor}\label{c:qconv-Qaff}
  Let $Q$ be a q-convex theory. If $M\models Q_\aff$, then there exists an extension $M\preceq^\aff N$ with $N\models Q$.
\end{cor}
\begin{proof}
  If $M\models Q_\aff$, then by \autoref{th:general-models-Qaff} there exists an extension $M\preceq^\aff N$ where $N$ is a direct integral of an elementarily measurable field of models of $Q$. By the preceding proposition, $N\models Q$.
\end{proof}

\begin{cor}
  \label{c:qconv-iota-Sc-surjective}
  Let $Q$ be a q-convex theory.
  Then the affine part map $\rho^\aff \colon \tS^\cont_x(Q) \to \tS^\aff_x(Q_\aff)$, as defined in \autoref{eq:RhoAff}, is surjective.
\end{cor}
\begin{proof}
  By the preceding corollary, any affine type $p\in\tS_x^\aff(Q_\aff)$ is realized in a model of $Q$. The result follows.
\end{proof}

We continue with a characterization of delta-convex formulas modulo q-convex theories.
This can be seen as a strengthening of a result of Bagheri, \cite[Prop.~5.4]{Bagheri2021}.

\begin{prop}
  \label{p:qconv-delta-convex-formulas}
  Let $Q$ be a q-convex theory and let $\varphi(x)$ be a continuous logic formula. The following are equivalent:
  \begin{enumerate}
  \item $\varphi(x)$ can be approximated modulo $Q$ by delta-convex formulas.
  \item For every extension $M \preceq^\aff N$ of models of $Q$ and every tuple $a \in M^x$, we have $\varphi^M(a) = \varphi^N(a)$.
  \end{enumerate}
\end{prop}
\begin{proof}
  \begin{cycprf}
  \item Clear.
  \item[\impfirst] Consider the map $\rho^\aff\colon \tS_x^\cont(Q)\to \tS_x^\aff(Q_\aff)$, which is surjective by \autoref{c:qconv-iota-Sc-surjective}.
    The formula $\varphi$ induces a continuous function
    $$\varphi\colon \tS^\cont_x(Q) \to\R.$$
    We claim that $\varphi$ factors trough $\rho^\aff$ via a continuous function $\tilde\varphi\colon\tS_x^\aff(Q_\aff)\to\R$.

    The map $\tilde\varphi$ must satisfy $\tilde\varphi(p) =\varphi(a)$, where $a$ is any realization of $p$ in some model $M\models Q$.
    Such an $a$ always exists, because $\rho^\aff\colon \tS_x^\cont(Q)\to \tS_x^\aff(Q_\aff)$ is surjective.
    Moreover, if $a'$ is any other realization of $p$ in a model $M'\models Q$, then $(M,a)\equiv^\aff (M',a')$.
    Thus, by \autoref{prop:AffineJEP}, there is a common affine extension, say $(N,b)$.
    As $N\models Q_\aff$, by \autoref{c:qconv-Qaff}, we may assume that $N\models Q$. Finally, using the hypothesis from the statement, we obtain that
    $$\varphi^M(a) = \varphi^N(b) = \varphi^{M'}(a').$$
    We conclude that $\tilde\varphi$ is well-defined.

    It follows that $\tilde\varphi$ is continuous, because $\varphi$ is continuous, $\rho^\aff$ is continuous and surjective, and $\tS_x^\cont(T)$ is compact Hausdorff.
    Indeed, if $F\subseteq \R$ is closed, then $\tilde\varphi^{-1}(F) = \rho^\aff(\varphi^{-1}(F))$, which is closed as well.

    By the lattice form of the Stone--Weierstrass theorem, $\tilde{\varphi}$, and therefore $\varphi$, can be uniformly approximated (on the respective type spaces) by lattice expressions in affine formulas, i.e., by delta-convex formulas.
  \end{cycprf}
\end{proof}

\begin{defn}
  Given an affine extension $M\preceq^\aff N$, we will say that $M$ is \emph{existentially closed for convex formulas} in $N$, and write $M\preceq^\eca N$, if for every $a \in M^x$ and every inf-convex formula $\varphi(x)$, we have $\varphi^M(a) = \varphi^N(a)$.
\end{defn}

Note that this is analogous, in the sense of \autoref{rmk:CounterpartAffineQF}, to a structure being existentially closed in another.

\begin{remark}\label{rmk:eca-closed}
  Assume that $M \preceq^\aff N$.
  \begin{enumerate}
  \item We have $M \preceq^\eca N$ if and only if every sup-delta-convex (equivalently, inf-delta-convex; equivalently, sup-concave) formula takes the same value at every $a \in M^x$ in $M$ and in $N$. This follows from the presentation \autoref{eq:delta-convex-form} of delta-convex formulas.
    Equivalently, if $N\models (D^\cont_M)_{\forall\Delta}$, where $D^\cont_M$ denotes the continuous logic diagram of $M$.
  \item\label{i:rmk:eca-closed-ueca} If $M\preceq^\eca N$ then $M$ is a model of the $\forall\exists\Delta$-theory of $N$. That is, $M$ satisfies every condition $\psi\leq 0$ satisfied by $N$ where $\psi$ is of the form $\psi=\sup_x\inf_y\varphi(x,y)$ with $\varphi$ a delta-convex formula.
  \end{enumerate}
\end{remark}

\begin{defn}
  A continuous logic theory $Q$ is \emph{model complete by affine} if whenever $M,N\models Q$, if  $M\preceq^\aff N$, then $M\preceq^\cont N$.
  In other words, for every model $M\models Q$, the continuous theory $D^\aff_M\cup Q$ (where $D^\aff_M$ is the affine diagram of $M$) is complete.
\end{defn}

\begin{prop}\label{p:affine-complete}
  Let $Q$ be a continuous logic theory. The following are equivalent.
  \begin{enumerate}
  \item $Q$ is model complete by affine.
  \item\label{i:p:affine-complete-eca-closed} Whenever $M,N\models Q$, if $M\preceq^\aff N$, then $M\preceq^\eca N$.
  \item Every continuous formula $\varphi(x)$ can be approximated modulo $Q$ by sup-delta-convex formulas.
  \end{enumerate}
\end{prop}
\begin{proof}
  \begin{cycprf}
  \item Clear.
  \item First we claim that every inf-convex formula $\varphi(x)$ can be approximated modulo $Q$ by sup-delta-convex formulas.
    We will use \autoref{p:universal-lattice-aff-formulas}. Consider an extension $M\preceq^\aff N$ of models of $Q$ and a tuple $a\in M^x$. We have that $M \preceq^\eca N$, so $\varphi^M(a)=\varphi^N(a)$, and the hypothesis of the proposition is satisfied.

    Now it follows from the fact that sup-delta-formulas are closed under $\wedge$ that every inf-delta-convex formula can be approximated by sup-delta-convex formulas, and (by considering $-\varphi$) that every sup-delta-convex formula can be approximated by inf-delta-convex formulas (modulo $Q$). From this, the general case follows easily by induction on the complexity of the formula $\varphi$.

  \item[\impfirst] If $M\preceq^\aff N$ and $\varphi$ is any continuous formula with parameters from $M$, we have by the hypothesis and by \autoref{p:universal-lattice-aff-formulas} that $\varphi^M\leq \varphi^N$. By considering the formula $-\varphi$, we have that $\varphi^N\leq \varphi^M$ as well. We can conclude that $M\preceq^\cont N$.
  \end{cycprf}
\end{proof}

\begin{defn}
  A continuous logic theory $Q$ has \emph{delta-convex reduction} if every continuous formula can be approximated, modulo $Q$, by delta-convex formulas.
\end{defn}

\begin{remark}\label{rmk:latt-red-iota-injective}
  By Stone--Weierstrass, $Q$ has delta-convex reduction if and only if the maps $\rho^\aff_x \colon \tS_x^\cont(Q)\to \tS_x^\aff(Q_\aff)$ are injective.
\end{remark}

\begin{prop}\label{prop:latt-red-for-closed-E(T)}
  Let $T$ be an affine theory such that the extreme type spaces $\cE_x(T)$ are closed. Then $T_\ext$ has delta-convex reduction.
\end{prop}
\begin{proof}
  Follows from the previous remark and \autoref{cor:E(T)-closed-iota-homeo}.
\end{proof}

\begin{prop}\label{p:lattice-reduction}
  Let $Q$ be a q-convex theory. The following are equivalent.
  \begin{enumerate}
  \item $Q$ is model complete by affine.
  \item $Q$ has delta-convex reduction.
  \item\label{i:p:lattice-reduction-eca-closed} Whenever $M\models Q$ and $M\preceq^\aff N$, we have $M\preceq^\eca N$.
  \end{enumerate}
\end{prop}
\begin{proof}
  \begin{cycprf}
  \item Follows readily using \autoref{p:qconv-delta-convex-formulas}.

  \item If $M\models Q$ and $M\preceq^\aff N$, then $N\models Q_\aff$, so by \autoref{c:qconv-Qaff} there is an extension $N\preceq^\aff N'$ with $N'\models Q$. It follows by delta-convex reduction of $Q$ that $M\preceq^\cont N'$, and hence also that $M\preceq^\eca N$.

  \item[\impfirst] Follows from \autoref{p:affine-complete}.
  \end{cycprf}
\end{proof}

A complete theory with delta-convex reduction need not be q-convex. For example, a continuous theory with affine reduction is never preserved by non-trivial direct multiples (as its models are the extremal models of its affine part).

\begin{prop}\label{p:qconvex-latt-red-iota-homeo}
  Let $Q$ be a q-convex theory with delta-convex reduction.
  Then the map $\rho^\aff_x\colon \tS_x^\cont(Q) \to \tS_x^\aff(Q_\aff)$ is a $\tau$-homeomorphism and a $\dtp$-isometry.
\end{prop}
\begin{proof}
  The map $\rho^\aff_x$ is surjective by \autoref{c:qconv-iota-Sc-surjective} and injective as per \autoref{rmk:latt-red-iota-injective}. Hence it is a homeomorphism. If $y$ is a variable of the same sort as $x$, then the surjectivity of the map $\rho^\aff_{xy}$ together with the injectivity of $\rho^\aff_x$ readily imply that $\rho^\aff_x$ is $\dtp$-isometric.
\end{proof}

We end with the following observation, related to \autoref{q:aff-extension-of-extremal-models-cont}.

\begin{prop}
  Let $T$ be an affine theory and let $M\preceq^\aff N$ be an affine extension of extremal models of $T$. Then $M\preceq^\eca N$.
  In particular, if $T$ is complete, then all extremal models of $T$ have the same universal-delta-convex theory.
\end{prop}
\begin{proof}
  Suppose there is $b \in N^x$ such that $\bigvee_{i<n} \phi_i(a, b) \leq 0$, with $\phi_i$ affine and $a \in M^y$. By \autoref{prop:ExtremeTypeTwoSteps}, $\tp(b/M) \in \cE_x(M)$ and by \autoref{l:realized-extreme-dense} (applied with $A = M$), for every $\eps>0$ there exists $b' \in M^x$ such that $\bigvee_{i<n} \phi_i(a, b') \leq \eps$. This proves the first part.

  The second part now follows from \autoref{prop:ExtremalJEP}.
\end{proof}


\section{The atomless continuous completion}
\label{sec:atomless-multiples}

\begin{defn}\label{defn:CR}
  We denote by $\CR_\cL$ be the collection of all continuous logic axioms of the form
  \begin{equation*}
    \sup_{x_0,\dots,x_{m-1},y} \qinf_z \bigvee_{i<n}\big|\varphi_i(z,y)-\sum_{j<m}\lambda_j\varphi_i(x_j,y)\big| = 0,
  \end{equation*}
  where $\varphi_i(x,y)$ are affine formulas and $\lambda_j\geq 0$ with $\sum_{j<m}\lambda_i=1$.
  If $\CR_\cL$ holds in an $\cL$-structure $M$, we say that $M$ has the \emph{convex realization property}.
\end{defn}

Since $|t| = t \vee (-t)$ and the left-hand side is always non-negative, $\CR_\cL$ is a $q$-convex theory.
If $\cL_A \supseteq \cL$ is an expansion by constants, then $\CR_\cL$ and $\CR_{\cL_A}$ are equivalent.

\begin{lemma}\label{l:non-atomic-has-CR}
  Let $(M_\Omega,e_I)$ be a measurable field of $\cL$-structures over an atomless probability space $(\Omega,\mu)$, and let $N=\int_\Omega^\oplus M_\omega\ud\mu(\omega)$ be its direct integral.
  Then $N$ has the convex realization property.
\end{lemma}
\begin{proof}
  Let $\varphi_i(x,y)$ and $\lambda_j$ be as in \autoref{defn:CR}.
  Let $a_j\in N^x$ and $b\in N^y$ be any tuples.
  Given $\eps>0$, we can find a finite measurable partition $\Omega=\bigsqcup_{k<\ell}A_k$ and real numbers $r_{ijk}$ such that
  $$|\varphi_i(a_j(\omega),b(\omega)) - r_{ijk}| < \eps$$
  for all $i<n$, $j<m$ and $\omega\in A_k$.
  Let
  $$r_i = \sum_j\lambda_j\varphi_i(a_j,b) = \sum_{j,k} \lambda_j\int_{A_k} \varphi_i(a_j(\omega),b(\omega)) \ud \mu(\omega).$$
  Let $\alpha_k=\mu(A_k)$. We observe that, for all $i<n$,
  \begin{align*}
    \Big|r_i - \sum_{j,k}\lambda_j\alpha_k r_{ijk}\Big| & = \Big|\sum_{j,k}\lambda_j\int_{A_k}\big(\varphi_i(a_j(\omega),b(\omega)) - r_{ijk}\big)\ud\mu(\omega)\Big| \\
                                                        & \leq \sum_{j,k}\lambda_j\int_{A_k} \eps \ud\mu  = \eps.
  \end{align*}

  Since $\Omega$ is atomless, for each $k$, we can choose a measurable partition $A_k = \bigsqcup_{j<m}A_{jk}$ such that $\mu(A_{jk}) = \lambda_j\alpha_k$.
  We may then define a tuple $f\in N^x$ by
  $$f(\omega) = a_j(\omega)\text{ if }\omega\in A_{jk}.$$
  Then, for every $i<n$, we have that
  \begin{align*}
    |\varphi_i(f,b) - r_i|& \leq \Big|\varphi_i(f,b) - \sum_{j,k}\lambda_j\alpha_k r_{ijk}\Big| +\eps \\
                          & \leq \sum_{j,k}\int_{A_{jk}}|\varphi_i(a_j(\omega),b(\omega)) - r_{ijk}|\ud\mu(\omega) +\eps \\
                          & \leq \sum_{j,k}\int_{A_{jk}} \eps \ud\mu(\omega) +\eps = 2\eps.
  \end{align*}
  We conclude that $N$ has the convex realization property.
\end{proof}

\begin{lemma}
  \label{l:CR-vs-eca-closed}
  Let $N$ be an $\cL$-structure. The following are equivalent:
  \begin{enumerate}
  \item $N$ has the convex realization property.
  \item $N$ is existentially closed for convex formulas in every affine extension.
  \end{enumerate}
\end{lemma}
\begin{proof}
  \begin{cycprf}
  \item Let $N\preceq^\aff N'$.
    We want to show that $N$ satisfies all existential-convex conditions with parameters from $N$ that hold in $N'$. So suppose that
    $$N'\models \inf_x \bigvee_{i<n}\varphi_i(x,b) \leq 0,$$
    where the $\varphi_i$ are affine formulas and $b\in N^y$.

    Given $\eps>0$, let $a'\in (N')^x$ be such that $\varphi_i(a',b)<\eps$ for each $i<n$.
    By \autoref{l:realized-convex-dense}, applied to $T = D^\aff_N$, the type $\tp^\aff(a'/N)$ can be approximated in the logic topology by convex combinations of types realized in $N$. In particular, there are tuples $a_j\in N^x$ and scalars $\lambda_j\geq 0$ with $\sum_{j<m}\lambda_j=1$, such that
    $$\big|\varphi_i(a',b) - \sum_{j<m}\lambda_j\varphi_i(a_j,b)\big| < \eps$$
    for every $i<n$. By the convex realization property, there is $a\in N^x$ such that
    $$\bigvee_{i<n}\varphi_i(a,b) < \bigvee_{i<n}\sum_{j<m}\lambda_j\varphi_i(a_j,b) + \eps < \bigvee_{i<n}\varphi_i(a',b) + 2\eps < 3\eps.$$
    This is enough to conclude.

  \item[\impfirst]
    By hypothesis, we have $N\preceq^\eca L^1(\Omega,N)$ for any $\Omega$.
    If we choose $\Omega$ atomless, then $L^1(\Omega,N)$ has the convex realization property by \autoref{l:non-atomic-has-CR}, and hence so does $N$, by \autoref{rmk:eca-closed}\autoref{i:rmk:eca-closed-ueca}.
  \end{cycprf}
\end{proof}

It follows from \autoref{l:non-atomic-has-CR} together with \autoref{l:CR-vs-eca-closed} that in an atomless direct integral $N$, the realized types (over $N$) are dense in $\tS^\aff(N)$, strengthening \autoref{rmk:DirectMultipleDenseRealizedTypes}.

\begin{prop}\label{p:qconvex-latt-red-CR}
  Let $Q$ be a q-convex theory. The following are equivalent:
  \begin{enumerate}
  \item $Q$ has delta-convex reduction.
  \item Every model of $Q$ has the convex realization property.
  \end{enumerate}
\end{prop}
\begin{proof}
  \begin{cycprf}
  \item If $M\models Q$ then $L^1(\Omega,M)\models Q$, by \autoref{p:qconv-direct-integral-preservation} and \autoref{rmk:elem-measurable-fields-convex-comb-direct-mult}. Hence by delta-convex reduction, $M\preceq^\cont L^1(\Omega,M)$. By choosing $\Omega$ atomless, we conclude from \autoref{l:non-atomic-has-CR} that $M$ has the convex realization property.
  \item[\impfirst] This follows from \autoref{l:CR-vs-eca-closed} and \autoref{p:lattice-reduction}.
  \end{cycprf}
\end{proof}

\begin{defn}
  \label{df:Tnat}
  Let $T$ be an affine $\cL$-theory.
  We denote the continuous logic theory $T \cup \CR_\cL$ by $T_\nat$, and call it the \emph{atomless completion} of $T$.
\end{defn}

\begin{remark}\label{rmk:Tnat} The following hold:
  \begin{enumerate}[wide]
  \item The theory $T_\nat$ is q-convex.

  \item By \autoref{l:non-atomic-has-CR}, any direct integral of models of $T$ over an atomless probability space is a model of $T_\nat$. In particular, if $T$ is consistent, so is $T_\nat$.

  \item\label{i:Tnat-direct-int-preservation}
    Any direct integral of models of $T_\nat$ is a model of $T_\nat$. Indeed, we may split any such direct integral into its atomless part and its atomic part. By \autoref{l:non-atomic-has-CR}, the atomless part is a model of $T_\nat$. Thus, the original direct integral is a convex combination of models of $T_\nat$, and we may apply \autoref{p:qconv-direct-integral-preservation}, because convex combinations are elementarily measurable.
  \end{enumerate}
\end{remark}

The following summarizes several of the preceding results.

\begin{theorem}\label{th:Tnat}
  Let $T$ be an affine theory. The following hold:
  \begin{enumerate}
  \item\label{i:th:T=Tnataff} $(T_\nat)_\aff \equiv T$.
  \item $T_\nat$ has delta-convex reduction.
  \item\label{i:th:Tnat-homeo} The canonical map $\rho^\aff\colon \tS_x^\cont(T_\nat) \to \tS_x^\aff(T)$ is a homeomorphism and a $\dtp$-isometry.
  \item\label{i:th:Tnat-completeness}  If $T$ is complete as an affine theory, $T_\nat$ is complete as a continuous logic theory.
  \end{enumerate}
\end{theorem}
\begin{proof}
  \begin{enumerate}[wide]
  \item The inclusion $T\subseteq (T_\nat)_\aff$ is clear. Conversely, if $M\models T$ then $M\preceq^\aff L^1(\Omega,M)\models T_\nat$ for any atomless $(\Omega,\mu)$, so $M\models (T_\nat)_\aff$.
  \item Follows from \autoref{p:qconvex-latt-red-CR}
  \item Follows from \autoref{p:qconvex-latt-red-iota-homeo}.
  \item If $T$ is complete and $N,N'\models T_\nat$, then by \autoref{prop:AffineJEP}, they admit a common affine extension $N,N'\preceq^\aff M$.
    Then also $N,N'\preceq^\aff L^1(\Omega,M)$, and the latter has the convex realization property if we choose $\Omega$ to be atomless.
    Since $T_\nat$ has delta-convex reduction, we have $N\equiv^\cont L^1(\Omega,M)\equiv^\cont N'$.\qedhere
  \end{enumerate}
\end{proof}

\begin{cor}\label{c:CR-CL-completeness}
  Let $M$, $N$ be $\cL$-structures such that $M\equiv^\aff N$.
  If they have the convex realization property, then $M\equiv^\cont N$.
\end{cor}

The following may be viewed as a Keisler--Shelah theorem for affine logic, compare with \cite{Bagheri2021}.

\begin{cor}
  Let $M$, $N$ be $\cL$-structures such that $M\equiv^\aff N$.
  Then for any atomless probability space $\Omega$, we have $L^1(\Omega, M) \equiv^\cont L^1(\Omega, N)$.
  Consequently, by the Keisler--Shelah theorem for continuous logic, there exists an ultrafilter $\cU$ such that $L^1(\Omega, M)^\cU \cong L^1(\Omega, N)^\cU$.
\end{cor}

\begin{cor}\label{c:Tnat-common-theory-atomless}
  The theory $T_\nat$ is the common continuous logic theory of all direct multiples (equivalently, direct integrals) of models of $T$ over atomless probability spaces.
\end{cor}
\begin{proof}
  If $T'$ and $T''$ are the common continuous theories of the atomless direct multiples and the atomless direct integrals of models of $T$, respectively, then by \autoref{l:non-atomic-has-CR}, $T' \supseteq T'' \supseteq T_\nat$.
  Conversely, if $M\models T_\nat$, then by delta-convex reduction $M\preceq^\cont L^1(\Omega,M)$ for any atomless $\Omega$, so $M\models T'$. Hence $T'$ and $T''$ are both equivalent to $T_\nat$.
\end{proof}

\begin{cor}
  \label{c:CR-iota-S(A)}
  Let $M$ have the convex realization property, and let $A\subseteq M$.
  Then $D^\cont_A = (D^\aff_A)_\nat$.
  Consequently, $\rho^\aff\colon \tS_x^\cont(A)\to \tS_x^\aff(A)$ is a $\tau$-homeomorphism and a $\dtp$-isometry.
\end{cor}
\begin{proof}
  On the one hand, $D^\cont_A \supseteq D^\aff_A \cup \CR_\cL \equiv D^\aff_A \cup \CR_{\cL_A} \equiv (D^\aff_A)_\nat$.
  On the other hand, $D^\aff_A$ is a complete affine theory, so $(D^\aff_A)_\nat$ is a complete continuous $\cL_A$-theory, and therefore coincides with $D^\cont_A$.
  The conclusion then follows from \autoref{th:Tnat}\autoref{i:th:Tnat-homeo}.
\end{proof}

Our machinery yields the following characterization of affinely saturated structures.

\begin{theorem}
  Let $\kappa$ be an infinite cardinal and let $M$ be an $\cL$-structure.
  The following are equivalent:
  \begin{enumerate}
  \item $M$ is affinely $\kappa$-saturated.
  \item $M$ is $\kappa$-saturated in continuous logic and has the convex realization property.
  \end{enumerate}
\end{theorem}
\begin{proof}
  \begin{cycprf}
  \item Clearly, if $M$ is an affinely $\kappa$-saturated structure, then it is existentially closed for convex formulas in every affine extension. Hence by \autoref{l:CR-vs-eca-closed}, $M$ has the convex realization property. The fact that $M$ is $\kappa$-saturated in continuous logic follows from \autoref{c:CR-iota-S(A)}.
  \item[\impfirst] Also follows readily from \autoref{c:CR-iota-S(A)}.
  \end{cycprf}
\end{proof}

As a consequence, we recover a result of Bagheri \cite[Prop.~4.5]{Bagheri2021}: if $M$ and $N$ are affinely $\aleph_0$-saturated structures and
$M\equiv^\aff N$, then $M\equiv^\cont N$.

We also obtain a sufficient condition for affine homogeneity.

\begin{theorem}
  Let $\kappa$ be an infinite cardinal and let $M$ be an $\cL$-structure.
  If $M$ has the convex realization property, then the following are equivalent:
  \begin{enumerate}
  \item $M$ is affinely $\kappa$-homogeneous.
  \item $M$ is $\kappa$-homogeneous in continuous logic.
  \end{enumerate}
\end{theorem}
\begin{proof}
  Every affinely $\kappa$-homogeneous structure is $\kappa$-homogeneous in continuous logic. Under the convex realization property, every partial affine map is partial elementary, by \autoref{th:Tnat}\autoref{i:th:Tnat-homeo}, so the converse holds.
\end{proof}

We write down the following easy corollaries.

\begin{cor}
  Let $M$ be any structure, let $\kappa$ be an infinite cardinal, and let $\Omega$ be an atomless probability space. Then every $\kappa$-saturated elementary extension of $L^1(\Omega,M)$ is an affinely $\kappa$-saturated affine extension of $M$.

  Similarly, every continuous logic $\kappa$-homogeneous elementary extension of $L^1(\Omega,M)$ is an affinely $\kappa$-homogeneous affine extension of $M$.
\end{cor}

\begin{cor}
  Let $T$ be an affine $\cL$-theory. Then $T_\nat$ is the common theory of all affinely $\kappa$-saturated models of $T$.
\end{cor}

Let us mention that there are theories in which extremal models have the convex realization property. See \autoref{th:Poulsen-complete} and \autoref{sec:PMP}.

\section{Continuity of the direct integral}
\label{sec:continuity-direct-integral}

In \cite{Farah2023p}, Farah and Ghasemi proved that direct integrals of measurable fields of structures (in the separable setting) preserve pointwise elementary equivalence. That is, if $M_\omega \equiv^\cont N_\omega$ almost surely, then $\int^\oplus_\Omega M_\omega\ud\mu \equiv^\cont \int^\oplus_\Omega N_\omega\ud\mu$. Similarly, for elementary extensions.

They obtained this preservation result as a corollary of a finer theorem, analogous to the Feferman--Vaught theorem for reduced products, for direct integrals and continuous formulas. More precisely: the value of a continuous formula $\varphi$ in a direct integral $M_{\Omega,I}=\int^\oplus_\Omega M_\omega\ud\mu$ is determined, up to $\eps$, by a neighborhood of the distribution of a tuple of real random variables $(\llbracket\psi_i\rrbracket)_i$, where the $\psi_i$ are continuous formulas depending only on $\varphi$ and $\eps$ and $\oset{\psi_i}(\omega) \coloneqq \psi_i^{M_\omega}$.

Our first theorem in this section is a generalization of the preservation result of \cite{Farah2023p}, valid under weaker hypotesis and for arbitrary measurable fields.

The following lemma handles the atomic case.

\begin{lemma}\label{l:ConvexCombinationsPreserveEquivalence}
  Let $M = \bigoplus_\omega \lambda_\omega M_\omega$ and $N = \bigoplus_\omega \lambda_\omega M_\omega$ be two convex combinations of $\cL$-structures over the same atomic probability space. Then:
  \begin{enumerate}
  \item If $M_\omega \equiv^\cont N_\omega$ for every $\omega$, then $M\equiv^\cont N$.
  \item If $M_\omega \preceq^\cont N_\omega$ for every $\omega$, then $M\preceq^\cont N$.
  \end{enumerate}
\end{lemma}
\begin{proof}
  We prove the second point, from which the first one follows by joint embedding. We consider the language $\cL'=\bigsqcup_\omega\cL_\omega$ consisting of disjoint copies $\cL_\omega$ of $\cL$ (with distinct corresponding sorts) for each $\omega$, and then consider the $\cL'$-structures $M' = \bigsqcup_\omega M_\omega$ and $N' = \bigsqcup_\omega N_\omega$, where the sorts and symbols of $\cL_\omega$ are interpreted by the structures $M_\omega$ and $N_\omega$, in the obvious manner. We may also consider the language $\cL''=\bigsqcup_\omega(\cL_\omega)_\Mor$.

  The $\cL''$-theory $\bigcup_\omega \Th^\cont\big((M_\omega)_\Mor\big)$ has quantifier elimination. Indeed, it is enough to show how to eliminate one quantifier at a time. Consider an $\cL''$-formula $\inf_x \phi(x, \bar a, \bar b)$ where $\phi$ is quantifier-free, $\bar a$ is a tuple of elements from the sort $\omega_x$ of $x$, and $\bar b$ is a tuple of elements of the other sorts, in some model of the theory. Then $\phi(x, \bar a, \bar b) = f(\psi(\bar a, x), \theta(\bar b))$ for some continuous function $f$,  quantifier-free $(\cL_{\omega_x})_\Mor$-formula $\psi$, and quantifier-free $\cL''$-formula $\theta$. If we set $r = \theta(\bar b)$, then $\inf_x f(\psi(\bar a, x), r)$ is an $(\cL_{\omega_x})_\Mor$-formula, so equivalent to a quantifier-free $(\cL_{\omega_x})_\Mor$-formula. We conclude that the value of $\inf_x \phi(x, \bar a, \bar b)$ is determined by the quantifier-free type of $\bar a \bar b$ in the language $\cL''$, and the claim follows.

  Hence if $M_\omega\preceq^\cont N_\omega$ for each $\omega$, then $M'\preceq^\cont N'$. On the other hand, $N$ is canonically interpretable in $N'$ (in the product sort $\prod_\omega N_\omega$), and the restriction of this interpretation to the elementary submodel $M'$ is an interpretation of $M$. Thus, $M\preceq^\cont N$, as desired.
\end{proof}

Let $(\Omega, \cB, \mu)$ be a probability space. We will say that an atom $A \in \cB$ is \emph{represented by a singleton} if there is $\omega \in A$ such that $\set{\omega} \in \cB$ and $\mu(\set{\omega}) = \mu(A)$. We will say that the space has \emph{singleton atoms} if all of its atoms are represented by singletons. The following remark shows that, with respect to elementary equivalence of direct integrals, it is harmless to replace atoms by singletons if the field is elementarily measurable.

\begin{remark}\label{rmk:degenerate-atoms}
  Let $(M_\Omega,e_I)$ be an elementarily measurable field of $\cL$-structures over a probability space $(\Omega,\mu)$. Let $A$ be an atom of $\Omega$, and let $\mu_A$ be the corresponding conditional probability measure. Consider the restricted field $(M_A,e_I\rest_A)$ and let $M_{A,I}=\int_A^\oplus M_\omega\ud\mu_A$. If $M_\omega\models Q$ for almost every $\omega\in A$, then $M_A\models Q$. Indeed, we may assume that the language is countable, and using that $\fI^\cont(M_\Omega,e_I,\cL)$ is closed and cofinal, we may find a countable subset $I_0\subseteq I$ such that $M_{\omega,I_0}\preceq^\cont M_\omega$ for almost every $\omega\in A$, and at the same time $M_{A,I_0}\preceq^\cont M_{A,I}$. Since $A$ is an atom and $I_0$ and the language are countable, the map $M_{\omega,I_0}\to M_{A,I_0}$, $e_{I_0}(\omega)\mapsto e_{I_0}\rest_A$ is an isomorphism for almost every $\omega\in A$. Thus $M_{A,I}\equiv^\cont M_{A,I_0}\cong M_{\omega,I_0}\models Q$, as desired.
\end{remark}

Given a probability space $\Omega$ with singleton atoms, let us denote
$$\Omega_{\rm{at}}=\set{\omega\in\Omega: \mu(\set{\omega}) > 0},\quad \Omega_{\rm{na}} = \Omega\setminus\Omega_{\rm{at}}.$$
\begin{theorem}
  \label{th:DirectIntegralPreservesEquivalence}
  Let $(\Omega, \mu)$ be a probability space with singleton atoms and let $(M_\Omega,e_I)$ and $(N_\Omega,e'_J)$ be measurable fields of $\cL$-structures.
  \begin{enumerate}
  \item If $M_\omega\equiv^\aff N_\omega$ for almost every $\omega\in\Omega_{\rm{na}}$ and $M_\omega\equiv^\cont N_\omega$ for every $\omega\in\Omega_{\rm{at}}$, then $\int_\Omega^\oplus M_\omega\ud\mu \equiv^\cont \int_\Omega^\oplus N_\omega\ud\mu$.
  \item Suppose that $M_\omega\subseteq N_\omega$ for almost every $\omega\in\Omega$, and that the maps $\omega\mapsto d(e_i(\omega),e'_j(\omega))$ are measurable for all $i\in I$, $j\in J$.
    Then
    $$\int_\Omega^\oplus M_\omega \ud\mu \subseteq \int_\Omega^\oplus N_\omega \ud\mu.$$
  \item If, in addition, $M_\omega \preceq^\aff N_\omega$ for almost every $\omega\in\Omega_{\rm{na}}$ and $M_\omega \preceq^\cont N_\omega$ for every $\omega\in\Omega_{\rm{at}}$, then
    $$\int_\Omega^\oplus M_\omega \ud\mu \preceq^\cont \int_\Omega^\oplus N_\omega \ud\mu.$$
  \end{enumerate}
  In particular, for any $\cL$-structures $M,N$, if $M\equiv^\cont N$ (or just $M\equiv^\aff N$ if $\Omega$ is atomless) then $L^1(\Omega,M)\equiv^\cont L^1(\Omega,N)$. And if $M\preceq^\cont N$ (or just $M\preceq^\aff N$ if $\Omega$ is atomless) then $L^1(\Omega,M) \preceq^\cont L^1(\Omega,N)$.
\end{theorem}
\begin{proof}
  We may assume that $\Omega$ is atomless, as the general case then follows by applying \autoref{l:ConvexCombinationsPreserveEquivalence}. For the first item, it suffices by \autoref{c:CR-CL-completeness} to note that $\int_\Omega^\oplus M_\omega\ud\mu$ and $\int_\Omega^\oplus N_\omega\ud\mu$ have the same affine theory.
  This follows from \autoref{th:Los}.
  The second item holds since every measurable section for $(M_\Omega,e_I)$ is a measurable section for $(N_\Omega,e'_J)$.
  For the third item, it suffices to observe that $\int_\Omega^\oplus M_\omega\ud\mu\preceq^\aff \int_\Omega^\oplus N_\omega\ud\mu$, which holds by \autoref{th:Los} as well.
\end{proof}
In view of \autoref{rmk:degenerate-atoms}, the assumption that $(\Omega, \mu)$ has singleton atoms in \autoref{th:DirectIntegralPreservesEquivalence} entails no loss of generality if we assume in addition that the measurable fields are elementarily measurable.

Next we will prove a continuity result for the direct integral construction with respect to continuous logic (\autoref{c:continuity-direct-integral} below), related to the Feferman--Vaught-like theorem of Farah and Ghasemi, which will be useful in the next section.

In fact, if one considers direct integral over atomless probability spaces only, the result of \cite{Farah2023p} is  equivalent to the quantifier elimination theorem for \emph{atomless randomizations} (see \cite[Thm.~3.32]{BenYaacovOnTheories}). In turn, this fundamental theorem about randomizations has the following analogue concerning Bauerizations with the convex realization property. (In \autoref{sec:Randomizations}, we will discuss how randomizations can be seen, indeed, as Bauerizations.)

\begin{theorem}
  \label{thm:QBauNatQE}
  For any continuous theory $Q$, the theory $(Q_\Bau)_\nat$ has quantifier elimination in continuous logic.
\end{theorem}
\begin{proof}
  The theory $Q_\Bau$, being the affine part of a Morleyization, has affine quantifier elimination.
  On the other hand, by \autoref{th:Tnat}, $(Q_\Bau)_\nat$ has delta-convex reduction.
  The result follows.
\end{proof}

Given a probability space $(\Omega,\mu)$, let us denote by $w(\Omega)$ the measure of its largest atom, or zero if it has no atoms.

\begin{lemma}\label{l:approximate-CR}
  Let $\cL$ be a signature and let $\Sigma\subseteq\CR_\cL$ consist of finitely many axioms of the convex realization property (see \autoref{defn:CR}), say $\Sigma = \set{ \theta_i = 0 : i<n}$. Then for every $\eps>0$ there is $\delta>0$ such that if $(M_\Omega,e_I)$ is any elementarily measurable field of $\cL$-structures with $w(\Omega)<\delta$, the direct integral $M_{\Omega,I}=\int_\Omega^\oplus M_\omega\ud\mu$ satisfies $\theta_i^{M_{\Omega,I}} < \eps$ for every $i<n$.
\end{lemma}
\begin{proof}
  Let $\cL_\PrA=\set{\mu,\cup,\cap,\neg,\bZero,\bOne}$ be the language of probability algebras, and let us consider the trivial probability algebra $A = \set{\bZero,\bOne}$ and its continuous theory $Q_A=\Th^\cont(A)$. Let $Q$ be the empty $\cL$-theory, let $\cL^+=\cL\sqcup\cL_\PrA$ (where $\cL_\PrA$ is interpreted on a separate sort), and let us consider the $\cL^+$-theory $Q^+=Q\cup Q_A$ and its Bauerization, $(Q^+)_\Bau$.

  By \autoref{rmk:Bau}, the models $M^+=(M,M_A)$ of $(Q^+)_\Bau$ can be decomposed as $\int^\oplus_\Omega M^+_\omega \ud \mu$, where the structures $M^+_\omega$ form a measurable field of models of $(Q^+)_\Mor$. In particular, $M^+_\omega = (M_\omega,M_{\omega,A})$ with $M_\omega \models Q_\Mor$ and $M_{\omega,A} = A$, and it follows that $M_A = L^1(\Omega,A) = \MALG(\Omega)$. Conversely, any elementarily measurable field $(M_\omega:\omega\in\Omega)$ of $\cL$-structures induces a measurable field $((M_\omega,A):\omega\in\Omega)$ of $(\cL^+)_\Mor$-structures, whose direct integral is a model of $(Q^+)_\Bau$.

  Now it suffices to note that the quantity $w(\Omega)$ is coded by the continuous $\cL_\PrA$-sentence
  $$\sup_x\inf_y\big|\mu(x\cap y)-\mu(x\cap \neg y)\big|,$$
  and that $(Q^+)_\Bau$ together with the condition $w(\Omega)=0$ imply $\CR_{(\cL^+)_\Mor}$ (and hence also $\CR_\cL$), by \autoref{l:non-atomic-has-CR}. The conclusion then follows by compactness.
\end{proof}

\begin{cor}
  \label{c:continuity-direct-integral}
  For every continuous $\cL$-formula $\varphi(x)$ and every $\eps>0$, there exist continuous $\cL$-formulas $\psi_0(x),\dots,\psi_{n-1}(x)$ and $\delta>0$ such that the following holds:

  Whenever $(M_\Omega,e_I)$ and $(N_{\Omega'},e'_J)$ are elementarily measurable fields of $\cL$-structures and $a \in M_{\Omega, I}^x$, $b \in N_{\Omega', J}^x$ are measurable sections such that
  \begin{itemize}
  \item $w(\Omega)<\delta$, $w(\Omega')<\delta$, and
  \item $\big|\int_\Omega \psi_i(a(\omega))\ud\mu - \int_{\Omega'} \psi_i(b(\omega))\ud\mu' \big| <\delta$ for every $i<n$,
  \end{itemize}
  the direct integrals $M_{\Omega,I}=\int_\Omega^\oplus M_\omega\ud\mu$ and $N_{\Omega',J}=\int_{\Omega'}^\oplus N_{\omega'}\ud\mu'$ satisfy
  $$\big|\varphi(a) - \varphi(b)\big| < \varepsilon.$$
\end{cor}
\begin{proof}
  Let $Q$ be the empty $\cL$-theory. By \autoref{thm:QBauNatQE}, $\varphi(x)$ can be approximated, up to $\eps>0$ and modulo $Q_\Bau\cup\CR_\cL$, by a quantifier-free $\cL_\Mor$-formula $\chi(x)$. By compactness, there are a finite set of axioms $\Sigma = \set{\theta_k:k<m}\sub\CR_\cL$ and $\eps'>0$ such that
  $$Q_\Bau\cup\set{\theta_k<\eps' : k<m}\models \sup_x |\varphi(x) - \chi(x)|\leq 2\eps.$$
  Let $M=\int_\Omega^\oplus M_\omega\ud\mu$ be a direct integral of an elementarily measurable field of $\cL$-structures. We can also see it as an integral of $\cL_\Mor$-structures which are models of $Q_\Mor$. In particular, $M \models Q_\Bau$. As $\chi$ is quantifier free, its value at a tuple $a \in M^x$ is determined by the values at $a$ of some basic $\cL_\Mor$-formulas $P_{\psi_i}$, i.e., by the values $P_{\psi_i}(a) = \int_\Omega P_{\psi_i}(a(\omega))\ud\mu =  \int_\Omega\psi_i(a(\omega))\ud\mu$. Combining this with \autoref{l:approximate-CR} and choosing $\delta$ small enough according to the continuity modulus of $\chi$, we obtain the desired result.
\end{proof}


\section{Characterizations of affineness}\label{sec:charact-affineness}

We use here the results from the previous sections to give a semantic characterization of continuous logic theories that are axiomatizable in affine logic and a syntactic characterization of \emph{complete} continuous logic theories preserved under direct integrals.

\begin{lemma}\label{l:half+half-preservation-implies-integral-pres}
  Let $Q$ be a continuous logic theory such that whenever $M,N\models Q$, we have $\half M\oplus\half N\models Q$. Then $Q$ is preserved under direct integrals of elementarily measurable fields over atomless probability spaces.
  If $Q$ has a separable language, then it is enough to check the condition for separable $M$ and $N$.
\end{lemma}
\begin{proof}
  Let $M_{\Omega,I} = \int_\Omega^\oplus M_\omega \ud\mu$ be the direct integral of an elementarily measurable field of models of $Q$ over an atomless probability space $(\Omega, \mu)$. Let $\varphi$ be a continuous $\cL$-sentence such that $Q\models \varphi=0$. Let $\eps>0$, and let $\psi_0,\dots,\psi_{n-1}$ and $\delta>0$ be as given by \autoref{c:continuity-direct-integral}. We may assume that $0\leq \psi_i \leq 1$ for all $i$. We may find a finite measurable partition $\Omega = \bigsqcup_{k<m} A_k$ and points $\omega_k\in A_k$ such that for every $i<n$, $k<m$ and almost every $\omega\in A_k$,
  \begin{equation}\label{eq:l:half+half}
    \big|\psi_i^{M_\omega} - \psi_i^{N_k}\big| < \delta/2,
  \end{equation}
  where $N_k = M_{\omega_k}$.
  Moreover, if we allow \autoref{eq:l:half+half} to fail for a set $K$ of indices $k<m$ such that $\mu\big(\bigcup_{k\in K}A_k)<\delta/2$, we may choose the sets $A_k$ so that $m = 2^\ell$ and $\mu(A_k)=2^{-\ell}<\delta$ for some $\ell\in\N$ and for every $k<m$.

  It follows that the atomic measurable field $(N_k : k<m)$ over the space $(m,\mu')$ defined by $\mu'(\set{k})=2^{-\ell}$ for each $k<m$ satisfies the hypotheses of \autoref{c:continuity-direct-integral} together with the field $M_\Omega$. On the other hand, by iterating the hypothesis of the statement, we have $\bigoplus_k 2^{-\ell}N_k \models Q$. We deduce that $|\varphi^{M_{\Omega,I}}|<\eps$, which is enough to conclude.

  If $Q$ has a separable language and the condition of the statement is valid for any separable $M$ and $N$, then it is valid for general models of $Q$, by Löwenheim--Skolem and \autoref{l:ConvexCombinationsPreserveEquivalence}.
\end{proof}

\begin{theorem}\label{th:affine-classes}
  Let $Q$ be a continuous logic theory. The following are equivalent:
  \begin{enumerate}
  \item\label{i:Q-aff-ax} $Q$ is affinely axiomatizable.
  \item\label{i:Q-aff-ax-conditions}
    The class of models of $Q$ satisfies the following properties:
    \begin{enumerate}
    \item\label{i:Q-aff-cond-half} if $M,N\models Q$, then $\half M\oplus\half N\models Q$;
    \item\label{i:Q-aff-cond-L1} if $L^1(\Omega,M)\models Q$, with $\Omega$ a standard atomless probability space, then $M\models Q$.
    \end{enumerate}
  \end{enumerate}
  Moreover, if $Q$ has a separable language, it is enough to check the conditions of \autoref{i:Q-aff-ax-conditions} for separable $M$ and $N$.
\end{theorem}
\begin{proof}
  One implication is clear by \autoref{th:Los}, and, moreover, \autoref{i:Q-aff-cond-L1} holds for arbitrary $\Omega$.

  Conversely, suppose that \autoref{i:Q-aff-ax-conditions} holds.
  Take $N\models Q_\aff$, and let us show that $N\models Q$. By \autoref{th:general-models-Qaff}, $N$ embeds affinely into a direct integral $M$ of an elementarily measurable field of models of $Q$.
  Moreover, we may assume that $M$ is an atomless direct integral. Indeed, it suffices to replace every atom $A$ of the measurable field of $M$ (say, with conditional direct integral $M_A=\int_A^\oplus M_\omega\ud\mu_A$, which is a model of $Q$ by \autoref{rmk:degenerate-atoms}) by an atomless measure space $A'$ of the same measure, and redefine $M_\omega = M_A$ for $\omega\in A'$. The resulting direct integral is an affine extension of $M$ by \autoref{th:Los}. Therefore, we may assume that $M$ has the convex realization property and also, by the hypothesis \autoref{i:Q-aff-cond-half} and \autoref{l:half+half-preservation-implies-integral-pres}, that $M$ is a model of $Q$.

  Let $\Omega$ be a standard, atomless probability space.
  Since $M \equiv^\aff L^1(\Omega,N)$ and both have the convex realization property, \autoref{th:Tnat}\autoref{i:th:Tnat-completeness} implies that $L^1(\Omega,N)\equiv^\cont M$.
  Hence, $L^1(\Omega,N)\models Q$.
  By the hypothesis \autoref{i:Q-aff-cond-L1}, we conclude that $N\models Q$.

  For the moreover part, assume that $Q$ has a separable language and assume that \autoref{i:Q-aff-ax-conditions} holds for separable structures.
  Then \autoref{l:half+half-preservation-implies-integral-pres} still applies.
  In addition, we may assume that the structures $N$ and $M$ considered in the proof are separable, and the rest of the proof goes through as well.
\end{proof}

The following result is a generalization of \cite[Thm.~5.3]{Bagheri2021}, which established that a continuous formula $\varphi$ can be approximated by affine formulas
if it is preserved by Bagheri's ultrameans and powermeans.

\begin{theorem}\label{th:charact-affine-formulas}
  Let $T$ be an affine theory, and let $\varphi(x)$ be a continuous logic formula (or definable predicate) in the same language. The following are equivalent.
  \begin{enumerate}
  \item\label{i:char-aff-form} $\varphi$ can be approximated modulo $T$ by affine formulas.
  \item\label{i:char-aff-form-cond-convex} Whenever $M,N\models T$, $K=\half M\oplus\half N$, $a\in M^x$ and $b\in N^x$, we have
    \begin{gather*}
      \varphi^K\big(\half a\oplus\half b\big) = \half\varphi^M(a) + \half\varphi^N(b).
    \end{gather*}
  \end{enumerate}
\end{theorem}
\begin{proof}
  One implication is clear.
  Conversely, suppose that \autoref{i:char-aff-form-cond-convex} is satisfied.
  The formula $\varphi$ induces a continuous map $\varphi\colon \tS^\cont_x(T_\nat) \to\R$.
  By \autoref{th:Tnat}, $\rho^\aff \colon \tS^\cont_x(T_\nat) \to \tS^\aff(T)$ is a homeomorphism, so we may consider $\phi$ as
  a continuous map $\tilde\varphi\colon \tS^\aff_x(T) \to\R$.

  By definition, $\varphi(a) = \tilde\varphi(\tp^\aff(a))$ for $a$ in any model of $T_\nat$.
  Let us show that this holds in all models of $T$.
  Let $M\models T$, $a\in M^x$, and set $r=\varphi^M(a)$.
  Consider the $\cL_a$-theory $Q = T \cup \set{r=\varphi(a)}$.
  The hypothesis \autoref{i:char-aff-form-cond-convex} implies that $Q$ satisfies the condition of \autoref{l:half+half-preservation-implies-integral-pres}, and therefore that $Q$ is preserved by atomless direct multiples.
  In particular, by choosing $\Omega$ atomless and letting $N=L^1(\Omega,M)$, we have $N\models r=\varphi(a)$.
  We deduce that
  $$\varphi^M(a) = \varphi^N(a) = \tilde\varphi(\tp^{\aff,N}(a)) = \tilde\varphi(\tp^{\aff,M}(a)),$$
  as desired.

  It is left to show that $\tilde\varphi$ is affine.
  Indeed, take $p,q\in\tS^\aff_x(T)$ and respective realizations $a,b$ in models $M,N\models T$.
  Let $K = \half M\oplus\half N$.
  Then
  \begin{align*}
    \tilde\varphi\big(\half p + \half q\big)
    & = \varphi^K\big(\half a\oplus\half b\big) \\
    & = \half\varphi^{M}(a) + \half\varphi^{N}(b) \\
    & = \half\tilde\varphi(p) + \half\tilde\varphi(q).
  \end{align*}
  Since $\tilde\varphi$ is continuous, this is enough to deduce that it is affine.
  It follows that $\tilde\varphi$ can be approximated by functions on $\tS_x^\aff(T)$ induced by affine formulas, which concludes the proof.
\end{proof}

\begin{cor}\label{th:charact-affinely-def-functions}
  Let $T$ be an affine theory, and let $f(x)$ be a uniformly definable function on models of $T$. The following are equivalent.
  \begin{enumerate}
  \item\label{i:char-aff-form} $f$ is affinely uniformly definable.
  \item\label{i:char-aff-form-cond-convex} Whenever $M,N\models T$, $K=\half M\oplus\half N$, $a\in M^x$ and $b\in N^x$, we have
    \begin{gather*}
      f^K\big(\half a\oplus\half b\big) = \half f^M(a) \oplus \half f^N(b).
    \end{gather*}
  \end{enumerate}
\end{cor}
\begin{proof}
  Apply the previous theorem to the definable predicate $d(f(x),y)$.
\end{proof}

\begin{theorem}
  \label{th:direct-integral-preservation}
  Let $Q$ be a complete continuous logic theory. The following are equivalent.
  \begin{enumerate}
  \item $Q$ is preserved under direct integrals of elementarily measurable fields.
  \item $Q$ is preserved under convex combinations of the form $\half M\oplus \half N$.
  \item $Q$ is preserved under direct multiples over standard, atomless probability spaces.
  \item Every model of $Q$ has the convex realization property.
  \item $Q$ is equivalent to $(Q_\aff)_\nat$.
  \item $Q$ is a q-convex theory.
  \end{enumerate}
\end{theorem}
\begin{proof}
  \begin{cycprf}
  \item Clear.
  \item By \autoref{l:half+half-preservation-implies-integral-pres}.
  \item If $M\models Q$ and $\Omega$ is atomless, then by hypothesis $L^1(\Omega,M)\models Q$, and the conclusion follows from \autoref{l:non-atomic-has-CR} and the completeness of $Q$.
  \item The hypothesis says that $Q\models (Q_\aff)_\nat$ and $(Q_\aff)_\nat$ is complete by the completeness of $Q$ and \autoref{th:Tnat}\autoref{i:th:Tnat-completeness}. Therefore the two theories are equivalent.
  \item Clear.
  \item[\impfirst] By \autoref{p:qconv-direct-integral-preservation}.
  \end{cycprf}
\end{proof}

\begin{question}
  In general, is a continuous theory preserved by direct integrals of elementarily measurable fields if and only if it is q-convex?
\end{question}


\section{Poulsen theories and a dichotomy for simplicial theories}
\label{sec:Poulsen}

Recall that the \emph{Poulsen simplex} is the unique, up to affine homeomorphism, non-trivial metrizable simplex with the property that the extreme points are dense.

\begin{defn}
  Let $T$ be an affine theory.
  We will say that $T$ is a \emph{Poulsen theory} if $T$ is simplicial and $\cE_x(T)$ is dense in $\tS^\aff_x(T)$ for every finite tuple $x$.
\end{defn}

If $T$ is a Poulsen theory, then $\cE_x(T)$ is dense in $\tS^\aff_x(T)$ for arbitrary $x$, by \autoref{l:inv-limit-dense-extreme}.

Let $T$ be an affine theory in a separable, single-sorted language.
Recall (\autoref{dfn:NonDegenerateTheory}) that it is non-degenerate if all its models have at least two elements.
Then, assuming that $T$ is non-degenerate, it is Poulsen if and only if $\tS^\aff_n(T)$ is a Poulsen simplex for every $n \geq 2$, if and only if $\tS^\aff_{\aleph_0}(T)$ is a Poulsen simplex, whence the name.

On the other hand, a complete theory is degenerate if and only if it has a unique model, consisting of a single point, if and only if all its type spaces are singletons.
Such a theory is both Poulsen and Bauer.

Next we give a characterization of complete Poulsen theories.
We start with two lemmas.
\begin{lemma}\label{l:extremal-model-with-CRP}
  Let $T$ be a complete affine theory. Suppose that some extremal model of $T$ has the convex realization property. Then for every $x$, the extreme types are dense in $\tS^\aff_x(T)$.
\end{lemma}
\begin{proof}
  Let $M$ be an extremal model of $T$ with the convex realization property. Let $U = \oset{\bigvee_{i < n} \phi_i < 0}$ be a non-empty basic open subset of $\tS^\aff_x(T)$ with $\phi_i$ affine. Every $p \in U$ is realized in some affine extension of $M$, so, by \autoref{l:CR-vs-eca-closed}, there is $a \in M$ with $\tp^\aff(a) \in U$. This implies that the types realized in $M$ are dense in $\tS^\aff_x(T)$, and they are extreme by hypothesis.
\end{proof}

\begin{lemma}
  \label{l:saturated-atomless}
  Let $T$ be a complete affine theory such that for some tuple $x$, $\tS^\aff_x(T)$ is a simplex and $\cE_x(T)$ is infinite.
  Let $M = \int_\Omega^{\oplus} M_\omega \ud \mu(\omega)$ be a direct integral of extremal models of $T$.
  If $M$ is realizes every type in $\tS^\aff_x(T)$, then $(\Omega, \mu)$ is atomless.
\end{lemma}
\begin{proof}
  Suppose, to the contrary, that $\Omega$ has an atom of measure $\lambda > 0$.
  Then $M = \lambda M_0 \oplus (1-\lambda)M_1$ with $M_0$ extremal. Let $k \in \N$ be such that $1/k < \lambda$, and choose $p_0, p_1, \ldots, p_{k-1}$ distinct elements of $\cE_x(T)$ .
  By hypothesis, $(1/k) \sum_{i < k} p_i$ is realized in $M$, say, by $\lambda a_0 \oplus (1-\lambda)a_1$ with $a_0 \in M_0^x$, $a_1 \in M_1^x$.
  This implies that $(1/k)\sum_{i < k} p_i = \lambda \tp^\aff(a_0) + (1-\lambda) \tp^\aff(a_1)$, which together with the fact that $\tp^\aff(a_0)$ is extreme, contradicts the assumption that $\tS^\aff_x(T)$ is a simplex.
\end{proof}

Since $M$ and surjectivity of $\rho^\aff$, the type

\begin{theorem}\label{th:Poulsen-complete}
  Let $T$ be a complete, simplicial theory. The following are equivalent:
  \begin{enumerate}
  \item $T$ is a Poulsen theory.
  \item Every model of $T$ has the convex realization property (i.e., $T\equiv T_\nat$).
  \item The maps $\rho^\aff\colon \tS_x^\cont(T)\to \tS_x^\aff(T)$ are homeomorphisms.
  \item $T$ has delta-convex reduction as a continuous logic theory.
  \item $T$ is complete as a continuous logic theory.
  \item Some extremal model of $T$ has the convex realization property.
  \end{enumerate}
\end{theorem}
\begin{proof}
  \begin{cycprf}
  \item If the complete theory $T$ is degenerate, then clearly $T\equiv T_\nat$.

    Otherwise, for some finite tuple $x$ the space $\tS^\aff_x(T)$ is an infinite simplex and $\cE_x(T)$ is dense (so infinite as well). Let $M\models T$ be a model, and let $M\preceq^\cont M'$ be an $\aleph_0$-saturated elementary extension. As $T$ is complete, the image of the map $\rho^\aff \colon \tS^\cont_x(\Th^\cont(M)) \to \tS^\aff_x(T)$ contains $\cE_x(T)$, by \autoref{l:image-rho-aff}. As $\cE_x(T)$ is dense in $\tS^\aff_x(T)$, $\rho^\aff$ must be surjective. Hence $M'$ realizes every type in $\tS^\aff_x(T)$.
    By \autoref{l:saturated-atomless}, the probability space of the extremal decomposition of $M'$ is atomless, so, by \autoref{l:non-atomic-has-CR}, $M'$, and therefore $M$, has the convex realization property.
  \item This follows from \autoref{th:Tnat}.
  \item By \autoref{rmk:latt-red-iota-injective}.

  \item If $T$ has delta-convex reduction, it is in particular model complete by affine.
    As by hypothesis, $T$ is also complete in affine logic, it follows from the affine joint embedding property that $T$ must be complete in continuous logic.

  \item If $M$ is any model of $T$ and $\Omega$ is an atomless probability space, then $M\equiv^\cont L^1(\Omega,M)$, because the latter is also a model of $T$. It follows from \autoref{l:non-atomic-has-CR} that $M$ has the convex realization property.

  \item[\impfirst] By \autoref{l:extremal-model-with-CRP}.
  \end{cycprf}
\end{proof}

\begin{cor}
  Let $T$ be a complete Poulsen theory, let $M$ be an extremal model of $T$, and let $A\subseteq M$ be any subset. Then the affine diagram $D^\aff_A$ is a Poulsen theory.
\end{cor}
\begin{proof}
  Since $T$ is Poulsen, $M$ has the convex realization property. By \autoref{p:face-simplicial}, the diagram $D^\aff_A$ is simplicial. Now, as an $\cL_A$-structure, $M$ is an extremal model of $D^\aff_A$ with the convex realization property. Hence $D^\aff_A$ is a Poulsen theory.
\end{proof}

\begin{ntn}
  If $T$ be a simplicial theory and $M \models T$, we will denote by $w(M)$ the maximum of the measures of all atoms in the probability space of the non-degenerate extremal decomposition of $M$, or zero if there are no atoms. This is well-defined by \autoref{th:uniqueness-decomposition}.
\end{ntn}

We have the following basic fact.

\begin{lemma}
  \label{l:weight-decreases}
  Suppose that $T$ is a simplicial theory, $M \models T$, and $M \preceq^\aff N$. Then $w(M) \geq w(N)$.
\end{lemma}
\begin{proof}
  Follows from \autoref{th:uniqueness-decomposition}.
\end{proof}

\begin{theorem}[Simplicial dichotomy]
  \label{th:dichotomy-simplicial-theories}
  Every non-degenerate, complete, simplicial theory is either Bauer or Poulsen.
\end{theorem}
\begin{proof}
  It is clear that if a theory is both Bauer and Poulsen, then it is degenerate. Let now $T$ be a non-degenerate, complete, simplicial theory and suppose that $T$ is not Bauer. This implies that every extremal model $M$ of $T$ has an elementary extension which is not extremal. Indeed, let $n$ be such that $\cE_n(T)$ is not closed in $\tS^\aff_n(T)$ and let $p \in \cl{\cE_n(T)} \sminus \cE_n(T)$. By completeness of $T$, the image of $\rho^\aff \colon \tS^\cont_n(\Th^\cont(M)) \to \tS^\aff_n(T)$ contains $\cl{\cE_n(T)}$, so $M$ has an elementary extension which realizes $p$.

  \begin{claim*}
    Let $M \models T$. If $w(M) > 0$, then there exists an elementary extension $M \preceq^\cont M'$ with $w(M') < w(M)$.
  \end{claim*}
  \begin{proof}
    Let $w(M) = \lambda$ and write $M = \lambda N_1 \oplus \cdots \oplus \lambda N_k \oplus (1-k \lambda) N_0$, where $(1- k \lambda)w(N_0) < \lambda$ and $N_1, \ldots, N_k$ are extremal. Apply the above observation to each $N_i$, $1 \leq i \leq k$, to find a non-extremal elementary extension $N_i \preceq^\cont N_i'$. Let $M' = \lambda N_1' \oplus \cdots \oplus \lambda N_k' \oplus (1-k \lambda) N_0$. By \autoref{l:ConvexCombinationsPreserveEquivalence}, $M \preceq^\cont M'$, and as for each $i$, we have that $w(N_i') < w(N_i) = 1$, this implies that $w(M') < \lambda$.
  \end{proof}

  Let $M_0$ be any extremal model of $T$. We will show that $M_0$ has an elementary extension $M'$ with $w(M') = 0$. If not, using the Claim and \autoref{l:weight-decreases}, we can build a continuous elementary chain $(M_\alpha)_{\alpha < \omega_1}$, starting from $M_0$, with $w(M_\alpha)$ a strictly decreasing sequence of real numbers of length $\omega_1$, which is impossible.  Now we can apply \autoref{l:non-atomic-has-CR} to deduce that $M'$, and therefore $M_0$, has the convex realization property. \autoref{th:Poulsen-complete} allows us to conclude that $T$ is Poulsen.
\end{proof}

\autoref{th:dichotomy-simplicial-theories} does not hold for incomplete theories (see \autoref{ex:PMP-Z}).
However, we can still draw the following conclusions.

\begin{cor}\label{cor:general-simplicial-theories}
  Let $T$ be a simplicial affine theory. Then:
  \begin{enumerate}
  \item\label{i:cor:general-simplicial-completions} Every extreme completion of $T$ is a Bauer theory or a Poulsen theory.
  \item\label{i:cor:general-simplicial-extensions} Every affine extension of extremal models of $T$ is elementary.
  \item\label{i:cor:general-simplicial-models} Every model of $T$ can be written as a convex combination $\bigoplus_{i\in\N}\lambda_i M_i$ of models of $T$, where $M_0$ has the convex realization property and the $M_i$ for $i\geq 1$ are extremal models of $T$ with a Bauer affine theory.
  \end{enumerate}
\end{cor}
\begin{proof}
  For \autoref{i:cor:general-simplicial-completions}, every extreme completion of $T$ is simplicial by \autoref{p:face-simplicial}, and the dichotomy theorem applies. For \autoref{i:cor:general-simplicial-extensions}, given an affine extension of models of $T$, we may apply \autoref{th:closed-extreme-types-Text} or \autoref{th:Poulsen-complete}, depending on whether the theory of the models is Bauer or Poulsen.

  For \autoref{i:cor:general-simplicial-models}, we may decompose every model of $T$ as a direct integral of extremal models. Then the model $M_0$ of the statement is obtained as the convex combination of the atomless part of the decomposition plus the atoms of the decomposition whose associated model has a Poulsen theory; and the $M_i$ for $i\geq 1$ are given by the atoms of the decomposition whose associated model has a Bauer theory.
\end{proof}

\begin{remark}\label{rmk:extreme-completion-of-Bauer}
  If $T$ is a Bauer theory, then every extreme completion $T'$ of $T$ is a Bauer theory. Indeed, $T'$ is simplicial and the extreme type spaces $\cE_x(T')=\cE_x(T)\cap\tS^\aff_x(T')$ are closed.
\end{remark}

\begin{cor}\label{c:simplicial-ext-elementarily-measurable}
  Let $T$ be a simplicial theory. Suppose that $T$ has a separable language, or, alternatively, that $T$ is complete. Then every measurable field of extremal models of $T$ is elementarily measurable.
\end{cor}
\begin{proof}
If $T$ has a separable language, this follows from \autoref{cor:general-simplicial-theories}\autoref{i:cor:general-simplicial-extensions}. If $T$ is complete, we have two cases. If $T$ is a Bauer theory, this was observed in \autoref{c:bauer-elementarily-measurable-fields} (even for incomplete theories). If $T$ is a complete Poulsen $\cL$-theory, then by \autoref{th:Poulsen-complete}, we have $T\equiv T_\nat$, and hence also $T\rest_{\cL_0}\equiv (T\rest_{\cL_0})_\nat$ for every sublanguage $\cL_0\subseteq\cL$, because $\CR_{\cL_0}\subseteq (\CR_\cL)\rest_{\cL_0}$. Thus, every reduct $T\rest_{\cL_0}$ has delta-convex reduction, and the result follows.
\end{proof}

Using extremal joint embedding (\autoref{prop:ExtremalJEP}) and \autoref{p:face-simplicial}\autoref{i:p:face-simplicial-constants}, we may restate \autoref{cor:general-simplicial-theories}\autoref{i:cor:general-simplicial-extensions} as follows.

\begin{cor}
  Let $M$ and $N$ be extremal models of a simplicial theory. If $M\equiv^\aff N$, then $M\equiv^\cont N$.
  More generally, if $a\in M^x$ and $b\in N^x$ satisfy $\tp^\aff(a)=\tp^\aff(b)$, then $\tp^\cont(a)=\tp^\cont(b)$.
\end{cor}

In fact, as we show below, a stronger statement holds: the affine theory a model $M$ of a simplicial theory determines the \emph{distribution of the continuous theories} of the extremal models $M_\omega$ appearing in its integral decomposition $M=\int^\oplus_\Omega M_\omega \ud\mu$.
In the particular case of tracial von Neumann algebras (and in a weaker form, considering the continuous theory of $M$, and only in the separable setting), this statement was conjectured in an earlier version of \cite{Farah2023p}, and a positive answer was given by Gao and Jekel in \cite{Gao2024pp}.
We thank Ilijas Farah for drawing our attention to this question and pointing out that our methods should allow to give a general answer.

Given an elementarily measurable field of $\cL$-structures $M_\Omega = (M_\omega:\omega\in\Omega)$, we may consider the map
\begin{gather*}
  \Th^\cont_{M_\Omega}\colon \Omega\to \tS_0^\cont(\cL),\ \omega\mapsto \Th^\cont(M_\omega).
\end{gather*}
This map is measurable, due to the elementarity of the field. It thus induces the pushforward measure $\big(\Th^\cont_{M_\Omega}\big)_*\mu \in \cM\big(\tS_0^\cont(\cL)\big)$ on the space of complete continuous $\cL$-theories.
In other words, this is the distribution of the continuous logic theories of the structures of the field.

Now let $T$ be a simplicial theory, and suppose that $T$ has a separable language or that it is complete.
If $M$ is a model of $T$, then by the extremal decomposition theorem and \autoref{c:simplicial-ext-elementarily-measurable}, there is an elementarily measurable, non-degenerate field $M_\Omega$ of extremal models of $T$ such that $M\cong\int^\oplus_\Omega M_\omega\ud\mu$.
We may then consider the distribution
\begin{gather*}
  \mu^\cont_M = \big(\Th^\cont_{M_\Omega}\big)_*\mu \in \cM\big(\tS_0^\cont(T_\ext)\big) \subseteq\cM\big(\tS_0^\cont(\cL)\big).
\end{gather*}
It follows from \autoref{th:uniqueness-decomposition} and the last part of \autoref{rmk:MeasurableFieldEmbedding-BijectiveMap} that $\mu_M$ is well-defined, in that it only depends on $M$.

Note that the canonical affine map $\iota_0^*\colon \cM\big(\tS^\cont_0(T_\ext)\big) \to \tS^\aff_0(T)$ sends $\mu^\cont_M$ to $\Th^\aff(M)$. When $T$ is a Bauer theory, the map $\iota_0^*$ is injective, so the affine theory of $M$ determines the distribution $\mu^\cont_M$.
(This is the case for tracial von Neumann algebras, as follows from our results in \autoref{sec:tracial-von-neumann}.)
In the opposite situation, when $T$ is Poulsen and complete, $\tS^\cont_0(T_\ext)$ is a singleton (by \autoref{th:Poulsen-complete}), so $\mu^\cont_M$ is trivially determined by $T=\Th^\aff(M)$.
More generally, we have the following.

\begin{theorem}\label{th:distribution-extremal-cont-theories}
  Let $T$ be a simplicial theory in a separable language, or complete.
  If $M,N$ are models of $T$ with $M\equiv^\aff N$, then $\mu^\cont_M = \mu^\cont_N$.
\end{theorem}
\begin{proof}
  By affine joint embedding (\autoref{prop:AffineJEP}), it is enough to consider the case in which $M\preceq^\aff N$.
  Now by \autoref{th:extremal-decomposition-for-pairs} and \autoref{c:simplicial-ext-elementarily-measurable}, there are non-degenerate, elementarily measurable extremal decompositions $(M_\omega : \omega\in\Omega)$ of $M$ and $(N_\xi : \xi\in\Xi)$ of $N$, together with a measure-preserving map $\pi\colon\Xi\to \Omega$ such that $M_{\pi(\xi)}\preceq^\aff N_\xi$ for every $\xi\in\Xi$. By \autoref{cor:general-simplicial-theories}\autoref{i:cor:general-simplicial-extensions}, $M_{\pi(\xi)}\preceq^\cont N_\xi$ for every $\xi\in\Xi$.
  Hence $\Th^\cont_{M_\Omega}\circ\pi = \Th^\cont_{N_\Xi}$, and we deduce that $\mu^\cont_N = \big(\Th^\cont_{M_\Omega}\big)_*(\pi_*\nu) = \mu^\cont_M$.
\end{proof}

We note that the previous result can also be stated for types: the affine type of a tuple determines the distribution of the continuous logic types of the corresponding tuples in its extremal decomposition.


\section{Randomizations and affine logic}
\label{sec:Randomizations}

In the present section, we show that Bauer theories and the Bauerization of a continuous logic theory, introduced in \autoref{sec:Bauer-theories}, provide a new approach to \emph{randomizations} in the sense of Keisler.
Roughly speaking, Keisler randomizations are structures whose elements represent random variables taking values in $\cL$-structures.
Randomizations were first introduced by Keisler \cite{KeislerRandomizing} in classical logic, in a language with two new sorts. One of them, call it $F$, is a real closed field, and the other, call it $A$, represents a Boolean algebra of events, with probabilities in $F$.
For each $\cL$-formula $\varphi$, a random structure $M$ admits a definable function $\llbracket \varphi  \rrbracket$ into $A$, where $\llbracket \varphi(a) \rrbracket$ represents the event on which $\varphi(a)$ holds.
With the subsequent development of continuous logic, it became possible to treat probabilities directly as values of formulas and do away with $F$.
This was done for $\cL$ a language in classical logic in \cite{BenYaacovKeisler} and was later generalized to the continuous case in \cite{BenYaacovOnTheories}. In both cases, the randomized theory is in continuous logic.

It is hardly surprising that randomizations and Bauerizations are related since \autoref{rmk:Bau} tells us that models of $Q_\Bau$ are direct integrals of models of $Q$, i.e., what one would intuitively imagine as ``a space of random variables which take values in models of $Q$''.
Indeed, we show that if $Q$ is a theory in continuous logic, then $Q^{\Ra}$ (the theory of random variables in models of $Q$ over probability spaces that may have atoms) is essentially the same theory as the Bauerization $Q_\Bau$. Similarly, $Q^\Rand$ (the atomless randomization of $Q$) can be identified with $(Q_\Bau)_\nat$.

Both papers \cite{BenYaacovKeisler,BenYaacovOnTheories} use an \emph{auxiliary sort} $A$ to represent the underlying probability space $\Omega$.
While it is not indispensable, the auxiliary sort makes it considerably easier to give an explicit axiomatization of the randomized theory. Even though one cannot recover $\Omega$ entirely from its probability algebra, the latter retains all the pertinent information (for our purposes) about the former, and the class of all probability algebras is easy to axiomatise in continuous logic.
When randomizing a theory in classical, two-valued logic, as in \cite{BenYaacovKeisler}, the ``random truth value'' of a formula $\varphi$ is indeed an event, denoted by $\llbracket \varphi \rrbracket$, and an auxiliary sort consisting of a probability algebra suffices.
When randomizing a theory in continuous logic, as in \cite{BenYaacovOnTheories}, a formula $\varphi$ takes real values, so $\llbracket \varphi \rrbracket$ should be a real-valued random variable on $\Omega$.
It is shown that the class of spaces of $[0,1]$-valued random variables equipped with the $\|{\cdot}\|_1$-metric, denoted by $L^1\bigl(\Omega, [0,1]\bigr)$, can be axiomatized in continuous logic.
There is some (unimportant) freedom in the choice of language -- in \cite{BenYaacovOnTheories}, this is done in a minimalistic language consisting of function symbols corresponding to a certain choice of connectives for $[0,1]$-valued continuous formulas, but then for any continuous function $\theta\colon [0,1]^n \rightarrow [0,1]$, the composition $\theta \circ (\cdots)$ is a definable function $L^1\bigl(\Omega, [0,1]\bigr)^n \rightarrow L^1\bigl(\Omega, [0,1]\bigr)$.
The expectation $\E$ can be an atomic predicate, or recovered as the $L^1$-distance to $0$.
The theory of $[0,1]$-valued random variables (in any of the possible languages) is bi-interpretable with that of probability algebras (see \cite[\textsection2.2]{BenYaacovOnTheories}).

\begin{remark}
  \label{rmk:RandVarCompleteLattice}
  Since $L^1\bigl( \Omega, [0,1] \bigr)$ is closed under the continuous operations $\wedge$ and $\vee$, it is a lattice.
  Moreover, it is a complete lattice.
  Indeed, for $A \sub L^1\bigl( \Omega, [0,1] \bigr)$, in order to calculate $\sup A$, we may assume that $0 \in A$ and that $A$ is closed under $\vee$.
  Let $e = \sup \, \bigl\{ \E[b] : b \in A \bigr\} \in [0,1]$ and choose an increasing sequence $(a_n)_n$ of elements of $A$ such that $\E[a_n] \rightarrow e$.
  Then $(a_n)$ is Cauchy (in $L^1$) and its limit is $\sup A$. The same argument works for $\inf$.
\end{remark}

Let us give a few formal reminders from \cite{BenYaacovOnTheories}.
We are given a language $\cL$ for continuous logic.
\begin{itemize}
\item The \emph{randomization language} $\cL^\Rand$ has the same sorts as $\cL$, referred to as \emph{main sorts}, along with a new \emph{auxiliary sort} $A$, which carries the language of $[0,1]$-valued random variables.
  If $S$ is a sort of $\cL$, we shall denote its counterpart in $\cL^\Rand$ by $S^\Rand$.
\item The main sorts carry the same function symbols as in $\cL$, with the same moduli of continuity.
\item We may assume, for simplicity of the presentation, that all predicate symbols in $\cL$ are $[0,1]$-valued.
  For each such predicate symbol $P(x)$ of $\cL$, we put in $\cL^\Rand$ a function symbol $\llbracket P(x) \rrbracket$ going from the domain sorts of $P$ into the auxiliary sort, with the same continuity modulus as $P$.
\item We equip each main sort $S^\Rand$ with a distance symbol $d_{S^\Rand}$, with the same bound as $d_S$ (having replaced $d_S$ itself with a binary function symbol $\llbracket d_S \rrbracket$).
\end{itemize}

The base randomization theory $T_0^{\Ra}$ (here, the zero tells us that this is the base theory, $\Rand$ stands for randomization, and the ``a'' tell us that we allow atoms in $\Omega$) consists of the following groups of axioms:
\begin{itemize}
\item $RV$: The auxiliary sort is a space of $[0,1]$-valued random variables of a probability space.
\item $R1$: \emph{Locality} axioms, meant to ensure that for each function symbol $f$ or predicate symbol $P$ of $\cL$, the corresponding symbols $f$ and $\llbracket P \rrbracket$ respect the same continuity moduli as $f$ and $P$, respectively, locally, i.e., at each $\omega \in \Omega$.
\item $R2$: A \emph{distance} axiom asserting that $d_{S^\Rand}(x,y) = E \llbracket d_S(x,y)\rrbracket$ for each main sort $S$.
\item $R3$: A \emph{gluing} axiom, asserting that for each event $B \in \MALG(\Omega)$ and $a,b$ in some sort $S^\Rand$, there exists $c \in S^\Rand$ that is equal to $a$ on $B$ and to $b$ on $\Omega \sminus B$.
\end{itemize}

For the precise formulation of the axioms, we refer the reader to \cite[\textsection3.2]{BenYaacovOnTheories}.
We also leave it to the reader to observe that these axioms are, in fact, affine.
It is shown in \cite[Thm.~3.11]{BenYaacovOnTheories} that models of $T_0^{\Ra}$ are, in a sense, spaces of random variables taking values in a random family of $\cL$-structures $(M_\omega : \omega \in \Omega)$.
Here we prove a better result, namely, that every such model is a direct integral in the sense of \autoref{sec:direct-integrals}.

The crucial result we need to quote from \cite[\textsection3.3]{BenYaacovOnTheories} is that one can construct for every $[0,1]$-valued continuous $\cL$-formula $\varphi$ on $S_1 \times \cdots \times S_n$, a definable function $\llbracket \varphi \rrbracket\colon S_1^\Rand \times \cdots \times S_n^\Rand \rightarrow A$.
It is uniquely determined by the following properties:
\begin{itemize}
\item For every predicate symbol $P$ of $\cL$, this is the function symbol $\llbracket P \rrbracket$.
\item Composition with function symbols, and more generally with terms: if $\varphi(x)$ is an $\cL$-formula, and for each $x_i$, $\tau_i(y)$ is a term of $\cL$, then $\psi(y) = \varphi(\tau)$ is another formula, the tuple $\tau$ defines a function in $\cL^\Rand$ as well, and $\llbracket \varphi(\tau)\rrbracket = \llbracket \varphi \rrbracket \circ \tau$.
\item Connectives: for any continuous function $\theta\colon [0,1]^n \rightarrow [0,1]$ and $\cL$-formulas $\varphi_i$, we have
  \begin{gather*}
    \bigl\llbracket \theta \circ (\varphi_1,\ldots, \varphi_n) \bigr\rrbracket = \theta \circ \bigl( \llbracket \varphi_1\rrbracket, \ldots, \llbracket \varphi_n \rrbracket\bigr).
  \end{gather*}
  On the left-hand side, $\theta\circ (\varphi_1,\ldots,\varphi_n)$ is another $\cL$-formula.
  On the right-hand side, the composition $\theta \circ$ is the one definable on spaces of random variables.
\item Quantifiers: in every model $M \models T_0^{\Ra}$ we have, using the completeness of the lattice $L^1$ (\autoref{rmk:RandVarCompleteLattice}) on the right-hand side,
  \begin{gather*}
    \bigl\llbracket \sup_y \varphi(a,y) \bigr\rrbracket = \sup_b \, \llbracket \varphi(a,b) \rrbracket.
  \end{gather*}
\end{itemize}

Moreover, the function $\llbracket \varphi \rrbracket$ is affinely definable.
Indeed, for the first three cases this is immediate.
As for quantifiers, the proof of \cite[Lemma~3.13]{BenYaacovOnTheories} tells us that $X \in A$ equals $\llbracket \inf_y \varphi(a,y) \rrbracket$ if and only if
\begin{gather*}
  \qsup_y \E \bigl( X \dotminus \llbracket \varphi(a,y) \rrbracket \bigr)
  +
  \qinf_y \E \bigl( \llbracket \varphi(a,y) \rrbracket \dotminus X \bigr) = 0,
\end{gather*}
and this is an affine condition (once $\llbracket \varphi \rrbracket$ is affinely definable).
In fact, one can go further and check that for arbitrary $X$, the left-hand side is exactly equal to $d\bigl( X, \llbracket \inf_y \varphi(a,y) \rrbracket\bigr)$.

If $Q$ is any continuous $\cL$-theory, then its \emph{randomization} (with atoms) $Q^{\Ra}$ is defined as $T_0^{\Ra}$ together with the axioms asserting that ``$Q$ holds almost surely'', namely, that $\llbracket \sigma \rrbracket$ vanishes for every sentence $\sigma$ for which $Q \models \sigma = 0$.
This is again an affine theory whose models can be viewed as spaces of random variables in a random family of models of $Q$.

We also define $T_0^\Rand$ as $T_0^\Ra$ together with an axiom asserting that $\Omega$ is atomless (this axiom is \emph{not} affine). We then define the \emph{atomless randomization} of $Q$ as $Q^\Rand = Q^{\Ra} \cup T_0^\Rand$.

The gluing axiom $R3$, together with the locality axiom $R1$, ensure that for every event $B \subseteq \Omega$ and tuples $b,b' \in M^y$ (in the main sorts), there exists a tuple $c$ such that for every formula $\varphi(x,y)$ and parameters $a \in M^x$:
\begin{gather}
  \label{eq:BBracketCases}
  \llbracket \varphi(a,c)\rrbracket = \chi_B \llbracket \varphi(a,b) \rrbracket + \chi_{\Omega \setminus B} \llbracket \varphi(a,b') \rrbracket.
\end{gather}
Indeed, one constructs $c$ as per \cite[Lemma~3.10]{BenYaacovOnTheories}, and then \autoref{eq:BBracketCases} follows from \cite[Thm.~3.14]{BenYaacovOnTheories}.
In particular, this means that the family $\bigl\{ \llbracket \varphi(a,b)\rrbracket : b \in M^y \bigr\}$ is closed under maximum, and therefore, that
\begin{equation}
  \label{eq:BBracketQuantifier}
  \begin{split}
    \sup \, \bigl\{ E \llbracket \varphi(a,b)\rrbracket : b \in M^y \bigr\}
    &= E \Bigl( \sup \bigl\{ \llbracket \varphi(a,b)\rrbracket : b \in M^y \bigr\}\Bigr) \\
    &= E \bigl\llbracket \sup_y \varphi(a,y) \bigr\rrbracket.
  \end{split}
\end{equation}

\begin{defn}
  Recall that the Morleyization language $\cL_\Mor$ contains a predicate symbol $P_\varphi(x)$ for each continuous $\cL$-formula $\varphi$.
  We define $T_0^{\Ra,\Mor}$ to be the definitional expansion of $T_0^{\Ra}$ which stipulates that $P_\varphi(x) = \E \llbracket \varphi(x) \rrbracket$ for any $[0,1]$-valued formula $\varphi$ and, more generally, $P_\varphi = (b-a)P_{\varphi'} + a$ whenever $\varphi$ is $[a,b]$-valued and $\varphi' = (\varphi-a)/(b-a)$.
  We let $Q^{\Ra,\Mor} \coloneqq T_0^{\Ra,\Mor} \cup Q^{\Ra}$.
\end{defn}

\begin{lemma}
  Let $\varphi$ be an affine $\cL_\Mor$-formula, and let $\hat{\varphi}(x)$ be the $\cL$-formula obtained by replacing each $P_\psi$ in $\phi$ by $\psi$.
  Then $T_0^{\Ra,\Mor}$ implies that $\varphi = P_{\hat{\varphi}}$.
\end{lemma}
\begin{proof}
  For atomic $\cL_\Mor$-formulas this holds by definition.
  Expectation respects affine combinations, and quantifiers are respected by \autoref{eq:BBracketQuantifier}.
\end{proof}

\begin{lemma}
  Let $\sigma$ be an affine $\cL_\Mor$-sentence.
  Then the following are equivalent:
  \begin{enumerate}
  \item $Q \models \hat{\sigma} \geq 0$, where $\hat{\sigma}$ is as in the previous lemma.
  \item $Q^{\Ra,\Mor} \models \sigma \geq 0$.
  \item $Q_\Bau \models \sigma \geq 0$.
  \end{enumerate}
\end{lemma}
\begin{proof}
  \begin{cycprf}
  \item
    Assume first that $Q \models \hat{\sigma} \geq 0$.
    Equivalently, $Q \models 0 \wedge \hat{\sigma} = 0$.
    Then $Q^{\Ra,\Mor}$ implies that $P_{\hat{\sigma}} \geq P_{0 \wedge \hat{\sigma}} = 0$.
    By the previous lemma, it implies that $\sigma \geq 0$.
  \item
    Let $M$ be a model of $Q$, and let $M'$ be the $\cL^\Rand$-structure in which the auxiliary sort is just $[0,1]$ (i.e., $\Omega$ consists of a single point), and $\llbracket P(a) \rrbracket$ is just (the constant random variable) $P(a)$.
    Then $M' \models Q^{\Ra}$, and $\llbracket \varphi(a) \rrbracket^{M'} = \varphi(a)^M$ for every continuous formula $\varphi$.
    With the definitional expansions to $\cL_\Mor$ we have $M \models Q_\Mor$, $M' \models Q^{\Ra,\Mor}$ and $P_\varphi(a)^{M'} = E \llbracket \varphi(a) \rrbracket^{M'} = \varphi(a)^M = P_\varphi(a)^M$.
    In other words, $M$ and $M'$ interpret $\cL_\Mor$ identically.

    Assume that $Q^{\Ra,\Mor} \models \sigma \geq 0$.
    By the previous argument, $\sigma \geq 0$ in every model of $Q_\Mor$.
    It is therefore a consequence of $Q_\Bau$.
  \item[\impfirst]
    If $Q_\Bau \models \sigma \geq 0$, then $Q_\Mor \models \sigma \geq 0$.
    Therefore $Q \models \hat{\sigma} \geq 0$.
  \end{cycprf}
\end{proof}

We have established the following:

\begin{theorem}
  \label{th:RandomMorBau}
  Let $Q$ be a continuous $\cL$-theory.
  Then the affine $\cL_\Mor$-theory $Q_\Bau$ consists exactly of the affine $\cL_\Mor$-conditions implied by $Q^{\Ra,\Mor}$.
\end{theorem}

\begin{remark}
  \label{rmk:RandomisedConstantSort}
  Consider the structure $V = [0,1]$ with the standard metric, in a language $\cL_V$ containing a predicate symbol $\id\colon V\to [0,1]$ interpreted as the identity function. Let us call this structure the \emph{value sort}.
  Being compact, it is the unique model of its complete continuous theory $Q_V$.
  This theory implies that $d(t,s) = |\id(t)-\id(s)|$ and $\inf_t |\id(t)-q| = 0$ for every constant $q \in [0,1]$ (and these conditions axiomatise $Q_V$).

  The language $\cL_V^\Rand$ has two sorts: the randomized value sort $V^\Rand$ and the auxiliary sort $A$.
  In a model of $Q_V^{\Ra}$, we have $A = L^1\bigl( \Omega, [0,1]\bigr)$ for some probability space $\Omega$, and for every $a,b \in V^\Rand$:
  \begin{align*}
    d_{V^\Rand}(a,b)
    & = \E\llbracket d_V(a,b) \rrbracket
      = \E \llbracket |\id(a) - \id(b)|\rrbracket
      = \E |\llbracket \id(a) \rrbracket - \llbracket \id(b) \rrbracket|
    \\
    &
      = d_A\bigl( \llbracket\id(a)\rrbracket, \llbracket\id(b)\rrbracket \bigr).
  \end{align*}
  Therefore, $\llbracket\id\rrbracket\colon V^\Rand \rightarrow A$ is an isometry.
  We shall identify $V^\Rand$ with its image in $A$, which is complete and therefore closed.

  For $q \in [0,1]$, we have $Q_V \models \inf_t \, |\id(t)-q| = 0$ and $\llbracket q\rrbracket \in A$ is the constant random variable on $\Omega$ equal to $q$.
  The theory $Q_V^{\Ra}$ implies that
  \begin{gather*}
    \inf_t d\bigl( \llbracket\id(t)\rrbracket, \llbracket q \rrbracket \bigr)
    = \inf_t \E |\llbracket\id(t)\rrbracket-\llbracket q \rrbracket|
    = \E \llbracket \inf_t \, |\id(t)-q| \rrbracket
    = 0.
  \end{gather*}
  Therefore, $\llbracket q\rrbracket$ lies in $\overline{V^\Rand} = V^\Rand$.

  By the gluing axiom, for any $f,g \in V^\Rand$ and event $F \in \MALG(\Omega)$, there exists $h \in V^\Rand$ that agrees with $f$ on $F$ and with $g$ on the complement.
  Therefore $V^\Rand$ contains all measurable step functions, and since these are dense, $V^\Rand = A$.
  Conversely, for every probability space $\Omega$, if $A = V^\Rand$ is the space of $[0,1]$-valued random variables, then $(V^\Rand,A) \models Q_V^{\Ra}$.
\end{remark}

\begin{theorem}
  \label{thm:QPlusBau}
  Let $Q$ be a continuous $\cL$-theory, let $\cL^+$ be the result of adjoining to $\cL$ a new value sort $V$, and let $Q^+ = Q \cup Q_V$.
  Then, up to identifying the sorts $V$ and $A$, the affine theories $Q^{\Ra}$ and $(Q^+)_\Bau$ are affinely interdefinable.
\end{theorem}
\begin{proof}
  In every model of $(Q^+)^{\Ra}$, the randomized value sort $V^\Rand$ equals the auxiliary sort.
  Moreover, if we drop $V^\Rand$, then we obtain a model of $Q^{\Ra}$, and if we drop $A$, then (after an expansion to $(\cL^+)_\Mor$, which uses the affinely definable functions $\llbracket \varphi \rrbracket$) we obtain a model of $(Q^+)_\Bau$.

  Conversely, as in \autoref{rmk:RandomisedConstantSort}, from every model of $Q^{\Ra}$ we can obtain a model of $(Q^+)^{\Ra}$ in a unique fashion, by repeating the sort $A$ as $V^\Rand$.
  Similarly, from every model of $(Q^+)_\Bau$ we can obtain a model of $(Q^+)^{\Ra}$ by repeating the sort $V$ as $A$.
  Indeed, if $(M,V^M) \models (Q^+)_\Bau$ then, by \autoref{th:RandomMorBau}, it embeds affinely in some $(N,V^N)$, where $(N,V^N,A^N) \models (Q^+)^{\Ra}$.
  Then $A^N = V^N$, so $(M,V^M,V^M) \preceq^\aff (N,V^N,A^N)$.
  Since $(Q^+)^{\Ra}$ is affine, it holds in $(M,V^M,V^M)$.

  Thus, up to identifying the sorts $V$ and $A$ and after an expansion to $(\cL^+)_\Mor$, the theory $(Q^+)^{\Ra}$ is an affine definitional expansion of both $Q^{\Ra}$ and $(Q^+)_\Bau$, whence our claim.
\end{proof}

\begin{cor}\label{c:models-of-randomizations}
  Every model $(M^\Rand,A)$ of $Q^{\Ra}$ is isomorphic to a direct integral of models of $Q$ over some probability space $(\Omega,\mu)$, with the auxiliary sort isomorphic to $L^1\bigl( \Omega,[0,1] \bigr)$.
  More precisely, $(M^\Rand,A)\cong \int_\Omega \, ( M_\omega, V )\ud\mu$, where $M_\omega\models Q$ and $V=[0,1]$ is the value sort.
\end{cor}
\begin{proof}
  Follows from the previous theorem and \autoref{rmk:Bau}.
\end{proof}

Let us now consider atomless randomizations. For the rest of this section, we employ the notation introduced in the statement of \autoref{thm:QPlusBau}.

\begin{theorem}\label{thm:QPlusBau-CR}
  Up to identification of the sorts $V$ and $A$, the continuous theories $Q^\Rand$ and $((Q^+)_\Bau)_\nat$ are interdefinable.
\end{theorem}
\begin{proof}
  It suffices to see that under the correspondence of \autoref{thm:QPlusBau}, models of $Q^\Rand$ correspond to models of $((Q^+)_\Bau)_\nat$ and vice versa. If $(M^\Rand,A)\models Q^\Rand$, then $A=L^1(\Omega,[0,1])$ with $\Omega$ atomless, so by the previous corollary $(M^\Rand,A)$ is a direct integral over an atomless space, and thus has the convex realization property by \autoref{l:non-atomic-has-CR}. Conversely, suppose that $(M^\Rand,A)$ has the convex realization property in the language $(\cL^+)_\Mor$. Since the bi-definition between $(Q^+)_\Bau$ and $Q^{\Ra}$ is affine, q-convex conditions in one language translate into q-convex conditions in the other language. Therefore, $(M^\Rand,A)$ has the convex realization property in the language $\cL^\Rand$, and so does the reduct $A=L^1(\Omega,[0,1])$ in the language of $[0,1]$-valued random variables. If $\Omega'$ is any probability space, then $A$ and $A'=L^1(\Omega',[0,1])$ have the same affine theory. If moreover $\Omega'$ is atomless, then $A\equiv^\cont A'$ by \autoref{c:CR-CL-completeness}. We conclude that $\Omega$ is also atomless, as desired.
\end{proof}

\begin{remark}
  The fundamental result of the theory of randomizations states that $Q^\Rand$ has quantifier elimination in the language expanded with the function symbols $\llbracket\varphi\rrbracket$ for every $\cL$-formula $\varphi$. We observe that one can deduce this theorem from \autoref{thm:QPlusBau} (and its proof) and \autoref{thm:QBauNatQE}. More precisely, the combination of these results gives quantifier elimination in the language expanded with the definable functions $\llbracket\varphi\rrbracket$ for $\cL^+$-fomulas $\varphi$. In turn, using basic properties of the theory $(Q^+)_\Mor$, one can reduce to the consideration of definable functions of the form $\llbracket\varphi\rrbracket$ for $\cL$-formulas~$\varphi$.
\end{remark}

To conclude, we prove some basic results to the effect that, in most cases, the passage from $Q$ to $Q^+$ is unnecessary.

Recall that two continuous logic theories are bi-interpretable if they have a common interpretational expansion; we refer the reader to \cite[\textsection1]{BenYaacovLelek} for the definitions. Observe that in the definition of an interpretational expansion in \cite{BenYaacovLelek} (specifically, in the operations allowed to builds the \emph{interpretable sorts}) one is allowed to freely add the so-called constant sort $\{0,1\}$ (or equivalently, any non-trivial compact metric sort, such as our value sort $V=[0,1]$). In particular, $Q^+$ is always an interpretational expansion of $Q$. We may say that an interpretational expansion is \emph{pure} if it does not use the operation of freely adding the constant sort, and say that two continuous logic theories are \emph{purely bi-interpretable} if they have a common pure interpretational expansion.

Recall that a theory $Q$ is non-degenerate if all models of $Q$ have at least two elements, in some finite product of basic sorts.

\begin{lemma}\label{l:Q-biint-Q+}
  Let $Q$ be a non-degenerate continuous logic theory.
  Then $Q^+$ is a pure interpretational expansion of $Q$ in continuous logic.
\end{lemma}
\begin{proof}
  It is enough to argue that the constant sort $\{0,1\}$ is interpretable in $Q$, as the value sort $[0,1]$ can then be obtain as quotient of the product sort $\{0,1\}^\bN$ by the definable pseudometric
  \begin{gather*}
    d'(x,y) = \left| \sum_{i\in\N}\frac{x_i - y_i}{2^{i+1}} \right|.
  \end{gather*}

  As in the discussion following \autoref{dfn:NonDegenerateTheory}, there exists a sort (or product of sorts) $D$, and $\delta > 0$, such that the diameter of $D$ is at least $\delta$ in every model of $Q$.
  We define a pseudometric on $D^2$ by
  \begin{gather*}
    \rho\big(x,y) = \frac{1}{\delta} \bigl|(d(x_0,x_1) \wedge \delta)-(d(y_0,y_1)\wedge \delta)\bigr|.
  \end{gather*}
  We may further form the imaginary sort $S = \widehat{D^2/\rho}$ (i.e., the completion of the quotient of $D^2$ by $\rho$), with the induced metric $\rho'$.
  Let $0_S$ denote the element of $S$ given by the class of any $x\in D^2$ with $x_0=x_1$.
  Then there is a unique element $1_S\in S$ with the property that $\rho'(0_S,1_S) = 1$.
  The set $C=\set{0_S,1_S}\subseteq S$ is definable in continuous logic, interpreting the constant sort in $Q$.
\end{proof}

We will say that an interpretational expansion is \emph{affine} if it is pure, both theories are affine, and all formulas involved in the definition are affine.
Accordingly, two affine theories are \emph{affinely bi-interpretable} if they have a common affine interpretational expansion.

\begin{lemma}\label{l:QBau-biint-Q'Bau}
  Let $Q$ and $Q'$ be purely bi-interpretable continuous logic theories. Then $Q_\Bau$ and ${Q'}_\Bau$ are affinely bi-interpretable.
\end{lemma}
\begin{proof}
  It suffices to prove that if $Q'$ is a pure interpretational expansion of $Q$, then ${Q'}_\Bau$ is an affine interpretational expansion of $Q_\Bau$.

  We denote the languages of $Q$ and $Q'$ by $\cL$ and $\cL'$. Under the assumption that $Q'$ is a pure interpretational expansion of $Q$ (so that, in particular, $\cL\subseteq \cL'$), for each $\cL'$-sort $S$ there exist some product of $\cL$-sorts, call it $P_S$, some $\cL$-definable (modulo $Q$) set $D_S\subseteq P_S$, an $\cL$-definable (modulo $Q$) pseudometric $\rho_S$ on $D_S$, and an $\cL'$-definable (modulo $Q'$) map $\alpha_S\colon D_S \to S$ such that:
  \begin{enumerate}
  \item $\alpha_S$ is a $\rho_S$-isometric map with dense image, thus inducing a surjective isometry $(\widehat{D_S/\rho_S},\rho'_S)\to (S,d)$, where $\rho'_S$ is the induced metric in the completion of the quotient of $D_S$ by $\rho_S$;
  \item for every basic $\cL'$-predicate $P'$ defined in a product sort $\prod_i S_i$, the predicate $P \colon \prod_i D_{S_i} \to \R$ given by $P\big((x_i)_i\big) = P'\big((\alpha_{S_i}(x_i))_i\big)$ is $\cL$-definable.
  \end{enumerate}
  When passing from $Q$ and $Q'$ to their Morleyizations, every formula $\varphi$ used to define $D_S$, $\rho_S$, $\alpha_S$, or the predicates $P$, can be replaced by the basic symbol $P_\varphi$. In this way, the sets $D_S$ and the pseudometrics $\rho_S$ are defined by affine formulas (and uniform limits thereof) modulo $Q_\Mor$. They are thus affinely definable modulo $Q_\Bau$, because the fact that a formula (or limit of formulas) defines a set or a pseudometric is expressible by affine conditions (for sets, this was proved in \autoref{lem:DefinableSetDistancePredicate}; for pseudometrics, it is just the fact that the triangle inequality is affine). Similarly, the maps $\alpha_S$ are defined by affine formulas modulo ${Q'}_\Mor$, and the properties enumerated above are expressible by affine $\cL'_\Mor$-conditions. This shows that ${Q'}_\Bau$ is an affine interpretational expansion of $Q_\Bau$.
\end{proof}

\begin{prop}\label{p:QBau-biint-QRa}
  Let $Q$ be a non-degenerate continuous logic theory. Then $Q_\Bau$ is affinely bi-interpretable with $(Q^+)_\Bau$.
  If moreover $Q$ has affine reduction, then $Q_\aff$ is affinely bi-interpretable with $(Q^+)_\Bau$.
\end{prop}
\begin{proof}
  The main assertion follows from \autoref{l:Q-biint-Q+} and \autoref{l:QBau-biint-Q'Bau}. If $Q$ has affine reduction, then $Q_\Bau$ is an affine definitional expansion of $Q_\aff$, by \autoref{l:interdef-affine-reduction-Bau}.
\end{proof}

We end with an application to Bauer theories.

\begin{theorem}
  Let $T$ be a non-degenerate Bauer theory. Then for any model $M\models T$, the following are equivalent:
  \begin{enumerate}
  \item $M$ is isomorphic to a direct integral of extremal models of $T$ over an atomless probability space.
  \item $M$ has the convex realization property.
  \end{enumerate}
\end{theorem}
\begin{proof}
  One implication is just \autoref{l:non-atomic-has-CR}. Conversely, suppose $M$ has the convex realization property. As $T$ is a Bauer theory, the theory $Q=T_\ext$ has affine reduction, by \autoref{th:Bauer-correspondence-T}. By \autoref{l:T=Textaff}, $M$ is a model of $Q_\aff$. On the other hand, by \autoref{p:QBau-biint-QRa}, $Q_\aff$ is affinely bi-interpretable with $(Q^+)_\Bau$, so $M$ can be seen as a model of $((Q^+)_\Bau)_\nat$. Finally, by \autoref{thm:QPlusBau-CR}, $M$ can be seen as a model of $Q^\Rand$, and hence as a direct integral over an atomless probability space, by \autoref{c:models-of-randomizations}.
\end{proof}

\section{An automorphism group criterion for affine reduction}
\label{sec:criterion-affine-reduction}

We give here a characterization of continuous logic theories with affine reduction in terms of the automorphism group of a separable approximately saturated model, if such a model exists. This criterion will be helpful in the analysis of some of the examples in \autoref{part:Examples}. In proving the criterion we also establish a result about transfer of approximate saturation by direct multiples (\autoref{thm:AtomlessTimesApproximatelySaturated}).

We begin by considering types over convex combinations of parameter sets.
Let $(M_i)_{i<n}$ be $\cL$-structures and let $\lambda_i\geq 0$ be scalars with $\sum\lambda_i=1$.
We consider the convex combination $M=\bigoplus_{i<n}\lambda_i M_i$, whose elements are denoted as usual by $\bigoplus\lambda_i a_i$ for $a_i\in M_i$.
Given subsets $A_i\subseteq M_i$, we define $\bigoplus\lambda_i A_i\subseteq M$ to consist of those $\bigoplus\lambda_i a_i$ with $a_i\in A_i$.
Given $p_i \in \tS^\aff_x(A_i)$, we define $\bigoplus \lambda_i p_i \in \tS^\aff_x\bigl( \bigoplus \lambda_i A_i \bigr)$ by
\begin{gather*}
  \varphi\left( x,\bigoplus\lambda_i a_i \right)^{\bigoplus\lambda_i p_i} = \sum \lambda_i \varphi(x,a_i)^{p_i}.
\end{gather*}
Equivalently, $\bigoplus\lambda_i p_i$ is the type of $\bigoplus\lambda_i b_i$, where $b_i$ is a realization of $p_i$ in some affine extension $M_i'\succeq^\aff M_i$, and $\bigoplus\lambda_i b_i$ is the corresponding element or tuple of the affine extension $\bigoplus\lambda_i M'_i\succeq^\aff M$.

\begin{lemma}
  \label{l:surjectivity-types-over-conv-comb}
  The map $\prod_{i<n}\tS_x^{\aff,M_i}(A_i) \rightarrow \tS_x^{\aff,M}(\bigoplus\lambda_i A_i)$ sending $(p_i) \mapsto \bigoplus \lambda_i p_i$ is continuous, affine and surjective.
\end{lemma}
\begin{proof}
  Continuity and affineness are clear.
  For surjectivity we use \autoref{l:surjective-affine-map}.
  Let $\varphi(x)$ be an affine formula with parameters from $\bigoplus\lambda_i A_i$ such that $\oset{\varphi>0}$ is non-empty in $\tS_x^{\aff,M}(\bigoplus\lambda_i A_i)$.
  Then there exists $b = \bigoplus \lambda_i b_i \in M^x$ such that $\varphi(b) > 0$.
  Then $\bigoplus \lambda_i \tp^\aff(b_i/A_i) \in \oset{\varphi>0}$, as desired.
\end{proof}

In the next results we will consider approximate saturation and homogeneity as introduced in \autoref{sec:Homogeneity-and-saturation} (see Definitions \ref{dfn:ApproximatelySaturated} and \ref{dfn:ApproximatelyHomogeneous}), as well as the corresponding notions in the sense of continuous logic. The definitions for continuous logic are the same as for affine logic, replacing affine types by continuous logic types, and the results proved in \autoref{sec:Homogeneity-and-saturation} are valid (and more or less well-known, see \cite{BenYaacov2007}) in the continuous setting, with the same proofs.

\begin{lemma}
  \label{lem:AtomlessTimesApproximatelySaturated}
  Let $M$ be a separable structure, approximately saturated in the sense of continuous logic, and let $(\Omega,\cB,\mu)$ be an atomless standard probability space.
  Then for every finite $a \in M^n$, every $p(x,a) \in \tS^\aff_1(a)$ and every $\varepsilon > 0$, there exists $b \in L^1(\Omega, M)$ such that $\partial(p,ba) < \varepsilon$ (see \autoref{lem:ApproximatelySaturatedByDense} and the preceding notation).
\end{lemma}
\begin{proof}
  Let $T = \Th^\aff(M)$ and observe that the hypotheses imply that $\fd_T(\cL) = \aleph_0$.
  Let $x,x'$ be singletons, and $y,y'$ $n$-tuples.
  Let $E \subseteq \tS^\aff_{xy}(T)$ consist of all affine types $r(x,y)$ such that $\pi_y(r) = \tp^\aff(a)$ and $r(x,a) \in \cE_1(a)$. In other words, $E = \cE(\pi_y^{-1}(\tp^\aff(a)))$.
  Then
  \begin{gather*}
    F = \bigl\{ (q,b) \in \tS^\aff_{xyx'y'}(T) \times M : \pi_{xy}(q) \in E, \ (b,a) \models \pi_{x'y'}(q), \ d(xy,x'y')^q < \varepsilon \bigr\}
  \end{gather*}
  is a Borel subset of a Polish space.

  For $(q,b) \in F$, let $r(x,y) = \pi_{xy}(q)$, and define $\tilde{\pi}(q,b) = r(x,a)$.
  Then $\tilde{\pi}\colon F \rightarrow \cE_1(a)$ is continuous, and we claim that it is surjective.
  Indeed, consider an extremal type $r(x,a) \in \cE_1(a)$.
  Applying \autoref{l:image-rho-aff} to the continuous diagram $D^\cont_a$, we see that there exists $\hat{r}(x,a) \in \tS^\cont_1(a)$ whose affine part is $r(x,a)$.
  By approximate saturation, there exist $\hat{q} \in \tS^\cont_{xyx'y'}(\cL)$ and $b \in M$ such that $\pi_{xy}(\hat{q}) = \hat{r}$, $(b,a) \models \pi_{x'y'}(\hat{q})$, and $d(xy,x'y')^{\hat{q}} < \varepsilon$.
  Then $(\rho^\aff(\hat{q}),b) \in F$ and $\tilde{\pi}$ sends it to $r(x,a)$.

  By the Jankov--von Neumann uniformization theorem (see \cite[Ex.~18.3, Thm.~21.10]{Kechris1995}), $\tilde{\pi}$ admits a universally measurable section $f\colon \cE_1(a) \rightarrow F$, which we may express as a pair $(g,h)$, where $g\colon \cE_1(a) \rightarrow \tS^\aff_{xyx'y'}(T)$ and $h\colon \cE_1(a) \rightarrow M$.

  By \autoref{l:choquet-order} and \autoref{th:Bishop-dL}, there is a probability measure $\nu\in\cM\bigl( \tS^\aff_1(a) \bigr)$ concentrated on $\cE_1(a)$ with barycenter $p(x,a)$.
  Let $q_0 \in \tS^\aff_{xyx'y'}(T)$ be the barycenter of $g_*\nu$, and consider $\hat{b} = h$ as a section in $\widehat{M} = L^1\bigl(\cE_1(a),\nu, M\bigr)$.
  Then
  \begin{gather*}
    p(x,y) = \pi_{xy}(q_0), \quad (\hat{b},a) \models \pi_{x'y'}(q_0), \And \quad d(xy,x'y')^{q_0} < \varepsilon,
  \end{gather*}
  so $\partial(p,\hat{b}a) < \varepsilon$.
  We may view $\cE_1(a)$ as a factor of $(\Omega',\mu') = (\Omega \times \cE_1(a), \mu \otimes \nu)$, which has the same measure algebra as $\Omega$, so
  \begin{gather*}
    \hat{b} \in L^1\bigl( \cE_1(a), M \bigr) \subseteq L^1(\Omega',M) \cong L^1(\Omega,M),
  \end{gather*}
  the isomorphism being over $M$ (i.e., over the set of constant sections).
\end{proof}

\begin{theorem}
  \label{thm:AtomlessTimesApproximatelySaturated}
  Let $M$ be a separable structure, approximately saturated in the sense of continuous logic, and let $(\Omega,\cB,\mu)$ be an atomless standard probability space.
  Then $L^1(\Omega,M)$ is approximately affinely saturated.
  In particular, it is approximately affinely homogeneous.
\end{theorem}
\begin{proof}
  Let $N = L^1(\Omega,M)$, and let $N_0 \subseteq N$ be the set of simple sections $f\colon \Omega\to M$, which we may express as formal sums $\sum a_i\chi_{Z_i}$ for elements $a_i\in M$ and sets $Z_i\in\cB$ forming a partition of $\Omega$.
  Then $N_0$ is dense in $N$.

  Consider a finite tuple $g \in N_0^n$, and an affine type $p(x,g) \in \tS^\aff_1(g)$.
  We can find a finite measurable partition $(\Omega^j :j < m)$ of $\Omega$ such that every $g_i$ is of the form $\sum_{j<m} a_i^j \chi_{\Omega^j}$.
  For $j < m$ let $\lambda_j = \mu(\Omega^j)$ and $N^j = L^1(\Omega^j,M)$, where we equip $\Omega^j$ with the normalized measure.
  Then $N = \bigoplus_{j<m} \lambda_j N^j$, where we identify a section $f\colon \Omega \rightarrow M$ with the tuple of its restrictions $f^j\colon \Omega^j \rightarrow M$.
  In particular, $g_i^j = a_i^j$ is a constant section, and $g_i = \bigoplus_{j<m} \lambda_j a_i^j$.

  Let $a^j = (a_i^j : i < n)$.
  By \autoref{l:surjectivity-types-over-conv-comb}, there exist $p^j \in \tS^\aff_1(a^j)$ such that $p(x,g) = \bigoplus_{j<m} \lambda_j p^j(x,a^j)$.
  It follows from \autoref{lem:AtomlessTimesApproximatelySaturated} that there exist $b^j \in N^j$ such that $\partial(p^j,b^ja^j) < \varepsilon$.
  We may find realizations of the types $p^j$ witnessing the latter in $K^j \succeq^\aff N^j$, and then in $\bigoplus_{j<m} \lambda_j K^j$ we would obtain a witness to $\partial\left( p, hg \right) < \varepsilon$, where $h = \bigoplus_{j<m} \lambda_j b^j \in N$.

  By \autoref{lem:ApproximatelySaturatedByDense}, $N$ is approximately affinely saturated and by \autoref{cor:ApproximatelySaturatedHomogeneousSeparable}, it is approximately affinely homogeneous.
\end{proof}


The following measurable selection lemma will be useful in the subsequent theorem as well as in later sections.
\begin{lemma}[Measurable choice of realization]
  \label{l:realization-selector}
  Let $T$ be a affine theory in a separable language, let $x$ be a countable variable, and let $M$ be a separable model of $T$. Let $(\Omega,\mu)$ be a complete probability space and let $\xi\colon\Omega \to \tS^\aff_x(T)$ be a measurable function with the property that $\xi(\omega)$ is realized in $M$ for every $\omega\in\Omega$.

  Then there exists a measurable function $f\colon \Omega\to M^x$ such that $f(\omega)\models \xi(\omega)$ for every $\omega\in\Omega$. Moreover, if every $\xi(\omega)$ can be realized by a tuple in a given Borel (or analytic) subset $B\subseteq M^x$, then we may choose $f$ so that $f(\omega)\in B$ for all $\omega\in\Omega$.
\end{lemma}
\begin{proof}
  Let us write $\tR_x^{\aff,M}(T)\subseteq \tS^\aff_x(T)$ for the set of affine types that are realized in $M$. Then $\tR_x^{\aff,M}(T)$ is an analytic subset of the Polish space $\tS^\aff_x(T)$, namely the projection to the first coordinate of the closed set
  $$F=\bigl\{(p,a)\in \tS^\aff_x(T)\times M^x : a\models p \bigr\}.$$
  We can apply the Jankov--von Neumann uniformization theorem to the set $F\cap (\tS^\aff_x(T)\times B)$, to obtain a universally measurable function
  $$\zeta \colon \tR_x^{\aff,M}(T) \to M^x$$
  such that $\zeta(p) \models p$ and $\zeta(p)\in B$ for all $p \in \tR_x^{\aff,M}(T)$. By hypothesis, the image of $\xi$ is contained in $\tR_x^{\aff,M}(T)$, so we can consider the composition $f=\zeta\circ\xi \colon\Omega \to M^x$. Thus, by construction, $f(\omega)$ realizes $\xi(\omega)$ and $f(\omega)\in B$ for every $\omega\in\Omega$. On the other hand, let $\nu=\xi_*\mu$ and let $\cB_\nu$ be the algebra of $\nu$-measurable subsets of $\tS^\aff_x(T)$. Since $(\Omega,\mu)$ is complete, $\xi$ is measurable with respect to $\cB_\nu$, and as $\zeta$ is universally measurable, we conclude that $f$ is measurable.
\end{proof}


Let $(\Omega,\mu)$ be a standard probability space. Given a separable structure $M$ with at least two elements, the direct multiple $L^1(\Omega,M)$ has some obvious automorphisms. On the one hand, the group $\Aut(\mu)$ of measure-preserving transformations of $\Omega$ acts faithfully by automorphisms on $L^1(\Omega,M)$ by the formula
$$(S \cdot f)(\omega) = f(S^{-1}(\omega)),$$
for $S\in\Aut(\mu)$, $f\in L^1(\Omega,M)$, and a.e.\ $\omega\in \Omega$. On the other hand, the group $L(\Omega,\Aut(M))$  of measurable maps $\Omega\to\Aut(M)$, up to almost sure equality, where $\Aut(M)$ carries the Borel $\sigma$-algebra of the pointwise convergence topology, also acts faithfully by automorphisms on $L^1(\Omega,M)$ by the formula
$$(T\cdot f)(\omega) = T(\omega)(f(\omega)),$$
for $T\in L(\Omega,\Aut(M))$, $f\in L^1(\Omega,M)$, and a.e.\ $\omega\in \Omega$.

Moreover, $\Aut(\mu)$ acts on $L(\Omega,\Aut(M))$ by
$$(S \cdot T)(\omega) = T(S^{-1}(\omega))$$
and this action defines a semidirect product $L(\Omega,\Aut(M))\rtimes\Aut(\mu)$ which can be identified with a subgroup of $\Aut(L^1(\Omega,M))$.

\begin{theorem}\label{th:criterion-affine-reduction}
  Let $M$ be a separable structure in a separable language with at least two elements in a given sort and let $Q = \Th^\cont(M)$. Let $(\Omega, \mu)$ be a standard atomless probability space. Then the following hold:
  \begin{enumerate}
  \item \label{i:criterion-ar-1} If $Q$ has affine reduction, then $\Aut(L^1(\Omega,M)) = L(\Omega,\Aut(M))\rtimes\Aut(\mu)$.
  \item \label{i:criterion-ar-2} If $M$ is approximately saturated in the sense of continuous logic and we have $\Aut(L^1(\Omega,M)) = L(\Omega,\Aut(M))\rtimes\Aut(\mu)$, then $Q$ has affine reduction.
  \end{enumerate}
\end{theorem}
\begin{proof}
  \autoref{i:criterion-ar-1} Since $Q$ has affine reduction, it follows from \autoref{p:QBau-biint-QRa} and \autoref{thm:QPlusBau} that $Q_\aff$ and $Q^{\Ra}$ are bi-interpretable.

  Thus, given any separable $M\models Q$, the automorphism group of $L^1(\Omega,M)$ is the same if $L^1(\Omega,M)$ is viewed as an $\cL$-structure (i.e., as a model of $Q_\aff$) or an $\cL^\Rand$-structure (i.e., a model of $Q^{\Ra}$). It was proved in \cite[Thm.~3.8]{IbarluciaRandomizations} that the automorphism group of $L^1(\Omega,M)$ as an $\cL^\Rand$-structure is precisely the semidirect product $L(\Omega,\Aut(M))\rtimes\Aut(\mu)$.

  \autoref{i:criterion-ar-2} We show that the canonical affine surjections $\iota_n^*\colon \cM(\tS^\cont_n(Q)) \to \tS_n^\aff(Q_\aff)$ are injective.

  As $(\Omega, \mu)$ is atomless, given $\mu_0,\mu_1\in\cM(\tS^\cont_n(Q))$, we can find measurable maps $f'_0,f'_1\colon \Omega\to \tS_n^\cont(Q)$ with $(f_i')_*(\mu)=\mu_i$. By approximate saturation, every type $f'_i(\omega)\in \tS^\cont_n(Q)$ (and hence also every affine type $\rho^\aff_n(f'_i(\omega))\in\tS^\aff_n(Q_\aff)$) is realized in $M$. We may thus apply \autoref{l:realization-selector} to obtain tuples $f_0,f_1\in L^1(\Omega,M)^n$ satisfying, for each affine formula $\varphi$ in $n$ variables,
  $$\varphi(f_i)^{L^1(\Omega,M)}=\int_\Omega \varphi^M(f_i(\omega)) \ud\mu(\omega) = \mu_i(\varphi), \quad \text{ for } i =0,1.$$
  In other words, $\tp^\aff(f_i) = \iota_n^*(\mu_i)$. Suppose now that $\iota_n^*(\mu_0)= \iota_n^*(\mu_1)$, so that $\tp^\aff(f_0)=\tp^\aff(f_1)$.

  By \autoref{thm:AtomlessTimesApproximatelySaturated}, $L^1(\Omega,M)$ is approximately affinely homogeneous. Hence there are automorphisms $U_k\in\Aut(L^1(\Omega,M))$ such that $U_k \cdot f_0$ converges to $f_1$. By hypothesis, we may write $U_k=T_k S_k$ for certain $T_k\in L(\Omega,\Aut(M))$ and $S_k\in \Aut(\mu)$.

  Thence, if we take a continuous function $\tS_n^\cont(Q)\to\R$ or, equivalently, a continuous logic formula $\psi$ in $n$ variables, we have for every $k$:
  \begin{align*}
    \int_\Omega \psi\big((U_k \cdot f_0)(\omega)\big) \ud\mu(\omega) & = \int_\Omega \psi\bigl(T_k(\omega)(f_0(S_k^{-1}\omega))\bigr) \ud\mu(\omega)\\
                                                                     & = \int_\Omega \psi(f_0(\omega)) \ud\mu(\omega) = \mu_0(\psi).
  \end{align*}
  By taking limits, we deduce that $\mu_0(\psi)=\mu_1(\psi)$. Since $\psi$ was arbitrary, we conclude that $\mu_0=\mu_1$, as desired.
\end{proof}


\part{Extremal categoricity}
\label{part:Categoricity}

\section{Extremal $\aleph_0$-categoricity}
\label{sec:aleph0-cat}

In this section we use our results from \autoref{sec:omitting-types} to establish a Ryll-Nardzewski type theorem for $\aleph_0$-categoricity of extremal models.
Using the decomposition theorem from \autoref{sec:ExtremalDecomposition}, we also describe all separable models of extremally $\aleph_0$-categorical, simplicial theories.

\begin{theorem}[Extremal Ryll-Nardzewski theorem]
  \label{th:aleph0-categoricity}
  Let $T$ be a complete, affine theory in a separable language. Then the following are equivalent:
  \begin{enumerate}
  \item $T$ admits a unique separable extremal model, up to isomorphism.
  \item For every $n$, the $\dtp$-topology and $\tau$ coincide on $\cE_n(T)$.
  \end{enumerate}
  Moreover, in that case, the unique separable extremal model is a prime model of $T$.
\end{theorem}
\begin{proof}
  \begin{cycprf}
  \item Suppose that the two topologies do not coincide on some $\cE_n(T)$. This is the same as saying that there exists $p \in \cE_n(T)$ which is not isolated. In particular, there is $r > 0$ such that $B_r(p)$ is $\tau$-meager in $\cE_n(T)$. Then, by \autoref{th:omitting-types}, $p$ is omitted in some extremal, separable model and by \autoref{th:ExtremalModelTypes} and \autoref{prop:DownwardLS}, it is also realized in some extremal, separable model, so these two models cannot be isomorphic.
  \item[\impfirst] Every extreme type being isolated means that every extremal model is atomic; then the claim and the moreover part follow from \autoref{th:uniqueness-atomic-model}.
  \end{cycprf}
\end{proof}

\begin{defn}
  We will call an affine theory in a separable language \emph{extremally $\aleph_0$-categorical} if it admits a unique separable extremal model, up to isomorphism.
\end{defn}
Note that, by \autoref{c:extremal-models-existence}, an extremally $\aleph_0$-categorical theory is automatically complete, and therefore satisfies the hypothesis and the conclusion of \autoref{th:aleph0-categoricity}.

If $M$ is a metric space and $G$ is a group of isometries of $M$, we denote by $M^n\sslash G$ the space $\set{\cl{G \cdot a}:a\in M}$ of orbit closures of the diagonal action of $G$ on $M^n$. It is endowed with the quotient metric
\begin{equation*}
  d(\cl{G \cdot a},\cl{G \cdot b}) \coloneqq \inf_{g\in G} d(g \cdot a,b) = \inf_{g\in G} \sum_{i<n} d(g \cdot a_i,b_i).
\end{equation*}

\begin{prop}\label{p:M-sslash-G}
  Let $T$ be an extremally $\aleph_0$-categorical theory with separable extremal model $M$. Let $G=\Aut(M)$. Then for every $n$, the map
  $$M^n\sslash G\to \cE_n(T), \quad \cl{G \cdot a} \mapsto \tp^\aff(a)$$
  is a $\dtp$-isometry (and a fortiori a $\tau$-homeomorphism).
\end{prop}
\begin{proof}
  By \autoref{th:aleph0-categoricity}, every extreme type is isolated, so realized in every model. Hence the map is surjective. On the other hand, given $a,b\in M^n$, it follows from \autoref{cor:ExtremeTypeDistance} and the fact that all extreme types are realized in $M$ that there are $a',b'\in M^n$ such that $\tp^\aff(a)=\tp^\aff(a')$, $\tp^\aff(b)=\tp^\aff(b')$ and $\dtp(\tp^\aff(a),\tp^\aff(b))=d(a',b')$. By approximate affine homogeneity of $M$ (see \autoref{th:uniqueness-atomic-model}), we have $a'\in \cl{G \cdot a}$ and $b'\in \cl{G \cdot b}$. We deduce that $d(\cl{G \cdot a},\cl{G \cdot b}) \leq \dtp(\tp^\aff(a),\tp^\aff(b))$ and the other inequality is obvious. In particular, the map is a homeomorphism for the $\dtp$-topology, which coincides with $\tau$ by the previous theorem.
\end{proof}

We record some basic connections with $\aleph_0$-categoricity in continuous logic.

\begin{prop}\label{prop:ext-aleph0cat-vs-cont-aleph0cat}
  Let $T$ be an affine theory. The following are equivalent:
  \begin{enumerate}
  \item $T$ is extremally $\aleph_0$-categorical theory and the spaces $\cE_n(T)$ are closed in $\tS^\aff_n(T)$.
  \item $T_\ext$ is $\aleph_0$-categorical in continuous logic.
  \end{enumerate}
\end{prop}
\begin{proof}
  \begin{cycprf}
  \item This follows from \autoref{cor:E(T)-closed-iota-homeo} and the Ryll-Nardzewski theorem for continuous logic.
  \item[\impfirst] If $T_\ext$ is $\aleph_0$-categorical, then $T$ is clearly extremally $\aleph_0$-categorical and they share the same unique separable (extremal) model, say $M$. Moreover, every type of $T_\ext$ is realized in $M$, and therefore its affine part is extreme. Hence the models of $T_\ext$ are precisely the extremal models of $T$, and the type spaces $\cE_n(T)$ are closed by \autoref{c:closed-extreme-types}.
  \end{cycprf}
\end{proof}

One of the previous implications can be partially generalized as follows.

\begin{prop}\label{p:a0-cat-vs-ext-a0-cat}
  Let $Q$ be a continuous theory in a separable language. If $Q$ is $\aleph_0$-categorical, then $Q_\aff$ is extremally $\aleph_0$-categorical.
  In that case, the separable extremal model of $Q_\aff$ is an affine substructure of the separable model of $Q$.
\end{prop}
\begin{proof}
  By \autoref{th:aleph0-categoricity}, if $Q_\aff$ is not extremally $\aleph_0$-categorical there are $n\in\bN$, $\eps>0$ and $q,q_i\in\cE_n(Q_\aff)$ such that $q_i\to^\tau q$ but $\dtp(q_i,q)\geq\eps$ for all $i\in\bN$. Say $q_i=\pi(p_i)$ for $p_i\in \tS_n^\cont(Q)$. Up to passing to a subsequence, there is $p$ such that $p_i\to^\tau p$. Hence $\pi(p)=q$ and $\dtp(p_i,p)\geq\dtp(q_i,q)\geq\eps$ for all $i$, implying that $Q$ is not $\aleph_0$-categorical.

  The last part is clear since the separable extremal model of $Q_\aff$ is prime.
\end{proof}

The argument of the previous proof is the same that one may use to show, in continuous logic, that a \emph{reduct} of an $\aleph_0$-categorical theory is $\aleph_0$-categorical.

\begin{question}
  Is a reduct of an extremally $\aleph_0$-categorical affine theory extremally $\aleph_0$-categorical?
\end{question}

Note that at the level of structures, a reduct of the unique separable extremal model of an extremally $\aleph_0$-categorical affine theory need not be extremal (e.g., consider a finite non-trivial probability algebra with names for each of its elements, as in \autoref{ex:PrA-with-2parameters}).

\begin{theorem}
  \label{th:sep-models-aleph0-cat}
  Let $T$ be an extremally $\aleph_0$-categorical, simplicial theory and let $M_0$ be its unique separable extremal model. Then every separable model of $T$ is of the form $L^1(\Omega, M_0)$, where $\Omega$ is a standard probability space.
\end{theorem}
\begin{proof}
  Let $M$ be a separable model of $T$. We apply \autoref{th:integral-decomposition} to write $M = \int_\Omega^\oplus M_\omega \ud \omega$, where $\Omega$ is a complete, standard probability space and each $M_\omega$ is a separable, extremal model of $T$. Let $(e_i : i \in \bN)$ be the pointwise enumeration of the field $(M_\omega:\omega\in\Omega)$ associated to the direct integral. By extremal $\aleph_0$-categoricity, each $M_\omega$ is isomorphic to $M_0$. Let us consider the measurable map $\xi\colon \Omega \to \tS^\aff_\bN(T)$, $\xi(\omega) = \tp^\aff(e(\omega))$ and the set $B\subseteq M_0^\bN$ of $\bN$-tuples that are dense in $M_0$; we may then apply \autoref{l:realization-selector} to get a measurable function $a \colon \Omega \to M_0^\N$ such that $\tp^\aff(a(\omega)) = \tp^\aff(e(\omega))$ and $a(\omega)$ is dense in $M_0$ for all $\omega\in\Omega$.
  For $\omega \in \Omega$, let $\theta_\omega \colon M_\omega \to M_0$ be the isomorphism sending $e(\omega)$ to $a(\omega)$.
  We define a map $\Theta \colon M \to L^1(\Omega, M_0)$ by
  \begin{equation*}
    \Theta(s)(\omega) = \theta_\omega(s(\omega)).
  \end{equation*}
  First, we check that $\Theta(s)$ is indeed a measurable function $\Omega \to M_0$. For this, it suffices to see that for every $b \in M_0$, the map $\omega \mapsto d\big(\theta_\omega(s(\omega)), b\big)$ is measurable. For every $n \in \N$, define
  \begin{equation*}
    j_n(\omega) = \min \set{j \in \N : d(b, a_j(\omega)) < 2^{-n}}.
  \end{equation*}
  It is clear that each $j_n$ is measurable. We have that
  \begin{equation*}
    d\big(\theta_\omega(s(\omega)), b\big) = \lim_{n \to \infty} d \big(\theta_\omega(s(\omega)), a_{j_n(\omega)}(\omega)\big)
    = \lim_{n \to \infty} d\big(s(\omega), e_{j_n(\omega)}(\omega)\big)
  \end{equation*}
  and this is measurable because $s$ is a measurable section.

  One can define an inverse of $\Theta$ by
  \begin{equation*}
    \Theta^{-1}(f)(\omega) = \theta_\omega^{-1}(f(\omega)),
  \end{equation*}
  so we see that it is a bijection. Finally, for a tuple $\bar s \in M^n$ and any affine formula $\varphi$, we have that
  \begin{equation*}
    \varphi^M(\bar s) = \int_\Omega \varphi^{M_\omega}(\bar s(\omega)) \ud \omega = \int_\Omega \varphi^{M_0}(\theta_\omega(\bar s(\omega))) \ud \omega = \varphi^{L^1(\Omega, M_0)}(\Theta(\bar s)),
  \end{equation*}
  and we conclude that $\Theta$ is an isomorphism.
\end{proof}

\begin{question}
  \label{q:aleph_0-cat-Bauer}
  Is every extremally $\aleph_0$-categorical simplicial theory Bauer?
\end{question}

In view of \autoref{th:dichotomy-simplicial-theories}, the question is equivalent to whether there exist extremally $\aleph_0$-categorical Poulsen theories.


\section{Absolutely categorical and bounded theories}
\label{sec:absolute-categoricity}

We present here some definitions and basic results concerning new categoricity phenomena that have no analogue in continuous logic.
Above all, we record some intriguing questions, which we do not address in the present work.

\begin{defn}
  \label{df:absolutely-cat}
  An affine theory $T$ is called \emph{absolutely categorical} if it has a unique extremal model, up to isomorphism.
\end{defn}

\begin{prop}\label{p:compact-absol-cat}
  Let $T$ be a complete affine theory with a compact model.
  Then $T$ is absolutely categorical, the extremal model is compact and prime, and the extreme type spaces $\cE_x(T)$ are closed.
\end{prop}
\begin{proof}
  If $T$ has a compact model, then by the main assertion of \autoref{p:compact-def-sets}, all extremal models are compact.
  In addition, by \autoref{prop:ExtremalJEP} and the moreover part of \autoref{p:compact-def-sets}, every two extremal models are isomorphic.
  It follows readily that each extreme type space $\cE_x(T)$ is $\partial$-compact, and therefore $\tau$-closed in $\tS^\aff_x(T)$.
  Finally, as $T$ is complete with a compact model, it has a separable language, in the sense of \autoref{defn:LanguageDensityCharacter}.
  Hence $T$ satisfies the conditions of \autoref{th:aleph0-categoricity}, and in particular, the extremal model of $T$ is prime.
\end{proof}

\begin{prop}
  \label{p:abs-cat-no-ee}
  Let $T$ be absolutely categorical and let $M$ be the extremal model of $T$. Then every affine embedding $M \to M$ is an isomorphism.
\end{prop}
\begin{proof}
  Suppose that $f \colon M \to M$ is an affine embedding that is not surjective and let $\kappa$ be a cardinal. We can construct a directed system
  \begin{equation*}
    M \xrightarrow{f} M \xrightarrow{f} M \to \cdots
  \end{equation*}
  of length $\kappa$ where at limit stages we take the direct limit of the sequence so far (and we use \autoref{p:affine-chains-extremal} and absolute categoricity to conclude that it is also isomorphic to $M$). If we take $\kappa$ sufficiently large, this produces an extremal model of $T$ of cardinal bigger than $|M|$, contradicting absolute categoricity.
\end{proof}

\begin{theorem}
  \label{th:abs-cat-models}
  Let $T$ be an absolutely categorical, simplicial theory in a separable language and let $M_0$ be the extremal model of $T$. Then every model of $T$ is of the form $L^1(\Omega, M_0)$, where $\Omega$ is a probability space.
\end{theorem}
\begin{proof}
  As $T$ has a separable language, it follows from Löwenheim--Skolem that $M_0$ is separable. Let $M \models T$.
  Using \autoref{th:integral-decomposition}, we write $M = \int_\Omega^\oplus M_\omega \ud \mu(\omega)$, where each $M_\omega$, being an extremal model of $T$, is isomorphic to $M_0$. We will use a strategy similar to the one in the proof of \autoref{th:sep-models-aleph0-cat}. Let $(e_i : i \in I)$ be the sequence of basic sections given by the direct integral decomposition. From the definition, it follows that there exists a countable $I_0 \sub I$ such that almost surely, $\cl{\set{e_i(\omega) : i \in I_0}} \preceq M_\omega$. As each $M_\omega$ is isomorphic to $M_0$ and $M_0$ has no proper affine substructures (\autoref{p:abs-cat-no-ee}), we obtain that these inclusions are actually equalities.

  As $I_0$ is countable, we can apply \autoref{l:realization-selector} to obtain a measurable function $a \colon \Omega \to M_0^{I_0}$ such that $\tp^\aff(a(\omega)) = \tp^\aff(e_{I_0}(\omega))$ and $a(\omega)$ is dense in $M_0$ for all $\omega\in\Omega$.
  Now for each $\omega$, we let $\theta_\omega \colon M_\omega \to M_0$ be the isomorphism that sends $e_{I_0}(\omega)$ to $a(\omega)$, and complete the proof exactly in the same way as the one of \autoref{th:sep-models-aleph0-cat}.
\end{proof}

There exist non-simplicial theories which are absolutely categorical, and even ones that admit a compact extremal model (see \autoref{ex:PrA-with-2parameters}).

\begin{defn}
  Let $T$ be an affine theory.
  \begin{enumerate}
  \item Let $\lambda$ be an infinite cardinal. The theory $T$ is \emph{$\lambda$-bounded} if every extremal model of $T$ has density character at most $\lambda$. It is \emph{bounded} if it is $\lambda$-bounded for some $\lambda$. Otherwise, $T$ is \emph{unbounded}.
  \item An extremal model $M$ of $T$ is \emph{maximal} if it does not admit any proper extremal affine extension.
  \end{enumerate}
\end{defn}

\begin{prop}
  \label{p:core-uniqueness-maximality}
  Let $T$ be a complete, affine, $\lambda$-bounded theory. Then $T$ admits a unique extremally $\lambda$-saturated model, up to isomorphism. Moreover, this model is maximal.
\end{prop}
\begin{proof}
  Observe first that by \autoref{prp:SaturatedExistence}, $T$ admits an extremally $\lambda$-saturated model. By \autoref{lem:SaturatedUniversal}, any extremal model of density $\leq \lambda$ (so any extremal model) embeds affinely in every $\lambda$-saturated model.
  Therefore any two extremally $\lambda$-saturated models embed affinely into each other, so they have the same density character, and are thus isomorphic by \autoref{prp:SaturatedHomogeneous}.

  Assume we have a proper affine extension $M \prec^\aff N$ with $N$ extremal and $M$ extremally $\lambda$-saturated.
  We can embed affinely $N$ in $M$, yielding a proper affine embedding $M \hookrightarrow M$.
  As in the proof of \autoref{p:abs-cat-no-ee}, we construct an  affine chain of copies of $M$, of length $\lambda^+$.
  At successor stages we use the proper embedding we have produced.
  At limit stages we embed the (extremal) direct limit of what we have so far in a copy of $M$.
  The limit of the chain is extremal and has density character $\lambda^+$, a contradiction.
\end{proof}

\begin{prop}
  Let $T$ be a complete affine theory and let $M\models T$ be an extremal model. The following are equivalent:
  \begin{enumerate}
  \item There is $\lambda$ such that $T$ is $\lambda$-bounded and $M$ is its extremally $\lambda$-saturated model.
  \item $M$ is maximal.
  \item Every $p\in \cE_1(M)$ is realized in $M$.
  \end{enumerate}
\end{prop}
\begin{proof}
  \begin{cycprf}
  \item By \autoref{p:core-uniqueness-maximality}.
  \item Because every extreme type over an extremal model is realized in an extremal affine extension.
  \item[\impprev] Clear.
  \item[\impprev] Because every extremal model of $T$ has an extremal affine joint embedding with $M$ (\autoref{prop:ExtremalJEP}).
  \end{cycprf}
\end{proof}

\begin{theorem}
  \label{th:bounded-Poulsen}
  Let $T$ be an affinely complete, bounded, simplicial theory with no compact models.
  Then $T$ is a Poulsen theory.
\end{theorem}
\begin{proof}
  A Bauer theory with a non-compact extremal model has arbitrarily large extremal models. Thus $T$ cannot be Bauer and so, by the dichotomy theorem for complete simplicial theories (\autoref{th:dichotomy-simplicial-theories}), $T$ is a Poulsen theory.
\end{proof}

\begin{question}
  \label{q:abs-cat-simplicial}
  Do theories satisfying the hypotheses of \autoref{th:bounded-Poulsen} exist?
\end{question}

\begin{question}
  Is every complete, bounded affine theory absolutely categorical?
  And if the language is separable?
\end{question}


\part{Examples}
\label{part:Examples}

\section{Absolutely categorical Bauer theories}
\label{sec:bauer-compact-extremal}

\subsection{Probability algebras}
\label{ss:ex:PrA}

Natural affine theories arise when we consider continuous theories with affine axiomatizations. A prime example of this is the theory of probability measure algebras, which was already considered by Bagheri. In this subsection, we show how it fits in the general theory we have developed.

We consider the language $\cL_\PrA=\set{\mu,\cup,\cap,\neg,\bZero,\bOne}$ of probability algebras. The theory $\PrA$ consists of the axioms of Boolean algebras together with $\mu(\bOne) = 1$ and
\begin{equation*}
  \sup_{x, y} \varphi(x, y) = \inf_{x, y} \varphi(x, y) = 0,
\end{equation*}
where
\begin{equation*}
  \varphi(x, y) = \mu(x \cup y) + \mu(x \cap y) - \mu(x) - \mu(y).
\end{equation*}
This theory is affine and its models are all probability measure algebras. In continuous logic, one usually considers its model completion $\APrA$, the theory of \emph{atomless} probability algebras with the additional axiom
\begin{equation*}
  \sup_x \inf_y \big|\mu(y \cap x) - \frac{1}{2} \mu(x)\big| = 0,
\end{equation*}
which is however not affine. The theory $\APrA$ is $\aleph_0$-categorical and eliminates quantifiers (see \cite[\textsection16]{BenYaacov2008}).

\begin{theorem}
  \label{th:prob-alg}
  The following statements hold:
  \begin{enumerate}
  \item \label{i:thpa:compl} The theory $\PrA$ is a complete affine theory, which has affine quantifier elimination.
  \item \label{i:thpa:abs-cat} It is an absolutely categorical, Bauer theory and its unique extremal model is the two-point algebra $\set{\bZero, \bOne}$.
  \item \label{i:thpa:type-spaces} For every $n$, $\tS^\aff_n(\PrA) \cong \cM(2^n)$.
  \item \label{i:thpa:type-spaces-param} More generally, for any model $M$,
    \begin{equation}
      \label{eq:thpa:param}
      \tS^\aff_n(M) \cong \set{f \in (L^\infty(M))^{2^n} : f_\eps \geq 0, \sum_{\eps \in 2^n} f_\eps = 1 },
    \end{equation}
    where the set on the right is equipped with the (product of) the weak$^*$ topology on $L^\infty(M)$. In particular, the extreme points in $\tS^\aff_n(M)$ are exactly the types realized in $M$.
  \end{enumerate}
\end{theorem}
Items \ref{i:thpa:compl} and \ref{i:thpa:type-spaces} of the theorem are due to Bagheri.
\begin{proof}
  \ref{i:thpa:compl}
  Let $M_\lambda$ denote the separable model of $\APrA$ (i.e., the measure algebra of a standard, atomless probability space $(\Omega, \lambda)$). If $M \models \PrA$ is separable, then $L^1(\Omega, M)$ is an atomless, separable probability algebra, so isomorphic to $M_\lambda$. As the diagonal embedding $M \hookrightarrow L^1(\Omega, M)$ is affine, we obtain that every separable model of $M$ embeds affinely in $M_\lambda$, which implies that $\PrA$ is complete and that $M_\lambda$ realizes all types in $\tS^\aff_n(\PrA)$ for every $n$. This, together with the fact that $M_\lambda$ is \emph{ultrahomogeneous} (i.e., for every two tuples with the same quantifier-free type in $M_\lambda$ there is an automorphism sending one to the other), implies that $\PrA$ eliminates quantifiers.

  \ref{i:thpa:abs-cat}, \ref{i:thpa:type-spaces}
  Let $M_1$ be the two-point algebra $\set{\bZero, \bOne}$. As its elements are affinely definable, $M_1$ is an extremal model. Now \autoref{p:compact-absol-cat} tells us that $\PrA$ is absolutely categorical with unique extremal model $M_1$. As $M_1$ has no non-trivial automorphisms, by \autoref{p:M-sslash-G}, $\cE_n(\PrA) \cong M_1^n = 2^n$. To check that the theory is Bauer, we will apply \autoref{th:criterion-affine-reduction} to the model $M_1$: we have that $L^1(\Omega, M_1) \cong \MALG(\Omega)$, so its hypothesis is trivially verified. Thus $\tS^\aff_n(\PrA) \cong \cM(2^n)$. An explicit formula for calculating the measure $\nu \in \cM(2^n)$ corresponding to the type of a tuple $\bar a$ is given by
  \begin{equation*}
    \nu(\set{\eps}) = \mu\Big(\bigcap_{i<n} a_i^{\eps_i}\Big) \quad \text{ for } \eps \in 2^n,
  \end{equation*}
  where $a^0 = a$ and $a^1 = \neg a$.

  \ref{i:thpa:type-spaces-param}
  Consider the map $\Phi$ from $\tS^\aff_n(M)$ to the right-hand side of \autoref{eq:thpa:param} given by
  \begin{equation*}
    \Phi(p)_\eps = \Pr\Big(\bigcap_i a_i^{\eps_i} \mid M\Big) \text{ for } \eps \in 2^n,
  \end{equation*}
  where $\bar a \models p$ and $\Pr(\cdot \mid M)$ denotes the conditional probability with respect to $M$. We claim that this map is a well-defined affine homeomorphism.

  For simplicity of notation, we only give the proof for $n = 1$, i.e., for
  \begin{equation*}
    \Phi \colon \tS^\aff_1(M) \to B = \set{f \in L^\infty(M) : 0 \leq f \leq 1}
  \end{equation*}
  given by $\Phi(p) = \Pr({a \mid M})$ with $a \models p$. (Here we identify $f$ with $f_0$ from the notation used above.)

  Note that for any simple function $h = \sum_j \lambda_j \chi_{b_j} \in L^1(M)$, we have that
  \begin{equation*}
    \ip{\Phi(p), h} = \sum_j \lambda_j \mu(a \cap b_j).
  \end{equation*}
  This shows that $\Phi(p)$ is well-defined (i.e., does not depend on the realization $a$ of $p$) and that $\Phi$ is affine, continuous, and injective. Surjectivity follows from the fact that for $b \in M$, $\Phi(\tp^\aff(b)) = \chi_b$ for $b \in M$ and that convex combinations of characteristic functions are dense in $B$. The final statement follows from \autoref{ex:Linfty-convex-set}.
\end{proof}

The theory of probability algebras with parameters provides several interesting examples.

\begin{example}\label{ex:PrA-with-2parameters}
  Let $M_2$ be the probability algebra $\set{\bZero, \bOne, a, \neg a}$, where $a$ is a named atom of measure $1/2$. By the theorem above, $\tS^\aff_1(M_2)$ has four extreme points, namely $0, 1, \chi_a, \chi_{\neg a}$, and it has dimension $2$ (being a subspace of $L^\infty(M_2) = \R^2$). Thus $\tS^\aff_1(M_2)$ is not a simplex and $M_2$ is a finite structure whose affine theory is not Bauer.

  The structure $M_2$ is the only extremal model of its theory. Every non-trivial convex combination of copies of $M_2$ has at least $4$ atoms and any direct integral of copies of $M_2$ over a probability space that is not atomic is infinite. However, $\Th^\aff(M_2)$ has models with $3$ atoms, which give a counterexample for the existence of extremal decompositions for models of non-simplicial theories.
\end{example}

\begin{example}
  \label{ex:PrA-with-atomless-parameters}
  Let $M_\lambda$ be the atomless, separable model of $\PrA$ and let $T$ be the affine diagram of $M_\lambda$. Then  by \autoref{th:prob-alg}, all extreme types of $T$ are realized in $M_\lambda$, i.e., $T$ is absolutely categorical with a non-compact extremal model $M_\lambda$. Note that extreme types are dense in $\tS^\aff_n(T)$, so there is no ``extremal compactness theorem'', i.e., a set of conditions that is finitely realized in extremal models need not be realized in an extremal model. Being an extremal model is also not closed under ultraproducts. Note finally that $\tS^\aff_n(T)$ is not a simplex. An example where extreme types are dense but the theory is simplicial is given in \autoref{sec:PMP}.
\end{example}

\subsection{More algebraic examples}
\label{sec:more-algebr-exampl}

In general, it is difficult to have quantifier elimination in affine theories simply because there are not enough affine quantifier-free formulas. One way to counteract this is to have sufficiently many function symbols in the language. In this subsection, we give three more such examples, the first of which also illustrates a different approach to the measure algebra (which is recovered as the case $q=2$ of the next proposition).

\begin{prop}
  \label{p:finite-fields}
  Let $\bF_q$ be the finite field with $q$ elements considered as a discrete structure in the language of rings. Then $\Th^\aff(\bF_q)$ is an absolutely categorical Bauer theory with affine quantifier elimination.
\end{prop}
\begin{proof}
  Let $T = \Th^\aff(\bF_q)$. By \autoref{p:compact-absol-cat}, $T$ is absolutely categorical and its extremal model is an affine substructure of $\bF_q$. Finite structures with the $0$--$1$ metric have no proper affine substructures because such a structure $F$ with $n$ elements satisfies
  \begin{equation*}
    F \models \sup_{\bar x} \sum_{i, j = 1}^n d(x_i, x_j) = n(n-1),
  \end{equation*}
  so we conclude that $\bF_q$ is the unique extremal model of $T$. It remains to check that $T$ is Bauer and has affine quantifier elimination. For this, it suffices to see that for every continuous logic formula $\varphi(\bar x)$, there exists a quantifier-free, affine formula $\psi(\bar x)$ such that $\varphi^{\bF_q} = \psi^{\bF_q}$. Let $G = \Aut(\bF_q)$ and recall that $G$ is the cyclic group generated by the Frobenius automorphism. The formula $\varphi$ is $G$-invariant and the linear space of $G$-invariant functions on $\bF_q^n$ is generated by the characteristic functions of orbits, so it suffices to find, for every orbit $G \cdot \bar a$, an atomic formula which evaluates to $\chi_{G \cdot \bar a}$. Let $\delta_{\bar a}(\bar x)$ be the polynomial defined by
  \begin{equation*}
    \delta_{\bar a}(\bar x) = \prod_{i } (1 - (x_i-a_i)^{q-1})
  \end{equation*}
  and note that it evaluates to the Dirac function of $\bar a$. Let
  \begin{equation*}
    P(\bar x) = \sum_{\bar b \in G \cdot \bar a} \delta_{\bar b}(\bar x)
  \end{equation*}
  and note that the polynomial $P(\bar x)$ is $G$-invariant, so its coefficients are in the prime field and it corresponds to a term in our language. Finally observe that
  \begin{equation*}
    \chi_{G \cdot \bar a}(\bar x) = d(0, P(\bar x))^{\bF_q}. \qedhere
  \end{equation*}
\end{proof}

\begin{prop}
  \label{p:Zp}
  Let $p$ be a prime and let $\Z_p$ be the ring of $p$-adic integers considered as a continuous logic structure in the ring language with the standard valuation metric. Then $\Th^\aff(\Z_p)$ is an absolutely categorical Bauer theory.
\end{prop}
\begin{proof}
  Let $T = \Th^\aff(\Z_p)$, and note that the ring axioms are affine. That $T$ is absolutely categorical with (compact) extremal model $\Z_p$ is proved exactly like in the previous proposition, observing that $\Z_p$ has no proper subrings. To see that $T$ is Bauer, we will apply \autoref{th:criterion-affine-reduction}. As the automorphism group of $\Z_p$ is trivial, we need to show that if $\Omega$ is a standard atomless probability space, then the natural embedding
  \begin{equation*}
    \Aut(\Omega) \hookrightarrow \Aut(L^1(\Omega, \Z_p))
  \end{equation*}
  is an isomorphism. As $\Z_p$ is an integral domain, the idempotent elements in the ring $L^1(\Omega, \Z_p)$ are exactly the characteristic functions $\set{\chi_a : a \in \MALG(\Omega)}$.
  The measure algebra operations are also definable in the ring language: $\chi_{a \cap b} = \chi_a \chi_b$ and $\chi_{\neg a} = 1 - \chi_a$. Thus $\Aut(L^1(\Omega, \Z_p))$ acts on $\MALG(\Omega)$ by automorphisms, which gives us a retraction $\Aut(L^1(\Omega, \Z_p))\to \Aut(\Omega)$. Finally, note that if an automorphism of $L^1(\Omega, \Z_p)$ is the identity on all idempotents, then it is the identity everywhere because finite sums of involutions are dense in $L^1(\Omega, \Z_p)$.
\end{proof}

\begin{prop}\label{p:ex:[0,1]}
  Consider $V = [0, 1]$ as a continuous structure with the metric inherited from the reals and a signature in which there are terms for a dense set of continuous functions $V^n \to V$ for every $n$ (for example, the signature $\set{0, \half, \dotminus, \neg}$ considered in \cite{BenYaacovOnTheories}). Then $\Th^\aff(V)$ is an absolutely categorical Bauer theory with affine quantifier elimination.
  Moreover, $\Th^\aff(V)$ is equivalent (as a continuous logic theory) to the theory $\RV$ of $[0,1]$-valued random variables discussed in \cite[\textsection2]{BenYaacovOnTheories}.
\end{prop}
\begin{proof}
  That $\Th^\aff(V)$ is absolutely categorical with extremal model $V$ follows as in the above propositions. To check simultaneously that it is Bauer and that it has affine quantifier elimination, we observe that it satisfies the following strong form of affine reduction: the set of atomic formulas on $V^n$ is dense in the set of all continuous logic formulas. Indeed, if $\varphi \colon V^n \to [0, 1]$ is the interpretation of a continuous logic formula, then, by identifying $V$ with a subset of $\R$, we can approximate $\varphi$ with a term $\tau$ and then the formula $d(0, \tau(\bar x))$ approximates $\varphi$.

  For the moreover part, recall that the models of $\RV$ are precisely the structures of the form $L^1(\Omega,V)$ for a probability space $\Omega$. Then the claim follows from \autoref{th:abs-cat-models}, which shows that $\Th^\aff(V)$ and $\RV$ have the same models. Alternatively, one may observe that $\RV_\aff = \Th^\aff(V)$ because $V\preceq^\aff L^1(\Omega,V)$ for every $\Omega$, and that the axioms for $\RV$ given in \cite[\textsection2]{BenYaacovOnTheories} are indeed affine.
\end{proof}


\section{Classical theories}
\label{sec:ex:classical}

A continuous logic structure is \emph{classical} if every basic predicate, including the metric, is $\set{0,1}$-valued. A continuous logic $\cL$-theory is a \emph{classical theory} if all its models are classical. This property can be axiomatized by the conditions
\begin{equation*}
  0\leq \inf_x P(x),\quad \sup_x P(x)\leq 1,\quad \sup_x\big( P(x)\wedge (1-P(x))\big) \leq 0,
\end{equation*}
where $P(x)$ varies over all basic $\cL$-predicates. We will denote the collection of all these conditions by $\ClassL$, so that an $\cL$-theory $C$ is classical if and only if $C\models \ClassL$. Note that $\ClassL$ is a universal theory. It is certainly not affinely axiomatizable because non-trivial direct integrals of classical structures are not classical.

If $C$ is any classical theory, we know from \autoref{p:ext-models-QAff} that the extremal models of $C_\aff$ are substructures of models of $C$, and are thus classical structures. In fact, much more is true.

\begin{theorem}\label{th:classical-theories-general}
  Let $C$ be a classical $\cL$-theory. The following hold:
  \begin{enumerate}
  \item \label{i:class-1} The canonical maps $\rho_x^\aff\colon \tS_x^\cont(C)\to \tS_x^\aff(C_\aff)$ satisfy $\rho_x^\aff(\tS_x^\cont(C)) = \cE_x(C_\aff)$, and restrict to $\dtp$-isometric $\tau$-homeomorphisms $\rho_x^\aff \colon \tS_x^\cont(C) \to \cE_x(C_\aff)$.
  \item \label{i:class-2} The theory $C$ has delta-convex reduction.
  \item \label{i:class-3} For a model $M\models C_\aff$ the following are equivalent:
    \begin{enumerate}
    \item $M$ is extremal;
    \item $M\models C$;
    \item $M$ is classical.
    \end{enumerate}
    In particular, $C\equiv C_\aff\cup\ClassL$.
  \end{enumerate}
\end{theorem}
\begin{proof}
  \autoref{i:class-1}, \autoref{i:class-2}.
  We prove first that $\rho_x^\aff(\tS_x^\cont(C))\subseteq \cE_x(C_\aff)$ (the inverse inclusion holds by \autoref{l:image-rho-aff}). More precisely, we argue by induction on the complexity of the affine formula $\varphi$ that for any tuple of variables $x$, if $q\in \tS_x^\cont(C)$ and $p_0,p_1\in \tS_x^\aff(C_\aff)$ are such that $\rho_x^\aff(q) = \half p_0 + \half p_1$, then $\varphi(p_0) = \varphi(p_1)$.

  If $\varphi$ is atomic, then $\varphi(q)\in \set{0,1}$ (as prescribed by $C$) and $\varphi(p_0), \varphi(p_1) \in [0,1]$ (as prescribed by $C_\aff$), so the fact that $\varphi(q) = \half \varphi(p_0) + \half \varphi(p_1)$ implies that $\varphi(p_0) = \varphi(p_1)$. The inductive step for affine combinations is obvious. Suppose finally that $\varphi(x) = \inf_y\psi(x,y)$ for some affine formula $\psi$ for which the hypothesis holds. Given $q\in\tS_x^\cont(C)$, take $q'\in\tS_{xy}^\cont(C)$ extending $q$ and such that $\varphi(q) = \psi(q')$. If $\rho_x^\aff(q) = \half p_0 + \half p_1$, then by \autoref{lemma:TypeConvexCombinationLift}\autoref{item:TypeConvexCombinationLiftFinite}, there exist $p_0',p_1'\in\tS_{xy}^\aff(C_\aff)$ extending $p_0$ and $p_1$, respectively, with the property that $\rho_{xy}^\aff(q') = \half p_0' + \half p_1'$. As $p_0'$ extends $p_0$, we have $\varphi(p_0)\leq \psi(p_0')$, and similarly $\varphi(p_1)\leq \psi(p_1')$. Hence
  $$\varphi(q) = \half \varphi(p_0) + \half \varphi(p_1) \leq \half \psi(p_0') + \half \psi(p_1') = \psi(q') = \varphi(q),$$
  and we deduce that $\varphi(p_0) = \psi(p_0')$ and $\varphi(p_1) = \psi(p_1')$. On the other hand, by the inductive hypothesis, $\psi(p_0') = \psi(p_1')$. It follows that $\varphi(p_0) = \varphi(p_1)$, as desired.

  We have thus established that $\rho_x^\aff(\tS^\cont_x(C)) = \cE_x(C_\aff)$. In particular, the spaces of extreme types of $C_\aff$ are closed. Also, $(C_\aff)_\ext \sub C$, so $\tS_x^\cont(C) \sub \tS_x^\cont((C_\aff)_\ext)$. Now it follows from \autoref{cor:E(T)-closed-iota-homeo} that the maps $\rho_x^\aff \colon \tS_x^\cont(C) \to \cE_x(C_\aff)$ are $\dtp$-isometric $\tau$-homeomorphisms and $(C_\aff)_\ext \equiv C$. \autoref{prop:latt-red-for-closed-E(T)} implies that $C$ has delta-convex reduction.

  \autoref{i:class-3} All this applies, in particular, to the base universal theory $\ClassL$. It follows that for any classical structure $N$ and any affine substructure $M\preceq^\aff N$, one has $M\preceq^\cont N$.

  Let $M\models C_\aff$. If we assume that $M$ is an extremal model, then $M$ embeds affinely into a model $N\models C$ (by \autoref{p:ext-models-QAff}), so by the remark in the previous paragraph, $M\preceq^\cont N$, and in particular $M\models C$ and $M$ is classical. Assume conversely that $M$ is a classical structure, and let $C'=C_\aff\cup\ClassL$, so that $M\models C'$. Note that $C'_\aff = C_\aff$ because $C\models C'$. By what we have proven above, we have $\rho_x^\aff(\tS_x^\cont(C'))\subseteq \cE_x(C_\aff)$ for every variable $x$. From this we can deduce that $M$ is an extremal model of $C_\aff$, as desired.
\end{proof}

In general, a classical theory need not have affine reduction, as illustrated by the following simple examples.

\begin{example}
  Let $P,P'$ be two unary predicates, let $\cL=\set{P,P'}$, and let $C$ be the classical theory stating that the model has a unique element. Then every type space $\tS_x^\cont(C)$ consists of precisely four points, corresponding to the four possibilities for the value of the pair $(P,P')$ on the unique element of the model. In turn, $\tS_x^\aff(C_\aff)$ is a convex set with four extreme points (say, $p_{00}$, $p_{10}$ ,$ p_{01}$, $p_{11}$, according to the value of $(P,P')$), but is not a simplex, because $\half p_{10} + \half p_{01} = \half p_{00} + \half p_{11}$. That is, $\tS_x^\aff(C_\aff)$ is not the tetrahedron $\cM(\tS_x^\cont(C))$, and hence $C$ does not have affine reduction, by \autoref{th:Bauer-correspondence-Q}.
\end{example}

\begin{example}
  Let $\cL = \set{S,d,\bZero,\bOne}\cup\set{S',d',\bZero',\bOne'}$ be a language with two distinct sorts $(S,d)$ and  $(S',d')$ (both metrics bounded by $1$) and two distinct constants on each sort. Let $C$ be the theory of the classical $\cL$-structure consisting of distinct interpretations for the constants and no other elements, i.e., the $\cL$-theory axiomatized by $d(\bZero,\bOne)=1$, $\sup_x d(x,\bZero)\wedge d(x,\bOne)=0$, $d(\bZero',\bOne')=1$ and $\sup_{x'}d(x',\bZero')\wedge d(x',\bOne')=0$. If $x$ and $x'$ are single variables of the sorts $S$ and $S'$, respectively, then similarly to the previous example, $\tS_{xx'}^\aff(C_\aff) \cong \tS_x^\aff(C_\aff)\times \tS_{x'}^\aff(C_\aff) \cong [0,1]^2$ is not a simplex. Hence $C$ does not have affine reduction.
\end{example}

On the other hand, somewhat surprisingly, \emph{complete, one-sorted} classical theories do have affine reduction. We will first establish this in the simplest case, that of sets with no structure.

Let $\cL_\emptyset$ the one-sorted empty language (i.e., containing only the distance symbol, say with bound $1$), and for each $\ell\in\set{1,2,\dots,\infty}$ let $M_\ell$ be the unique countable, classical $\cL_\emptyset$-structure with precisely $\ell$ elements. In the statement of following theorem we use the notation from \autoref{sec:criterion-affine-reduction}.

\begin{theorem}\label{th:L1Mk-automorphisms}
  Let $M_\ell$ be as above and let $(\Omega,\mu)$ be a standard probability space. If $\ell\geq 2$, then
  $$\Aut(L^1(\Omega,M_\ell)) = L(\Omega,\Sym(M_\ell)) \rtimes \Aut(\mu).$$
\end{theorem}
\begin{proof}
  We make the preliminary observation that if $U\in \Aut(L^1(\Omega,M_\ell))$ and $f,f'\in L^1(\Omega,M_\ell)$ are two elements that are distinct almost everywhere, then so are $U(f)$ and $U(f')$, because
  $$\mu(\oset{U(f)\neq U(f')}) = d(U(f),U(f')) = d(f,f') = 1.$$
  Now let us fix $U\in \Aut(L^1(\Omega,M_\ell))$ and an enumeration without repetitions $(b_i)_{i<\ell}$ of $M_\ell$. Let $e_i\in L^1(\Omega,M_\ell)$ be the constant section taking on the value $b_i$, and define $f_i = U(e_i)$, which we identify with a Borel function $f_i\colon \Omega\to M_\ell$ (i.e., we choose some Borel representative for $f_i$). Given $\omega\in\Omega$, let $T(\omega)\colon M_\ell\to M_\ell$ be the map defined by
  $$T(\omega)(b_i) = f_i(\omega).$$
  We claim that $T(\omega)$ is bijective almost surely.

  Indeed, by our preliminary observation, $T(\omega)$ is injective almost surely. For finite $\ell$ this is enough, so let us consider the case $\ell=\infty$. Since $M_\infty$ is countable, if $T(\omega)$ is not surjective almost surely, then there exist $b\in M_\infty$ and a Borel set $Z\subseteq\Omega$ with $\mu(Z)>0$ such that $b\neq f_i(\omega)$ for every $i\geq 0$ and every $\omega\in Z$. Let $f\colon \Omega\to M_\infty$ be given by $f(\omega)=b$ for $\omega\in Z$ and $f(\omega)=f_0(\omega)$ for $\omega\notin Z$. Hence $f$ is distinct almost everywhere from $f_i$, for every $i\geq 1$. By our preliminary observation, it follows that $U^{-1}(f)(\omega)\neq b_i$ almost surely, for every $i\geq 1$. This is only possible if $U^{-1}(f)=e_0$, that is, if $f=f_0$, which is in contradiction with the fact that $f(\omega)=b\neq f_0(\omega)$ for any $\omega\in Z$.

  We have thus established that $T(\omega)\in\Sym(M_\ell)$ for almost every $\omega\in\Omega$. In addition, the map $\omega\mapsto T(\omega)$ is measurable, so it defines an element $T\in L(\Omega,\Aut(M_\ell))\leq \Aut(L^1(\Omega,M_\ell))$. Let us now consider the composition $S=T^{-1}U$, which has the property that $S(e_i)=e_i$ for every $i$. We will prove that $S\in\Aut(\mu)$, which is enough to conclude.

  For each pair $i,j<k$ with $i\neq j$ and measurable set $A\subseteq\Omega$, let $e^A_{ij}\in L^1(\Omega,M_\ell)$ be the measurable section that takes the constant value $b_i$ on $A$ and the constant value $b_j$ on $\Omega\setminus A$. As $e^A_{ij}$ is everywhere distinct from $b_k$ for any $k\notin\set{i,j}$, so has to be its image $S(e^A_{ij})$, up to measure zero. In other words, there exists a measurable set $S_*^{ij}(A)\subseteq\Omega$ such that $S(e^A_{ij}) = e^{S_*^{ij}(A)}_{ij}$.

  We observe first that
  $$\mu(S_*^{ij}(A)) = d(e_{ij}^{S^{ij}_*(A)},e_j) = d(e_{ij}^A,e_j) = \mu(A).$$
  Next we claim that $S_*^{ij}(A)$ does not depend on $i$ or $j$. Indeed, we have:
  \begin{itemize}
  \item $1 = d(e_{ij}^A,e_{ji}^A) = d(e_{ij}^{S^{ij}_*(A)},e_{ji}^{S_*^{ji}(A)})$, whence $S^{ij}_*(A) = S^{ji}_*(A)$;
  \item for $k\notin\set{i,j}$, $\mu(A) = d(e_{ij}^A,e_{kj}^A) = d(e_{ij}^{S_*^{ij}(A)},e_{kj}^{S_*^{kj}(A)}) = \mu(S_*^{ij}(A)\cup S_*^{kj}(A))$, whence $S_*^{ij}(A) = S_*^{kj}(A)$.
  \end{itemize}
  We may thus write $S_*(A)=S_*^{ij}(A)$.

  We claim now that the map $S_*$ defines an automorphism of the measure algebra of $\Omega$. Note that we may also define a map $(S^{-1})_*$ by applying the previous construction to the inverse $S^{-1}$ of $S$. This map satisfies $e^A_{ij} = e^{S_*((S^{-1})_*(A))}_{ij}$, which implies $A=S_*((S^{-1})_*(A))$ and shows that $S_*$ is surjective. We know as well that $S_*$ is measure-preserving. Therefore, we are only left to show that it preserves disjoint unions. If $A, B \sub \Omega$ with $A \cap B = \emptyset$, then:
  \begin{itemize}
  \item $\mu(A)+\mu(B) = d(e_{ij}^A,e_{ij}^B) = d(e_{ij}^{S_*(A)},e_{ij}^{S_*(B)}) = \mu(S_*(A)\triangle S_*(B))$, whence $\mu(S_*(A))+\mu(S_*(B)) = \mu(S_*(A)\cup S_*(B))$, showing that $S_*(A)$ and $S_*(B)$ are disjoint;
  \item $\mu(S_*(A) \triangle S_*(A\cup B)) = d(e_{ij}^{S_*(A)},e_{ij}^{S_*(A\cup B)}) = d(e_{ij}^A,e_{ij}^{A\cup B}) = \mu(S_*(B))$, and similarly, $\mu(S_*(B) \triangle S_*(A\cup B)) = \mu(S_*(A))$.
  \end{itemize}
  We deduce that $S_*(A \cup B) = S_*(A) \cup S_*(B)$, for if $X$ and $Y$ are disjoint elements of a measure algebra, then the union $X\cup Y$ is the unique $Z$ such that $\mu(Z)=\mu(X)+\mu(Y)$, $\mu(X\triangle Z)=\mu(Y)$, and $\mu(Y\triangle Z)=\mu(X)$.

  To conclude, we check that $S$ coincides with the automorphism in $\Aut(\mu)\leq \Aut(L^1(\Omega,M_\ell))$ induced by $S_*$, which we also denote by $S_*$. If we let $I=S^{-1}S_*$, then, by construction, $I$ fixes the measurable sections $e_i$ and~$e^A_{ij}$. Now let $f\in L^1(\Omega,M_\ell)$ be arbitrary, and let us consider the measurable partitions $(A_i)_{i<\ell}$ and $(B_i)_{i<\ell}$ of $\Omega$ given by $A_i=f^{-1}(b_i)$, $B_i=I(f)^{-1}(b_i)$. We then have, for all distinct $i,j<\ell$,
  \begin{itemize}
  \item $\sum_{k\neq i}\mu(A_k) = d(f,e_i) = d(I(f),e_i) = \sum_{k\neq i}\mu(B_k)$, whence $\mu(A_i)=\mu(B_i)$;
  \item $\sum_{k\notin\set{i,j}}\mu(A_k) = d(f,e_{ij}^{A_i}) = d(I(f),e_{ij}^{A_i}) = \mu(B_i\setminus A_i) + \mu(B_j\cap A_i) + \sum_{k\notin\set{i,j}}\mu(B_k)\geq \mu(B_i\setminus A_i) + \sum_{k\notin\set{i,j}}\mu(A_k)$, which implies that $B_i=A_i$ as they are of equal measure.
  \end{itemize}
  Hence $I(f)=f$, completing the proof.
\end{proof}

For each $\ell\in\set{1,2,\dots,\infty}$, let $C_\ell$ be the theory of the structure $M_\ell$.

\begin{theorem}\label{th:affine-reduction-pure-sets}
  The theory $C_\ell$ has affine reduction. Moreover, every continuous formula is equal to an affine formula modulo $C_\ell$ (rather than just approximated by affine formulas).
\end{theorem}
\begin{proof}
  The case $\ell=1$ is trivial; for $\ell\geq 2$, affine reduction follows from \autoref{th:L1Mk-automorphisms} and \autoref{th:criterion-affine-reduction}. For the moreover part, it suffices to observe that since the type spaces $\tS^\cont_n(C_\ell)$ are finite, the vector spaces $C(\tS^\cont_n(C_\ell))$ (i.e., the spaces of continuous logic formuals) are finite-dimensional.
\end{proof}

\begin{question}
  Does the theory $\Class_{\cL_\emptyset}$ have affine reduction?
\end{question}

As every complete, one-sorted classical theory implies $C_\ell$ for some $\ell$, we obtain the following.

\begin{cor}\label{c:classical-d-max-is-affine}
  Let $C$ be a complete, classical theory in a one-sorted language. Then for every $n$, there is an affine $\cL_\emptyset$-formula $\theta_n(\bar x,\bar y)$ such that for every $M\models C$ and $\bar a,\bar b\in M^n$, we have $\theta_n(\bar a,\bar b) = 0$ if $\bar a=\bar b$ and $\theta_n(\bar a,\bar b) = 1$ otherwise.
\end{cor}

\begin{question}
  Is there a ``simple'' explicit expression for the affine formula $\theta_n$?
\end{question}

Let us say that a formula is \emph{classical} if it is constructed from the atomic formulas using only the connectives $\vee$, $\varphi\mapsto 1-\varphi$, the constant formula $1$, and quantifiers.
The following general theorem is an easy consequence of the previous corollary.

\begin{theorem}\label{th:complete-classical-theories}
  Let $C$ be a complete, classical theory in a one-sorted language. The following hold:
  \begin{enumerate}
  \item The theory $C$ has affine reduction. Moreover, every classical formula is equal to an affine formula modulo $C$ (rather than just approximated by affine formulas).
  \item The theory $C_\aff$ is a Bauer theory, and its models are precisely the direct integrals of models of $C$.
  \end{enumerate}
\end{theorem}
\begin{proof}
  Continuous formulas modulo a classical theory can be uniformly approximated by continuous formulas with finite range, and the latter are linear combinations of classical formulas. So to prove that $C$ has affine reduction it is enough to show that classical formulas are affine. For this, we proceed by induction on the construction of the formula, with the only non-obvious case being that of the connective $\vee$. Now, if $\varphi(\bar x)$ and $\psi(\bar x)$ are classical formulas, then either $\varphi\vee\psi = 1$ in every model of $C$, or, by completeness, $C$ implies $\inf_{\bar y}\varphi(\bar y)+\psi(\bar y) = 0$. In the first case we are done, and in the second, $C$ implies the equality
  $$\varphi(\bar x)\vee\psi(\bar x) = \inf_{\bar y} \varphi(\bar y) + \psi(\bar y) + \theta_{|\bar x|}(\bar x,\bar y),$$
  where $\theta_{|\bar x|}$ is as given by \autoref{c:classical-d-max-is-affine}. From this, we deduce that if $\varphi$ and $\psi$ are equivalent to affine formulas, then so is $\varphi\vee\psi$.

  The second part follows from the first one together with \autoref{th:Bauer-correspondence-Q} and \autoref{th:classical-theories-general}.
\end{proof}

\begin{example}
  Consider the language $\cL=\set{\bZero,\bOne}$ and let $T_2$ be the affine theory of the $\cL$-structure $\set{\bZero,\bOne}$ consisting of two named points at distance one. Then $T_2$ is an absolutely categorical Bauer theory with $\tS^\aff_n(T_2)\cong\cM(2^n)$. On the other hand, $T_2$ is a reduct of the theory $\PrA$ of probability algebras discussed in \autoref{ss:ex:PrA}. It follows from the description of the type spaces of both theories that the canonical restriction map $\tS^\aff_n(\PrA)\to \tS^\aff_n(T_2)$ is an affine homeomorphism. Hence, $\PrA$ is an affine definitional expansion of $T_2$, i.e., the Boolean algebra operations (and of course the measure, which is just the distance to $\bZero$) can be defined by affine formulas in the language $\cL$. These formulas can also be computed explicitly.

  It is interesting to point out that while the theory $\PrA$ has affine quantifier elimination (see \autoref{th:prob-alg}), the theory $T_2$ does not. Hence, the theory of probability algebras can be seen as an ``affine Morleyization'' of the theory of two named points at distance one.
\end{example}

If $C$ is a classical theory, the probability measures over the type spaces $\tS_x^\cont(C)$ are known as \emph{Keisler measures}, and have been the subject of intensive study by model-theorists in recent years. As is well-known, Keisler measures can also be seen as types over the atomless randomization $C^\Rand$, but this requires a change of language (essentially, a Morleyization) which enforces this correspondence. The previous theorem provides a subtler identification bewteen Keisler measures and types. Indeed, if $C$ satisfies the hypothesis of the theorem, then the natural maps $\iota_x^*\colon \cM(\tS_x^\cont(C))\to \tS_x^\aff(C_\aff)$ are affine homeomorphisms.

\begin{cor}
  Keisler measures over a complete, one-sorted classical theory can be identified with the affine types over its affine part.
\end{cor}

We end this section with an example of a direct integral of infinite sets which shows that in the statement of the uniqueness theorem of the extremal decomposition, one cannot, in general, obtain isomorphism almost surely.
\begin{example}
  \label{ex:random-set}
  We denote by $\omega_1$ the first uncountable ordinal and by $\omega$ the first infinite ordinal. Consider the compact space $2^{\omega_1}$ equipped with the Bernoulli measure $\mu = (\half \delta_0 + \half \delta_1)^{\otimes \omega_1}$, which is a Radon measure. We will work with the $\sigma$-algebra of $\mu$-measurable sets (i.e., the completion of the product $\sigma$-algebra with respect to $\mu$). As usual, we will identify $2^{\omega_1}$ with the family of subsets of $\omega_1$. Our first goal is to make sense of the direct integral
  \begin{equation*}
    \int_{2^{\omega_1}}^\oplus z \ud \mu(z),
  \end{equation*}
  where $z$ is a considered as a classical structure in the language consisting only of the $\set{0, 1}$-valued metric.

  For $z \sub \omega_1$, we denote by $\ot(z)$ the \emph{order type} of $z$, i.e., the unique ordinal isomorphic to $(z, <)$. Note that for every $\alpha < \omega_1$, almost surely $\ot(z) > \alpha$. This is because for every $\beta \leq \alpha$, the intersection $z \cap [\omega \cdot \beta, \omega \cdot (\beta+1))$ is infinite almost surely. In particular, the set $\Omega = \set{z \in 2^{\omega_1} : z \text{ is infinite}}$ has measure $1$.

  Next we make $(z : z \in \Omega)$ into a measurable field of structures. Define an enumeration $(e_\alpha : \alpha < \omega_1)$ by:
  \begin{equation*}
    e_\alpha(z) = \alpha \text{th element of } z.
  \end{equation*}
  This is only defined if $\ot(z) > \alpha$, but this holds with probability $1$. (On the complement, we can define $e_\alpha(z)$ arbitrarily.)
  It is trivial to verify that the conditions of \autoref{defn:MeasurableField} hold and thus we can form the direct integral $M = \int_{\Omega}^\oplus z \ud \mu(z)$.

  Now let $\Omega_k = \set{z \in \Omega : |z| = \aleph_k}$ for $k=0,1$. The sets $\Omega_0$ and $\Omega_1$ form a partition of $\Omega$, neither of them is measurable, and the measure $\mu$ concentrates on both. To see this, note that the only Baire set in $2^{\omega_1}$ disjoint from $\Omega_k$ is the empty set, because a Baire set in $2^{\omega_1}$ depends on only countably many coordinates. Thus we may apply \autoref{l:concentration-direct-integral} and obtain the isomorphic direct integrals
  $$M \cong \int_{\Omega_0}^\oplus z \ud \mu_0(z) \cong \int_{\Omega_1}^\oplus z \ud \mu_1(z)$$
  over the induced probability spaces $(\Omega_0,\mu_0)$ and $(\Omega_1,\mu_1)$.
  By our preceding results, the measurable fields $(z:z\in\Omega_0)$ and $(z:z\in\Omega_1)$ consist of extremal models of the affine theory of infinite sets, which are simplicial and even Bauer. However, the models from the two fields are non-isomorphic, showing that \autoref{th:uniqueness-decomposition} cannot possibly be strengthened to a statement about isomorphisms between the models of the fields.

  One may also observe that the structure $M$ is isomorphic to the direct multiple $L^1(\Omega,\mu,\omega_1)$. Indeed, the map $L^1(\Omega,\mu,\omega_1)\to \int_{\Omega}^\oplus z \ud \mu(z)$ sending a simple section $f$ defined on a partition $(A_i)_{i<n}$ of $\Omega$ by $f(z)=\alpha$ if and only if $z\in A_i$, to the simple section $f'$ defined by $f'(z)=e_{\alpha}(z)$ if and only if $z\in A_i$, is an isomorphism between the two structures.
\end{example}


\section{Hilbert spaces}
\label{sec:ex:HS}

The usual axioms that define Hilbert spaces through the vector space structure and the inner product look very much first-order, and even affine. A well-known pitfall, though, is that as metric sorts in continuous logic are required to be bounded, Hilbert spaces cannot be seen as metric structures on their own, but only via some distinguished bounded subset or a family of bounded subsets. The canonical choices are the unit ball or the unit sphere (if one prefers a one-sorted language), or the collection of all balls with integer radius (for a multi-sorted, countable language). In continuous logic, this is not a serious hurdle, since one can write axioms whose models are precisely the unit balls (or spheres, or families of balls) of Hilbert spaces, and all the information about the corresponding Hilbert spaces is encoded in those structures. In affine logic, the situation is more delicate, but rather interesting.

In what follows, we will consider Hilbert spaces through their unit spheres. Since the unit spheres are affinely definable in the unit balls (the distance function being given by the formula $1-d(x,0)$), the formulas that we write are also expressible in the unit balls, and the proofs below can be adapted to that setting.

We denote by $\HS$ the common continuous logic theory of the unit spheres of real, non-trivial Hilbert spaces, in the language consisting only of the inner product and the metric. We let $\AHS\coloneqq \HS_\aff$ be its affine part.

As it is easy to see, a non-trivial convex combination of unit spheres gives a model of $\AHS$ that is not a model of $\HS$. That is, the unit spheres of Hilbert spaces are not affinely axiomatizable. However, by the GNS construction, every model $M\models\AHS$ generates a Hilbert space $\cH_M$ in a canonical way, so that $M$ becomes a subset of the unit sphere of $\cH_M$. Indeed,
\begin{equation}\label{eq:AHS-positive-type}
  \AHS\models \qinf_{x_1,\dots,x_n}\sum_{i,j}\lambda_i \lambda_j\ip{x_i,x_j} \geq 0 \quad \text{for all } n\in\N, \bar \lambda \in\R^n.
\end{equation}
and
\begin{equation}
  \label{eq:AHS-sphere}
  \AHS \models \ip{x, x} = 1.
\end{equation}
Thus the map $\ip{\cdot, \cdot}\colon M\times M\to \R$ is a kernel of positive type on any model $M\models \AHS$, and as is well-known, this guarantees the existence of a uniquely determined Hilbert space $\cH_M$ containing $M$ whose inner product extends $\ip{\cdot, \cdot}$, and such that the subspace generated by $M$ is dense in $\cH_M$. See, for example, \cite[Appendix~C]{Bekka2008} (the development there is for complex Hilbert spaces but for real ones it is similar). By the functoriality of the construction, if $M' \sub M$, then $\cH_{M'} \sub \cH_M$ and, in particular, the group of automorphisms of $M$ embeds into the group of automorphisms of $\cH_M$.

We also note that in a model of $\HS$, we have that
\begin{equation*}
  \big\| \sum_i \lambda_i x_i \big\| = \qsup_z \sum_i \lambda_i \ip{z, x_i}
\end{equation*}
for all finite tuples $\bar \lambda$ of real numbers and $\bar x$ of elements of the model. Thus we may consider $\nm{\sum_i \lambda_i x_i}$ as an affine formula. In particular, we have that
\begin{equation}
  \label{eq:HS-d-from-norm}
  \AHS \models d(x, y) = \nm{x-y},
\end{equation}
so the metric is affinely definable from the inner product.

The models of $\HS$ are precisely the structures satisfying the universal axioms \autoref{eq:AHS-positive-type}, \autoref{eq:AHS-sphere}, and \autoref{eq:HS-d-from-norm} and for which $M$ is the whole unit sphere of $\cH_M$. This condition is easy to axiomatize in continuous logic. On the other hand, it is unclear to us what are explicit affine axioms for $\AHS$.

\begin{remark}
  If $M'$ is the sphere of $\cH_M$, then the inclusion $M\subseteq M'$ is typically not an embedding of structures, because the metric on $M$ will not be the restriction of the metric on $M'$ (as the formula which defines it is not quantifier-free).
\end{remark}

We will show next that the extremal models of $\AHS$ are precisely the models of $\HS$.
We recall from \autoref{sec:convex-formulas} that a continuous logic theory $Q$ is universal-delta-convex if every affine submodel of a model of $Q$ is also a model of $Q$.

\begin{lemma}
  The theory $\HS$ is universal-delta-convex.
\end{lemma}
\begin{proof}
  Let $M\models\HS$ be a Hilbert sphere and let $M'\preceq^\aff M$ be an affine submodel. Take finitely many elements $a_i\in M'$ and scalars $\lambda_i\in\R$ and suppose $k=\|\sum_i \lambda_i a_i\|\neq 0$. Then $b=\frac{1}{k}\sum_i \lambda_i a_i$ belongs to $\cH_{M'}\cap M$, and we claim that $b\in M'$. For this it suffices to observe that $b$ is definable over the $a_i$ in $M$ by the affine formula
  $$d(x,b) = \big\| x - \frac{1}{k}\sum_i \lambda_i a_i \big\|.$$
  Hence $M'\models \inf_x d(x,b)=0$, and we see that $b\in M'$.

  By the density of finite linear combinations, we conclude that $M'$ is the sphere of $\cH_{M'}$, and so $M'\models\HS$, as desired.
\end{proof}

Given $d\in\set{1,2,\dots,\infty}$, denote by $\HS_d$ the common continuous logic theory of the unit spheres of Hilbert spaces of dimension $d$. We also let $\AHS_d\coloneqq (\HS_d)_\aff$.

\begin{lemma}
  For every $d\in\set{1,2,\dots,\infty}$, the theory $\HS_d$ is universal-delta-convex.
\end{lemma}
\begin{proof}
  Let $M'\preceq^\aff M\models \HS_d$. We have that $M'\models \HS$ by the previous lemma. We consider the affine formula
  $$\varphi(x,y)= \sqrt2 \nm{x + y} - \ip{x,y}.$$
  We note that in models of $\HS$, we have $\varphi(x,y)=2\sqrt{1+\ip{x,y}}-\ip{x,y}$, and one sees readily that this function attains its maximum $2$ precisely when $\ip{x,y}=0$. For each $n\geq 1$, let
  \begin{equation}\label{eq:dim-n-hilbert-spaces}
    \theta_n\coloneqq \qsup_{x_1,\dots,x_n}\sum_{i<j}\varphi(x_i,x_j).
  \end{equation}
  Hence for every finite $n\leq d$, the axiom
  $$\theta_n=n(n-1)$$
  holds in $M$, and therefore in $M'$. Since the spheres of Hilbert spaces are saturated in continuous logic, there are elements $x_1,\dots,x_n$ in $M'$ realizing the supremum. It follows that $\ip{x_i,x_j}=0$ for each $i<j$, and so that $\cH_{M'}$ has dimension at least $n$ for each $n\leq d$. Since $\dim \cH_{M'}\leq\dim \cH_M= d$, we conclude that $\cH_{M'}$ has dimension $d$.
\end{proof}

\begin{theorem}\label{th:AHS-extremal-models}
  The extremal models of $\AHS$ are precisely the models of $\HS$. Similarly, for any $d\in\set{1,2,\dots,\infty}$, the extremal models of $\AHS_d$ are precisely the models of $\HS_d$.
\end{theorem}
\begin{proof}
  By \autoref{c:uda-ext-models}, the extremal models of $\AHS$ are models of $\HS$, and similarly for $\AHS_d$. For the converse, it is enough to show that each theory $\AHS_d$ is an extreme completion of $\AHS$, and that the models of $\HS_d$ are extremal models of $\AHS_d$.

  We start with the latter. In the case $d<\infty$, as the extremal models of $\AHS_d$ are models of $\HS_d$, which has only one model up to isomorphism, we deduce that the unique model of $\HS_d$ is an extremal model of $\AHS_d$. We can deduce similarly that the only separable model of $\HS_\infty$ is an extremal model of $\AHS_\infty$. Since for checking extremality it is enough to consider separable affine substructures, we conclude as well that every model of $\HS_\infty$ is an extremal model of $\AHS_\infty$.

  Now we show that each $\AHS_d$ is an extreme point of $\tS^\aff_0(\AHS)$. For each $n\geq 1$, let $\theta_n$ be the formula defined in \autoref{eq:dim-n-hilbert-spaces}. Since $\HS\models \theta_n\leq n(n-1)$, the completions of $\AHS$ that satisfy $\theta_n= n(n-1)$ form a face of $\tS^\aff_0(\AHS)$ that we denote by $F_n$. Since every extremal model of $\AHS$ is a model of $\HS$, we see that the extreme points of $F_n$ are contained in the set $\set{\AHS_d:d\geq n}$. If $\AHS_n$ is not an extreme point of $F_n$, then every extreme point of $F_n$ satisfies the condition $\theta_{n+1}=(n+1)n$,
  and hence so does every theory in $F_n$. Since $\AHS_n$ does not satisfy this condition, this is a contradiction. We can conclude that $\AHS_d$ is an extreme completion of $\AHS$ for every $d<\infty$. On the other hand, we see that the intersection $F_\infty=\bigcap_{n\geq 1}F_n$ is a face of $\tS^\aff_0(\AHS)$ containing $\AHS_\infty$ which cannot contain any other extreme point. We conclude that $F_\infty=\set{\AHS_\infty}$, and so $\AHS_\infty$ is also an extreme completion of $\AHS$.
\end{proof}

\begin{cor}
  The sets of extreme types $\cE_x(\AHS)$ and $\cE_x(\AHS_d)$ are closed in the respective type spaces.
\end{cor}
\begin{proof}
  Immediate from the previous theorem and \autoref{c:closed-extreme-types}.
\end{proof}

\begin{remark}
  Contrary to the situation in \autoref{ex:PrA-with-atomless-parameters}, the affine diagram of a model $M\models\HS_\infty$ is not absolutely categorical. Indeed, any proper continuous logic elementary extension of $M$ is extremal (with or without names for the elements of $M$).
\end{remark}

In what follows we will strengthen the previous corollary for the theories $\AHS_d$, by showing they are Bauer theories. Indeed, we will prove that the theories $\HS_d$ have affine reduction, for which we proceed similarly as we did for the complete classical theories over the empty language (\autoref{th:affine-reduction-pure-sets}), by applying the criterion of \autoref{sec:criterion-affine-reduction}. The following theorem is stated using the notation from that section.

\begin{theorem}
  \label{th:Aut(L1(mu,H))}
  Let $M \models \HS$ be separable and let $(\Omega, \mu)$ be a standard probability space. Then
  $$\Aut(L^1(\Omega,M)) = L(\Omega,\Aut(M)) \rtimes \Aut(\mu).$$
\end{theorem}
\begin{proof}
  Note that $\cH_{L^1(\Omega,M)} = L^2(\Omega,\cH_M)$. We identify $\Aut(M)$ with the orthogonal group
  $\cO(\cH_M)$ and observe that
  $$\Aut(L^1(\Omega,M)) = \set{U \in\cO(\cH_{L^1(\Omega,M)}) : U(L^1(\Omega,M))= L^1(\Omega,M)}.$$

  We start by showing that if $f,f'\in L^1(\Omega,M)$ are such that $f(\omega)\perp f'(\omega)$ for almost every $\omega \in \Omega$, and $U \in \Aut(L^1(\Omega,M))$, then $U(f)(\omega) \perp U(f')(\omega)$ for almost every $\omega$. Indeed, since $\|U(f)(\omega)\|=\|U(f')(\omega)\|=1$ almost surely, we have that
  $$\ip{U(f)(\omega), U(f')(\omega)}=1-\half \|U(f)(\omega) - U(f')(\omega)\|^2.$$
  On the other hand, since $f(\omega)\perp f'(\omega)$, we have that $\frac{1}{\sqrt{2}}(f-f')$ belongs to $L^1(\Omega,M)$. Hence $U\bigl(\frac{1}{\sqrt{2}}(f-f')\bigr)$ belongs to $L^1(\Omega,M)$, which implies that $\|U(f)(\omega)-U(f')(\omega)\|=\sqrt{2}$ almost surely. We conclude that
  $\ip{U(f)(\omega),U(f')(\omega)}=0$
  for almost every $\omega$, as desired.

  Now fix some $U \in\Aut(L^1(\Omega,M))$. Let $(b_i)$ be an orthonormal basis of $\cH_M$, and for each $i$, let $e_i\in L^1(\Omega,M)$ be the measurable section equal to the constant $b_i$. Let also $f_i=U(e_i)$, which we see as a concrete Borel function $f_i\colon \Omega \to M$. Then, for $\omega \in \Omega$, let $T(\omega)\colon \cH_M\to \cH_M$ be the operator determined by
  $$T(\omega)(b_i)=f_i(\omega).$$
  We claim that $T(\omega)\in\cO(\cH_M)$ almost surely, which is equivalent to saying that $(f_i(\omega))$ is an orthonormal basis of $\cH_M$ for almost every $\omega$.

  By the preservation of pointwise orthogonality, the set $(f_i(\omega))$ is orthonormal for almost every $\omega$. In the finite-dimensional case, this is enough to conclude that $(f_i(\omega))$ is an orthonormal basis, so let us consider the case $M\models \HS_\infty$. We want to show that for almost every $\omega$, there is no $a \in M$ is such that $a \perp f_i(\omega)$ for every $i$. To that end, let
  $$E = \set{(\omega, a) \in \Omega \times M : a \perp f_i(\omega)\text{ for every }i \in \N}.$$
  Denote by $Z$ the projection of $E$ on the first coordinate. The set $E$ is Borel, so by the Jankov--von Neumann uniformization theorem \cite[Thm.~18.1]{Kechris1995}, there is a measurable function $w \colon Z\to M$ such that $w(\omega)\perp f_i(\omega)$ for all $i$ and $\omega \in Z$. We extend $w$ to an element of $\cH_{L^1(\Omega,M)}$ by setting $w = 0$ outside of $Z$. For each $j \in \N$, define
  \begin{gather*}
    w_j(\omega)\coloneqq
    \begin{cases}
      w(\omega)\text{ if } \omega \in Z, \\
      f_j(\omega)\text{ if } \omega \in \Omega \setminus Z.
    \end{cases}
  \end{gather*}
  Thus we have a measurable function $w_j\colon \Omega\to M$ which satisfies $w_j(\omega)\perp f_i(\omega)$ almost surely whenever $i<j$. Again by the preservation of pointwise orthogonality, we have that
  \begin{equation}\label{eq:u^-1wj-perp-vi}
    U^{-1}(w_j)(\omega) \perp b_i \quad \text{for a.e.\ $\omega$ and $i<j$}.
  \end{equation}
  On the other hand, the sequence $(w_j)$ converges weakly in $\cH_{L^1(\Omega,M)}$ to  $w$. Hence, $(U^{-1}(w_j))$ converges weakly to $U^{-1}(w)$. Thus, using \autoref{eq:u^-1wj-perp-vi}, for every measurable $A \sub \Omega$, we have that $\ip{U^{-1}(w), \chi_A b_i} = \lim_j \ip{U^{-1}(w_j), \chi_A b_i} = 0$. This allows us to conclude that $U^{-1}(w)(\omega)\perp b_i$ almost surely for all $i$. We deduce that $w=0$, which implies that $Z$ has measure zero. This completes the proof that $T(\omega) \in \cO(\cH_M)$ almost surely.

  The map $\omega \mapsto T(\omega)$ is measurable, so we get an element $T\in L(\Omega,\cO(\cH_M))$
  which satisfies $T^{-1}U(e_i)=e_i$ for every $i$. Let $S = T^{-1}U$. Given $i$ and a measurable set $A\subseteq \Omega$, define
  \begin{equation*}
    e_i^A = b_i \chi_A - b_i \chi_{\Omega\setminus A} \in L^1(\Omega,M).
  \end{equation*}
  Then $e_i^A(\omega) \perp e_j(\omega)$ for a.e.\ $\omega$ and $j \neq i$, implying that $S(e_i^A)(\omega)\perp S(e_j)(\omega)$ almost surely. Since $S(e_j)(\omega) = e_j(\omega) = b_j$, we deduce that $S(e_i^A)(\omega)$ belongs almost surely to the one-dimensional subspace spanned by $b_i$, i.e., there is a measurable set $S_i(A)$ such that $S(e_i^A)=e_i^{S_i(A)}$.

  We argue that $S_i(A) = S_j(A)$ up to measure zero for every $i$, $j$ and $A$. Indeed, fix $i \neq j$. We have
  $$e_i^A(\omega)-e_j^A(\omega)\perp b_i+b_j \quad \text{for a.e.\ $\omega$}.$$
  Also, $e_i^A-e_j^A$ has constant norm because $e_i^A(\omega)\perp e_j^A(\omega)$ for every $\omega$. So we can deduce from the preservation of orthogonality that
  $$e_i^{S_i(A)}(\omega) - e_j^{S_j(A)}(\omega)\perp b_i+b_j \quad \text{for a.e.\ } \omega.$$
  This implies that $S_i(A) = S_j(A)$. Thus $S_i$ does not depend on $i$, and we denote it simply by $S_*$.

  We prove next that the map $A\mapsto S_*(A)$ defines an automorphism of the measure algebra of $\Omega$. By performing the same construction as above but for the transformation $S^{-1}$, we see readily that $(S^{-1})_*$ is the inverse of $S_*$, so $S_*$ is bijective. Now let us fix $i \in \N$ and denote $c_i^A = b_i\chi_A = \half(e_i^A+e_i)$. We have $S(c_i^A)=c_i^{S_*(A)}$, and thus
  $$\mu(A)=\big\|c_i^A\big\|= \big\|S(c_i^A)\big\| = \mu(S_*(A)),$$
  so $S_*$ is measure-preserving. It then suffices to show that for $A, B \sub \Omega$ with $A \cap B = \emptyset$, we have $S_*(A \cup B) = S_*(A) \cup S_*(B)$. If $A\cap B=\emptyset$, then $c_i^{A\cup B}=c_i^A+c_i^B$, so
  $$c_i^{S_*(A\cup B)} = S(c_i^{A\cup B}) = S(c_i^A) + S(c_i^B) = c_i^{S_*(A)}+ c_i^{S_*(B)},$$
  and it follows that $S_*(A\cup B) = S_*(A)\cup S_*(B)$.

  Finally, we show that $S$ coincides with the orthogonal transformation of $L^2(\Omega,\cH_M)$ induced by the measure automorphism $S_*$. It is enough to check that $S(f) = S_*(f)$ for those $f \in L^2(\Omega,\cH_M)$ of the form $f=\sum_j a_j \chi_{A_j}$ where $(A_j)$ is a finite measurable partition of $\Omega$ and the vectors $a_j$ are finite linear combinations of the $b_i$, say $a_j = \sum_i r_{ij}b_i$. But then
  $$S(f)=\sum_{i,j}r_{ij}S(c_i^{A_j})=\sum_{i,j}r_{ij}c_i^{S_*(A_j)}= \sum_j a_j\chi_{S_*(A_j)} = S_*(f),$$
  as desired.

  Thus $S$ belongs to the subgroup $\Aut(\mu) \leq \Aut(L^1(\Omega,M))$, and we have that
  $$U = T S \in L(\Omega,\Aut(M))\rtimes\Aut(\mu),$$
  concluding the proof.
\end{proof}

We can now apply \autoref{th:criterion-affine-reduction} and \autoref{th:Bauer-correspondence-Q} to obtain the following.

\begin{theorem}
  The continuous logic theories $\HS_d$ for $d\in\set{1,2,\dots,\infty}$ have affine reduction, and the corresponding affine theories $\AHS_d$ are Bauer.

  The models of $\AHS_d$ are direct integrals of spheres of $d$-dimensional Hilbert spaces.
\end{theorem}

We do not know a syntactic proof of this affine reduction result.

\begin{question}
  Give an explicit way of approximating continuous logic formulas by affine formulas modulo $\HS_d$.
\end{question}

\begin{question}
  Is $\AHS$ a Bauer theory?
\end{question}


\section{Measure-preserving systems}
\label{sec:PMP}

Measure algebras equipped with a countable group of automorphisms are some of the most interesting examples we have for affine theories. This setup has already been considered in continuous logic: for a single aperiodic automorphism in \cite{BenYaacov2008}, for free actions of amenable groups by Berenstein and Henson in \cite{Berenstein2018p}, for general hyperfinite actions by Giraud in \cite{Giraud2019p} (based on results of Elek~\cite{Elek2012}), and for existentially closed actions of free groups by Berenstein, Ibarluc\'ia and Henson in \cite{BerHenIba}. Some applications of model theory to ergodic theory in a non-hyperfinite setting were also developed in \cite{IbarluciaTsankov}. The theory of a single automorphism in an affine setting was also studied by Bagheri~\cite{Bagheri2014}.

In this section, we identify the extremal models for general group actions and show that the relevant theories are simplicial. In the case of hyperfinite actions, we use the quantifier elimination results of Giraud to precisely identify the type spaces and show that the theories are Poulsen. More generally, we give a criterion whether the theory of a given ergodic action is Bauer or Poulsen.

Let $\Gamma$ be a countable group. A \emph{probability measure-preserving system} of $\Gamma$ is an action of $\Gamma$ by measure-preserving automorphisms on a probability space $(Y, \cB, \mu)$. Any such action induces a dual action of $\Gamma$ by automorphisms on the measure algebra $\MALG(Y, \mu)$, which encodes all ergodic-theoretic information about the system. Conversely, any action of $\Gamma$ by automorphisms on a probability measure algebra can be realized as a concrete measure-preserving system. A system is called \emph{ergodic} if it has no fixed points in the measure algebra except $\bZero$ and $\bOne$. In that case, we will also say that the measure $\nu$ is ergodic for the action $\Gamma \actson (Y, \cB)$.

For a fixed countable group $\Gamma$, we consider the language $\cL_{\Gamma}=\cL_\PrA\cup\set{\gamma:\gamma\in\Gamma}$, i.e., the language of probability algebras expanded with  unary function symbols $\gamma$ for each $\gamma \in \Gamma$.
We let $\PMP_\Gamma$ be the affine theory consisting of the axioms of $\PrA$ together with the following:
\begin{itemize}
\item each $\gamma$ is an embedding;
\item for $\gamma_1, \gamma_2 \in \Gamma$:
  \begin{equation*}
    \sup_a d((\gamma_1\gamma_2)(a), \gamma_1(\gamma_2(a))) = 0
  \end{equation*}
  and
  \begin{equation*}
    \sup_a d(1_\Gamma(a), a) = 0.
  \end{equation*}
\end{itemize}
In words, the interpretations of the function symbols $\set{\gamma : \gamma \in \Gamma}$ define an action of $\Gamma$ on a model of $\PrA$ by automorphisms. Hence, we think of the models of $\PMP_\Gamma$ as the probability measure-preserving systems of $\Gamma$. Since $\PrA$ is a universal affine theory, $\PMP_\Gamma$ is also universal and affine.

If $\Gamma \actson Z$ is an action by homeomorphisms on a compact space, we denote by $\cM_\Gamma(Z)$ the compact convex set of $\Gamma$-invariant Radon probability measures on $Z$. The strong interplay between ergodic theory and Choquet theory is due to the following fundamental fact (see \cite[Prop.~12.3, 12.4]{Phelps2001}).

\begin{prop}\label{p:simplex-of-invariant-measures}
  Let $\Gamma \actson Z$ be an action by homeomorphisms on a compact space. Then $\cM_\Gamma(Z)$ is either empty or a simplex. The extreme points of this simplex are precisely the ergodic measures of the action.
\end{prop}

We start by characterizing the quantifier-free types of the theories $\PMP_\Gamma$.

\begin{prop}
  \label{p:PMPG-qf}
  Let $\Gamma$ be a countable group. Then the following hold:
  \begin{enumerate}
  \item \label{i:p:PMPG:Sqf} For every cardinal $\kappa$, there is an affine homeomorphism
    \begin{equation*}
      \Phi\colon \tS_\kappa^\qf(\PMP_\Gamma) \to \cM_\Gamma((2^\kappa)^\Gamma),
    \end{equation*}
    where the action $\Gamma \actson (2^\kappa)^\Gamma$ is by left shift: $(\gamma \cdot z)(h) = z(\gamma^{-1}h)$.
    The extreme types correspond precisely to the ergodic measures.
  \item \label{i:p:PMPG:models} Let $\cY\models \PMP_\Gamma$ and let $p\in \tS_\kappa^\qf(\PMP_\Gamma)$ be the quantifier-free type of an enumeration of $\cY$ (or of a subset densely generating $\cY$). If $\nu=\Phi(p)$, then $\cY$ is isomorphic to the measure-preserving system $\Gamma\actson ((2^\kappa)^\Gamma, \cB, \nu)$, where $\cB$ denotes the Borel $\sigma$-algebra.
    In particular, $\cY$ is ergodic if and only if $p$ is an extreme type.
  \end{enumerate}
\end{prop}
\begin{proof}
  \ref{i:p:PMPG:Sqf} For finite subsets $F \sub \Gamma$, $I \sub \kappa$ and a function $\eps \colon F \times I \to 2$, define the cylinder set:
  \begin{equation}
    \label{eq:PMPG-cylinder}
    C_\eps = \set{z \in (2^\kappa)^\Gamma : z\rest_{I \times F} = \eps},
  \end{equation}
  where we identify $(2^\kappa)^\Gamma$ with $2^{\kappa \times \Gamma}$.
  Consider the map $\Phi \colon \tS_\kappa^\qf(\PMP_\Gamma) \to \cM_\Gamma((2^\kappa)^\Gamma)$ defined by
  \begin{equation*}
    \Phi(p(\bar x))(C_\eps) = \mu\big(\bigcap_{(i, \gamma) \in I \times F} \gamma \cdot x_i^{\eps_i}\big)^p.
  \end{equation*}
  Perhaps the easiest way to see that $\Phi(p)$ defines a finitely additive measure on the clopen subsets of $(2^\kappa)^\Gamma$ is to realize $p$ as a tuple $\bar a$ in a model $\cY = (Y, \cY, \mu)$ and observe that $\Phi(p)$ is the pushforward of $\mu$ by the $\Gamma$-equivariant map
  \begin{equation}
    \label{eq:PMPG-pushf}
    \Psi_{\bar a} \colon Y \to (2^\kappa)^\Gamma, \quad \Psi_{\bar a}(y) = z, \quad \text{where} \quad z(\gamma)(i) = 0 \iff y \in \gamma \cdot a_i.
  \end{equation}
  The details are similar to the ones in the proof of \autoref{th:prob-alg} and we omit them.
  This also shows that $\Phi(p)$ is $\Gamma$-invariant. As continuous functions $(2^\kappa)^\Gamma \to \R$ can be uniformly approximated by continuous functions with finite image, we see that $\Phi(p)$ defines a bounded linear functional on $C((2^\kappa)^\Gamma)$, so a Radon probability measure. It is also clear that $\Phi$ is continuous and affine. To see that it is surjective, let $\nu \in \cM_\Gamma((2^\kappa)^\Gamma)$. Then the left shift $\Gamma\actson ((2^\kappa)^\Gamma, \cB, \nu)$, where $\cB$ denotes the Borel $\sigma$-algebra, is a measure-preserving system, so a model of $\PMP_\Gamma$. For $i \in \kappa$, let
  \begin{equation*}
    a_i = \set{z \in (2^\kappa)^\Gamma : z(1_\Gamma)(i) = 0}
  \end{equation*}
  and set $p = \tp^\qf(\bar a)$. Then it is easy to check that $\Phi(p) = \nu$.

  Finally, by \autoref{p:simplex-of-invariant-measures}, the extreme points of the simplex of invariant measures on a compact space are precisely the ergodic measures.

  \ref{i:p:PMPG:models} Let $\cY\models\PMP_\Gamma$ and let $\bar a\in \cY^\kappa$ densely generate $\cY$ (i.e., $\cY$ is equal to the closed substructure generated by $\bar a$). Let also $p=\tp^\qf(\bar a)$ and $\nu=\Phi(p)$. Then, conversely to what we did above, the map $a_i\mapsto \set{z\in (2^\kappa)^\Gamma : z(1_\Gamma)(i) = 0}$ extends uniquely to a $\Gamma$-invariant measure algebra isomorphism from $\cY$ to $\MALG((2^\kappa)^\Gamma, \nu)$. The map is surjective because its image includes the clopen algebra of $(2^\kappa)^\Gamma$, which is dense in $\MALG((2^\kappa)^\Gamma, \nu)$ by regularity.
\end{proof}

\begin{theorem}
  \label{th:PMPG-simplicial}
  Let $\Gamma$ be a countable group. Then the following hold:
  \begin{enumerate}
  \item \label{i:th:PMPG:ergodic} The extremal models of $\PMP_\Gamma$ are the ergodic systems.
  \item \label{i:th:PMPG:simplicial} The theory $\PMP_\Gamma$ is simplicial.
  \end{enumerate}
\end{theorem}
\begin{proof}
  \ref{i:th:PMPG:ergodic} Suppose first that $\cY \models \PMP_\Gamma$ is ergodic. Let $\bar a\in \cY^\kappa$ be an enumeration of a dense subset of $\cY$ and let $q = \tp^\aff(\bar a) \in \tS^{\fM,\aff}_\kappa(\PMP_\Gamma)$. By \autoref{p:PMPG-qf}\ref{i:p:PMPG:models}, the projection $\rho^\qf_\kappa(q)\in \tS_\kappa^\qf(\PMP_\Gamma)$ is an extreme type. Thus, by \autoref{lem:ExtremeTypeProjectionQF}, $q$ is an extreme type, i.e., $\cY$ is an extremal model.

  Conversely, suppose that $\cY$ is not ergodic. Then there exists $a \in \cY$, $a \neq \bZero, \bOne$, which is fixed by $\Gamma$. We let $\cY_1$ and $\cY_2$ be the normalized restrictions of $\cY$ to $a$ and $\neg a$, respectively, and we note that $\cY_1$ and $\cY_2$ are models of $\PMP_\Gamma$ and that $\cY = \mu(a)\cY_1 \oplus (1-\mu(a)) \cY_2$. This implies that $\tp^\aff(a) = \mu(a) \tp^{\aff,\cY_1}(\bOne) + (1 - \mu(a)) \tp^{\aff,\cY_2}(\bZero)$, so $\tp^\aff(a)$ is not an extreme type and thus $\cY$ is not extremal.

  \ref{i:th:PMPG:simplicial} By \autoref{p:PMPG-qf}\ref{i:p:PMPG:Sqf} and \autoref{p:simplex-of-invariant-measures}, the quantifier-free type spaces of the theory $\PMP_\Gamma$ are simplices. Moreover, if $q\in\cE^\fM_\bN(\PMP_\Gamma)$ then, by \ref{i:th:PMPG:ergodic}, the model enumerated by $q$ is ergodic and hence, by \autoref{p:PMPG-qf}\ref{i:p:PMPG:models}, the projection $\rho^\qf_\bN(q)\in\tS_\bN^\qf(\PMP_\Gamma)$ is an extreme type. The hypotheses of \autoref{prop:SimplicialCriterionQF} are thus satisfied, which allow us to conclude that $\PMP_\Gamma$ is simplicial.
\end{proof}

As a consequence of our extremal decomposition result, \autoref{th:integral-decomposition}, we obtain a generalization of the classical ergodic decomposition theorem (see, e.g., \cite[Thm.~3.22]{Glasner2003}) to arbitrary probability measure-preserving systems (i.e., not necessarily standard) in terms of their dual actions.

\begin{cor}\label{cor:decomposition-PMP}
  Every model of $\PMP_\Gamma$ is a direct integral of ergodic models.
\end{cor}

We recall that a model $\cX$ of $\PMP_\Gamma$ \emph{admits almost invariant sets} if there exists a sequence $(a_n)_n$ of elements of $\cX$ such that the measures $\mu(a_n)$ are bounded away from $0$ and $1$, and $\lim_n \mu(\gamma \cdot a_n \sdiff a_n) = 0$ for all $\gamma \in \Gamma$. The system $\cX$ is called \emph{strongly ergodic} if it does not admit almost invariant sets. By \cite[Prop.~2.7]{IbarluciaTsankov}, $\cX$ is strongly ergodic if and only if every model of $\Th^\cont(\cX)$ is ergodic. A countable group $\Gamma$ has \emph{property (T)} if all of its ergodic actions are strongly ergodic. This is not the original definition but it is equivalent by \cite{Schmidt1981, Connes1980}. By \cite{Glasner1997}, it is also equivalent to $\cM_\Gamma((2^\omega)^\Gamma)$ being a Bauer simplex.
\begin{cor}
  \label{c:PMPG-strongly-ergodic-propT}
  Let $\Gamma$ be a countable group and let $\cX$ be an ergodic model of $\PMP_\Gamma$. Then:
  \begin{enumerate}
  \item \label{i:se:extremal-models} $\Th^\aff(\cX)$ is simplicial and its extremal models are the ergodic models of $\Th^\cont(\cX)$.
  \item \label{i:se:str-erg} If $\cX$ is strongly ergodic, then $\Th^\aff(\cX)$ is a Bauer theory.
  \item \label{i:se:not-str-erg} If $\cX$ is not strongly ergodic, then $\Th^\aff(\cX)$ is a Poulsen theory.
  \item \label{i:se:propT} $\Gamma$ has property (T) if and only if $\PMP_\Gamma$ is a Bauer theory.
  \end{enumerate}
\end{cor}
\begin{proof}
  \autoref{i:se:extremal-models} The theory $\Th^\aff(\cX)$ is simplicial because it is an extreme completion of the simplicial theory $\PMP_\Gamma$ (\autoref{p:face-simplicial}). Every ergodic model of $\Th^\cont(\cX)$ is an ergodic, hence extremal, model of $\PMP_\Gamma$, and thus also an extremal model of the completion $\Th^\aff(\cX)$. Conversely, since $\cX$ is ergodic, $\Th^\aff(\cX)$ is an extreme completion of $\PMP_\Gamma$ and so every extremal model of $\Th^\aff(\cX)$ is an extremal, hence ergodic, model of the $\PMP_\Gamma$. Moreover, by \autoref{cor:general-simplicial-theories}\autoref{i:cor:general-simplicial-extensions}, and \autoref{prop:ExtremalJEP}, any extremal model of $\Th^\aff(\cX)$ satisfies $\Th^\cont(\cX)$.

  \autoref{i:se:str-erg} This follows from the previous point, \autoref{c:closed-extreme-types} and the characterization of strong ergodicity recalled above.

  \autoref{i:se:not-str-erg} Similarly, $\Th^\aff(\cX)$ is not Bauer by \autoref{c:closed-extreme-types}. As it is simplicial and complete, it is Poulsen by \autoref{th:dichotomy-simplicial-theories}.

  \autoref{i:se:propT} This follows from \autoref{th:PMPG-simplicial}, \autoref{c:closed-extreme-types}, and the fact that a group has property (T) if and only if the ergodic models of $\PMP_\Gamma$ form an elementary class in continuous logic \cite[Prop.~2.8]{IbarluciaTsankov}.
\end{proof}

The theory $\PMP_\Gamma$ is never complete if $\Gamma$ is non-trivial. A basic reason for this is that the sets of fixed points (in $Y$) of the elements of $\Gamma$ are definable and their measures can be expressed by affine formulas. The notion of an IRS that we recall below serves to capture this information. Let $\Sub(\Gamma)$ denote the compact space of subgroups of $\Gamma$ (seen as a subspace of $2^\Gamma$) and let
\begin{equation*}
  \IRS(\Gamma) = \cM_\Gamma(\Sub(\Gamma)),
\end{equation*}
where the action $\Gamma \actson \Sub(\Gamma)$ is by conjugation. An \emph{invariant random subgroup} (\emph{IRS}, for short) of $\Gamma$ is an element of $\IRS(\Gamma)$. The convex set $\IRS(\Gamma)$ is a simplex and its extreme points are the ergodic IRSs. Some examples of ergodic IRSs are the Dirac measures on normal subgroups and uniform measures on conjugacy classes of finite index subgroups. Many groups also admit atomless IRSs.

The relation between measure-preserving systems and IRSs is the following: if $\Gamma \actson (Y, \cB, \mu)$ is a measure-preserving action on a probability space and we denote by $\Stab \colon Y \to \Sub(\Gamma)$ the stabilizer map given by
\begin{equation*}
  \Stab(y) = \set{\gamma \in \Gamma : \gamma \cdot y = y},
\end{equation*}
then $\Stab_*\mu$ is an IRS of $\Gamma$, called the \emph{stabilizer IRS} of the action. It was proved in \cite{Abert2014} that, conversely, any IRS can be realized as the stabilizer IRS of some measure-preserving system. An action is called \emph{free} if its IRS is the Dirac measure on the trivial subgroup of $\Gamma$, equivalently, if for all $\gamma \neq 1_\Gamma$, $\mu(\set{y \in Y : \gamma \cdot y = y}) = 0$.

For reasons of quantifier elimination, it will be convenient to augment the language with constant symbols $\Fix_\gamma$ for every $\gamma \in \Gamma$. Each $\Fix_\gamma$ is interpreted as $\set{y \in Y : \gamma \cdot y = y}$, or, in terms of the measure algebra,
\begin{equation*}
  \Fix_\gamma = \bigcup \set{a \in \cY : \forall b \sub a, \ \gamma \cdot b = b}.
\end{equation*}
This an affine definitional expansion of $\PMP_\Gamma$, as follows from the proof of \cite[Lemma~3.15]{Giraud2019p}, so we continue to denote it by $\PMP_\Gamma$.

With this expanded language, it is clear that the stabilizer IRS of the action is part of its quantifier-free theory. Indeed, for finite subsets $F_1, F_2 \sub \Gamma$, denote
\begin{equation*}
  D_{F_1, F_2} \coloneqq \set{H \in \Sub(\Gamma) : F_1 \sub H, F_2 \cap H = \emptyset}
\end{equation*}
and note that if $\theta$ is the stabilizer IRS of $\cY$, we have that
\begin{equation}
  \label{eq:formula-IRS}
  \theta\big(D_{F_1, F_2}\big) = \mu\big(\bigcap_{\gamma \in F_1} \Fix_\gamma \cap \bigcap_{\gamma \in F_2} \neg \Fix_\gamma\big)^\cY.
\end{equation}
For $\theta \in \IRS(\Gamma)$, we denote by $\PMP_\theta$ the affine theory of measure-preserving systems of $\Gamma$ with stabilizer IRS equal to $\theta$.

\begin{prop}
  \label{p:PMP-theta-face}
  If $\theta$ is an ergodic IRS, then $\PMP_\theta$ is a closed face of $\PMP_\Gamma$.
\end{prop}
\begin{proof}
  There is a natural continuous affine map $\tS^\aff_0(\PMP_\Gamma) \to \IRS(\Gamma)$ defined by \autoref{eq:formula-IRS}, and $\tS^\aff_0(\PMP_\theta)$ is  the preimage of the extreme point $\theta$ by this map, so it must be a closed face.
\end{proof}

\begin{cor}
  \label{c:IRS-simplicial}
  If $\theta$ is an ergodic IRS, then the theory $\PMP_\theta$ is simplicial.
\end{cor}
\begin{proof}
  This follows from \autoref{p:face-simplicial} and \autoref{th:PMPG-simplicial}.
\end{proof}

Next we will use the results of Giraud to see that for a hyperfinite IRS $\theta$, the theory $\PMP_\theta$ is complete.
A measure-preserving action $\Gamma \actson (Y, \cB, \mu)$ is called \emph{hyperfinite} if its orbit equivalence relation is hyperfinite: i.e., an increasing union of equivalence relations with finite classes. Equivalently, the action is hyperfinite if for every finite $S \sub \Gamma$ and every $\eps > 0$, there exists a finite subgroup $G \leq \Aut(Y, \mu)$ such that
\begin{equation*}
  \mu(\set{y \in Y : S \cdot y \sub G \cdot y}) > 1 - \eps.
\end{equation*}
It is a classical theorem of Ornstein and Weiss~\cite{Ornstein1980} that every action of an amenable group is hyperfinite and it is also well-known that free actions of non-amenable groups are never hyperfinite (see, e.g., \cite{Kechris2004} for more details). However, there do exist hyperfinite actions that do not factor through an action of an amenable group.

It is a result of Elek~\cite{Elek2012} that whether an action is hyperfinite or not only depends on its stabilizer IRS. We call an IRS $\theta$ \emph{hyperfinite} if every (equivalently, some) action with stabilizer IRS $\theta$ is hyperfinite.  We will say that an IRS $\theta$ is \emph{nowhere of finite index} if for $\theta$-a.e.\ $H$, $[\Gamma : H] = \infty$.

\begin{theorem}[Giraud~\cite{Giraud2019p}]
  \label{th:Giraud}
  Let $\theta$ be a hyperfinite IRS of a countable group $\Gamma$ that is nowhere of finite index. Then the theory $\PMP_\theta$ is complete in continuous logic and eliminates quantifiers in the language augmented by the constants $\Fix_\gamma$.
\end{theorem}
\begin{remark}
  \label{rem:Giraud}
  The condition that $\theta$ is nowhere of finite index is not present in \cite{Giraud2019p} but there is an axiom stating that the measure algebra is atomless. As this axiom is not affine, we omit it and then it is implied by our condition that the stabilizer IRS is nowhere of finite index.
\end{remark}

In order to explain what happens in affine logic and describe the type spaces, we will need to introduce some further notation. If $\Gamma \actson Z$ is an action on a compact space by homeomorphisms, we let
\begin{equation*}
  A_\Gamma(Z) = \set{(H, z) \in \Sub(\Gamma) \times Z : H \cdot z = z}
\end{equation*}
and we note that $A_\Gamma(Z)$ is a closed, $\Gamma$-invariant subspace of $\Sub(\Gamma) \times Z$. If $\theta \in \IRS(\Gamma)$, we let
$$\fJ_\theta(Z) = \set{\nu \in \cM_\Gamma(A_\Gamma(Z)) : \pi_*(\nu) = \theta},$$
where $\pi\colon A_\Gamma(Z)\to \Sub(\Gamma)$ is the projection to the first coordinate. Note that when $\theta = \delta_{\set{1_\Gamma}}$, $\fJ_\theta(Z) \cong \cM_\Gamma(Z)$.

\begin{theorem}
  \label{th:PMP-hyperfinite}
  Let $\Gamma$ be a countable group and let $\theta$ be a hyperfinite, ergodic IRS of $\Gamma$ that is nowhere of finite index. Then the following hold:
  \begin{enumerate}
  \item \label{i:PMP-hf-complete} The theory $\PMP_\theta$ is a complete, simplicial, affine theory with affine quantifier elimination in the language augmented by the constants $\Fix_\gamma$.
  \item \label{i:PMP-hf-ext-models} The extremal models of $\PMP_\theta$ are precisely its ergodic models.
  \item \label{i:PMP-hf-types} For every cardinal $\kappa$, $\tS^\aff_\kappa(\PMP_\theta) \cong \fJ_\theta((2^\kappa)^\Gamma)$.
  \item \label{i:PMP-hf-Poulsen} $\PMP_\theta$ is a Poulsen theory.
  \end{enumerate}
\end{theorem}
\begin{proof}
  \ref{i:PMP-hf-complete} Completeness of the theory follows from \autoref{th:Giraud}, and that it is simplicial was proved in \autoref{c:IRS-simplicial}. For affine quantifier elimination, we observe that if an affine theory $T$ eliminates quantifiers in continuous logic, as in our case, then it also eliminates quantifiers in affine logic: as atomic formulas separate points in $\tS^\cont_n(T)$, they also separate points in $\tS^\aff_n(T)$, which is enough for affine quantifier elimination.

  \ref{i:PMP-hf-ext-models} As $\PMP_\theta$ is a face of $\PMP_\Gamma$, $\tS^\aff_n(\PMP_\theta)$ is a closed face of $\tS^\aff_n(\PMP_\Gamma)$ for every $n$, so the extremal models of $\PMP_\theta$ are precisely the extremal models of $\PMP_\Gamma$ which are also models of $\PMP_\theta$. Now the claim follows from \autoref{th:PMPG-simplicial}.

  \ref{i:PMP-hf-types} We define a map $\Phi \colon \tS^\aff_n(\PMP_\theta) \to \fJ_\theta((2^\kappa)^\Gamma)$ as follows. For a type $p \in \tS^\aff_n(\PMP_\theta)$, we take a realization $\bar a$ of $p$ in a model $\cY = (Y,\cB,\mu)$ and we define a map
  \begin{equation*}
    \Theta_{\bar a} \colon Y \to \Sub(\Gamma) \times (2^\kappa)^\Gamma, \quad
    \Theta_{\bar a}(y) = (\Stab(y), \Psi_{\bar a}(y)),
  \end{equation*}
  where $\Psi_{\bar a}$ is defined as in \autoref{eq:PMPG-pushf}. Finally, we set $\Phi(p) = {\Theta_{\bar a}}_*\mu$. We claim that $\Phi$ is an affine homeomorphism.

  First, it is clear that $\Phi$ is well-defined, that is, it does not depend on the realization $\bar a$. Indeed, ${\Theta_{\bar a}}_*\mu$ only depends on the measures of the elements of the subalgebra of $\cY$ generated by $\bar a$ and $\set{\Fix_\gamma : \gamma \in \Gamma}$, whose isomorphism type is determined by the (quantifier-free) type of $\bar a$. Also, the measure ${\Theta_{\bar a}}_* \mu$ belongs to $\fJ_\theta((2^\kappa)^\Gamma)$: it is $\Gamma$-invariant because $\Theta_{\bar a}$ is $\Gamma$-equivariant and it concentrates on $A_\Gamma((2^\kappa)^\Gamma)$ because $\Stab(y) \cdot \Psi_{\bar a}(y) = \Psi_{\bar a}(\Stab(y) \cdot y) = \Psi_{\bar a}(y)$ for $\mu$-a.e.\ $y$.

  To see that $\Phi$ is continuous, note that it can be alternatively defined as follows. For finite subsets $F \sub \Gamma$, $I \sub \kappa$ and a function $\eps \colon F \times I \to 2$, define the cylinder set $C_\eps$ as in \autoref{eq:PMPG-cylinder}. Let $F_1, F_2$ be finite subsets of $\Gamma$. Then if we denote $\Phi(p(\bar x)) = \nu$, we have:
  \begin{equation}
    \label{eq:PMPth:Phi-alternative}
    \nu\big( D_{F_1, F_2} \times C_\eps\big)
    = \mu\big(\bigcap_{\gamma \in F_1} \Fix_\gamma \cap \bigcap_{\gamma \in F_2} \neg \Fix_\gamma
    \cap \bigcap_{(i, \gamma) \in I \times F} \gamma \cdot x_i^{\eps_i} \big)^p.
  \end{equation}
  It is also clear $\Phi$ is affine. Injectivity follows from quantifier elimination and the fact that atomic formulas as on the right-hand side of \autoref{eq:PMPth:Phi-alternative} determine a quantifier-free type.

  It remains to check that $\Phi$ is surjective. Let $\nu \in \fJ_\theta((2^\kappa)^\Gamma)$. Let $\cW = (W, \lambda)$ be an ergodic system with IRS $\theta$, as constructed in \cite[Prop.~13]{Abert2014}. The systems $(A_\Gamma((2^\kappa)^\Gamma), \nu)$ and $\cW$ have $(\Sub(\Gamma), \theta)$ as a common factor: the map $\pi \colon A_\Gamma((2^\kappa)^\Gamma) \to \Sub(\Gamma)$ is given by projection on the first coordinate and the map $W \to \Sub(\Gamma)$ is the stabilizer map. We let $\cY$ be the relatively independent joining of these two systems with respect to their common factor (see \cite[Ch.~6]{Glasner2003}). We can realize $\cY = (Y, \mu)$ with
  \begin{equation*}
    Y = \set{(w, H, z) \in W \times A_\Gamma : \Stab(w) = H}
  \end{equation*}
  and $\mu$ an invariant measure on $Y$ with marginals $\lambda$ and $\nu$. First, we claim that $\cY \models \PMP_\theta$. For this it suffices to see that for all $(w, H, z) \in Y$,
  \begin{equation*}
    \Stab((w, H, z)) = \Stab(w).
  \end{equation*}
  The left to right inclusion is clear. For the other, note that if $\gamma \in \Stab(w)$, then $\gamma \in H$ and, by the definition of $A_\Gamma$, $\gamma \cdot z = z$. Now we let
  \begin{equation*}
    a_i = \set{(w, H, z) \in Y : z(1_\Gamma)(i) = 0}
  \end{equation*}
  and observe that $\Phi(\tp^\aff(\bar a)) = \nu$. Indeed, one sees readily that the map $\Psi_{\bar a}\colon Y\to (2^\kappa)^\Gamma$ defined as in \autoref{eq:PMPG-cylinder} satisfies $\Psi_{\bar a}((w,H,z)) = z$, and thus the corresponding map $\Theta_{\bar a}\colon Y\to\Sub(\Gamma)\times (2^\kappa)^\Gamma$ is given by $\Theta_{\bar a}((w,H,z)) = (\Stab(w),z)= (H,z)$. In other words, $\Theta_{\bar a}$ is the factor map from $Y$ to $A_\Gamma((2^\kappa)^\Gamma)$, and we have ${\Theta_{\bar a}}_*\mu = \nu$.

  \ref{i:PMP-hf-Poulsen} This follows from \autoref{th:Giraud} and \autoref{th:Poulsen-complete}.
\end{proof}

\begin{example} When $\Gamma$ is a countably infinite, amenable group, the preceding theorem applies, in particular, to the theory $\FPMP_\Gamma$ of free measure-preserving actions of $\Gamma$ (i.e., corresponding to $\theta=\delta_{\set{1_\Gamma}}$), which is the model completion of the theory $\PMP_\Gamma$, both in continuous and affine logic.
\end{example}

\begin{remark}
  \label{rem:finite-Gamma}
  It remains to see what happens when $\theta$ is of finite index with positive measure. Then, if $\theta$ is  ergodic, it concentrates on the conjugacy class of a subgroup $H \leq \Gamma$ of finite index. There is a unique ergodic system with this IRS, namely the transitive action $\Gamma \actson \Gamma/H$ equipped with the uniform measure, and the corresponding affine theory is absolutely categorical and Bauer. The situation is very similar to the one of the pure measure algebra, which corresponds to the case where $H = \Gamma$. We leave the details to the reader.
\end{remark}

\begin{prop}
  \label{p:completions-amenable}
  Let $\Gamma$ be a countable, amenable group. Then $\tS^\aff_0(\PMP_\Gamma) \cong \IRS(\Gamma)$.
\end{prop}
\begin{proof}
  Consider the map $\IRS(\Gamma) \to \tS^\aff_0(\PMP_\Gamma)$, $\theta \mapsto \PMP_\theta$. By \autoref{th:Giraud} and \autoref{rem:finite-Gamma}, it takes values in $\tS^\aff_0(\PMP_\Gamma)$ and it is continuous, injective, and affine. To see that it is surjective, let $\cX \models \PMP_\Gamma$ and let $\theta$ be the IRS of $\cX$. Then $\cX \models \PMP_\theta$ and by completeness of $\PMP_\theta$, it is equal to $\Th^\aff(\cX)$.
\end{proof}

\begin{example}
  \label{ex:PMP-Z}
  Let $\Z$ denote the group of integers.
  The theory $\PMP_\Z$ is simplicial (by \autoref{th:PMPG-simplicial}) but it is neither Bauer nor Poulsen. Indeed, it is not Bauer by \autoref{c:PMPG-strongly-ergodic-propT} (as $\Z$ does not have property (T)), and it is not Poulsen because $\IRS(\Z)$ has only countably many extreme points (the Dirac measures on the subgroups of $\Z$), so it is not a Poulsen simplex.
\end{example}


\section{Tracial von Neumann algebras}
\label{sec:tracial-von-neumann}

Recall that a \emph{von Neumann algebra} $M$ is a unital, self-adjoint subalgebra of the algebra of bounded operators $B(\cH)$ of some complex Hilbert space $\cH$, closed in the strong or, equivalently, the weak operator topology. We will denote by $M^1$ the unit ball of $M$ in the operator norm. A \emph{faithful, normal tracial state} (or just a \emph{trace}) on $M$ is a linear functional $\tau \colon M \to \bC$ which satisfies:
\begin{itemize}
\item $\tau$ is \emph{positive}, i.e., satisfies $\tau(a^*a) \geq 0$ for all $a \in M$;
\item $\tau\rest_{M^1}$ is continuous for the weak operator topology;
\item $\tau$ is \emph{faithful}, i.e., $\tau(a^*a) = 0$ implies that $a = 0$;
\item $\tau(\bOne) = 1$;
\item for all $a, b \in M$, $\tau(ab) = \tau(ba)$.
\end{itemize}
A \emph{tracial von Neumann algebra} is a pair $(M, \tau)$, where $M$ is a von Neumann algebra and $\tau$ is a trace on $M$. A von Neumann algebra is called a \emph{factor} if its center is trivial. A factor admitting a trace is called \emph{finite} and in that case, the trace is unique.
Commutative tracial von Neumann algebras are of the form $(L^\infty(X), \mu)$, where $(X,\mu)$ is a probability space. The measure $\mu$ defines a trace on $L^\infty(X)$ by integrating. Apart from the operator norm $\nm{\cdot}$, there is a variety of other norms on a tracial von Neumann algebra; two that we will use are defined by:
\begin{itemize}
\item $\nm{a}_2 = \tau(a^*a)^{1/2}$;
\item $\nm{a}_1 = \tau(|a|)$.
\end{itemize}
Here $|a| = (a^*a)^{1/2}$ and we recall that as $a^*a$ is a self-adjoint operator on a Hilbert space, we have available the continuous functional calculus that allows us to make sense of the square root. The algebra $M$ is not complete with respect to either of these two norms. However, we have the following basic fact.
\begin{lemma}
  \label{l:vN-top-same}
  Let $(M, \tau)$ be a tracial von Neumann algebra with $M \sub B(\cH)$. Then on $M^1$ the norms $\nm{\cdot}_1$ and $\nm{\cdot}_2$ are complete and equivalent and the topology they define coincides with the strong operator topology.
\end{lemma}
\begin{proof}
  For the fact that $\nm{\cdot}_2$ defines the strong operator topology and the completeness of the norm, see (the proof of) \cite[Prop.~III.5.3]{Takesaki1979}. For the equivalence of the two norms, first note that by the Cauchy--Schwarz inequality, we have
  \begin{equation*}
    \nm{a}_1 = \tau(|a| \bOne) \leq \tau(|a|^2)^{1/2}\tau(\bOne)^{1/2} = \nm{a}_2.
  \end{equation*}
  For the other direction, recall that in an arbitrary \Cstar-algebra, for all $x, y$ with $x \geq 0$, we have that $y^*xy \geq 0$. We have then:
  \begin{equation*}
    \nm{a}_2^2 = \tau(|a|^2) = \tau(|a|^{1/2}|a||a|^{1/2}) \leq \nm{a} \tau(|a|^{1/2}|a|^{1/2}) \leq \nm{a}_1. \qedhere
  \end{equation*}
\end{proof}

Let $A$ be a unital \Cstar-algebra. A \emph{tracial state} on $A$ is a positive linear functional $\tau$ that satisfies $\tau(\bOne) = 1$ and $\tau(ab) = \tau(ba)$ for all $a, b \in A$. We denote by $\cT(A)$ the compact convex set of all traces on $A$, equipped with the weak$^*$ topology. We have the following important fact (see \cite[Thm.~3.1.18]{Sakai1971}).
\begin{prop}
  \label{p:traces-simplex}
  Let $A$ be a unital \Cstar-algebra. Then $\cT(A)$ is a simplex.
\end{prop}

The following lemma collects some standard facts about tracial representations that we will need. The proofs can be found for example in \cite{Dixmier1969} for a more general situation but we have also provided them below for the convenience of the reader. If $A$ is a subset of $B(\cH)$, we denote by $A'$ the \emph{commutant} of $A$, i.e., the set of operators in $B(\cH)$ that commute with all elements of $A$. If $B$ is a \Cstar-algebra, we denote by $B^\op$ the opposite algebra of $B$.
We also recall that every injective homomorphism of \Cstar-algebras is isometric.
\begin{lemma}
  \label{l:tracial-rep-Cstar}
  Let $A$ be a unital \Cstar-algebra and let $\tau \in \cT(A)$. Let $\cI_\tau = \set{a \in A : \tau(a^*a) = 0}$. Then the following hold:
  \begin{enumerate}
  \item \label{i:ltr:ideal} $\cI_\tau$ is a closed ideal of $A$.
  \item \label{i:ltr:ip} The sesquilinear form $\ip{a, b} \coloneqq \tau(b^*a)$ is positive definite on $B \coloneqq A/\cI_\tau$.
  \item \label{i:ltr:rep} Denote by $\cH_\tau$ the completion of $B$ with respect to the Hilbert space norm $\nm{a}_2 = \ip{a, a}^{1/2}$. Then the left multiplication $b \mapsto ab$ of $B$ extends to a faithful representation $\lambda \colon B \to B(\cH_\tau)$. Similarly, the right multiplication $b \mapsto ba$ extends to a faithful representation $\rho \colon B^\op \to B(\cH_\tau)$.
  \item \label{i:ltr:factor} Denote by $M_\lambda$ the closure of $\lambda(B)$ in $B(\cH_\tau)$ in the weak operator topology, and similarly for $M_\rho$. Then $\tau$ extends by continuity to $M_\lambda$ and $M_\rho$, so that $(M_\lambda, \tau)$, $(M_\rho, \tau)$ are tracial von Neumann algebras, $M_\lambda' = M_\rho$, $M_\rho' = M_\lambda$, and they are factors if and only if $\tau$ is an extreme point of $\cT(A)$.
  \end{enumerate}
\end{lemma}
\begin{proof}
  \ref{i:ltr:ideal} We note that if $x \in A$ is positive, then $\tau(x) = 0$ implies that $\tau(x^2) = 0$. Indeed, we may as well suppose that $\nm{x} \leq 1$, so $x^2 \leq x$, and the claim follows from the positivity of $\tau$. Now if $a \in \cI_\tau$ or $b \in \cI_\tau$, using the Cauchy--Schwarz inequality, we have that
  \begin{equation*}
    |\tau((ab)^*ab)| = |\tau(b^*a^*ab)| = |\tau((b^*b)^* a^*a)| \leq |\tau((b^*b)^2) \tau((a^*a)^2)|^{1/2} = 0,
  \end{equation*}
  so $ab \in \cI_\tau$.

  \ref{i:ltr:ip} The form $\ip{\cdot, \cdot}$ is positive semi-definite because $\tau \geq 0$ and it is definite because of the definition of $\cI_\tau$.

  \ref{i:ltr:rep} We have to check that for all $a, b \in B$, $\nm{ab}_2 \leq \nm{a} \nm{b}_2$ and $\nm{ab}_2 \leq \nm{b} \nm{a}_2$. We have:
  \begin{equation*}
    \nm{ab}_2^2 = \tau\big( (ab)^*ab \big) = \tau(b^*a^*ab) \leq \tau \big( b^*(\nm{a}^2\bOne) b \big) = \nm{a}^2\nm{b}_2^2.
  \end{equation*}
  The other inequality is similar, using the fact that $\tau\big( (ab)^*ab \big) = \tau\big( ab (ab)^* \big)$.
  Faithfulness follows from the fact that the map $B \to \cH_\tau$, $a \mapsto a \cdot \bOne$ is injective.

  \ref{i:ltr:factor} By von Neumann's bi-commutant theorem, we have that $M_\lambda = \lambda(B)''$ and $M_\rho = \rho(B)''$, so they are von Neumann algebras and $B$ is strongly dense in both. We have that for $a \in B$, $\tau(a) = \ip{a \bOne, \bOne}$, so $\tau$ extends to a positive, weakly continuous functional on $M_\lambda$ given by the same formula. That $\tau$ is tracial follows from continuity: for a fixed $a \in B$, the set
  \begin{equation*}
    \set{b \in M_\lambda : \tau(ab) = \tau(ba)} = \set{b \in M_\lambda : \ip{b\bOne, a^*} = \ip{ba, \bOne}}
  \end{equation*}
  is weakly closed and contains $B$, so it is equal to $M_\lambda$. Applying this a second time, we get that the identity holds for all $a, b \in M_\lambda$. Similarly for $M_\rho$. The fact that $M_\lambda' = M_\rho$ and $M_\rho' = M_\lambda$ follows from \cite[Part~I, Ch.~5, Thm.~1]{Dixmier1981}.

  Suppose now that $M_\lambda$ is not a factor and let $z \in M_\lambda$ be a non-trivial central projection. Define $\tau_1(a) = \frac{1}{\tau(z)} \tau(az)$ and $\tau_2(a) = \frac{1}{1-\tau(z)} \tau(a(\bOne - z))$ for all $a \in A$. Then $\tau_1$ and $\tau_2$ are traces and $\tau = \tau(z) \tau_1 + (1-\tau(z))\tau_2$. To see that $\tau_1$ and $\tau_2$ are distinct, note that if $(a_n)$ is a sequence of elements of $B$ converging to $z$ in the weak operator topology, then $\tau_1(a_n) \to \tau(z)$ and $\tau_2(a_n) \to 0$. The argument for $M_\rho$ is similar.

  Finally, suppose that $\tau$ is not extreme and $\tau = \half \tau_1 + \half \tau_2$ with $\tau_1, \tau_2$ distinct elements of $\cT(A)$. Note that $\ip{x, y}_1 \coloneqq \tau_1(y^*x)$ defines a positive sesquilinar form on $B$ such that $\ip{x, x}_1 \leq 2 \ip{x, x}$, so it extends to a positive sesquilinear form on $\cH_\tau$. Thus we can define an operator $T \in B(\cH_\tau)$ by $\ip{Tx, y} = \ip{x, y}_1$ and it satisfies $0 \leq T \leq 2 \cdot \bOne_{\cH_\tau}$. We show that $T \in M_\lambda'$. For $a, x, y \in B$, we have:
  \begin{equation*}
    \ip{(T \lambda(a))x, y} = \tau_1(y^*ax) = \ip{Tx, a^*y} = \ip{Tx, \lambda(a)^*y} = \ip{\lambda(a)Tx, y}.
  \end{equation*}
  As $B$ is dense in $\cH_\tau$, we obtain that $T\lambda(a) = \lambda(a)T$ everywhere. Similarly, $T \in M_\rho'$. Now it only remains to notice that $M_\lambda' \cap M_\rho' = M_\rho \cap M_\lambda$ is the common center of $M_\lambda$ and $M_\rho$ and that $T$ is not a multiple of the identity because $\tau_1$ is not a multiple of $\tau$.
\end{proof}

In continuous logic, tracial von Neumann algebras have been axiomatized (see, for example, \cite{Farah2014}) and there is considerable literature on their model theory. What is perhaps most relevant for us is that, under a mild non-triviality condition, no completion of the continuous theory of tracial von Neumann algebra has quantifier elimination \cite{Farah2023ppp} and, except possibly for very strange counterexamples, no completion is model complete \cite{Farah2023pp}.

Here we consider the theory of tracial von Neumann algebras in a richer (but equivalent) language which will allow us to use functional calculus directly. An element $x$ of a \Cstar-algebra is called a \emph{contraction} if $\nm{x} \leq 1$. For a cardinal $\kappa$, we let $\cC_\kappa$ denote the universal unital \Cstar-algebra
\begin{equation*}
  \mathrm{C}^\ast \gen{(\scx_i : i < \kappa) : \nm{\scx_i} \leq 1},
\end{equation*}
generated by $\kappa$ contractions. More concretely, $\cC_\kappa$ is the completion of the free $\ast$-algebra of non-commutative $\ast$-polynomials $P(\bar\scx)$ in the variables $(\scx_i : i < \kappa)$ with respect to the norm
\begin{equation*}
  \nm{P(\bar\scx)} \coloneqq \sup \set{\nm{P(\bar T)} : T_i \in B(\cH), \nm{T_i} \leq 1},
\end{equation*}
where $\cH$ is some infinite-dimensional Hilbert space. See \cite{Blackadar1985} for more details on universal \Cstar-algebras. The algebra $\cC_\kappa$ has the following universal property: for every \Cstar-algebra $A$ and any tuple of contractions $(v_i : i < \kappa)$ in $A$ there exists a unique homomorphism $\cC_\kappa \to A$ sending $\scx_i$ to $v_i$ for every $i$. If $u \in \cC_\kappa$, we denote by $u(\bar v)$ the image of $u$ under this homomorphism.

We consider a language $\cL_\TvN$ consisting of the following symbols. For every $n$, we include $n$-ary function symbols $u$ for every contraction $u \in \cC_n$. We also include a distance predicate $d(\cdot, \cdot)$ and a complex-valued predicate $\tau$ for the trace. (A complex-valued predicate is represented by two real-valued predicates, for the real and imaginary parts.) Given a tracial von Neumann algebra $M$, we see its unit ball $M^1$ as an $\cL_\TvN$-structure in the natural way, with $d$ interpreted according to the metric axiom below.

We define the affine theory $\TvN$ by the following universal axioms.
\begin{description}
\item[Algebraic axioms]
  For every $\ell, k \in \N$ and contractions $\bar v \in (\cC_\ell)^k$, $u \in \cC_k$, we include the axiom
  \begin{equation*}
    u(\bar v)(\bar x) = u(\bar v(\bar x)),
  \end{equation*}
  where $\bar x = (x_0, \ldots, x_{\ell-1})$ are variables, the element $u(\bar v) \in \cC_\ell$ is as defined above, and the right-hand side is composition of function symbols.
  These axioms allow us to have unambiguous terms for all non-commutative $\ast$-polynomials (of norm $\leq 1$) and also for more complicated functions such as $|x| = (x^*x)^{1/2}$.

\item[Trace axioms] $\tau$ is linear: $\tau(\lambda_1 x_1 + \lambda_2 x_2) = \lambda_1 \tau(x_1) + \lambda_2 \tau(x_2)$ for all $\lambda_1, \lambda_2 \in \bC$ with $|\lambda_1| + |\lambda_2| \leq 1$; positive: $\tau(x^*x) \geq 0$; tracial: $\tau(xy) = \tau(yx)$; and normalized: $\tau(\bOne) = 1$.

\item[Metric] $d(x, y) = \tau(|x - y|)$.
\end{description}

Note that with these axioms, every atomic formula is equivalent to a formula of the form $\Re \tau(u(\bar x))$, where $u \in \cC_k$ is a contraction. We will also slightly abuse notation and use (complex-valued) formulas $\tau(u(\bar x))$ without the restriction $\nm{u} \leq 1$; they can simply be interpreted as $\nm{u} \tau\big(\frac{u}{\nm{u}}(\bar x)\big)$. With this convention, every (complex-valued) quantifier-free formula is equivalent to one of the form $\tau(u(\bar x))$ for some $u \in \cC_k$.

\begin{prop}
  \label{p:axioms-vN}
  The axioms above axiomatize the class of (unit balls of) tracial von Neumann algebras.
\end{prop}
\begin{proof}
  Let $(M, \tau)$ be a tracial von Neumann algebra. Then the algebraic axioms are satisfied (for all $\bar x$ in $M^1$) because of the universal property of $\cC_k$ and the fact that $M$ is a \Cstar-algebra. The trace axioms are satisfied because $\tau$ is a trace, and the metric axiom is simply the definition of $d$.

  Conversely, let $M^1$ be a model of the theory. For every $r > 0$, we let $M^r$ be a copy of $M^1$ and denote by $J_r \colon M^1 \to M^r$ the corresponding bijection. For $r < s$, we consider the map $I_{r,s}\colon M^r \to M^s$, $a \mapsto J_s((r/s)J_r^{-1}(a))$. We will identify $M^r$ (as a set) with its image in $M^s$, and we let $M = \bigcup_{r > 0} M^r$ be the direct limit of these inclusions. Note that $\bigcap_{r > 0} M^r = \set{\bZero}$.

  For $a \in M^r\sub M$ and a non-zero $\lambda \in \bC$, we define the scalar multiplication $\lambda a = J_{|\lambda|r}((\lambda/|\lambda|) J_r^{-1}(a)) \in M^{|\lambda|r}\sub M$. Notice that this does not depend on $r$. We also define $0a = \bZero$. Similarly, for $a,b\in M^r\sub M$ we define addition by $a+b = J_{2r}((1/2)J_r^{-1}(a) + (1/2)J_r^{-1}(b))$, multiplication by $ab = J_{r^2}(J_r^{-1}(a)J_r^{-1}(b))$, and the involution by $a^* = J_r(J_r^{-1}(a)^*)$. We also define the trace by $\tau(a) = r\tau(J_r^{-1}(a))$. Finally, for $a \in M$, we define $\nm{a} = \inf \set{r > 0 : a \in M^r}$.

  The algebraic axioms imply that $M$ is a $\ast$-algebra. Since for all $r, s \in \R^+$, $(r+s)^{-1}(r \scx_1 + s \scx_2)$ is a contraction in $\cC_2$, $M^1$ is closed under this polynomial, which implies that the norm satisfies the triangle inequality. Homogeneity of the norm is also clear.
  As $M^1$ is closed under multiplication, it follows from the definition that $\nm{ab} \leq \nm{a}\nm{b}$ for all $a, b \in M$, so $M$ is a normed algebra. To check the \Cstar-identity, let $a \in M$ with $\nm{a^*a} \leq 1$ in order to show that $\nm{a} \leq 1$. From the approximate polar decomposition in $\cC_1$ (see \cite[Prop.~II.3.2.1]{Blackadar2006}), there exists $u \in \cC_1$ with $\nm{u} \leq 1$ such that $\scx_1 = u(\scx_1^*\scx_1)^{1/4}$. Note that $(\scx_1^*\scx_1)^{1/8}$ is also a contraction in $\cC_1$, so $1 \geq \nm{(a^*aa^*a)^{1/8}} = \nm{(a^*a)^{1/4}}$. Thus we have that
  \begin{equation*}
    \nm{a} = \nm{u(a)(a^*a)^{1/4}} \leq \nm{u(a)}\nm{(a^*a)^{1/4}} \leq 1.
  \end{equation*}

  Let $A$ denote the completion of $M$ in the $\nm{\cdot}$ norm. Then $A$ is a \Cstar-algebra and the trace axioms imply that $\tau$ is a trace on $A$. The metric axiom implies that $\tau$ is faithful. Now \autoref{l:tracial-rep-Cstar} gives us a faithful representation $\lambda \colon A \to B(\cH_\tau)$. Let $M_\lambda$ be the von Neumann algebra generated by $\lambda(A)$. As $\lambda$ is isometric, we can just consider $A$ (and thus $M$) as a subalgebra of $M_\lambda$.

  Next we show that $M^1$ is closed in $M_\lambda^1$ in the strong operator topology of $B(\cH_\tau)$. Let $a_i \to a$ strongly with $a_i \in M^1$ and $a \in M_\lambda^1$. Then, in particular, $a_i \cdot \bOne \to a \cdot \bOne$ and thus $\nm{a_i - a}_2 \to 0$. By \autoref{l:vN-top-same}, $\nm{a_i - a}_1 \to 0$, so $(a_i)$ is $d$-Cauchy. By the completeness of $M^1$, we obtain that $a \in M^1$. The Kaplansky density theorem implies that $M^1$ is dense in $M_\lambda^1$, so $M = A = M_\lambda$ is a von Neumann algebra. The function $\tau$ is a trace on $M$ and it is weakly continuous on $M^1$ by \autoref{l:vN-top-same} and \cite[Thm.~II.2.6]{Takesaki1979}.

  We finally note that the interpretation of the function symbols $u \in \cC_k$ is the intended one. Fix $\bar a \in (M^1)^k$ and denote by $\Phi_{\bar a} \colon \cC_k \to A$ the homomorphism defined by $\Phi_{\bar a}(\bar \scx) = \bar a$. Let $u^A$ denote the interpretation of the symbol $u$ in $M^1$. Using scalar multiplication, we can define a function $u^A\colon M^1\to A$ for all $u \in \cC_k$, and not only for contractions.
  Our goal is to show that $\Phi_{\bar a}(u) = u^A(\bar a)$ for all $u$. This holds for $u = \scx_1 + \scx_2, \scx_1 \scx_2, \scx_1^*, \lambda \scx_1$ by definition. Now it follows from the algebraic axioms that it holds for all non-commutative $*$-polynomials. On the other hand, by definition, if $\nm{u} \leq 1$, then $\nm{u^A(\bar a)} \leq 1$. By applying this to $u/\nm{u}$, we obtain that $\nm{u^A(\bar a)} \leq \nm{u}$. So if $P$ is a polynomial in the variables $\bar \scx$ with $\nm{P - u} \leq \eps$, we have
  \begin{equation*}
    \begin{split}
      \nm{\Phi_{\bar a}(u) - u^A(\bar a)} &\leq \nm{\Phi_{\bar a}(u) - \Phi_{\bar a}(P)} + \nm{P^A(\bar a) - u^A(\bar a)} \\
                                          &= \nm{\Phi_{\bar a}(u-P)} + \nm{(P-u)^A(\bar a)} \leq 2 \eps
    \end{split}
  \end{equation*}
  and we are done.
\end{proof}

Next we show that the quantifier-free types of the theory $\TvN$ can be identified with the traces of the universal unital \Cstar-algebras $\cC_\kappa$. As each $\cC_n$ is separable, this implies, in particular, that $\TvN$ has a separable language.

\begin{prop}\label{p:vN-qftp}
  The following hold:
  \begin{enumerate}
  \item\label{i:p:vN:Sqf} For every cardinal $\kappa$, the map $\Phi \colon \tS_\kappa^\qf(\TvN) \to \cT(\cC_\kappa)$ given by
    \begin{equation*}
      \Phi(p(\bar x))(u) = \tau(u(\bar x))^p.
    \end{equation*}
    is an affine homeomorphism.
  \item \label{i:p:vN:models} Let $(M,\sigma)$ be a tracial von Neumann algebra and let $p\in \tS_\kappa^\qf(\TvN)$ be the quantifier-free type of an enumeration of the unit ball $M^1$ (or of a subset densely generating $M^1$). If $\tau=\Phi(p)$, then $(M,\sigma)$ is isomorphic to the von Neumann algebra $(M_\lambda, \tau)$ as constructed in \autoref{l:tracial-rep-Cstar}\ref{i:ltr:factor}, for $A=\cC_\kappa$.
    In particular, $M$ is a factor if and only if $p$ is an extreme type.
  \end{enumerate}
\end{prop}
\begin{proof}
  First, we observe that the trace axioms imply that $\Phi$ takes values in $\cT(\cC_\kappa)$. It is also obvious that it is continuous, affine, and injective. To see that it is surjective, let $\tau \in \cT(\cC_\kappa)$. Construct the Hilbert space $\cH_\tau$ and the tracial von Neumann algebra $(M_\lambda, \tau)$ as in \autoref{l:tracial-rep-Cstar}\autoref{i:ltr:factor}. It then follows from the definition that $\Phi(\tp^\qf(\lambda(\bar \scx/\cI_\tau))) = \tau$. This proves \autoref{i:p:vN:Sqf}.

  For \autoref{i:p:vN:models}, it is enough to observe that if $\bar a$ is the enumeration of $M$ for which $p=\tp^\aff(\bar a)$, then the map $\lambda(\bar\scx/\cI_\tau) \mapsto \bar a$ extends to an isomorphism ${M_\lambda \to M}$. Finally, as per \autoref{l:tracial-rep-Cstar}\autoref{i:ltr:factor}, the von Neumann algebra $M=M_\lambda$ is a factor if and only if $\tau$ is an extreme point of $\cT(\cC_\kappa)$, i.e., if and only if $p$ is an extreme type.
\end{proof}

\begin{prop}\label{p:factors-elementary-class}
  The class of finite von Neumann factors is axiomatizable in continuous logic.
\end{prop}
\begin{proof}
  See, e.g., \cite[Prop.~3.4]{Farah2014} or \cite[\textsection2]{Goldbring2023}.
\end{proof}

With the same strategy as in the proof of \autoref{th:PMPG-simplicial}, we now establish the main result of this section.

\begin{theorem}
  The following hold:
  \label{th:vN-simplicial}
  \begin{enumerate}
  \item\label{i:th:vN:factors} The extremal models of $\TvN$ are precisely the finite factors.
  \item\label{i:th:vN:Bauer} $\TvN$ is a Bauer theory.
  \end{enumerate}
\end{theorem}
\begin{proof}
  \autoref{i:th:vN:factors} Let $(M,\tau)$ be a finite factor, and let $q \in \tS^{\fM,\aff}_\kappa(\TvN)$ be the type of an enumeration of a dense subset of its unit ball. Then by \autoref{p:vN-qftp}\ref{i:p:vN:models}, the projection $\rho^\qf_\kappa(q)\in \tS_\kappa^\qf(\TvN)$ is an extreme type. By \autoref{lem:ExtremeTypeProjectionQF}, $q$ is an extreme type, i.e., $M^1$ is an extremal model of $\TvN$.

  Conversely, suppose $(M, \tau)$ is a tracial von Neumann algebra with a non-trivial center and let $z$ be a non-trivial central projection. Denote by $M_1$ and $M_2$ the algebras $zM$ and $(1-z)M$, with units $z$ and $1-z$, respectively, and traces appropriately normalized. Then $M_1, M_2$ are tracial von Neumann algebras, and as $\cL_\TvN$-structures their unit balls satisfy $M^1 \cong \tau(z) (M_1)^1 \oplus \tau(\bOne - z) (M_2)^1$. To see this, observe that if $a \in M^1$, then $za \in (M_1)^1$ and $(\bOne-z)a \in (M_2)^1$, and conversely, if $a_i \in (M_i)^1$, then $za_1 + (\bOne - z)a_2 \in M^1$. So $\tp^\aff(z) = \tau(z) \tp^{\aff,M_1} (\bOne) + \tau(\bOne - z) \tp^{\aff,M_2} (\bZero)$ is not extreme and thus $M$ is not an extremal model.

  \autoref{i:th:vN:Bauer} We first use \autoref{prop:SimplicialCriterionQF}. By \autoref{p:vN-qftp}\autoref{i:p:vN:Sqf} and \autoref{p:traces-simplex}, the type spaces $\tS_n^\qf(\TvN)$ are simplices. On the other hand, if $q\in\cE^\fM_\bN(\TvN)$, then, by \autoref{i:th:vN:factors}, the von Neumann algebra corresponding to the model enumerated by $q$ is a factor, and by \autoref{p:vN-qftp}\autoref{i:p:vN:models}, the quantifier-free type $\rho_\bN^\qf(q)$ is extreme. We conclude that $\TvN$ is simplicial. That it is a Bauer theory then follows from \autoref{p:factors-elementary-class} and \autoref{c:closed-extreme-types}.
\end{proof}

\begin{remark}
  \label{rem:CF_n-Poulsen}
  It is interesting to contrast \autoref{th:vN-simplicial} with the recent result by Orovitz, Slutsky, and Vigdorovich~\cite{Orovitz2023p} that $\cT(\mathrm{C}^\ast(\bF_n))$ (for $n \geq 2$) is a Poulsen simplex. We do not know whether the same is true for the simplices of the traces on the algebras $\cC_n$ (which are our quantifier-free types), but in any case, \autoref{th:vN-simplicial} says that the simplices of full types (with quantifiers) are Bauer.
\end{remark}

\begin{cor}
  \label{c:factors-affine-reduction}
  The continuous logic theory of finite von Neumann factors has affine reduction.
\end{cor}

Finally, as in the ergodic-theoretic case, we obtain a generalization of the classical integral decomposition theorem into factors, for arbitrary (i.e., not necessarily separable) tracial von Neumann algebras. Note that the classical result is not limited to the tracial case; see \cite[Thm.~8.21]{Takesaki1979}.

\begin{cor}
  \label{c:decomposition-vN}
  Every tracial von Neumann algebra decomposes as a direct integral of factors.
\end{cor}


\newpage
\bibliographystyle{affine-logic}

\addtocontents{toc}{\protect\vspace{10pt}}
\bibliography{affine-logic}
\end{document}